\documentclass[a4paper,oneside, 11pt]{amsart}
\usepackage{amssymb}
\usepackage{amsmath}
\usepackage{amsfonts}
\usepackage{amsthm}
\usepackage{subfigure}
\usepackage{epsfig}
\usepackage[all]{xy}
\usepackage[utf8]{inputenc}
\usepackage{mathabx}
\usepackage{graphicx}
\usepackage{xcolor}
\usepackage[T1]{fontenc}

\usepackage{imakeidx}
\usepackage[refpage]{nomencl}
\makeindex[columns=3]
\makenomenclature
\usepackage[pageanchor]{hyperref}

\DeclareFontFamily{T1}{calligra}{}
\DeclareFontShape{T1}{calligra}{m}{n}{<->s*[1.5]callig15}{}
\DeclareMathAlphabet\mathcalligra   {T1}{calligra} {m} {n}
\DeclareMathAlphabet\mathzapf       {T1}{pzc} {mb} {it}
\DeclareMathAlphabet\mathchorus     {T1}{qzc} {m} {n}
\DeclareMathAlphabet\mathrsfso      {U}{rsfso}{m}{n}
\usepackage{ulem}
\usepackage{comment}
\usepackage{stmaryrd}
\usepackage{enumitem}
\usepackage{dynkin-diagrams}
\usepackage{hyperref} 
\hypersetup{colorlinks=true,linkcolor=blue,citecolor=red,urlcolor=blue}

\newtheorem{theorem}{Theorem}[section] 
\newtheorem{lemma}[theorem]{Lemma}     
\newtheorem{corollary}[theorem]{Corollary}
\newtheorem{proposition}[theorem]{Proposition}

\theoremstyle{definition}
\newtheorem{example}[theorem]{Example}
\newtheorem{remark}[theorem]{Remark}
\newtheorem{notation}[theorem]{Notation}
\newtheorem{definition}[theorem]{Definition}

\title[Quotients of the Bruhat-Tits building by $\{\p\}$-arithmetic subgroups]{Arithmetic subgroups of Chevalley group schemes over function fields I: quotients of the Bruhat-Tits building by $\{\p\}$-arithmetic subgroups}

\newcommand{\Chi}{X}
\newcommand{\stab}{\operatorname{Stab}}
\newcommand{\fix}{\operatorname{Fix}}
\newcommand{\cl}{\operatorname{cl}}
\newcommand{\Sec}{\operatorname{Sec}}
\newcommand{\sSec}{\operatorname{sp-Sec}}

\newcommand{\spec}{\operatorname{Spec}}

\newcommand{\p}{P}

\begin{document}

\maketitle

\author{\textbf{Claudio Bravo}
\footnote{Universidad de Chile, Facultad de Ciencias, Casilla 653, Santiago, Chile. Email: \email{claudio.bravo.c@ug.uchile.cl}}
\footnote{Centre de Mathématiques Laurent Schwartz, École Polytechnique, Institut Polytechnique de Paris, 91128 Palaiseau Cedex, France. Email: \email{claudio.bravo-castillo@polytechnique.edu}}
}
\author{\textbf{Benoit Loisel}
\footnote{Université de Poitiers (Laboratoire de Mathématiques et Applications, UMR7348), Poitiers, France. Email: \email{benoit.loisel@math.univ-poitiers.fr}}
\footnote{ENS de Lyon, (Unité de Mathématiques Pures et Appliquées, UMR5669), Lyon, France}
}

\vspace{0.5cm}

\begin{abstract}
Let $\mathbf{G}$ be a reductive Chevalley group scheme (defined over $\mathbb{Z}$). Let $\mathcal{C}$ be a smooth, projective, geometrically integral curve over a field $\mathbb{F}$. Let $\p$ be a closed point on $\mathcal{C}$. Let $A$ be the ring of functions that are regular outside $\lbrace \p \rbrace$. The fraction field $k$ of $A$ has a discrete valuation $\nu=\nu_{\p}: k^{\times} \rightarrow \mathbb{Z}$ associated to $\p$.
In this work, we study the action of the group $ \textbf{G}(A)$ of $A$-points of $\mathbf{G}$ on the Bruhat-Tits building $\mathcal{X}=\mathcal{X}(\textbf{G},k,\nu_\p)$ in order to describe the structure of the orbit space $ \textbf{G}(A)\backslash \mathcal{X}$.
We obtain that this orbit space is the ``gluing'' of a closed connected CW-complex with some sector chambers.
The latter are parametrized by a set depending on the Picard group of $\mathcal{C} \smallsetminus \{\p\}$ and on the rank of $\mathbf{G}$.
Moreover, we observe that any rational sector face whose tip is a special vertex contains a subsector face that embeds into this orbit space. \\


\end{abstract}

\tableofcontents

\section{Introduction}\label{section intro}

In Lie theory, symmetric spaces are useful to study arithmetic groups via their action.
In order to study reductive groups over local or global fields, Tits introduced certain complexes, called buildings, that are analogous to the symmetric spaces associated to reductive groups over $\mathbb{R}$.
More precisely, to any Henselian discretely valued field $K$ and any reductive $K$-group $\mathbf{G}$, Bruhat and Tits associated a polysimplicial complex $\mathcal{X} = \mathcal{X}(\mathbf{G},K)$, called the Bruhat-Tits building of $\mathbf{G}$ over $K$ (c.f.~\cite{BT} and~\cite{BT2}). When $\mathbf{G}$ is a reductive group of semisimple rank $1$, for instance when $\mathbf{G}=\mathrm{SL}_2$, the associated building is actually a semi-homogeneous tree.
Buildings allow us to study subgroups of the group of $K$-points of a reductive group scheme.

Let $\mathcal{C}$ be a smooth, projective, geometrically integral curve over a field $\mathbb{F}$. The function field $k$ of $\mathcal{C}$ is a separable extension of $\mathbb{F}(x)$, where $x \in k$ is transcendental over $\mathbb{F}$.
Hence, it follows from \cite[I.1.5, p.6]{Stichtenoth} that the closure $\tilde{\mathbb{F}}$ of $\mathbb{F}$ in $k$ is a finite extension of $\mathbb{F}$.
In all that follows, without loss of generality, \emph{we assume that $\tilde{\mathbb{F}}=\mathbb{F}$}, i.e.~\emph{$\mathbb{F}$ is algebraically closed in $k$}.
Let $\p$ be a closed point on $\mathcal{C}$, and let $A$ be the ring of functions of $\mathcal{C}$ that are regular outside $\lbrace \p \rbrace$.
Let $\nu=\nu_{\p}: k^{\times} \rightarrow \mathbb{Z}$ be the discrete valuation defined from $\p$, and let us denote by $K=k_{\p}$ the completion of $k$ with respect to the valuation $\nu$.

In~\cite{S}, Serre describes the structure of the orbit space for the action of $\mathrm{SL}_2(A)$ on its Bruhat-Tits tree $\mathcal{X}(\mathrm{SL}_2,K)$.

\begin{theorem}\cite[Ch. II, \S 2, Th. 9]{S}\label{serre graph}
The quotient graph $\mathrm{SL}_2(A)\backslash \mathcal{X}(\mathrm{SL}_2,K)$ is the union of a connected graph of finite diameter, and a family of rays $\lbrace r(\sigma) \rbrace$ called cusps, such that:
\begin{itemize}
\item the vertex set of $\mathcal{Z} \cap r(\sigma)$ consists only in a single element and the edge set of $\mathcal{Z} \cap r(\sigma)$ is empty;
\item one has $r(\sigma) \cap r(\sigma')=\emptyset$ if $\sigma \neq \sigma'$.
\end{itemize}
Such a family of rays $r(\sigma)$ can be indexed by elements $\sigma$ in $\mathrm{Pic}(A)$.
\end{theorem}

Theorem \ref{serre graph} has some interesting consequences on the involved groups.
For instance, by using Theorem \ref{serre graph} and the Bass-Serre theory (c.f.~\cite[Ch. I, \S 5]{S}), Serre describes the structure of groups of the form $\mathrm{SL}_2(A)$ as an amalgamated product of simpler subgroups (c.f.~\cite[Ch. II, \S 2, Th. 10]{S}).
Then, Serre applies this description on groups in order to study its homology and cohomology groups with coefficient in certain modules (c.f.~\cite[Ch. II, \S 2.8]{S}). 
For instance, Serre obtains that $H_i(\mathrm{SL}_2(A), \mathbb{Q})=0$, for all $i >1$ and $H_1(\mathrm{SL}_2(A), \mathbb{Q})$ is a finite dimensional $\mathbb{Q}$-vector space.
Another interesting application of Theorem~\ref{serre graph} is the description of the conjugacy classes $\mathfrak{U}$ of maximal unipotent subgroups of certain subgroups $G$ of $\mathrm{SL}_2(A)$, as principal congruence subgroups.
In the same way, Serre describes the relative homology groups of $G$ modulo $\mathfrak{U}$ in terms of its Euler-Poincar\'e characteristic (c.f.~\cite[Ch. II, \S 2.9]{S}).


In a building of arbitrary rank, one can consider that a sector chamber (c.f.~\S \ref{intro vector faces}) has a role analogous as a ray in a rank $1$ building.
When $\mathcal{C}=\mathbb{P}^1_{\mathbb{F}}$, $\p=\infty$ and $\mathbf{G}$ is a split simply connected semisimple $k$-group, Soul\'e
describes, in \cite{So}, the topology and combinatorics of the corresponding orbit space as follows:

\begin{theorem}\cite[\S 1, Th. 1]{So}\label{soule quotient}
Let $\mathbf{G}$ be a split simply connected semisimple $k$-group.
Then $\mathbf{G}(\mathbb{F}[t])\backslash \mathcal{X}(\mathbf{G},\mathbb{F}(\!(t^{-1})\!))$ is isomorphic to a sector chamber of $\mathcal{X}(\mathbf{G},\mathbb{F}(\!(t^{-1})\!))$.
\end{theorem}

In the same article, Soul\'e describes $\mathbf{G}(\mathbb{F}[t])$ as an amalgamated sum of certain well-known subgroups of $\mathbf{G}(\mathbb{F}[t])$ (c.f.~\cite[\S 2, Th. 3]{So}).
Moreover, analyzing the preceding action, Soul\'e obtains some results on the homology groups $H_{\bullet}(\mathbf{G}(\mathbb{F}[t]), \ell)$, for some fields $\ell$ of suitable characteristic.
More specifically, Soul\'e obtains an homotopy-invariance property of the homology group of $\mathbf{G}(\mathbb{F})$ (c.f.~\cite[\S 3, Th. 5]{So}).
In \cite{Margaux}, Margaux extends this work of Soul\'e to the case of an isotrivial simply connected semisimple $k$-group, that is a group $\mathbf{G}$ splitting over an extension of the form $\ell = \mathbb{E}k$, where $\mathbb{E}/\mathbb{F}$ is a finite extension.

Let $\mathcal{S}$ be a finite set of closed places of $\mathcal{C}$.
Let us denote by $\mathcal{O}_{\mathcal{S}}$ the ring of functions that are regular outside $\mathcal{S}$. 
In particular, we have that $\mathcal{O}_{\lbrace \p\rbrace}=A$.
Let $\mathbf{G}$ be a $k$-isotropic and non-commutative algebraic group.
Choose a particular realization $\mathbf{G}_{\mathrm{real}}$ of $\mathbf{G}$ as an algebraic $k$-set of some affine space.
Given this realization, we define $G$ as the group of $\mathcal{O}_{\mathcal{S}}$-points of $\mathbf{G}_{\mathrm{real}}$.
The group $G$ is called an $\mathcal{S}$-arithmetic subgroup of $\mathbf{G}(k)$.
The $\mathcal{S}$-arithmetic group $G$ depends on the chosen realization of $\mathbf{G}$, and any two such choices lead to commensurable $\mathcal{S}$-arithmetic subgroups.
Let $\mathcal{X}(\mathbf{G},\mathcal{S})$ be finite product of buildings $\prod_{\p \in \mathcal{S}}\mathcal{X}(\mathbf{G},k_{\p})$.

A general result on rational cohomology, due to Harder in~\cite{H2}, is the following:
given a split simply connected semisimple group scheme $\mathbf{G}$, we have that $H^u(\mathbf{G}(\mathcal{O}_{\mathcal{S}}), \mathbb{Q})=0$ for $u \not\in\{0, \mathbf{t}s\}$, where $\mathbf{t}$ is the rank of $\mathbf{G}$ and $s=\mathrm{Card}(\mathcal{S})$.
Moreover, Harder also describes the dimension of the non-trivial cohomology groups in terms of representations of the group $\prod_{\p' \in \mathcal{S}} \mathbf{G}(k_{\p'})$. 
In the proof, one of the main arguments is
the interpretation of the action of $\mathbf{G}(\mathcal{O}_{\mathcal{S}})$ on $\mathcal{X}(\mathbf{G},\mathcal{S})$ in terms of the reduction theory, which mainly consists in the description of a fundamental domain for the action of $\mathbf{G}(k)$ on subgroups of the group of the adelic points $\mathbf{G}(\mathbb{A}_{\mathcal{S}})$ of $\mathbf{G}$.
More specifically, Harder uses this fact in order to describe a covering of the orbit space $\mathbf{G}(\mathcal{O}_{\mathcal{S}}) \backslash \mathcal{X}(\mathbf{G},\mathcal{S})$ in terms of some spaces indexed by $\mathbf{G}(\mathcal{O}_{\mathcal{S}})$-conjugacy classes of parabolic subgroups of $\mathbf{G}(k)$ (c.f.~\cite[Lemma 1.4.6]{H2}).


Heuristically, the translation of the Harder's reduction theory into the language of quotient of buildings proceeds via ``pretending'' that the building $\mathcal{X}(\mathbf{G}, \mathcal{S})$ can be identified with the orbit space $\mathbf{G}(\mathbb{A}_{\mathcal{S}})/\mathbf{G}(\mathcal{O}_{\mathcal{S}})$ (c.f.~\cite[\S 12]{B}).
Using this idea, Bux, Köhl and Witzel~\cite{B} study the action of $G$ on $\mathcal{X}(\mathbf{G},\mathcal{S})$.
They describe some finiteness properties of the orbit space by exhibiting the following cover of $\mathcal{X}(\mathbf{G},\mathcal{S})$:

\begin{theorem}\cite[Prop 13.6]{B}\label{witzel}
Assume that $\mathbb{F}$ is finite. Let $\mathbf{G}$ be a $k$-isotropic and non-commutative algebraic group. Let $G$ and $\mathcal{X}(\mathbf{G},\mathcal{S})$ as above. Then, there exists a constant $\kappa$ and finitely many sector chambers $Q_1, \cdots, Q_s$ of $\mathcal{X}(\mathbf{G},\mathcal{S})$ such that:
\begin{itemize}
\item the $G$-translates of the $\kappa$-neighborhood of $\bigcup_{i=1}^{s} Q_i$ cover $\mathcal{X}(\mathbf{G},\mathcal{S})$, and
\item for $i\neq j$, the $G$-orbits of $Q_i$ and $Q_j$ are disjoint.
\end{itemize} 
\end{theorem}

Bux, Köhl and Witzel in \cite[Prop 13.8]{B} use the preceding result in order to prove that $G$ is a lattice in the automorphism group of $\mathcal{X}(\mathbf{G}, \mathcal{S})$.

\bigskip

In the present work, one goal is to determine the quantity $s$ in Theorem~\ref{witzel} (see Th.~\ref{main theorem 3}).
Another objective is to understand what becomes the $\kappa$-neighborhood of a sector chamber $Q_i$ in the quotient of the building by $\mathbf{G}(A)$ (see Th.~\ref{main theorem 2}).


Roughly speaking, the first part of the article (cf.~\S~\ref{section commutative}-\S~\ref{section Stabilizer of points in the Borel variety}) is devoted to develop an alternative method than the one defined from reduction theory, in order to describe a fundamental domain for the action of an arithmetic subgroup.
This method focuses in the structure of the root group datum of $\mathbf{G}$. 
More specifically, in order to obtain it, we study in \S~\ref{section commutative} subsets $\Psi$ containing the highest root of $\mathbf{G}$.
We show that suitable subsets $\Psi$ define commutative unipotent subgroups $\mathbf{U}_{\Psi}$ with a linear action of a Borel of $\mathbf{G}$.
In \S~\ref{section Stabilizer of points in the Borel variety}, given an arbitrary Dedekind domain $A_0$, we bound the group of $A_0$-points of $\mathbf{U}_{\Psi}$ by a finite sum of fractional ideals of $A_0$.

In the second part of this work (cf.~\S~\ref{section structure of X}-\S~\ref{sec number cusp faces}), we apply the aforementioned general results on the integral points of unipotent subgroups to the case of $A_0= \mathcal{O}_{\lbrace \p \rbrace}$. 
More specifically, using a method analogous to that of the Mason's approach in \cite{M}, we prove that, under suitable assumptions, any (rational) sector face of $\mathcal{X}$ has a subset which injects in the quotient $\mathbf{G}(A) \backslash \mathcal{X}$.
Then, by controlling the neighborhood relation of sector chambers, we describe the orbit space $\mathbf{G}(A) \backslash \mathcal{X}$ in terms of a gluing of some ``sector chambers'' and a remaining complex, extending Theorem~\ref{serre graph} and Theorem~\ref{soule quotient} to the context of general Chevalley groups over general rings of the form $A=\mathcal{O}_{\lbrace \p \rbrace}$.
This description precises Theorem~\ref{witzel}, since we do not only obtain a covering of a fundamental domain, but also a description of such a fundamental domain.
We also get a precise control on the images in the orbit space of a particular set of sector chambers.
They constitute all the ``sector chambers'' that appear in the orbit space when $\mathbb{F}$ is finite.


Contrary to some works of Serre \cite{S} and Stuhler \cite{Stuhler}, this method does not rely on the language of vector bundles but on the language of Euclidean buildings and arithmetic groups.

\section{Main results}\label{main}

Let $\mathcal{C}$ be an arbitrary smooth, projective, geometrically integral curve defined over a field $\mathbb{F}$. As in \S~\ref{section intro}, let $\mathcal{O}_{\lbrace \p \rbrace}$ be the ring of functions of $\mathcal{C}$ that are regular outside a closed point $\p$.
For simplicity, we denote it by $A=\mathcal{O}_{\{\p\}}$.
Let $K$ be the completion of $k=\mathbb{F}(\mathcal{C})$ defined from the discrete valuation induced by $\p$.
Let $\mathbf{G}_k$ be an arbitrary split reductive linear algebraic group over $k$.
By splitness, $\mathbf{G}_k$ can be realized as the scalar extension to $k$ of a Chevalley group scheme $\mathbf{G}$ over $\mathbb{Z}$ (c.f.~\S~\ref{intro Chevalley groups}).

We denote by $\mathcal{X}$ the Euclidean Bruhat-Tits building of $\mathbf{G}$ over $K$.
It appears to be useful to not consider the complete system of apartments of the Bruhat-Tits building of $\mathbf{G}$ over $K$, but an incomplete subsystem of apartments arising from the field $k$ of rational functions over $\mathcal{C}$.
Let us denote by $\mathcal{X}_k$ the Bruhat-Tits building (c.f.~\S~\ref{Aff build}).
We call it the ``rational'' building of $\mathbf{G}$ over $k$.
It will be useful to work with certain conical cells of the building, called sector chambers whose faces are called sector faces (see \S \ref{intro sector faces}). More specifically, we will focus on the $\mathbf{G}(k)$-translates of a ``standard'' sector chamber. These $\mathbf{G}(k)$-translates are called the ``rational'' sector chambers (see \S \ref{intro sector faces}).

We consider the action of the $\{\p\}$-arithmetic subgroup $\mathbf{G}(A)$ on the rational building $\mathcal{X}_k$.
In this work, we want to get a better understanding of the geometry of the orbit space $\mathbf{G}(A) \backslash \mathcal{X}_k$. More precisely, we find suitable sector faces that embed in this quotient space. The proof along \S \ref{section structure of X}, consists in finding, successively, suitable subsector faces of a given sector face having suitable properties with respect to the action of $\mathbf{G}(A)$.
Among this section, we firstly obtain the following result that, heuristically, means that, far enough, the orbit space $\mathbf{G}(A)\backslash\mathcal{X}_k$ does not branch.

\begin{theorem}\label{main theorem 0}
Given a rational sector chamber $Q$, there exists a subsector chamber $Q'$ such that for any $y \in Q'$, the $1$-neighborhood of $y$ in $\mathcal{X}_k$ is covered by the $\mathbf{G}(A)$-orbits of the $1$-neighborhood of $y$ in $Q$.
\end{theorem}

For rational sector faces, a more technical result is given in Proposition~\ref{prop starAction}.
The main idea of the proof of Theorem \ref{main theorem 0} consists in finding unipotent elements provide enough foldings, at each $y \in Q'$, of the rational building onto the sector chamber $Q$.
These unipotent elements are obtained by applying the Riemann-Roch Theorem in the Dedekind ring $A$ using parametrization of the root subgroups of $\mathbf{G}$.
Choosing suitable subsector faces of the sector faces given by Theorem~\ref{main theorem 0}, it provides the central result:
\begin{theorem}\label{main theorem 1}
Let $Q$ be a rational sector face of $\mathcal{X}_k$.
Assume that one of the following cases is satisfied:
\begin{itemize}
    \item $\mathbf{G} = \mathrm{SL}_n$ or $\mathrm{GL}_n$, for some $n \in \mathbb{N}$;
    \item $Q$ is a sector chamber;
    \item $\mathbb{F}$ is a finite field.
\end{itemize}
Then, there exists a subsector face $Q'$ of $Q$ which is embedded in $\mathbf{G}(A) \backslash \mathcal{X}_k$.
\end{theorem}
A more technical statement is given in Corollary~\ref{cor conclusion without special hyp}.
Note that, in particular, Theorem~\ref{main theorem 1} describes the images in the quotient of the sector chambers introduced in Theorem~\ref{witzel} in the context where $\mathbb{F}$ is finite.

The proof of this theorem consists in finding a suitable invariant.
This invariant consists in associating to each special vertex $x$ a family of finite dimensional $\mathbb{F}$-vector spaces whose dimensions characterize the coordinates of $x$ in a given apartment.
In order to construct such vector spaces, we firstly prove that:
\begin{proposition}\label{main prop commutative}
Let $k$ be a field of characteristic different from $2$ and $3$.
Let $\mathbf{G}$ be a split reductive $k$-group scheme.
Let $\mathbf{B}$ be a Borel subgroup of $\mathbf{G}$ and let $\mathbf{U}$ be a closed commutative unipotent subgroup of $\mathbf{G}$ normalized by $\mathbf{B}$.
Then $\mathbf{U}(k)$ is a $k$-vector space and the action by conjugation of $\mathbf{B}(k)$ onto $\mathbf{U}(k)$ is $k$-linear.
\end{proposition}
For a more general statement that includes the cases of characteristic $2$ and $3$, see Proposition~\ref{prop suitable polynomials}.
It fact, suitable parabolic subgroups also acts linearly onto such commutative unipotent subgroups (see Corollary~\ref{cor k-linear action parabolic}).
Secondly, in \S \ref{section Stabilizer of points in the Borel variety}, we study the intersection with such well-chosen commutative unipotent subgroups of the pointwise stabilizer of sector faces, denoted by $M_\Psi(h)$, for some parameters $h$ and $\Psi$. Using the fact that $A$ is a Dedekind domain, one can bound the commutative groups $M_\Psi(h)$ between some suitable direct products of fractional $A$-ideals (see Proposition~\ref{prop ideal contained} and Proposition~\ref{prop fractional ideals}).
Using the bounds of these commutative groups, one can define finite dimensional vector spaces $V_\Psi(h)$ (see Lemma~\ref{lemma finite dimensional}).
The linearity, given by Proposition~\ref{main prop commutative}, of the action of a suitable parabolic subgroup associated to a sector face will ensure the invariance of the dimension the $\mathbb{F}$-vector space $V_\Psi(h)$.

In \S \ref{section structure orbit space}, describing the image in the quotient of the neighboring faces of vertices in the sector faces given by Theorem~\ref{main theorem 1}, we obtain the following theorem that describes the orbit space $\mathbf{G}(A) \backslash \mathcal{X}_k$. 

\begin{theorem}\label{main theorem 2} 
There exists a set of $k$-sector faces $\{Q_{i,\Theta}: \Theta \subseteq \Delta, i \in \mathrm{I}_\Theta\}$ of the Bruhat-Tits building $\mathcal{X}_k$, called cuspidal rational sector faces, parametrized by the set of types of sector faces $\Theta \subseteq \Delta$ and some sets $\mathrm{I}_\Theta$ such that:
\begin{itemize}
    \item the $\mathbf{G}(A)$-orbits of two different cuspidal rational sector chambers $Q_{i,\emptyset}$, for $i \in \mathrm{I}_\emptyset$, do not intersect;
    \item the orbit space $ \mathbf{G}(A) \backslash \mathcal{X}_k $ can be realized as a connected CW-complex obtained as the attaching space of CW-complexes isomorphic to the closure $\overline{Q_{i,\emptyset}}$ ($i \in \mathrm{I}_\emptyset$) of the cuspidal sector chambers and a closed connected CW-complex $\mathcal{Z}$ along some well-chosen subspaces;
    \item Moreover, when $\mathbf{G}$ is $\mathrm{SL}_n$ or $\mathrm{GL}_n$, or when $\mathbb{F}$ is finite, any cuspidal sector face (non necessarily a sector chamber) embeds in $\mathbf{G}(A) \backslash \mathcal{X}_k$, and any two of these cuspidal sector faces with same type $\Theta\subseteq \Delta$ have visual boundaries in the same $ \mathbf{G}(A)$-orbit whenever they have two points in the same $ \mathbf{G}(A)$-orbit.
\end{itemize}
\end{theorem}
Theorem~\ref{main theorem 2} generalizes Theorem~\ref{serre graph} to higher dimensions and Theorem~\ref{soule quotient} to an arbitrary Dedekind domain of the form $\mathcal{O}_{\lbrace \p \rbrace}$. 
This theorem summarizes some results of \S~\ref{section structure orbit space}. Indeed, Theorem~\ref{main teo 2 new} and Theorem~\ref{theorem quotient and cuspidal rational sector chambers} provide a more precise description of the orbit space.

There may be a connection of this quotient with moduli stack of vector bundles and their Hecke correspondence, when $\mathbf{G}=\mathrm{GL}_n$.

The proof of Theorem~\ref{main theorem 1} and Theorem~\ref{main theorem 2} are totally independent from Theorem~\ref{witzel}.
Moreover, Theorem~\ref{main theorem 1} precises Theorem~\ref{witzel} since it interprets its covering in the orbit space $\mathbf{G}(A) \backslash \mathcal{X}$.
We also precise the set of cuspidal rational sector chambers in the quotient through the following results.

\begin{theorem}\label{main theorem 3}
There exists a one-to-one correspondence between the set of cuspidal rational sector chambers and $\ker(H^1_{\text{\'et}}(\mathrm{Spec}(A), \mathbf{T})) \to H^1_\text{\'et}(\mathrm{Spec}(A), \mathbf{G}))$. Moreover, if $\mathbf{G}$ is semisimple and simply connected, the image of the preceding correspondence is isomorphic to $\mathrm{Pic}(A)^{\mathbf{t}}$, where $\mathbf{t}=\mathrm{rk}(\mathbf{G})$ is the dimension of $\mathbf{T}$. In particular, this implies that the number of sector chambers in Theorem~\ref{witzel}, for $\mathcal{S}=\lbrace \p \rbrace$ as above, is exactly $\mathrm{Card}\left( \mathrm{Pic}(A)^{\mathbf{t}}\right)$.
\end{theorem}


An application of Theorem \ref{main theorem 1} together with further investigations provides a characterization of the conjugacy classes of maximal unipotent subgroups in the $\mathcal{S}$-arithmetic group $\mathbf{G}(\mathcal{O}_{\mathcal{S}})$, and in its finite index subgroups (work in progress generalizing \cite[\S~10]{BravoLoisel}).
In particular, for any finite subset of places $\mathcal{S}$, the number of conjugacy classes of maximal unipotent subgroups of the $\mathcal{S}$-arithmetic subgroup $\mathbf{G}(\mathcal{O}_\mathcal{S})$ will be given in a further work thanks to an analogous Theorem to that of Theorem~\ref{main theorem 3}.

By using Theorem \ref{main theorem 2} and the theory of small categories without loops, which generalizes Bass-Serre theory in higher dimension, a description of the structure of the considered arithmetic groups as amalgamated sums of groups is given in \cite[\S~11]{BravoLoisel}. This is a work in progress.
These results have natural consequences on the (co)homology of these groups, as in the work of
Soul\'e~\cite[\S 3, Th. 5]{So}, Harder~\cite{H2} and Serre~\cite[Ch. II, \S 2.8]{S}.


\section{Notation on reductive group schemes and their Bruhat-Tits buildings}\label{build}

 Let $\mathcal{C}$ be a smooth, projective, geometrically integral curve defined over a field $\mathbb{F}$. Let $k$ be the function field of $\mathcal{C}$ and let $K=k_{\p}$ be its completion with respect to the discrete valuation map $\nu_{\p}$ defined from a closed point $\p$ of $\mathcal{C}$. We denote by $\mathcal{O}$ the ring of integers of $K$ and we fix a uniformizer $\pi \in \mathcal{O} \cap k$.  

In this section, we introduce and recall some classical definitions and notations used along the paper. 
We denote respectively by $\mathbb{G}_a$ and $\mathbb{G}_m$ the additive and the multiplicative group scheme.

Along the paper, we will have to use some canonical embeddings of Bruhat-Tits buildings, that are functorial with respect to field extensions or group embeddings. These functorial properties are introduced in \S~\ref{intro building embedding}. Thus, we recall, in \S~\ref{Aff build}, some elements of construction of the affine Bruhat-Tits building $\mathcal{X}(\mathbf{G},\ell)$ associated to the datum of any discretely valued field $(\ell,\nu)$ and any reductive split $\ell$-group $\mathbf{G}$.
Typically, $\ell/k$ will be a univalent extension such as $\ell = k$, $\ell = K = k_{\p}$ or an algebraic extension of degree $[\ell:k] < \infty$.
To take in account the field extensions, we introduce some definitions on rational combinatorial structures in \S~\ref{intro rational building}.

The visual boundary of a building will be an important tool to understand the quotient space $\mathbf{G}(A) \backslash \mathcal{X}(\mathbf{G},K)$. We introduce it in \S~\ref{intro visual boundary}. It is naturally linked with the spherical building structure deduced from the Borel-Tits structure results~\cite{BoTi}, introducing faces called vector faces whose we recall the definition in~\ref{intro vector faces}.

\subsection{Notation on Weyl Groups}\label{Weyl groups}
Along this paper we adopt the following notation on Weyl groups.
A character of $\mathbf{G}$ is an element of $\Chi(\mathbf{G})=\operatorname{Hom}_{k-\mathrm{group}}(\mathbf{G}, \mathbb{G}_m)$.
Choose an arbitrary maximal $k$-torus $\textbf{T}$ of $\textbf{G}$. Since $\mathbf{G}$ is $k$-split, $\mathbf{T}$ isomorphic to $\mathbb{G}_m^{\mathbf{t}}$ for some $\mathbf{t} \in \mathbb{Z}_{\geqslant 0}$.
Let $\Phi=\Phi(\mathbf{G}, \mathbf{T}) \subset \Chi(\mathbf{T})$ be the set of roots of $\textbf{G}$ with respect to $\textbf{T}$.

Let $\mathbf{B}$ be a Borel subgroup containing $\mathbf{T}$.
Since $\mathbf{B}$ is a solvable group, the set of unipotent elements $\mathbf{U}$ of $\textbf{B}$ is a subgroup.
Moreover, since $\mathbf{G}$ is a split $k$-group, there exists a $k$-isomorphism between $\textbf{B}$ and the semi-direct product $\textbf{T}\ltimes \textbf{U}$ \cite[10.6(4)]{BoA}.
From a Borel subgroup, we define a subset of positive roots $\Phi^+ = \Phi(\mathbf{T},\mathbf{B})$ as in \cite[20.4]{BoA}.
This induces a basis of simple roots $\Delta = \Delta(\mathbf{B})$ \cite[VI.1.6]{Bourbaki} of $\mathbf{G}$ relatively to the Borel $k$-subgroup $\mathbf{B}$.

For any $\alpha \in \Phi= \Phi(\mathbf{G}, \mathbf{T})$, let $\mathbf{U}_{\alpha}$ be the $\mathbf{T}$-stable unipotent subgroup of $\mathbf{G}$ defined from this.
Since $\mathbf{G}$ splits over $k$, every unipotent group $\mathbf{U}_{\alpha}$ is isomorphic to $\mathbb{G}_{a}$ via a $k$-isomorphism $\theta_{\alpha}: \mathbb{G}_{a} \rightarrow U_{\alpha}$ called an \'epinglage \cite[18.6]{BoA}.
We define $\mathbf{U}^{+}= \langle \mathbf{U}_{\alpha}: \alpha \in \Phi^{+} \rangle$ (resp. $\mathbf{U}^{-}= \langle \mathbf{U}_{\alpha}: \alpha \in \Phi^{-} \rangle$). By construction, $\mathbf{U}^+$ is the unipotent radical of $\mathbf{B}$.

Let $\mathbf{N}=\mathcal{N}_{\mathbf{G}}(\mathbf{T})$ be the normalizer of $\textbf{T}$ in $\mathbf{G}$.
We denote by $W^{\mathrm{sph}}$ the finite algebraic quotient $\mathbf{N}/\mathbf{T}$.
Since $\mathbf{G}$ is assumed to be $k$-split, $W^{\mathrm{sph}}$ is a constant $k$-group so that, for any extension $\ell/k$, we identify $W^{\mathrm{sph}}$ with the group of rational points $\mathbf{N}(\ell)/\mathbf{T}(\ell)$ \cite[21.4]{BoA}. It identifies with the Weyl group of the root system $\Phi=\Phi(\mathbf{G},\mathbf{T})$ via an action $\mathfrak{w}^{\mathrm{sph}}: \mathbf{N}(k) \to W(\Phi)$ satisfying $\mathbf{U}_{\mathfrak{w}^{\mathrm{sph}}(n)(\alpha)} = n \mathbf{U}_{\alpha} n^{-1}$ for any $\alpha \in \Phi(\mathbf{G},\mathbf{T})$ and any $n \in \mathbf{N}(k)$ \cite[14.7]{BoA}.
Moreover, the group $W^{\mathrm{sph}}$ is the group generated by the set of reflections $r_{\alpha}$ induced by $\alpha \in \Phi$ (c.f.~\cite[14.7]{BoA}).

Let $\Chi_*(\mathbf{T})= \operatorname{Hom}_{k-\text{gr}}(\mathbb{G}_m,\mathbf{T})$ be the $\mathbb{Z}$-module of cocharacter of $\mathbf{T}$.
Recall that there is a perfect dual pairing $\langle \cdot,\cdot \rangle: \Chi_*(\mathbf{T}) \times \Chi(\mathbf{T}) \to \mathbb{Z}$ \cite[8.11]{BoA}.
This pairing extends naturally over $\mathbb{R}$ into a perfect dual pairing between the finite dimensional $\mathbb{R}$-vector space $V_1 = \Chi_*(\mathbf{T}) \otimes_{\mathbb{Z}} \mathbb{R}$ and its dual $V_1^* = \Chi(\mathbf{T}) \otimes_{\mathbb{Z}} \mathbb{R}$.
Using the valuation $\nu=\nu_\p$, we have a group homomorphism $\rho: \mathbf{T}(K) \to V_1$ \cite[4.2.3]{BT2} given by 
\[
 \langle \rho(t), \chi \rangle = - \nu\big(\chi(t)\big),\qquad \forall t\in \mathbf{T}(K),\ \forall \chi \in \Chi(\mathbf{T}).
\]
We denote by $\mathbf{T}_b(K)$ its kernel.
Let $V_0=\left\{v \in V_1,\ \langle v,\alpha \rangle =0\ \forall \alpha \in \Phi\right\}$ and $V=V_1/V_0$.
Hence, for any root $\alpha \in \Phi$, the map $\langle \cdot,\alpha \rangle : V_1 \to \mathbb{R}$ induces a linear map $V \to \mathbb{R}$ and we denote by $\alpha(v)$ the value of this map at $v \in V$.
Let $\mathbb{A}$ be any $\mathbb{R}$-affine space over $V$.
One can define an action by affine transformations $\mathfrak{w}: \mathbf{N}(K) \to V \rtimes \mathrm{GL}(V) = \operatorname{Aff}(\mathbb{A})$ extending $\mathfrak{w}^{\mathrm{sph}}$ as in \cite[1.6]{L}.
The affine Weyl group $W^{\mathrm{aff}}$ is the image of $\mathfrak{w}$.
Since the kernel of $\mathfrak{w}$ is $\mathbf{T}_b(K)$, the action $\mathfrak{w}$ induces an isomorphism $\mathbf{N}(K) / \mathbf{T}_b(K) \simeq W^{\mathrm{aff}}$.
This group decomposes as $W^{\mathrm{aff}} =\mathbf{\Lambda} \ltimes W^{\mathrm{sph}}$, where $\mathbf{\Lambda}$ is a free abelian group.

\subsection{Chevalley group schemes over \texorpdfstring{$\mathbb{Z}$}{Z}}\label{intro Chevalley groups}

Let $S = \spec(\mathbb{Z})$.
According to \cite[Exp.~XXV, Cor.~1.3]{SGA3-3}, there exists a reductive $S$-group scheme $\mathbf{G}_\mathbb{Z}$ that is a $\mathbb{Z}$-form of $\mathbf{G}$, i.e.~$\mathbf{G}_\mathbb{Z} \otimes_S k = \mathbf{G}$.
There exists a maximal torus $\mathbf{T}_\mathbb{Z} \cong D_S(M)$ defined over $\mathbb{Z}$ (c.f.~\cite[Exp.~XXV, 2.2]{SGA3-3}). It follows from \cite[Exp.~IX, Def. 1.3]{SGA3-1} that, in fact, we have $\mathbf{T}_\mathbb{Z} \cong D_S(M) \cong \mathbb{G}_{m,\mathbb{Z}}^{\mathbf{t}}$ by local considerations. So, for $\mathbf{G}$, we pick the scalar extension $\mathbf{T} = \mathbf{T}_\mathbb{Z} \otimes_S k$ as maximal split torus.
With respect to this torus, the root groups $\mathbf{U}_\alpha$ for $\alpha \in \Phi(\mathbf{G},\mathbf{T})$ are defined over $\mathbb{Z}$ and there exists a Chevalley system (i.e.~a family of \'epinglages that are compatible to each other with respect to the group structure of $\mathbf{G}$) such that the \'epinglages $\theta_\alpha : \mathbb{G}_{a,S} \to \mathbf{U}_\alpha$ for $\alpha \in \Phi(\mathbf{G},\mathbf{T})$ are defined over $\mathbb{Z}$ \cite[Exp.~XXV, 2.6]{SGA3-3}.
There exists an injective homomorphism of $S$-groups $\rho: \mathbf{G}_{\mathbb{Z}} \rightarrow \mathrm{SL}_{n,\mathbb{Z}}$ (c.f.~\cite[9.1.19(c)]{BT}).

The Chevalley group scheme $\mathbf{G}_\mathbb{Z}$ provides, for any ring $R$, an abstract group $\mathbf{G}_\mathbb{Z}(R)$ that we denote $\mathbf{G}(R)$ for simplicity.

\begin{remark}
If $\mathbf{G}$ is semisimple, then $V_0 = 0$ according to \cite[8.1.8(ii)]{Springer}.
Moreover, we know that $\mathbf{T}_\mathbb{Z} \cong \big( \mathbb{G}_{m,\spec \mathbb{Z}} \big)^\mathfrak{t}$ by \cite[Exp.~XXII, 4.3.8]{SGA3-3} whenever $\mathbf{G}$ is simply connected, where $\mathfrak{t}$ denotes the rank of $\mathbf{G}$.
In this case, one can easily check from the definition and the dual pairing that $\mathbf{T}_b(K) = \mathbf{T}(\mathcal{O})$.
\end{remark}

Let $\mathbb{A}$ be the vector space $V$ seen as an affine space.
Consider the ground field $\mathbb{F} \subset A$ and let $\mathbf{N}_\mathbb{F}$ be the normalizer of $\mathbf{T}_\mathbb{F}= \mathbf{T}_\mathbb{Z} \otimes \mathbb{F}$ in $\mathbf{G}_\mathbb{F}= \mathbf{G}_\mathbb{Z} \otimes \mathbb{F}$.
For $\alpha \in \Phi$, define $m_\alpha = \theta_\alpha(1) \theta_{-\alpha}(-1) \theta_\alpha(1) \in \mathbf{N}(\mathbb{F})$.
Let $\mathbf{N}(\mathbb{F})^{\mathrm{sph}}$ be the abstract subgroup of $\mathbf{N}(\mathbb{F})$ generated by the $m_\alpha$.
Let $T_1 \cong (\mu_2(\mathbb{F}))^\mathbf{t}$ be the finite subgroup of $\mathfrak{t}$-copies of $\mu_1(\mathbb{F})$ contained in $\mathbf{T}(\mathbb{F}) \cong \big(\mathbb{G}_m(\mathbb{F})\big)^\mathbf{t}$.
We have that $\mathfrak{w}^{\mathrm{sph}}(m_\alpha) = r_\alpha$ for $\alpha \in \Phi$.
Moreover, for any $\alpha, \beta \in \Phi$, there exists a sign $\varepsilon \in \{\pm 1\}$ such that $m_\alpha m_\beta m_\alpha = m_{r_\alpha(\beta)} \big( r_\alpha(\beta)\big)^\vee(\varepsilon)$ according to \cite[Exp.~XXIII, 6.1]{SGA3-3}, where $\big( r_\alpha(\beta)\big)^\vee \in \Chi_*(\mathbf{T})$ denotes the coroot associated to $r_\alpha(\beta)$.
Hence the natural group homomorphism $\mathbf{N}(\mathbb{F})^{\mathrm{sph}} \to W^{\mathrm{sph}}$ is surjective with kernel contained in $T_1$.
In particular, the group $\mathbf{N}(\mathbf{F})^{\mathrm{sph}}$ is finite. For each $w \in W^{\mathrm{sph}}$, we denote by $n_w$ a chosen representative of $w$ in $\mathbf{N}(\mathbb{F})^{\mathrm{sph}}$ and, for each for $\alpha \in \Phi$, we denote by $n_\alpha = n_{r_\alpha}$. 
We denote this subset of representatives by $N^{\mathrm{sph}} = \lbrace n_w : w \in W^{\mathrm{sph}} \rbrace \subset \mathbf{G}(\mathbb{F})$.

\subsection{Affine Bruhat-Tits buildings of split reductive groups}\label{Aff build}

Let $(\ell,\nu)$ be a discretely valued field and $\mathbf{G}$ be a split reductive $\ell$-group.
We denote by $\mathcal{X}_\ell = \mathcal{X}(\mathbf{G},\ell)$ its Bruhat-Tits building.
Formally, it is a set $\mathcal{X}$ of points together with a collection $\mathcal{A}_\ell$ of subsets $\mathbb{A} \subset \mathcal{X}$ called apartments. Any apartment is an $\mathbb {R}$-affine space on a $\mathbb{R}$-vector space isomorphic to $V$ that is tessellated with respect to the reflections of the group $W^{\mathrm{aff}}$ (c.f.~\cite[\S 1.4]{Brown}).
Inside an apartment, the hyperplane of fixed points of a reflection in $W^{\mathrm{aff}}$ is called a wall.
A polysimplex (resp. maximal, resp. minimal) of this tessellation is called a face (resp. chamber, resp. vertex).

Let $\mathbb{A}_0$ be an affine space over $V$, together with an action of $\mathfrak{w}:\mathbf{N}(\ell)\to W^\mathrm{aff}$ on it, as defined in \S~\ref{Weyl groups}.
Let $v_0 \in \mathbb{A}_0$ be the vertex fixed by the subgroup $0 \rtimes W^{\mathrm{sph}}$ of $W^{\mathrm{aff}}$.
We realize roots $\alpha \in \Phi$ as linear maps by setting $\alpha(x) := \alpha(x-v_0)$ for any point $x \in \mathbb{A}_0$.
For any point $x \in \mathbb{A}_0$
and any root $\alpha \in \Phi$, we apply \cite[\S 7]{BT} to the valuation of the root group datum defined in \cite[4.2.3]{BT2}, in order to define 
\[\mathbf{U}_{\alpha,x}(\ell)= \left\lbrace  u \in \mathbf{U}_{\alpha}(\ell): \nu\left( \theta^{-1}_{\alpha}(u) \right) \geq -\alpha(x) \right\rbrace.\]
For any non-empty subset $\Omega\subset \mathbb{A}_0$ we can write $\mathbf{U}_{\alpha, \Omega}(\ell)=\bigcap_{x \in \Omega }\mathbf{U}_{\alpha,x}(\ell)$ (see \cite[7.1.1]{BT}). We define the unipotent group of $\Omega$ over $\ell$ as
\footnote{Note that, despite the notations, this does not define an algebraic group $\mathbf{U}_{\alpha,\Omega}$, nor $\mathbf{U}_{\Omega}$, but only a collection of topological groups.} \[\mathbf{U}_{\Omega}(\ell) = \left\langle \mathbf{U}_{\alpha, \Omega}(\ell): \alpha \in \Phi \right\rangle.\]

The Bruhat-Tits building of $(\mathbf{G},\ell)$ is the cellular complex defined by gluing of multiple copies of $\mathbb{A}$:
\[
\mathcal{X}(\textbf{G},\ell)= \textbf{G}(\ell) \times \mathbb{A}/\backsim,
\]
with respect to the equivalence relation $(g,x) \backsim (h,y)$ if and only if there is $n \in \mathbf{N}(\ell)$ such that $y=\mathfrak{w}(n)(x)$ and $g^{-1}hn \in \mathbf{U}_{x}(\ell)$ as defined in \cite[7.1.2]{BT}.
For any $g \in \mathbf{G}(\ell)$, the map $\iota_g: \mathbb{A}_0 \to \mathcal{X}(\mathbf{G},\ell)$ given by $x \mapsto (g,x)/\backsim$ is injective and we identify $\mathbb{A}_0$ with its image via $\iota_e$ where $e$ denotes the identity element of $\mathbf{G}(\ell)$ as in \cite[7.1.2]{BT}.
The apartment $\mathbb{A}_0 \subset \mathcal{X}$ is called the standard apartment associated to $\mathbf{T}$.
Thus, we formally see $g \cdot \mathbb{A}$ as an affine space over the vector space $g \cdot V$ corresponding to the torus $g \mathbf{T} g^{-1}$.

Note that the $\mathbb{R}$-affine space $\mathbb{A}_0$ only depends on $\mathbf{T}$ but its tessellation depends on the valued field $(\ell,\nu)$ and the $\ell$-group $\mathbf{G}$.

\subsection{Combinatorial structure and rationality questions in buildings}\label{intro rational building}


All the previous definitions and characterizations hold, in particular, when $\ell$ is the global field $k$, its completion $K$ with respect to the valuation $\nu = \nu_\p$ or a suitable extension of $k$.

Assume that $\ell/k$ is a univalent extension (see definition in \cite[1.6]{BT2}) and that $\mathbf{G}$ is a split reductive $k$-group. We build such an extension in Lemma~\ref{lem curve extension}. Then $\mathbf{G}_\ell := \mathbf{G} \otimes_k \ell$ is a split reductive $\ell$-group.
There is a natural building embedding $\iota: \mathcal{X}(\mathbf{G},k) \to \mathcal{X}(\mathbf{G},\ell)$ given by \cite[9.1.19 (a)]{BT}.

The collection $\mathcal{A}_\ell$ of $\ell$-apartments is given by the subsets $g \cdot \mathbb{A}_0$, for $g \in \mathbf{G}(\ell)$.
For instance, the buildings $\mathcal{X}_k = \mathcal{X}(\mathbf{G},k)$ and $\mathcal{X}_K = \mathcal{X}(\mathbf{G},K)$ have the same polysimplicial structure but not the same collections $\mathcal{A}_k$ and $\mathcal{A}_K$ of apartments (we detail this in Lemma~\ref{lema1}).
In other words, $\mathcal{X}(\mathbf{G},K)$ is the building $\mathcal{X}_k$ equipped with the complete apartment system \cite[2.3.7]{Rousseau77}.
 In fact, by rational conjugacy of maximal split tori, there is a natural one-to-one correspondence between the set of $\ell$-apartment and the set of maximal $\ell$-split tori.

In this context, a maximal $k$-split torus $\mathbf{T}$ induces a maximal $\ell$-split torus $\mathbf{T}_\ell := \mathbf{T} \otimes_k \ell$ of $\mathbf{G}_\ell$ and, therefore, the apartment $\mathbb{A}_0$ of $\mathcal{X}(\mathbf{G},k)$ associated to $\mathbf{T}$ naturally identifies via $\iota(\mathbb{A}_0) \cong \mathbb{A}_0$ with the apartment of $\mathcal{X}(\mathbf{G},\ell)$ associated to $\mathbf{T}_\ell$.
Thus any root $\alpha \in \Phi$ induces a linear map on $\mathbb{A}_0$.

We call an $\ell$-wall of $\mathbb{A}_0$, an hyperplane of the form $H_{\alpha,r} = \{x \in \mathbb{A}_0,\ \alpha(x) = r\}$ for $r \in \nu(\ell^\times)$.
The $\ell$-walls provides the tessellation of the apartment of $\mathcal{X}(\mathbf{G}_\ell,\ell)$ according to \cite[6.2.22, 9.1.19(b)]{BT}.
For $x \in \mathbb{A}_0$, we denote
$$\Phi_{x,\ell} = \{ \alpha \in \Phi,\ \exists r \in \nu(\ell^\times),\ H_{\alpha,r} \text{ is an } \ell\text{-wall}\}.$$
This is the subroot system\footnote{More generally, $\Phi_x$ is the set of roots $\alpha \in \Phi$ such that $\alpha(x) \in \Gamma'_\alpha$ where $\Gamma'_\alpha$ denotes the set of values defining the walls directed by the root $\alpha$. Here $\Gamma'_\alpha = \nu(\ell^\times)$ because we assume that $\mathbf{G}$ is split.} of $\Phi$ \cite[6.4.10]{BT} of the local spherical building at $x$.

We call an $\ell$-half apartment a half-space of $\mathbb{A}_0$ of the form $D_{\alpha,r} = \{x \in \mathbb{A}_0,\ \alpha(x) \geqslant r\}$ for $r \in \nu(\ell^\times)$.
A subset $\Omega \subseteq \mathbb{A}_0$ is said to be $\ell$-enclosed if it is the intersection of $\ell$-half apartment.
The $\ell$-enclosure of a subset $\Omega \subseteq \mathbb{A}_0$ is the intersection of $\ell$-half apartments that contain $\Omega$, as defined in \cite[7.4.11]{BT}.
We denote it by $\cl_\ell(\Omega)$ the $\ell$-enclosure of $\Omega$.

The set of $\ell$-half apartments $\Sigma$ induces an equivalence relation on the set of points of $\mathbb{A}_0$ given by
$$x \sim y \Longleftrightarrow \cl_\ell(\{x\}) = \cl_\ell(\{y\}). $$
The equivalence classes are called the $\ell$-faces.
A $\ell$-face with maximal dimension is called an $\ell$-chamber.
A $\ell$-face with codimension $1$ is called an $\ell$-panel.
An $\ell$-face with minimal dimension is a single point. It is called an $\ell$-vertex.
A vertex $x$ is called $\ell$-special if $\Phi_{x,\ell}= \Phi$.
The $\ell$-enclosure of an $\ell$-chamber is called an $\ell$-alcove.

If $\mathbb{A} = g \cdot \mathbb{A}_0$ is another $\ell$-apartment and $\Omega \subset \mathbb{A}$, we say that $\Omega$ is an $\ell$-vertex (resp. panel, alcove, wall, half-apartment) if so is $g^{-1} \cdot \Omega \subset \mathbb{A}_0$.
The enclosure of $\Omega$ is then $g \cdot \cl( g^{-1} \cdot \Omega)$ that is the intersection of $\ell$-half apartments of $\mathbb{A}$.
For simplicity, if $\ell=k$, we omit the ``$k$-'' in the following and we will say vertex, panel, alcove, wall, half-apartment, apartment instead of $k$-vertex, $k$-panel, etc.

Note that, if $\Omega$ is a subset of some apartment $\mathbb{A}$, then the enclosure $\cl(\Omega)$ is closed for the real topology and that, if $\Omega$ is a face of $\mathcal{X}$, then the enclosure $\cl(\Omega)$ is equal to the closure $\overline{\Omega}$ of $\Omega$ in $\mathcal{X}$ for the real topology.

\subsection{Spherical building: vector faces and sector faces}\label{intro spherical building}

\subsubsection{Spherical building structure}\label{intro vector building}

For $\alpha \in \Phi$, define the closed half-space $D_\alpha:=\alpha^{-1}(\mathbb{R}_+) \subset V$.
There is an equivalence relation on $V$ defined by \[x \sim y \Longleftrightarrow \left( \forall \alpha \in \Phi,\ x \in D_\alpha \Leftrightarrow y \in D_\alpha \right).\]
A vector face is an equivalence class for this relation.
A vector chamber $D$ of $V$ is a connected component of the space $V \smallsetminus \bigcup_{\alpha \in \Phi} \ker(r_{\alpha})$.
The action $\mathfrak{w}^{\mathrm{sph}}$ of the spherical Weyl group $W^{\mathrm{sph}}$ on $V$ is transitive on the set of vector chambers \cite[1.69]{Brown}.
We say that $D$ is a vector face of the apartment $g \cdot V$ if $g^{-1} \cdot D$ is a vector face of $V$.
From the BN-pair $\big(\mathbf{B}(\ell), \mathbf{N}(\ell)\big)$, one can define a (vectorial) building of $(\mathbf{G},\ell)$ (see~\cite[Cor. 10.6]{RouEuclidean}) whose faces are the vector faces on which $\mathbf{G}(k)$ acts strongly transitively.

\subsubsection{Standard vector faces}\label{intro vector faces}

Let $\Delta=\Delta(\mathbf{B})$ be the basis of the root system $\Phi=\Phi(\mathbf{G},\mathbf{T})$ associated to the Borel subgroup $\mathbf{B}$, as in section~\ref{Weyl groups}.

We define a standard vector chamber $D_0 \subset V$ associated to $(\mathbf{T},\mathbf{B})$ by
\[D_0= \{ x \in V,\ \alpha(x) > 0, \forall \alpha \in \Delta \}.\]

We denote by $W_\Theta$ the subgroup of the Weyl group $W(\Phi) = W^{\mathrm{sph}}$ generated by the reflections $r_\alpha$ for $\alpha \in \Theta$.
For any $\Theta \subset \Delta$, we denote by $\Theta^\perp = \{ x \in V,\ \alpha(x) = 0\ \forall \alpha \in \Theta \}.$
Note that $\Theta^\perp$ is the subspace of invariants of $V$ with respect to the group $W_\Theta$, i.e.~$\Theta^\perp = V^{W_\Theta} = \{x \in V,\ w(x) = x,\ \forall w \in W_\Theta\}$.
We define
\[\Phi_{\Theta}^0 = \{\alpha \in \Phi,\ \alpha \in \operatorname{Vect}_\mathbb{R}(\Theta)\}\]
the subset of $\Phi$ consisting in linear combination of elements in $\Theta$.
According to~\cite[VI.1.7, Cor.4 of Prop.20]{Bourbaki}, it is a root system with basis $\Theta$ and, by definition, its Weyl group\footnote{Because $\mathbf{G}$ is assumed to be split, $\Phi$ is reduced. But we could extend this definition to an arbitrary non reduced root system $\Phi$ by considering the sub-root system of non-divisible roots, which the same Weyl group \cite[VI.1.4~Prop.13(i)]{Bourbaki}.} is $W_\Theta$.

We denote by $\Phi_\Theta^+ = \Phi^+ \smallsetminus \Phi_{\Theta}^0$ and $\Phi_\Theta^- = \Phi^- \smallsetminus \Phi_{\Theta}^0$ so that $\Phi = \Phi_\Theta^+ \sqcup \Phi_\Theta^0 \sqcup \Phi_\Theta^-$.
We define the standard vector face of type $\Theta \subset \Delta$ by:
\begin{align*}
    D_0^\Theta &=\{x \in V,\ \alpha(x) > 0\ \forall \alpha \in \Delta \smallsetminus \Theta \text{ and } \alpha(x)=0\ \forall \alpha \in \Theta\}\\
    &= \{ x \in V,\ \alpha(x) > 0\ \forall \alpha \in \Phi_\Theta^+ \text{ and } \alpha(x)=0\ \forall \alpha \in \Phi_\Theta^0\}.
\end{align*}
Note that the notation is natural because $D_0 = D_0^\emptyset$ and the closure $\overline{D_0^\Theta}$ of $D_0^\Theta$ is the subset of invariants of the closure $\overline{D_0}$ of $D_0$ with respect to $W_\Theta$, i.e.~$\overline{D_0^\Theta} = \overline{D_0} \cap V^{W_\Theta}$.

\subsubsection{Sector faces and sector chambers}\label{intro sector faces}
We say that a subset $Q=Q(x,D)$ is an $\ell$-sector face of $\mathcal{X}(\mathbf{G},\ell)$ if there is an $\ell$-apartment $\mathbb{A} = g \cdot \mathbb{A}_0$, for some $g \in \mathbf{G}(\ell)$, a vector face $D$ of $g \cdot V$ and a point $x$ of $\mathbb{A}$ such that $Q = x +D$  (it is called a conical cell in \cite{Brown}).
The point $x=x(Q)$ is called the tip of $Q(x,D)$ and the vector face $D=D(Q)$ is called the direction of $Q(x,D)$.
Note that the tip of an $\ell$-sector face is not necessarily a vertex of $\mathcal{X}_\ell$.
We denote by $\overline{Q}(x,D)$ the closure in $\mathbb{A}$ of $Q(x,D)$.

If $D$ is a vector chamber, then $Q(x,D)$ is called an $\ell$-sector chamber.
We denote by $Q_0 = Q(v_0,D_0) \subset \mathbb{A}_0$ and we call it the standard $\ell$-sector chamber of $\mathcal{X}_\ell$. Let $L$ be the completion of $\ell$ with respect to $\nu$. Note that every $\ell$-sector chamber is an $L$-sector chamber but there may exist $L$-sector chambers of $\mathcal{X}(\mathbf{G},\ell)$ that are not $\ell$-sector chambers: i.e.~contained in the set of points $\mathcal{X}(\mathbf{G},\ell)$ but not in any $\ell$-apartment of $\mathcal{X}(\mathbf{G},\ell)$.
We denote by $\Sec(\mathcal{X}_\ell)$ the set of $\ell$-sector chambers of $\mathcal{X}_\ell$.

\subsection{Visual boundary}\label{intro visual boundary}

On the set of rays of an Euclidean building $\mathcal{X}$, we define an equivalence relation corresponding to parallelism, as in \cite[11.70]{Brown}.
We denote by $\partial_\infty \mathcal{X}$ the set of equivalence classes, which is called the visual boundary of $\mathcal{X}$.
Note that, if $\mathcal{X}$ is not equipped with a complete system of apartments, the definition of ray is restricted to those contained in some apartment.
This induces an equivalence relation on the set of sector faces by
\[
Q \sim Q' \text{ if and only if } D(Q)=D(Q'),
\]
For a subset $Q$ of $\mathcal{X}_k$ and, in particular, a sector face, we denote by $\partial_\infty Q$ the set of equivalence classes of geodesic rays contained in $Q$. We call it the visual boundary of $Q$.
Note that, for sector faces, the above equivalence relation becomes $Q \sim Q' \text{ if and only if } \partial_\infty Q = \partial_\infty Q'$.
We deduce from~\cite[11.75]{Brown} that there is a $1$ to $1$ correspondence $\partial_\infty Q \leftrightarrow D(Q)$ between the visual boundaries of sector faces and the vector faces.
Thus, for an arbitrary vector face $D$, one can denote by $\partial_\infty D := \partial_\infty Q(x,D)$ which does not depend on the tip $x$ of $Q(x,D)$.


Thus, visual boundaries of sector faces form a partition on $\partial_\infty \mathcal{X}$ and the visual boundaries of apartments form a system of apartments which equips $\partial_\infty \mathcal{X}$ of a structure of building whenever the system of apartment is ``good'' \cite[\S 11.8.4]{Brown}.
For instance, $\partial_\infty \mathcal{X}(\mathbf{G},K)$ is a spherical building since $\mathcal{X}(\mathbf{G},K)$ is equipped with a complete system of apartments.

\subsection{Functorial embeddings of the Bruhat-Tits buildings}\label{intro building embedding}

Let $\mathbf{G}$ be a connected reductive $k$-group and $\rho : \mathbf{G} \to \mathrm{SL}_{n,k}$ be a closed embedding.
Let $\mathcal{X}_k = \mathcal{X}(\mathbf{G},k)$ and $\mathcal{X}' = \mathcal{X}(\mathrm{SL}_n,k)$ be respectively the Bruhat-Tits building of $(\mathbf{G},k)$ and $(\mathrm{SL}_n,k)$.
By \cite[4.2.12]{BT2}, the Euclidean building $\mathcal{X}_k = \mathcal{X}(\mathbf{G},k)$ is isometric to $ \mathcal{X}(\rho(\mathbf{G}),k)$.
Moreover, by \cite[2.2.1]{Landvogt-functoriality}, there exists a $\rho(\mathbf{G}(K))$-equivariant closed immersion $j_K : \mathcal{X}(\rho(\mathbf{G}),K) \hookrightarrow \mathcal{X}'_K$, mapping the standard apartment $\mathbb{A}_0$ of $\mathcal{X}$ into the standard apartment $\mathbb{A}_0'$ of $\mathcal{X}'$ and multiplying the distances by a fixed constant (depending on $\rho$).
In fact, because the buildings $\mathcal{X}_k$ and $\mathcal{X}_K$ (resp. $\mathcal{X}'$ and $\mathcal{X}'_K$) only differ by their system of apartment, and since $\mathbb{A}_0$ and $\mathbb{A}'_0$ are $k$-apartments, there is a $\rho\big(\mathbf{G}(k)\big)$-equivariant closed embedding $j: \mathcal{X}\big( \rho(\mathbf{G}),k\big) \to \mathcal{X}'$ sending $\mathbb{A}_0$ onto $\mathbb{A}'_0$. If we furthermore assume that $\mathbf{G}_{\mathbb{Z}}$ is a $\mathbb{Z}$-Chevalley group scheme and that $\rho$ embeds the maximal torus $\mathbf{T}$ of $\mathbf{G}$ into a maximal torus of $\mathrm{SL}_n$ over $\mathbb{Z}$, then there is such a map $j$ that sends the special vertex $v_0 \in \mathbb{A}_0 \subset \mathcal{X}$ onto the special vertex $v'_0 \in \mathbb{A}'_0 \subset \mathcal{X}'$ \cite[9.1.19(c)]{BT}, by a descent of the valuation.

\section{The action of \texorpdfstring{$\mathbf{G}(k)$}{G(k)} on \texorpdfstring{$\mathcal{X}_k$}{X(G,k)}}\label{section action}

In all that follows, given a group $\Gamma$ acting on a set $X$ and a subset $Y \subseteq X$, we denote by $\stab_\Gamma(Y)$ the setwise stabilizer subgroup in $\Gamma$ of $Y$, and by $\fix_{\Gamma}(Y)$ the pointwise stabilizer subgroup in $\Gamma$ of $Y$.

The goal of this work is to study some properties of the abstract group $\mathbf{G}(A)$, where $A = \mathcal{O}_{\{\p\}}$ is as defined in \S~\ref{main}. In order to use the Borel-Tits theory, it is more convenient to work with $\mathbf{G}(k) \supset \mathbf{G}(A)$ where $k$ is the fraction field of $A$.
The Bruhat-Tits buildings are built from the completion $K$ of $k$ with respect to the valuation $\nu_\p$ and properties of the action of $\mathbf{G}(K)$ on $\mathcal{X}_K$ are well-known.
In this section, we recall some properties of this action that remain true for $\mathbf{G}(k)$ acting on $\mathcal{X}_k$ even if the valued field $(k,\nu_\p)$ is not complete.

At first, in order to work on $\mathcal{X}_k$ which is not equipped with a complete system of apartment in general, we introduce a subset of $k$-sectors built from the standard sector and from a subset of $\mathbf{G}(k)$ that covers the set of points of the building $\mathcal{X}_k$.

\begin{lemma}\label{Goodsector}
Let $\mathcal{X}$ be an affine building with a complete system of apartments and $\mathbb{A}$ be a given apartment of $\mathcal{X}$.
Let $x \in \mathcal{X}$ be any point and $x_0$ be a point of $\mathbb{A}$.
Then there exists a sector chamber $Q$ of $\mathbb{A}$ with tip $x_0$ and an apartment $\mathbb{A}'$ that contains $x$ and $Q$.
\end{lemma}

\begin{proof}
Let $\mathbb{A}_1$ be an apartment of $\mathcal{X}$ that contains both $x$ and $x_0$ (existence is given by axiom~(A3) of affine buildings \cite[Def.~1.1]{Parreau}).
There is a sector chamber $Q_1$ of $\mathbb{A}_1$ with tip $x_0$ so that $x \in \overline{Q_1}$.

Consider the spherical building $\operatorname{germ}_{x_0}(\mathcal{X})$ of germs of sector faces at $x_0$ as defined in \cite[1.3.3]{Parreau}.
The spherical building $\operatorname{germ}_{x_0}(\mathcal{X})$ can be endowed with a $W^{\mathrm{sph}}$-distance $\delta: \mathcal{C}_{x_0} \times \mathcal{C}_{x_0} \to W^{\mathrm{sph}}$ \cite[4.81]{Brown} where $\mathcal{C}_{x_0}$ denotes the set of germs of sector chambers at $x_0$ that are the chambers of $\operatorname{germ}_{x_0}(\mathcal{X})$.

Fix a set of generators $S$ of the spherical Coxeter group $W^{\mathrm{sph}}$ so that $\left( W^{\mathrm{sph}},S \right)$ is a Coxeter system and recall that there is a length map $\ell : W^{\mathrm{sph}} \to \mathbb{Z}_{\geqslant 0}$ with respect to $S$.
Let $\mathcal{C}_{0} \subset \mathcal{C}_{x_0}$ be the set of germs of sector chambers of $\mathbb{A}$ at $x_0$.
Consider a germ of sector chamber $D \in \mathcal{C}_0$ that realizes the maximum of the map $D \mapsto \ell(\delta(D,D_1))$ over $\mathcal{C}_0$ where $D_1$ denotes the germ at $x_0$ of the sector chamber $Q_1$.
Denote by $w_0$ the unique element having maximal length $\ell(w_0) = \ell_0$ \cite[2.19]{Brown}.

Suppose by contradiction that $\ell(\delta(D,D_1)) < \ell_0$.
Then $\delta(D,D_1) \neq w_0$ and there exists $s \in S$ such that $\ell(s\delta(D,D_1))= 1 + \ell(\delta(D,D_1))$ since $\ell(\delta(D,D_1))$ is not maximal.
The set $\operatorname{germ}_{x_0}(\mathbb{A})$ of germs at $x_0$ of sector faces of $\mathbb{A}$ is an apartment of $\operatorname{germ}_{x_0}(\mathcal{X})$ \cite[1.3.3]{Parreau}.
In the apartment $\operatorname{germ}_{x_0}(\mathbb{A})$, there exists a germ of sector chamber $D' \in \mathcal{C}_0$ such that $s=\delta(D',D)$ \cite[4.84]{Brown}.
Hence $\ell(\delta(D',D_1)) = \ell(s\delta(D,D_1)) > \ell(\delta(D,D_1))$ \cite[5.1 (WD2)]{Brown}.
This contradicts the maximality assumption.

Thus, $\delta(D,D_1)=w_0$.
By \cite[Cor.~1.11(GG)]{Parreau}, there is an apartment of $\operatorname{germ}_x(\mathcal{X})$ containing both $D$ and $D_1$.
In this apartment, the germs $D$ and $D_1$ are opposite since $\delta(D,D_1)=w_0$.
Let $Q$ be a sector chamber of $\mathbb{A}$ with germ at $x_0$ equal to $D$.
Then $Q$ and $Q_1$ are opposite sector chambers by definition.
Thus, by property~(CO) of affine buildings with a complete system of apartments \cite[Prop.~1.12(CO)]{Parreau}, there exists an apartment $\mathbb{A}'$ containing $Q$ and $Q_1$.
But since $x$ is in $Q_1$, we get that $x \in \mathbb{A}'$.
\end{proof}

Unipotents groups $\textbf{U}_{\Omega}(K)$ are very useful to analyze the action of $\mathbf{G}(K)$ on $\mathcal{X}_K$, since these groups act transitively on the set of apartments containing $\Omega\subset \mathcal{X}$ (c.f.~\cite[9.7]{L}).
Indeed, for $\Omega= Q_{0}$ we have that $\mathbf{U}_{\Omega}(K)$ is equal to a subgroup of $\mathbf{U}^{+}(K)$, because $\mathbf{U}_{\alpha, \Omega}(K)= \lbrace 0 \rbrace$, if $\alpha \in \Phi^{-}$.
It follows from \cite[Exp.XXV~2.5]{SGA3-3} that, using the \'epinglages $\theta_{\alpha}$, there exists an isomorphism of $\mathbb{Z}$-schemes between $\mathbf{U}^{+}$ and $\prod_{\alpha \in \Phi^{+}} \mathbf{U}_{\alpha}$.
Since $\mathbb{G}_a(k)$ is dense in $\mathbb{G}_a(K)$, we can deduce that $\mathbf{U} ^{+}(k)$ is dense in $\mathbf{U}^{+}(K)$.
From this density, we get that the points of the respective buildings $\mathcal{X}(\mathbf{G},k)$ and $\mathcal{X}(\mathbf{G},K)$ are the same whereas the system of apartments may be different (whenever there exists a maximal $K$-split torus that is not defined over $k$).
More precisely, we obtain all the points of $\mathcal{X}_K$ translating a single $k$-sector as follows:

\begin{lemma}\label{lema1}
 Let $k$ be a discretely valued field with completion $K$ and $\mathbf{G}$ be a split reductive $k$-group. 
Denote respectively by $\mathcal{X}_K=\mathcal{X}(\mathbf{G},K)$ (resp. $\mathcal{X}_k=\mathcal{X}(\mathbf{G},k)$) the Bruhat-Tits building associated to $\mathbf{G}$ over $K$ (resp. $k$).
Fix a maximal split torus $\mathbf{T}$ of $\mathbf{G}$, a Borel subgroup $\mathbf{B}$ containing $\mathbf{T}$ with unipotent radical $\mathbf{U}^+$ and let $Q_0$ be the standard $k$-sector chamber defined in \S \ref{build}, associated to $(\mathbf{B},\mathbf{T})$. Then the points of $\mathcal{X}_k$ and $\mathcal{X}_K$ are the same and there exists a subset $N^{\mathrm{sph}}$ in $\mathbf{N}_\mathbf{G}(\mathbf{T})(k)$ of representatives
\footnote{Typically, $N^{\mathrm{sph}}$ will be the system of representative of $W^{\mathrm{sph}}$ in $\mathbf{N}(\mathbb{F})$ introduced in \S \ref{intro Chevalley groups}.}
of $W^{\mathrm{sph}}$ such that this set is
\[
N^{\mathrm{sph}} \mathbf{U}^{+}(k) N^{\mathrm{sph}} \cdot\, \overline{Q_0}.
\]
\end{lemma}

\begin{proof}
Let $x \in \mathcal{X}_K$ be any point.
Let $v_0$ be the tip of $Q_0$. Let $\mathbb{A}'$ be a $K$-apartment and $Q'\subset \mathbb{A}_0$ a $K$-sector that satisfy Lemma~\ref{Goodsector} applied to $\mathbb{A}_0$, $v_0$ and $x$ in $\mathcal{X}_K$ (the system of apartments of $\mathcal{X}_K$ is complete since $K$ is complete \cite[2.3.7]{Rousseau77}).
Let $N^{\mathrm{sph}}\subset \mathbf{N}(\mathbb{F}) \subset \mathbf{N}(k)$ be the lift of $W^{\mathrm{sph}}$ as defined in section~\ref{intro Chevalley groups}.
Since $Q_0$ and $Q'$ have the same tip, by transitivity of the action of $W^{\mathrm{sph}}$ on the Coxeter complex, there exists $w' \in W^{\mathrm{sph}}$ such that $Q_0 = n_{w'} \cdot Q'$.
Since $\mathbb{A}_0$ and $n_{w'} \cdot \mathbb{A}'$ contain $Q_0$, by \cite[Cor.~9.7]{L}, there is $u \in \mathbf{U}_{Q_0}(K) \subset \mathbf{U}^+(K)$ satisfying $x' = (u n_{w'}) \cdot x \in \mathbb{A}_0$.
We know that the stabilizer of $x'$ in $\mathbf{U}^{+}(K)$ is an open subgroup \cite[12.12(ii)]{L} and that $\mathbf{U}^{+}(K)$ acts continuously on $\mathcal{X}_K$.
Hence, by density of $\mathbf{U}^{+}(k)$ in $\mathbf{U}^{+}(K)$, we can deduce that there exists $u' \in \mathbf{U}^{+}(k)$ such that $x'= (u'n_{w'}) \cdot x$.
Since $x'$ belongs to some closure of a sector chamber of $\mathbb{A}_0$ with tip $x_0$, we can deduce that $(n_{w}u'n_{w'}) \cdot x \in \overline{Q_0}$, for some $w \in W^{\mathrm{sph}}$.
\end{proof}

We recall the following classical result relating the standard parabolic subgroups and the standard vector chambers thanks to the spherical Bruhat decomposition.

\begin{lemma}\label{lemma stab standard face}
Let $\mathbf{B}$ be a Borel subgroup of $\mathbf{G}$ inducing a basis $\Delta$ of $\Phi$. Consider the action of $\mathbf{G}(k)$ on its spherical building. Let $\Theta \subset \Delta$.
The stabilizer in $\mathbf{G}(k)$ of the standard vector face $D_0^\Theta$ of type $\Theta$ is the group $\mathbf{P}_\Theta(k)$ of rational points of the standard parabolic subgroup associated to $\Theta$.
\end{lemma}

\begin{proof}
Consider the spherical Bruhat decomposition $\mathbf{G}(k) = \mathbf{B}(k) N^{\mathrm{sph}} \mathbf{B}(k)$ where $N^{\mathrm{sph}}$ is a lift in $\mathbf{N}(k)$ of the spherical Weyl group $W^{\mathrm{sph}} = \mathbf{N}(k) / \mathbf{T}(k)$.
According to~\cite[§6.2.4]{Brown}, the stabilizer of $D_0^\Theta$ is the standard parabolic subgroup associated to this Bruhat decomposition, that is
\[ \stab_{\mathbf{G}(k)}(D_0^\Theta) = \mathbf{B}(k) N_\Theta \mathbf{B}(k)\]
where $N_\Theta \subset N^{\mathrm{sph}}$ is a lift of $W_\Theta$.
It is precisely the rational points of the standard parabolic subgroup $\mathbf{P}_\Theta$ \cite[21.16]{BoA}.
\end{proof}

The strong transitive action of $\mathbf{G}(k)$ on its spherical building induces, in particular, a transitive action on the $k$-vector chambers.
We will make a strong use of the subset $H = \{n^{-1}u,\ n \in N^{\mathrm{sph}},\ u \in \mathbf{U}^+(k)\}$, which is not a subgroup in general. The following lemma explains how any $k$-vector face can be built from $H$ and the standard vector faces.

\begin{lemma}\label{lemma WUW}
For any $k$-vector face $D$ there exists $n \in N^{\mathrm{sph}}$, $u \in \mathbf{U}^{+}(k)$ and $\Theta \subset \Delta$ such that $u \cdot D = n \cdot D_0^\Theta$.
\end{lemma}

\begin{proof}
Let $D$ be a $k$-vector chamber.
Since $\mathbf{G}(k)$ acts strongly transitively on its spherical building (c.f.~§\ref{intro vector building}), there is $g \in \mathbf{G}(k)$ such that $ D = g \cdot D_0$.
Hence, by the spherical Bruhat decomposition (c.f.~\cite[21.15]{BoA}), $g$ can be written as $g=u^{-1} n_{w} b$ with $u \in \mathbf{U}^+(k)$, $w \in W^{\mathrm{sph}}$ and $b \in \mathbf{B}(k)$. 
Since $b \in \stab_{\mathbf{G}(k)}(D_0)$ according to Lemma~\ref{lemma stab standard face}, we deduce that $D = u^{-1} n_w \cdot D_0$, whence the result follows with $\Theta = \emptyset$.

Now, let $\widetilde{D}$ be a $k$-vector face. It is the face of some $k$-vector chamber $D$.
Thus, we have shown that there exists $u \in \mathbf{U}^+(k)$ and $n \in N^{\mathrm{sph}}$ such that $u \cdot D = n \cdot D_0$.
Thus $n^{-1} u \cdot \widetilde{D}$ is a face of $D_0$, whence there exists $\Theta\subset \Delta$ such that $n^{-1} u \cdot \widetilde{D} = D_0^\Theta$.
\end{proof}

We consider the visual boundary of $\mathcal{X}_k$ as introduced in \S \ref{intro visual boundary}.
Note that two sector chambers $Q$ and $Q'$ are equivalent if, and only if, $Q \cap Q'$ contains a common subsector chamber of $\mathcal{X}$ or, equivalently, if $Q$ and $Q'$ define the same chamber $\partial_\infty Q$ in the spherical building at infinity $\partial_{\infty}(\mathcal{X})$ (c.f.~\cite[11.77]{Brown}).

We denote by $\partial_\infty \Sec(\mathcal{X}_k)$ the set of visual boundaries of $k$-sector chambers.

\begin{lemma}\label{lemma-1}
Let $x,y \in \mathbb{A}_0$ be two special vertices.
Then $x$ and $y$ are in the same $\mathbf{G}(k)$-orbit if, and only if, they are in the same $\mathbf{T}(k)$-orbit.
\end{lemma}

\begin{proof}
Assume that $x$ and $y$ are in the same $\mathbf{G}(k)$-orbit.
Let $g \in \mathbf{G}(k)$ be such that $y=g \cdot x$.
Then, by definition of the building, there exists $n \in \mathbf{N}(k)$ such that $y = n \cdot x$ and $g^{-1} n \in \mathbf{U}_x(k)$.
Let $N_y = \stab_{\mathbf{N}(k)}(y)$.
Since $y$ is special, we know that $\mathfrak{w}^{\mathrm{sph}}(N_y) = W^\mathrm{sph}$. Thus $\mathbf{N}(k) = N_y \cdot \mathbf{T}(k)$.
Write $n = n' t$ with $n' \in N_y$ and $t \in \mathbf{T}(k)$.
Then $(n')^{-1} \cdot y = y = (n')^{-1} n \cdot x = t \cdot x$.
The converse is immediate.
\end{proof}

\begin{lemma}\label{lema0}
Let $Q$ and $Q'$ be two $k$-sector chambers of $\mathcal{X}_k$ whose tips $x$ and $x'$ are special $k$-vertices.
Then, there exists $g \in \mathbf{G}(k)$ such that $g\cdot Q=Q'$ if, and only if, $x$ and $x'$ are in the same $\mathbf{G}(k)$-orbit. In particular, $\mathbf{G}(k)$ acts transitively on $\partial_\infty \Sec(\mathcal{X}_k)$.
\end{lemma}

\begin{proof}
The necessary condition is clear.
Now, suppose that the tips $x$ and $x'$ of $Q$ and $Q'$ are in the same $\mathbf{G}(k)$-orbit and let $\mathbb{A}, \mathbb{A}'$ be two apartments of $\mathcal{X}_k$ satisfying $\mathbb{A} \supset Q$, $\mathbb{A}' \supset Q'$. By transitivity of the action of $\mathbf{G}(k)$ on the $k$-apartments of $\mathcal{X}_k$, there exists $g' \in \mathbf{G}(k)$ satisfying $g' \cdot \mathbb{A}=\mathbb{A}'$. Then, we can assume, without loss of generality, that $\mathbb{A}' = \mathbb{A}_0$.

Since $W^{\mathrm{sph}}$ acts transitively on the set of vector chambers of $\mathbb{A}_0$ we have that this group acts also transitively on the set of directions of sectors in $\mathbb{A}_0$. Hence, there exists $w \in W^{\mathrm{sph}}$ such that the directions of $(n_{w} g') \cdot Q$ and of $Q'$ are equals, whence they have a common subsector.
Observe that the tips of $(n_{w} g') \cdot Q$ and $Q'$ are special $k$-vertices and in the same $\mathbf{G}(k)$-orbit.
Thus, by Lemma~\ref{lemma-1}, there is an element $t \in \mathbf{T}(k)$ such that $(t n_{w} g') \cdot Q=Q'$. The result follows by taking $g=t n_{w} g' \in \mathbf{G}(k)$.
\end{proof}

\begin{lemma}\label{visual limit lemma}
The set $\partial_{\infty}(\Sec(\mathcal{X}_k))$ of visual boundaries of $k$-sectors of $\mathcal{X}_k$ is in bijection with the quotient $\mathbf{G}(k)/\mathbf{B}(k)$, which equals the $k$-rational points $(\mathbf{G}/\mathbf{B})(k)$ of the Borel variety $\mathbf{G}/\mathbf{B}$.
\end{lemma}

\begin{proof}

We have to find a bijective function between $\partial_{\infty}(\Sec(\mathcal{X}_k))$ and the quotient set $\mathbf{G} (k)/\mathbf{B}(k)$.
From the correspondence between vector chambers and visual boundaries of sector chambers (see §\ref{intro visual boundary}), we identify $\partial_{\infty}(Q_0)$ with $D_0=D_0^\emptyset$.
By Lemma~\ref{lemma stab standard face}, we deduce that $\stab_{\mathbf{G}(k)} (\partial_{\infty}(Q_0)) = \stab_{\mathbf{G}(k)} (D_0^\emptyset) = \mathbf{B}(k)$.
It follows from the Lemma~\ref{lema0}, that every element $\partial_{\infty}(Q) \in \partial_{\infty}(\Sec(\mathcal{X}_k))$ can be written as $g \cdot \partial_{\infty}(Q_0)$, for some $g \in \mathbf{G}(k)$.
Hence, we have $\mathrm{Stab}_{\mathbf{G}(k)} (\partial_{\infty}(Q))= g \mathbf{B}(k) g^{-1}$.
Let $\phi: \partial_{\infty}(\Sec(\mathcal{X}_k))\rightarrow \mathbf{G} (k)/\mathbf{B}(k)$ be the function defined by $\partial_{\infty}(Q) \mapsto \overline{g}$, where $\mathrm{Stab}_{\mathbf{G}(k)} (\partial_{\infty}(Q))= g \mathbf{B}(k) g^{-1}$. Note that the previous function is well defined and injective since $\mathbf{N}_{\mathbf{G}}(\mathbf{B})(k) = \mathbf{B}(k)$ (c.f.~\cite[Prop. 21.13]{BoA}). Obviously the function $\phi$ is surjective and then we get a bijective function between $\partial_{\infty}(\Sec(\mathcal{X}_k))$ and the set $\mathbf{G} (k)/\mathbf{B}(k)$. Finally we claim that
$$ \mathbf{G} (k)/\mathbf{B}(k) \cong  (\mathbf{G}/\mathbf{B})(k).$$
Indeed, consider the quotient exact sequence of $k$-algebraic varieties
$$
0 \rightarrow \mathbf{B} \xrightarrow{} \mathbf{G} \xrightarrow{} \mathbf{G}/\mathbf{B} \rightarrow 0.
$$
It follows from \cite[\S 4, 4.6]{DG} that there exists a long exact sequence
\begin{equation}\label{ex seq}
\lbrace 0 \rbrace \rightarrow \mathbf{B}(k) \rightarrow \mathbf{G}(k) \rightarrow \mathbf{G}/\mathbf{B}(k) \rightarrow H^1_{\text{\'et}}(\operatorname{Spec}(k), \mathbf{B}).
\end{equation}
On the other hand, by \cite[Exp.~XXVI, Cor. 2.3]{SGA3-3}, we have $H^1_{\text{\'et}}(\operatorname{Spec}(k), \mathbf{B})= H^1_{\text{\'et}}(\operatorname{Spec}(k), \mathbf{T})$. But, by hypothesis $ \mathbf{T}=  \mathbb{G}_m^{\mathbf{t}}$. So, it follows from the Hilbert 90's that $$H^1_{\text{\'et}}(\operatorname{Spec}(k), \mathbf{B})= H^1_{\text{\'et}}(\operatorname{Spec}(k), \mathbf{T})=H^1_{\text{\'et}}(\operatorname{Spec}(k), \mathbb{G}_m^{\mathbf{t}})= 0,$$
whence, from the exact sequence~\eqref{ex seq}, we get $\mathbf{G} (k)/\mathbf{B}(k) \cong  (\mathbf{G}/\mathbf{B})(k)$. Then, we result follows.
\end{proof}

\section{Commutative unipotent subgroups normalized by standard parabolic subgroups}\label{section commutative}

In this section, we assume that $k$ is a ring and that $S$ is a scheme (typically, $S = \spec \mathbb{Z}$ in subsection~\ref{sec linear action}). We consider a split reductive smooth affine $S$-group scheme $\mathbf{G}$ together with a Killing couple $(\mathbf{T},\mathbf{B})$, i.e.~a maximal split torus $\mathbf{T}=D_S(M) \cong \mathbb{G}_{m, \mathbb{Z}}^{r}$, where $M$ is a finitely generated free $\mathbb{Z}$-module, contained in a Borel subgroup $\mathbf{B}$. The existence of such couple is given in~\cite[Exp.~XXII, 5.5.1]{SGA3-3}).
Let $\Phi$ be the root system of $\mathbf{G}$ associated to $\mathbf{T}$ and $\Phi^+$ be the subset of positive roots associated to $\mathbf{B}$ (see~\cite[Exp.~XXII, 5.5.5(iii)]{SGA3-3}).

Let $\Psi \subseteq \Phi^+$ be a subset. Recall that $\Psi$ is said to be closed (see \cite[Chap.~VI, \S 1.7]{Bourbaki}) if
\begin{equation}\label{eq closed subset}
\forall \alpha,\beta\in \Psi,\ \alpha + \beta \in \Phi \Rightarrow \alpha + \beta \in \Psi.
\tag{C0}
\end{equation}

Let $\Delta$ be the corresponding basis to $\mathbf{B}$ and let $\Theta \subset \Delta$.
In this section, we want to provide a ``large enough'' smooth connected commutative unipotent subgroup $\mathbf{U}_\Psi$ normalized by the standard parabolic subgroup $\mathbf{P}_\Theta$ such that $\mathbf{P}_\Theta$ acts linearly by conjugation on $\mathbf{U}_\Psi$.
We cannot apply in our context the calculations in~\cite[Exp.~XXII, §5.5]{SGA3-3} because the ordering given by~\cite[exp~XXI, 3.5.6]{SGA3-3} may not be compatible with the commutativity condition or the ``large enough'' condition.

We consider the following conditions on $\Psi \subseteq \Phi^+$:
\begin{enumerate}[label={(C\arabic*{})}]
\item \label{cond closure} $\forall \alpha \in \Phi^+,\ \forall \beta \in \Psi,\ \alpha + \beta \in \Phi \Rightarrow \alpha + \beta \in \Psi$;
\item \label{cond commutative} $\forall \beta,\gamma \in \Psi, \beta+\gamma \not\in \Phi$;
\end{enumerate}

These conditions are related. For instance, if $\Psi$ closed and is a linearly independent family of $\operatorname{Vect}_\mathbb{R}(\Phi)$, then it satisfies~\ref{cond commutative}.
In this section, we firstly discuss the relation between these conditions and properties on groups, then we construct subset satisfying these conditions and we conclude by providing a consequence of these conditions, which is large enough with respect to some generating condition.

Iterating the conditions~\ref{cond closure} and~\ref{cond commutative}, we get that:

\begin{lemma}\label{lem sum of good roots}
The condition~\ref{cond closure} is equivalent to
\begin{equation}\label{cond closure prime}\tag{C1'}
    \forall \alpha \in \Phi^+,\ \forall \beta \in \Psi,\ \forall r,s \in \mathbb{N},\ r\alpha + s\beta \in \Phi \Rightarrow r\alpha+s\beta \in \Psi.
\end{equation}
The condition~\ref{cond commutative} is equivalent to
\begin{equation}\label{cond commutative prime}\tag{C2'}
    \forall \beta,\gamma \in \Psi,\ \forall r,s \in \mathbb{N},\ r\beta + s\gamma \not\in \Phi.
\end{equation}
If $\Psi \subseteq \Phi^+$ satisfies condition~\ref{cond commutative}, then the condition~\ref{cond closure} is equivalent to
\begin{equation} \label{cond closure second}\tag{C1''}
\forall \alpha \in \Phi^+,\ \forall \beta \in \Psi,\ \forall r,s \in \mathbb{N},\ r\alpha +s\beta \in \Phi \Rightarrow r\alpha+s\beta \in \Psi \text{ and } s = 1.
\end{equation}
\end{lemma}

\begin{proof}
Obviously, \eqref{cond closure second}~$\Rightarrow$~\eqref{cond closure prime}~$\Rightarrow$~\ref{cond closure} and~\eqref{cond commutative prime}~$\Rightarrow$~\ref{cond commutative}.

Conversely, we assume that $\Psi$ satisfies~\ref{cond closure} (resp.~\ref{cond commutative}, resp.~\ref{cond closure} and~\ref{cond commutative}) and we prove~\eqref{cond closure prime} (resp.~\eqref{cond commutative prime}, resp.~\eqref{cond closure second}). We proceed by induction on $r+s \geq 2$.
Let $\alpha \in\Phi^+$ and $\beta \in \Psi$.
Basis is obvious since $r=s=1$ whenever $r,s \in \mathbb{N}$, $r+s \geq 2$.
If $r+s > 2$ and $r \alpha + s \beta \in \Phi$, then $(r-1) \alpha + s \beta \in \Phi$ or $r \alpha + (s-1) \beta \in \Phi$ according to \cite[Chap.~VI, \S 1.6, Prop.~19]{Bourbaki}.

In the case $\alpha \in \Psi$ and $\Psi$ satisfies~\ref{cond commutative}, by induction assumption on~\eqref{cond commutative prime}, we get a contradiction so that $r\alpha + s \beta \not\in\Phi$.
Whence we deduce \ref{cond commutative}~$\Rightarrow$~\eqref{cond commutative prime}.

If $(r-1) \alpha + s \beta \in \Phi$, then $(r-1) \alpha + s \beta \in \Psi$ (resp. and $s=1$) by induction assumption on~\eqref{cond closure prime} (resp. on~\eqref{cond closure second}). Thus $\Big( (r-1) \alpha + s \beta \Big) + \alpha \in \Psi$ by condition~\ref{cond closure}.
 Otherwise $r \alpha + (s-1) \beta \in \Phi^+$ and $\beta \in \Psi$. Thus $\Big( r \alpha + (s-1) \beta \Big) + \beta \in \Psi$ by condition~\ref{cond closure}.
Whence we deduce \ref{cond closure}~$\Rightarrow$~\eqref{cond closure prime}.
 
Moreover, if $s > 1$, then $r \alpha + (s-1) \beta \in \Psi$ by induction assumption on~\eqref{cond closure prime}. Thus $\Big( r \alpha + (s-1) \beta \Big) + \beta \not\in \Phi$ under condition~\ref{cond commutative}. This is a contradiction with $r\alpha+s\beta \in\Phi$ so that $s=1$ whenever \ref{cond commutative} is satisfied.
Whence we deduce \ref{cond closure} and \ref{cond commutative}~$\Rightarrow$~\eqref{cond closure second}.
\end{proof}

\subsection{Associated closed subset of positive roots}

Recall that it follows from \cite[Exp.~XXII, 5.6.5]{SGA3-3} that, for any closed subset $\Psi$, there is a unique smooth closed $S$-subgroup with connected and unipotent fibers, denoted by $\mathbf{U}_\Psi$, which is normalized by $\mathbf{T}$ and with suitable action of $\mathbf{T}$ on its Lie algebra.
Moreover, for any given ordering on elements of $\Psi \subseteq \Phi^+$, there is an isomorphism of schemes
\begin{equation}\label{eq isom psi}
    \psi = \prod_{\alpha \in \Psi} \theta_\alpha : \left(\mathbb{G}_{a,k}\right)^{\Psi} \to \mathbf{U}_\Psi.
\end{equation}
The combinatorial properties on $\Psi$ correspond to the following in terms of root groups:

\begin{proposition}
\label{fact conditions}
If $\Psi \subset \Phi^+$ satisfies~\ref{cond closure}, then $\Psi$ is closed and $\mathbf{B}$ normalizes $\mathbf{U}_\Psi$.

If $\Psi$ satisfies~\ref{cond commutative}, then $\Psi$ is closed and $\mathbf{U}_{\Psi}$ is commutative.
\end{proposition}

\begin{proof}
\ref{cond closure} or \ref{cond commutative}~$\Rightarrow$~\eqref{eq closed subset} is obvious.

Assume that $\Psi$ satisfies~\ref{cond closure}. For any $\alpha \in \Phi^+$, using commutation relations given by~\cite[Exp.~XXII, 5.5.2]{SGA3-3} and isomorphism~\eqref{eq isom psi}, the subgroup $\mathbf{U}_\alpha$ normalizes $\mathbf{U}_\Psi$.
Hence $\mathbf{B}$ normalizes $\mathbf{U}_\Psi$ since $\mathbf{B}$ is generated by $\mathbf{T}$ and the $\mathbf{U}_\alpha$ for $\alpha \in \Phi^+$ according to~\cite[Exp.~XXII, 5.5.1]{SGA3-3}.

Assume that $\Psi$ satisfies~\eqref{cond commutative prime}. Then the commutativity of $\mathbf{U}_\Psi$ is an immediate consequence of \cite[Exp.~XXII, 5.5.4]{SGA3-3} and isomorphism~\eqref{eq isom psi}.
\end{proof}


Conversely, suppose that $k$ is a field and $S = \spec(k)$.
Consider a smooth connected unipotent subgroup $\mathbf{U}$.
If $\mathbf{U}$ is normalized by $\mathbf{T}$, we know by \cite[3.4]{BoTi} that $\mathbf{U} = \mathbf{U}_\Psi$ where $\Psi$ is a quasi-closed unipotent subset of $\Phi$ (see definition in \cite[3.8]{BoTi}).
In fact, $\Psi$ is closed whenever $\operatorname{char}(k) \not\in \{2,3\}$ by \cite[2.5]{BoTi} and, under these assumptions, we get a converse for Proposition~\ref{fact conditions}.

\begin{proposition}\label{fact conditions inverse}
Assume that $k$ is a field of characteristic different from $2$ and $3$. Let $\mathbf{G}$ be a smooth affine split reductive $k$-group.
Let $\mathbf{T}$ be a maximal split $k$-torus of $\mathbf{G}$.

(1) For any smooth affine connected unipotent $k$-group $\mathbf{U}$ normalized by $\mathbf{T}$, there exists a Borel subgroup $\mathbf{B}$ containing $\mathbf{T}$ with associated positive roots $\Phi^+$ and a closed subset $\Psi \subset \Phi^+$ such that $\mathbf{U} = \mathbf{U}_\Psi \subset \mathbf{B}$.

(2) Moreover, if $\mathbf{U}$ is normalized by $\mathbf{B}$, then $\Psi$ satisfies condition~\ref{cond closure}.

(3) If $\mathbf{U}$ is commutative, then $\Psi$ satisfies condition~\ref{cond commutative}.
\end{proposition}

\begin{proof}
(1) is an immediate consequence of \cite[3.4]{BoTi} and remark in \cite[2.5]{BoTi}.

(2) If $\mathbf{B}$ normalizes $\mathbf{U}$, then so does $\mathbf{U}_\alpha$ for any $\alpha \in \Phi^+$.
Thus by \cite[2.5, 3.4]{BoTi} again, we have that $(\alpha,\beta) \subset \Psi$ for any $\alpha \in \Phi^+$ and any $\beta \in \Psi$. 

(3) If $\mathbf{U}$ is commutative, then for any $\alpha,\beta \in \Psi$, we have that $\big[ \mathbf{U}_\alpha, \mathbf{U}_\beta \big] = \mathbf{U}_{(\alpha,\beta)}$ is trivial.
Thus $(\alpha,\beta) = \emptyset$ and therefore $\alpha+\beta \not\in \Phi$.
\end{proof}

In order to provide counter-examples, we introduce the structure constants in the following commutation relations (see \cite[Exp.~XXIII, 6.4]{SGA3-3}).

For $\alpha,\beta \in \Phi$, there are elements $c^{r,s}_{\alpha,\beta} \in \mathbb{Z}$ such that we have the commutation relation:
\begin{equation}\label{commutation relation}
\big[\theta_\alpha(x),\theta_\beta(y)\big] = \prod_{\substack{r,s \in \mathbb{N}\\ r\alpha+s\beta \in \Phi}} \theta_{r\alpha+s\beta}\left( c^{r,s}_{\alpha,\beta} x^r y^s\right)
\end{equation}

\begin{example}[Counter-examples in characteristic $2$ and $3$]
Consider $\mathbf{G} = \mathrm{Sp}(4)$ (of type $B_2$) with root system $\Phi = \{ \pm \alpha, \pm \beta, \pm (\alpha + \beta), \pm (2\alpha + \beta)\}$ where $\Delta = \{\alpha,\beta\}$ is a basis with $\alpha$ short and $\beta$ long.
We have that
\[ \big[ \theta_\alpha(x),\theta_{\alpha+\beta}(y) \big] = \theta_{2\alpha+\beta}\big( 2xy \big).\]
If $\operatorname{char}(k) = 2$, then $\Psi = \{\alpha,\alpha+\beta\}$ is then quasi-closed unipotent and $\mathbf{U}_\Psi$ is commutative but $\Psi$ does not satisfies~\ref{cond commutative}.

Consider a split reductive $k$-group $\mathbf{G}$ of type $G_2$ with a maximal split $k$-torus $\mathbf{T}$.
Let $\Phi$ be the root system relatively to $\mathbf{T}$ and let $\Delta = \{\alpha,\beta\}$ be a basis of $\Phi$ with $\alpha$ short and $\beta$ long.
Consider $\Psi = \{2\alpha + \beta, 3\alpha+2\beta\}$.
Any element in $\mathbf{U}_{3\alpha+2\beta}$ commutes with any element in $\mathbf{U}^+$ since $3\alpha+2\beta$ is the highest root and, for $2\alpha+\beta$, we have that
\begin{align*}
     \big[ \theta_\alpha(x),\theta_{2\alpha+\beta}(y) \big] & = \theta_{3\alpha+\beta}\big( 3xy \big),\\
     \big[ \theta_{\alpha+\beta}(x),\theta_{2\alpha+\beta}(y) \big] &= \theta_{3\alpha+2\beta}\big( -3xy \big).
\end{align*}
Thus, if $\operatorname{char}(k)=3$, we have that $\mathbf{U}_\Psi$ is commutative and normalized by the Borel subgroup associated to $\Delta$.
The closed subset $\Psi$ satisfies the condition~\ref{cond commutative} but not the condition~\ref{cond closure}.
\end{example}

\subsection{Construction of such subset of roots}

\begin{lemma}\label{lem high roots}
Let $\Phi$ be an irreducible root system and let $\Delta$ be a basis of $\Phi$.
Denote by $\Phi^+$ the subset of positive roots and by $h \in \Phi$ the highest root with respect to the basis $\Delta$.
For any $\Theta \subseteq \Delta$, define $\Psi(\Theta) = \big(h + \operatorname{Vect}_\mathbb{R}(\Theta)\big) \cap \Phi$.

If $\Theta \neq \Delta$, then $\Psi(\Theta)$ is a non-empty subset of $\Phi^+$ that satisfies conditions~\ref{cond closure}, \ref{cond commutative} and it is stabilized by the natural action of $W_\Theta$.
Moreover, there exists $\alpha \in \Delta \smallsetminus \Theta$ such that $\Psi(\Theta \cup \{\alpha\}) \supsetneq \Psi(\Theta)$.
\end{lemma}

\begin{proof}
Consider any $\Theta \subsetneq \Delta$.
Note that $W_\Theta$ is generated by the $s_\alpha$ for $\alpha \in \Theta$.
Let $x \in \Psi(\Theta)$. Then $s_\alpha(x) = x - \langle x,\alpha \rangle \alpha$ belongs to $s_\alpha(\Phi)=\Phi$ where $\langle x,\alpha \rangle \in \mathbb{R}$. Whence, by definition of $\Psi(\Theta)$, we deduce $s_\alpha\big( \Psi(\Theta)\big) \subseteq \Psi(\Theta)$ for any $\alpha \in \Theta$.
Thus $\Psi(\Theta)$ is stabilized by $W_\Theta$.

For any element $x \in \operatorname{Vect}_\mathbb{R} (\Phi)$, write $x = \sum_{\alpha \in \Delta} n_x^{\alpha} \alpha$ with coordinates $n_x^\alpha \in \mathbb{R}$ in the basis $\Delta$. Recall that if $x\in \Phi$, then the coordinates are either all non-negative, or all non-positive and that the highest root $h$ satisfies $n_h^\alpha \geqslant n_x^\alpha$ for any $x \in \Phi$ and any $\alpha \in \Delta$ \cite[VI.1.7 Cor.3 and VI.1.8 Prop.25]{Bourbaki}.

Let $\beta \in \Delta \smallsetminus \Theta$. Then for any $x \in \Psi(\Theta)$ we have that $n_x^\beta = n_h^\beta \geqslant 1$. In particular, $\Psi(\Theta) \subset \Phi^+$.
Moreover, if $x,y \in \Phi(\Theta)$, then $n_{x+y}^\beta = n_x^\beta + n_y^\beta = 2 n_h^\beta > n_h^\alpha$. Whence $x+y \not\in \Psi(\Theta)$. Thus $\Psi(\Theta)$ satisfies~\ref{cond commutative}.

Let $x \in \Phi^+$ and $y \in \Psi(\Theta)$ such that $x+y \in \Phi$.
For any $\beta \in \Delta \smallsetminus \Theta$, we have that $n^\beta_h \geqslant n^\beta_{x+y} = n^\beta_x + n^\beta_y \geqslant n^\beta_x + n^\beta_h$. Then $n^\beta_x = 0$ for all $\beta$, whence $x \in \operatorname{Vect}_{\mathbb{R}}(\Theta)$. Thus $x+y \in \big(\Psi(\Theta) + \operatorname{Vect}_{\mathbb{R}}(\Theta)\big) \cap \Phi = \Psi(\Theta)$.
Hence $\Psi(\Theta)$ satisfies condition~\ref{cond closure}.

Finally, suppose by contradiction that $\Psi(\Theta) = \Psi(\Theta \cup \{\alpha\})$ for all the $\alpha \in \Delta \smallsetminus \Theta$.
Then $\Psi(\Theta)$ would be stabilized by $s_\alpha \in W$ for any $\alpha \in \Delta$, thus by the whole Weyl group $W = W_\Delta$.
But $h \in \Psi(\Theta)$, $s_h \in W$ and $s_h(h) = -h \not\in \Phi^+$, which is a contradiction with a previous property of $\Psi(\Theta)$.
\end{proof}

\begin{proposition}\label{prop good increasing subsets}
Let $\Phi$ be a root system and $\Delta$ be a basis of $\Phi$.
For any $\Theta \subsetneq \Delta$, there exists a sequence of subsets $\Psi_i^\Theta$ of $\Phi^+$ for $0 \leq i \leq \operatorname{dim}(\Theta^\perp)$ such that:
\begin{itemize}
    \item for any $i$, the subset $\Psi_i^\Theta$ satisfies conditions~\ref{cond closure},~\ref{cond commutative} and is stabilized by the natural action of $W_\Theta$;
    \item the sequence of $\mathbb{R}$-vector spaces $\operatorname{Vect}\left({\Psi_i^\Theta}_{|\Theta^\perp}\right)$ is a complete flag of $(\Theta^\perp)^*$;
    \item for any $i \geq 1$ and $z \in \left(\Theta \cup \Psi_{i-1}^\Theta\right)^\perp$, the map $\alpha \in \Psi_i^\Theta \smallsetminus \Psi_{i-1}^\Theta \mapsto \alpha(z) \in \mathbb{R}$ has a constant sign (either positive, negative or zero).
\end{itemize}
\end{proposition}

\begin{proof}
We firstly treat the case of an irreducible root system $\Phi$ with basis $\Delta$.
Define $\Theta_1 = \Theta$.
Let $m := \operatorname{Card}(\Delta \smallsetminus \Theta)$.
According to Lemma~\ref{lem high roots}, there is a sequence of pairwise distinct simple roots $\alpha_2, \dots, \alpha_m \in \Theta \smallsetminus \Delta$ such that $\Psi(\Theta_i) \subsetneq \Psi(\Theta_{i+1})$ for any $1 \leqslant i < m$ where $\Theta_i = \Theta \cup \{\alpha_2,\dots,\alpha_i\}$.
Set $\Psi_i^\Theta = \Psi(\Theta_i)$.
Then, according to Lemma~\ref{lem high roots}, the subset $\Psi_i^\Theta$ satisfies~\ref{cond closure}, \ref{cond commutative} and it is stabilized by $W_{\Theta_i}$, therefore by $W_\Theta$.

Let $h$ be the highest root with respect to $\Delta$. Then, ${\Psi_1^\Theta}_{|\Theta^\perp} = \{ h_{|\Theta^\perp}\}$ is a nonzero single point since $\alpha_{|\Theta^\perp} = 0$ for any $\alpha \in \Theta$ and $h \not\in \operatorname{Span}_\mathbb{Z}(\Theta)$.
By construction, $\operatorname{Vect}_\mathbb{R}\big( {\Psi_i^\Theta}_{|\Theta^\perp} \big) = \operatorname{Vect}_\mathbb{R}\big( h_{|\Theta^\perp}, {\alpha_2}_{|\Theta^\perp},\dots,{\alpha_i}_{|\Theta^\perp}\big)$ for any $1 < i \leqslant m$.
Since the family $\{h,\alpha_2, \dots,\alpha_m\} \cup \Theta$ is a generating family of $\operatorname{Vect}(\Phi)$, we deduce that $\{h_{|\Theta^\perp}, {\alpha_2}_{|\Theta^\perp},\dots,{\alpha_m}_{|\Theta^\perp}\}$ is a generating family of $(\Theta^\perp)^*$, hence a basis because of cardinality.
We conclude that the family $\operatorname{Vect}_\mathbb{R}\big( {\Psi_i^\Theta}_{|\Theta^\perp} \big)$, with $1 \leq i \leq m$, forms a complete flag.

Finally,  let $z \in \left(\Theta \cup \Psi_{i-1}^\Theta\right)^\perp$ and let $\alpha,\alpha' \in \Psi_i^\Theta$.
Write respectively
\[\alpha = h - \sum_{\beta \in\Theta_{i-1}} n_\alpha^\beta \beta - m \alpha_i = \gamma - m \alpha_i
 \text{ and }
\alpha' = h - \sum_{\beta \in\Theta_{i-1}} {n'}_\alpha^\beta \beta - m' \alpha_i = \gamma' - m' \alpha_i
\]
with $n_\alpha^\beta,{n'}_\alpha^\beta,m,m' \in \mathbb{Z}_{\geq 0}$ for any $\beta \in \Theta_{i-1}$.
By construction and because the restriction to $\Theta^\perp$ gives a complete flag, for any $0 \leq i \leq m$, we have that $\operatorname{Vect}_\mathbb{R}\big( \Psi_i^\Theta \cup \Theta\big) = \operatorname{Vect}_\mathbb{R}\big( \{h\} \cup \Theta_i \big)$.
Then $\gamma,\gamma' \in \operatorname{Vect}_\mathbb{R}\big( \{h\} \cup  \Theta_{i-1} \big) = \operatorname{Vect}(\Psi_{i-1}^\Theta \cup \Theta)$.
If, moreover, $\alpha,\alpha' \not\in \Psi_{i-1}^\Theta$, we deduce that $m,m' > 0$.
Thus $\alpha(z)=-m \alpha_i(z)$ and $\alpha'(z) =- m' \alpha_i(z)$ both have the sign of $-\alpha_i(z)$.
Then, the result follows for irreducible root systems.

Now, consider a general root system $\Phi$ with basis $\Delta$ and denote by $(\Phi_j,\Delta_j)_{1 \leqslant j \leqslant t}$ its irreducible components with $V_j = \operatorname{Vect}(\Phi_j)^*$.
For $1 \leqslant j \leqslant t$, let $\Theta_j := \Theta \cap \Delta_j$ and let $m_j := \operatorname{Card}(\Delta_j - \Theta_j)$.
Note that $\Theta^\perp = \bigoplus_{1 \leqslant j \leqslant t} (\Theta_j^\perp \cap V_j)$.
For $1 \leqslant j \leqslant t$, if $\Theta_j=\Delta_j$ (i.e.~$m_j=0$), then $\Theta_j^\perp \cap V_j = 0$ and we can omit this factor.
Otherwise $m_j >0$ and we obtained an increasing sequence of subsets $\Psi_i^{\Theta_j} \subset \Phi_j^+$ for $1 \leq i \leq m_j$ which are~\ref{cond closure}, \ref{cond commutative} and stabilized by $W_\Theta$.
By an elementary proof, for $j_1 \neq j_2$, the subset $\Psi_{i_1}^{\Theta_{j_1}} \cup \Psi_{i_2}^{\Theta_{j_2}}$ also is~\ref{cond closure}, \ref{cond commutative} and stabilized by $W_\Theta$.
Thus it suffices to define the concatenation of the $\Psi_i^{\Theta_j}$ by:
\[\forall 1 \leqslant j \leqslant t,\ \forall \sum_{k=1}^{j-1} m_k < i \leqslant   \sum_{k=1}^{j} m_k,\ \Psi_i^\Theta :=  \left( \bigcup_{1 \leqslant k < j} \Psi_{m_k}^{\Theta_k} \right) \cup \Psi_{i-\sum_{k=1}^{j-1} m_k}^{\Theta_j}\]
This sequence provides successively a complete flag on each direct factor $\Theta_j^\perp \cap V_j$ whence the result follows.
\end{proof}

\subsection{\texorpdfstring{$k$}{k}-linear action of \texorpdfstring{$\mathbf{B}(k)$}{B(k)}}\label{sec linear action}

Recall that any subset of positive roots $\Phi^+$ with basis $\Delta$ is a poset with the following ordering:
$\beta < \gamma$ if there exists $m \in \mathbb{N}$ and simple roots $\alpha_1,\dots,\alpha_m \in \Delta$ such that $\forall 1 \leq i \leq m,\ \beta + \alpha_1 + \dots + \alpha_i \in \Phi^+$ and $\gamma = \beta + \alpha_1 + \dots + \alpha_m$.
We will use it several times.
Since $\Phi^+$ is a poset, for any subset $\Psi \subset \Phi^+$, there is a numbering $\Psi = \{\alpha_1,\dots,\alpha_m\}$ such that:
\begin{equation}\label{cond decreasing}
    \alpha_i < \alpha_j\ \Rightarrow\ i > j\tag{C3}
\end{equation}

\begin{lemma}\label{lem any subset}
Let $\Psi \subset \Phi^+$ satisfying conditions~\ref{cond closure} and~\ref{cond commutative}.
Write $\Psi = \{\alpha_1,\dots,\alpha_m\}$ with a numbering satisfying~\eqref{cond decreasing}.
Then, for any $1 \leq \ell \leq m$, the subset $\Psi_\ell = \{\alpha_1,\dots,\alpha_\ell\}$ satisfies conditions~\ref{cond closure} and~\ref{cond commutative}.
\end{lemma}

\begin{proof}
The condition~\ref{cond commutative} is automatically satisfied for any subset of $\Psi$.
Let $\ell \in \llbracket 1,m\rrbracket$.
If $\alpha \in \Phi^+$ and $\beta \in \Psi_k$ are such that $\beta'=\alpha+\beta \in\Phi$, then $\beta' \in \Psi$ by~\ref{cond closure} and $\beta' > \beta$.
Hence there are $i,j \in \llbracket 1,m\rrbracket$ such that $\beta=\alpha_i$ and $\beta'=\alpha_j$.
By condition~\eqref{cond decreasing}, we have that $i > j$ so that $\beta' = \alpha_j \in \Psi_\ell$.
Thus condition~\ref{cond closure} is satisfied for $\Psi_\ell$.
\end{proof}

\begin{proposition}\label{prop suitable polynomials}
Let $\mathbf{G}$, $\mathbf{B}$, $\mathbf{T}$ as before and assume that $S = \spec(\mathbb{Z})$.
Let $\Psi \subseteq \Phi^+$ be a subset satisfying conditions~\ref{cond closure} and~\ref{cond commutative}.
Let $\{\alpha_1,\dots,\alpha_m\} = \Psi$ be a numbering of $\Psi$ satisfying condition~\eqref{cond decreasing} and extends it to a numbering $\{\alpha_1,\dots,\alpha_M\} = \Phi^+$ of positive roots (without further conditions on $\{\alpha_{m+1},\dots,\alpha_M\}$).
Then, there exist polynomials $P_{i,j} \in \mathbb{Z}[X_1,\dots,X_M]$ for $i,j \in \llbracket 1,m\rrbracket$, depending on the choice of the numbering such that:
\begin{itemize}
    \item for any $v = \prod_{i=1}^{m} \theta_{\alpha_i}(y_{\alpha_i}) \in \mathbf{U}_\Psi$ and any $u = \prod_{i=1}^M \theta_{\alpha_i}(x_{\alpha_i}) \in \mathbf{U}^+$, we have
\[uvu^{-1} = \prod_{j=1}^{m} \theta_{\alpha_j}\left( \sum_{i=1}^{m} P_{i,j}\left( x_{\alpha_1},\dots,x_{\alpha_M}\right) y_{\alpha_i} \right)\]
where $x_\alpha, y_\beta \in \mathbb{G}_a$ for $\alpha \in \Phi^+$ and $\beta \in \Psi$,
    \item if $i < j$, then $P_{i,j} = 0$,
    \item if $i=j$, then $P_{i,j} = 1$.
\end{itemize}
\end{proposition}

\begin{proof}
Let $\mathbf{U}_\Psi = \prod_{\beta \in \Psi} \mathbf{U}_\beta$.
It is a commutative subgroup normalized by $\mathbf{B}$ according to Proposition~\ref{fact conditions}.

Let $\alpha \in \Phi^+$ and $\beta \in \Psi$.
Let $x,y \in \mathbb{G}_a$.
Then, using condition~\eqref{cond closure second} given by Lemma~\ref{lem sum of good roots}, the commutation relation~\eqref{commutation relation} becomes
\begin{equation}\label{eq simplified comm rel}
\left[ \theta_\alpha(x), \theta_\beta(y) \right] = \prod_{\substack{r\in\mathbb{N}\\r\alpha+\beta \in \Phi}} \theta_{r \alpha + \beta}\left( c^{r,1}_{\alpha,\beta} x^r y \right).
\end{equation}

We prove, for $i,j \in \llbracket 1,m\rrbracket$, by induction on the integer $N \geqslant 0$ such that $x_{\alpha_{N+1}} = \dots = x_{\alpha_M} = 0$, the existence of polynomials $P_{i,j;N} \in \mathbb{Z}[X_{1},\dots,X_{N}]$ satisfying:
\begin{itemize}
    \item for any $1 \leqslant \ell \leqslant N$, $P_{i,j;N}(X_{1},\dots,X_{\ell},0,\dots,0) = P_{i,j;\ell}$,
    \item and for any $v = \prod_{i=1}^{m} \theta_{\alpha_i}(y_{\alpha_i}) \in \mathbf{U}_\Psi$ and any $u = \prod_{i=1}^N \theta_{\alpha_i}(x_{\alpha_i}) \in \mathbf{U}^+$, we have
\[uvu^{-1} = \prod_{j=1}^{m} \theta_{\alpha_j}\left( \sum_{i=1}^{m} P_{i,j;N}\left( x_{\alpha_1},\dots,x_{\alpha_M}\right) y_{\alpha_i} \right)\]
where $x_\alpha, y_\beta \in \mathbb{G}_a$ for $\alpha \in \Phi^+$ and $\beta \in \Psi$.
\end{itemize}

\textbf{Basis:} when $N=0$, we immediately have $uvu^{-1} = v$ and therefore:
\[\forall i,j \in \llbracket 1,m\rrbracket,\ P_{i,j;0} = \left\{\begin{array}{rl}1 & \text{ if } i=j,\\0& \text{ otherwise.}\end{array}\right.\]

\textbf{Induction step:} assume $N \geq 1$ and write $u = \theta_{\alpha_N}(x_{\alpha_N}) \cdots \theta_{\alpha_1}(x_{\alpha_1})$.
Let $u_0 = \theta_{\alpha_{N-1}}(x_{\alpha_{N-1}}) \cdots \theta_{\alpha_1}(x_{\alpha_1})$ so that $u = \theta_{\alpha_N}(x_{\alpha_N}) u_0$.
Then, by the inductive assumption, there exist polynomials $P_{i,j;N-1} \in \mathbb{Z}[X_{1},\dots,X_{N-1}]$ such that 
\[u_0 v u_0^{-1} =  \prod_{j=1}^{m} \theta_{\alpha_j}\left( \sum_{i=1}^{m} P_{i,j;N-1}\left( x_{\alpha_1},\dots,x_{\alpha_{N-1}}\right) y_{\alpha_i} \right),\]
$P_{i,j;N-1} = 1$ for $i=j$ and $P_{i,j;N-1} = 0$ for $i<j$.
Hence 
\[uvu^{-1} = \prod_{j=1}^{m} \theta_{\alpha_N}(x_{\alpha_N}) \theta_{\alpha_j}\left( \sum_{i=1}^m P_{i,j;N-1}\left( x_{\alpha_1},\dots,x_{\alpha_{N-1}}\right) y_{\alpha_i} \right) \theta_{\alpha_N}(x_{\alpha_N})^{-1}.\]
Let $j \in \llbracket 1,m\rrbracket$.
If $\alpha_N +\alpha_j \not\in \Phi$, then
\begin{multline*}
\theta_{\alpha_N}(x_{\alpha_N}) \theta_{\alpha_j}\left( \sum_{i=1}^m P_{i,j;N-1}\left( x_{\alpha_1},\dots,x_{\alpha_{N-1}}\right) y_{\alpha_i} \right) \theta_{\alpha_N}(x_{\alpha_1})^{-1} \\= \theta_{\alpha_j} \left( \sum_{i=1}^m P_{i,j;N-1}\left( x_{\alpha_1},\dots,x_{\alpha_{N-1}}\right) y_{\alpha_i} \right).
\end{multline*}
Otherwise, $\alpha_N + \alpha_j \in \Phi$ and, according to Lemma~\ref{lem sum of good roots} and commutation relations~\eqref{eq simplified comm rel}, we have
\begin{multline*}
\theta_{\alpha_N}(x_{\alpha_N}) \theta_{\alpha_j}\left( \sum_{i=1}^m P_{i,j;N-1}\left( x_{\alpha_1},\dots,x_{\alpha_{N-1}}\right) y_{\alpha_i} \right) \theta_{\alpha_N}(x_{\alpha_N})^{-1} =\\
 \left(\prod_{\substack{r\in \mathbb{N}\\r\alpha_N + {\alpha_j} \in \Psi}}\theta_{r\alpha_N+{\alpha_j}}\left( \sum_{i=1}^m c^{r,1}_{\alpha_N,{\alpha_j}} P_{i,j;N-1}\left( x_{\alpha_1},\dots,x_{\alpha_{N-1}}\right) x_{\alpha_N}^r y_{\alpha_i} \right)
 \right)\\
 \cdot \theta_{\alpha_j}\left( \sum_{i=1}^m P_{i,j;N-1}\left( x_{\alpha_1},\dots,x_{\alpha_{N-1}}\right) y_{\alpha_i} \right).
\end{multline*}
Thus, since $\mathbf{U}_\Psi$ is a commutative group scheme, we get the desired formula by setting:
\[
P_{i,j;N} := 
P_{i,j;N-1} + \sum_{\substack{r \in \mathbb{N},\ \ell \in \llbracket 1,m\rrbracket\\ \alpha_j - r \alpha_N = \alpha_\ell\in \Psi}} c^{r,1}_{\alpha_N,\alpha_\ell} P_{i,\ell,N-1} X_{N}^r.
\]
Moreover if $i=j$ and $\alpha_\ell:=\alpha_j-r \alpha_N \in \Psi$, then $\alpha_\ell < \alpha_j$. Thus $i<\ell$ by condition~\eqref{cond decreasing} and $P_{i,j;N}= P_{i,j;N-1} = 1$.
If $i <j$ and $\alpha_j-r \alpha_N \in \Psi$, then there exists $\ell\in\llbracket 1,m\rrbracket$ such that $\alpha_j-r \alpha_N = \alpha_\ell < \alpha_j$. Thus $i<j<\ell$ by condition~\eqref{cond decreasing} and $P_{i,j;N}= 0$ since $P_{i,\ell,N-1}=0$ for all such $\ell$.
\end{proof}

\begin{corollary}\label{cor k-linear action borel}
Assume that $\Psi \subset \Phi^+$ satisfies conditions~\ref{cond closure} and~\ref{cond commutative}.
Then the subgroup $\mathbf{U}_\Psi(k)$ is naturally isomorphic to a free $k$-module of finite rank on which the action by conjugation of $\mathbf{B}(k)$ is $k$-linear.
\end{corollary}

\begin{proof}
Write $\Psi = \{\alpha_1,\dots,\alpha_m\}$ with a numbering satisfying~\eqref{cond decreasing} and extend the numbering in $\Phi^+ = \{\alpha_1,\dots,\alpha_M\}$.
The $k$-module structure on $\mathbf{U}_\Psi(k)$ is given by the isomorphism
\[\psi = \prod_{i=1}^m \theta_{\alpha_i}:  k^m \to  \mathbf{U}_\Psi(k)\]
Let $b \in \mathbf{B}(k)$.
Since $\mathbf{U}_\Psi(k)$ is normalized by $\mathbf{B}(k)$, there is a natural map $f_b : \mathbf{U}_\Psi(k) \to \mathbf{U}_\Psi(k)$ given by $f_b(x) = bxb^{-1}$.

Let $x=\prod_{j=1}^m \theta_{\alpha_i}(x_{i}) \in \mathbf{U}_{\Psi}(k)$ and $\lambda \in k$.
Let us write $b=t \cdot u$, where $u=\prod_{j=1}^{M} \theta_{\alpha_j}(y_j) \in \mathbf{U}^{+}(k)$ and $t \in \mathbf{T}(k)$.
Then by Proposition~\ref{prop suitable polynomials} applied to $\Psi$, we have
\begin{align*}
f_{b}(x) =& b x b^{-1}= t \prod_{j=1}^m \theta_{\alpha_j}\left( \sum_{i =1}^m P_{i,j} (y_1, \dots, y_M) x_{i}\right) t^{-1} ,\\
=& \prod_{j=1}^m \theta_{\alpha_j}\Big( \alpha_j(t) \big( \sum_{i=1}^m P_{i,j} (y_1, \dots, y_M) x_{i}\big)\Big).
\end{align*}
So, we get that
\begin{align*}
\lambda \cdot f_{b}(x) =& \prod_{j=1}^m \theta_{\alpha_j}\bigg( \lambda \alpha_j(t) \Big( \sum_{i=1}^m P_{i,j} (y_1, \dots, y_M) x_{i}\Big)\bigg) ,\\
=& \prod_{j=1}^m \theta_{\alpha_j}\bigg( \alpha_j(t) \Big( \sum_{i=1}^m P_{i,j} (y_1, \dots, y_M) \lambda x_{i}\Big)\bigg), \\
=& f_{b} \Big( \prod_{j=1}^m \theta_{\alpha_j}(\lambda x_{j}) \Big), \\
=&  f_{b}(\lambda \cdot x).
\end{align*}
Hence $f_b$ is $k$-linear for every $b \in \mathbf{B}(k)$.
\end{proof}

\begin{corollary}\label{cor k-linear action parabolic}
Let $\Theta \subset \Delta$ be any proper subset and let $\mathbf{P}_\Theta$ be the standard parabolic subgroup containing $\mathbf{B}$ associated to $\Theta$.
Assume that $\Psi \subset \Phi^+$ satisfies conditions~\ref{cond closure} and~\ref{cond commutative} and that $W_\Theta(\Psi) = \Psi$.
Then the action by conjugation of $\mathbf{P}_\Theta(k)$ on $\mathbf{U}_\Psi(k)$ is $k$-linear.
\end{corollary}

\begin{proof}
According to \cite[14.18]{BoA}, there is a subgroup $W_\Theta$ of the Weyl group such that $\mathbf{P}(k)$ admits the Bruhat decomposition $\mathbf{P}(k) = \mathbf{B}(k) W_\Theta \mathbf{B}(k)$.
Moreover, $W_\Theta$ is by definition spanned by the reflections $s_\alpha$ for $\alpha \in \Theta$.
Let $n_\alpha \in \mathbf{N}(k)$ be elements lifting the $s_\alpha$.
Since $\mathbf{B}(k)$ acts $k$-linearly on $\mathbf{U}_\Psi(k)$ by Corollary~\ref{cor k-linear action borel}, it suffices to prove that each $n_\alpha$ acts $k$-linearly on $\mathbf{U}_\Psi(k)$.

Since $\big( \theta_\alpha\big)_{\alpha \in \Psi}$ is a Chevalley system, one can write $n_\alpha = t_\alpha m_\alpha$ where $t_\alpha \in \mathbf{T}(k)$ and
\[ m_\alpha \theta_\beta( x ) m_\alpha^{-1} = \theta_{s_\alpha(\beta)}(\pm x),\qquad \forall x \in \mathbb{G}_a\]
for any $\beta \in \Phi$ \cite[3.2.2]{BT2}.
Hence, for any $u = \prod_{\beta \in \Psi} \theta_\beta(x_\beta) \in \mathbf{U}_\Psi(k)$, we have that
\[ n_\alpha u n_\alpha^{-1} = 
\prod_{\beta \in \Psi} \theta_{s_\alpha(\beta)}\big(\pm s_\alpha(\beta)(t_\alpha) x_\beta\big)
= \prod_{\beta \in s_\alpha(\Psi)} \theta_{\beta} \big(\pm \beta(t_\alpha) x_{s_\alpha(\beta)}\big)
\]
where the ordering of the product does not matter by commutativity of $\mathbf{U}_\Psi$, as consequence of~\ref{cond commutative}.
Thus, by $W_\Theta$-stability of $\Psi$, we deduce that $n_\alpha$ acts $k$-linearly on $\mathbf{U}_\Psi(k)$.
Whence the result follows for $\mathbf{P}_\Theta$.
\end{proof}

\subsection{Refinement for Borel subgroups}

In the case where $\Theta = \emptyset$, i.e.~$\mathbf{P}_\Theta = \mathbf{B}$ is a Borel subgroup, one can pick smaller subsets $\Psi_i^\emptyset$ satisfying conditions of Proposition~\ref{prop good increasing subsets} so that, moreover, $\Psi_{\operatorname{rk}(\mathbf{G})}^\emptyset$ forms an adapted basis to the complete flag.
In other words, here we show that $\Psi_i^\emptyset$ is a linearly independent subset of roots for every $i$.

\begin{proposition}\label{prop good roots}
Let $\Phi$ be a root system and $\Phi^+$ a choice of positive roots in $\Phi$.
There is a subset $\Psi \subset \Phi^+$ that satisfies conditions~\ref{cond closure}, \ref{cond commutative} and that is a basis of $\operatorname{Vect}_\mathbb{R}(\Phi)$.
\end{proposition}

The strategy is to enlarge the $\Psi_i^\emptyset$ from the highest root by subtracting a different simple root to some root of $\Psi_i^\Theta$ at each step, and check that the conditions~\ref{cond closure} and~\ref{cond commutative} are satisfied at each step.

\begin{proof}
Let $\Phi_1, \Phi_2$ be two root systems with positive roots $\Phi_1^+,\Phi_2^+$. If $\Psi_1 \subset \Phi_1^+$ and $\Psi_2 \subset \Phi_2^+$ satisfy the conditions~\ref{cond closure}, \ref{cond commutative} and are bases of $\operatorname{Vect}_\mathbb{R}(\Phi_1)$ and $\operatorname{Vect}_\mathbb{R}(\Phi_2)$ respectively, then the union $\Psi_1 \sqcup \Psi_2$ also satisfy the conditions in $\Phi_1^+ \sqcup \Phi_2^+ = \big( \Phi_1 \sqcup \Phi_2 \big)^+$.
Hence, without loss of generalities, we assume that $\Phi$ is irreducible.
Moreover, up to considering the subroot system of non-multipliable roots, we assume, without loss of generalities, that $\Phi$ is reduced.

Let $\Delta$ be the basis of $\Phi$.
Recall that any positive root $\alpha \in \Phi^+$ is a linear combination of elements in $\Delta$ with non-negative integer coefficients \cite[VI.1.6 Thm.~3]{Bourbaki}.
We proceed in a case by case consideration, using the classification \cite[VI.\S 4]{Bourbaki} and realisation of positive roots of irreducible root systems.
For each case, we provide an example of a subset $\Psi$ satisfying the conditions.
We detail why the conditions are satisfied only in the cases $D_\ell$ and $E_\ell$, which are the most technical cases.
Other cases work in the same way.

\textbf{Type $A_\ell$ ($\ell \geq 1$, see \cite[Planche I]{Bourbaki}):}
Write $\operatorname{Dyn}(\Delta)$:
\[\dynkin[labels={1,,,\ell},label macro/.code={\alpha_{#1}},edge length=1.2cm] A{} \]
Then
\[\Phi^+ = \left\lbrace \sum_{i \leq k \leq j} \alpha_k: 1 \leq i \leq j\leq \ell\right\rbrace.\]
Define $\displaystyle \beta_i = \alpha_i + \dots + \alpha_\ell$ for $1 \leq i \leq \ell$.
Then $\Psi = \lbrace \beta_1,\dots,\beta_\ell\rbrace$ answers to the lemma.

\textbf{Type $B_\ell$ ($\ell \geq 2$, see \cite[Planche II]{Bourbaki}):}
Write $\operatorname{Dyn}(\Delta)$:
\[\dynkin[labels={1,,,\ell-1,\ell},label macro/.code={\alpha_{#1}},edge length=1.2cm] B{} \]
Then
\[\Phi^+ = \left\lbrace \sum_{i \leq k \leq j} \alpha_k: 1 \leq i \leq j\leq \ell\right\rbrace
\sqcup \left\lbrace \sum_{i \leq k \leq j-1} \alpha_k + \sum_{j \leq k \leq \ell} 2\alpha_k: 1 \leq i < j \leq \ell\right\rbrace.\]
Define $\displaystyle \beta_i = (\alpha_1 + \dots + \alpha_\ell) + \sum_{i \leq k \leq \ell} \alpha_k$ for $2 \leq i \leq \ell + 1$.
Then $\Psi = \lbrace \beta_1,\dots,\beta_\ell\rbrace$ answers to the lemma.

\textbf{Type $C_\ell$ ($\ell \geq 3$ since $B_2 = C_2$, see \cite[Planche III]{Bourbaki}):}
Write $\operatorname{Dyn}(\Delta)$:
\[\dynkin[labels={1,,,\ell-1,\ell},label macro/.code={\alpha_{#1}},edge length=1.2cm] C{} \]
Then
\[\Phi^+ = \left\lbrace \sum_{i \leq k \leq j} \alpha_k: 1 \leq i \leq j\leq \ell\right\rbrace
\sqcup \left\lbrace \sum_{i \leq k \leq \ell} \alpha_k + \sum_{j \leq k \leq \ell-1} \alpha_k: 1 \leq i \leq j \leq \ell-1\right\rbrace.\]
Define $\displaystyle \beta_i = (\alpha_1 + \dots + \alpha_\ell) + \sum_{i \leq k \leq \ell-1} \alpha_k$ for $1 \leq i \leq \ell$.
Then $\Psi = \lbrace \beta_1,\dots,\beta_\ell\rbrace$ answers to the lemma.

\textbf{Type $D_\ell$ ($\ell \geq 4$, see \cite[Planche IV]{Bourbaki}):}
Write $\operatorname{Dyn}(\Delta)$:
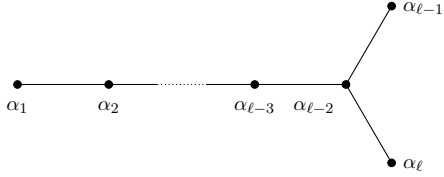
\[\dynkin[labels={1,2,\ell-3,\ell-2,\ell-1,\ell},label macro/.code={\alpha_{#1}},edge length=1.2cm] D{} \]
Then
\begin{multline*}
    \Phi^+ = \left\lbrace \sum_{k=i}^{j} \alpha_k : 1 \leq i \leq j \leq \ell \text{ and } (i,j) \neq (\ell-1, \ell)\right\rbrace\\ \sqcup \left\lbrace \alpha_\ell + \sum_{k=i}^{\ell-2} \alpha_k : 1 \leq i \leq \ell-2 \right\rbrace\\ \sqcup \left\lbrace\alpha_\ell + \alpha_{\ell-1} + 2 \sum_{k=j}^{\ell-2} \alpha_k + \sum_{k=i}^{j-1} \alpha_k: 1 \leq i < j \leq \ell-2  \right\rbrace.
\end{multline*}
Define
\begin{itemize}
    \item $\beta_1 =\alpha_1 + \dots + \alpha_{\ell-1}$,
    \item $\displaystyle \beta_i = (\alpha_1 + \dots + \alpha_\ell) + \sum_{i \leq k \leq \ell-2} \alpha_k$ for $2 \leq i \leq \ell-1$,
    \item $\beta_\ell =\alpha_1 + \dots + \alpha_{\ell-2} + \alpha_\ell$.
\end{itemize}
Consider $\Psi = \lbrace\beta_1,\dots,\beta_\ell\rbrace$.
For any $i,j$, we have $\beta_i + \beta_j \not\in \Phi$ since $\beta_i+\beta_j = 2 \alpha_1 + \dots$ and there is not a root in $\Phi$ with coefficient $2$ in $\alpha_1$.
Hence condition~\ref{cond commutative} is satisfied.

Let $V = \operatorname{Vect}_\mathbb{R}(\Psi)$.
We have $\alpha_\ell = \beta_{\ell-1} - \beta_1$ and $\alpha_{\ell - 1} = \beta_{\ell-1} - \beta_\ell$.
Moreover, $\beta_i - \beta_{i+1} = \alpha_i$ for $2 \leq i \leq \ell-2$.
Thus $\{\alpha_2,\dots,\alpha_\ell\} \subset V$.
Hence $\alpha_1 = \beta_1 - \alpha_2 - \dots - \alpha_{\ell-1} \in V$.
Therefore $V \supset \operatorname{Vect}_\mathbb{R}(\Delta)$ whence $\Psi$ is a basis of $\operatorname{Vect}_\mathbb{R}(\Phi)$.

If $\alpha + \beta_1 \in \Phi$, write $\alpha = \sum_{i=1}^{\ell} n_i \alpha_i$ with $n_i \in \mathbb{Z}_{\geq 0}$.
Necessarily, $n_\ell = 1$ and $n_{\ell-1} = 0$.
If $\alpha = \alpha_\ell$, then $\alpha + \beta_1 = \beta_{\ell-1} \in \Psi$.
Otherwise, $\alpha = \alpha_\ell + \sum_{k=i}^{\ell-2}$ for some $1 \leq i \leq \ell-2$.
Hence $\alpha + \beta_1 = \alpha_\ell + \alpha_{\ell - 1} + 2 \sum_{k=i}^{\ell-2} \alpha_k + \sum_{k=1}^{i-1} \alpha_k = \beta_i \in \Psi$.
Therefore, condition~\ref{cond closure} is satisfied for $\beta_1$.

If $\alpha + \beta_\ell \in \Phi$, write $\alpha = \sum_{i=1}^{\ell} n_i \alpha_i$ with $n_i \in \mathbb{Z}_{\geq 0}$.
Necessarily, $n_\ell = 0 = n_1$ and $n_{\ell-1} = 1$.
Hence $\alpha = \sum_{k=i}^{\ell-1} \alpha_k$ for some $2 \leq i \leq \ell-1$.
Thus $\alpha + \beta_\ell = \alpha_\ell + \alpha_{\ell - 1} + 2 \sum_{k=i}^{\ell-2} \alpha_i + \sum_{k=1}^{i-1} \alpha_k = \beta_i \in \Psi$.
Therefore, condition~\ref{cond closure} is satisfied for $\beta_\ell$.

If $\alpha + \beta_i \in \Phi$ for some $2 \leq i \leq \ell-1$, write $\alpha = \sum_{i=1}^{\ell} n_i \alpha_i$ with $n_i \in \mathbb{Z}_{\geq 0}$.
Necessarily, $n_1 = n_{\ell-1} = n_\ell = 0$ and $n_i = n_{i+1} = \dots = n_{\ell-1} = 0$.
Hence $\alpha = \sum_{k=i'}^{j'} \alpha_k$ with $2 \leq i' \leq j' < i$.
Thus $\alpha + \beta_i = \alpha_\ell + \alpha_{\ell - 1} + 2 \sum_{k=i}^{\ell-2} \alpha_i + \sum_{k=j'+1}^{i-1} \alpha_k + 2 \sum_{k= i'}^{j'} \alpha_k + \sum_{k=1}^{i'-1} \alpha_k \in \Phi^+$.
Necessarily, $j' = i -1$ and $\alpha + \beta_i= \beta_{i'} \in \Psi$.
Therefore, condition~\ref{cond closure} is satisfied for $\beta_i$.

\textbf{Type $E_6$ (see \cite[Planche V]{Bourbaki}):}
Write $\operatorname{Dyn}(\Delta)$:
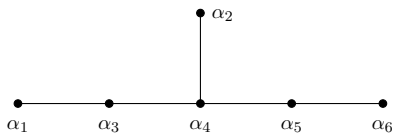
\[\dynkin[label,label macro/.code={\alpha_{#1}},edge length=1.2cm] E{6} \]
Denote by $n_1 n_2 n_3 n_4 n_5 n_6$ the root $\sum_{i=1}^6 n_i \alpha_i \in \Phi^+$.
Then the roots of $\Phi^+$ having a coefficient $n_i > 1$ for some $i$ are
\begin{multline*}\lbrace 011210, 111211, 011211, 112210, 111211, 011221,\\ 112211,111221,112221,112321,122321\rbrace.
\end{multline*}
One can take $\Psi = \lbrace 011221, 112211,111221,112221,112321,122321 \rbrace$.
Condition~\ref{cond closure} is satisfied because there are all the roots having at least $2$ coefficients greater or equal to $2$ and a non-zero coefficient $n_6$ in $\alpha_6$.
Condition~\ref{cond commutative} is satisfied considering the coefficient in $\alpha_4$.
The generating condition comes from the differences
$\alpha_1=111221-011221$,
$\alpha_2=122321-112321$,
$\alpha_3=112221-111221$,
$\alpha_4=112321-112221$,
$\alpha_5=112221-112211$.
Hence $\Psi$, being of cardinality $\ell=6$, answers to the lemma.

\textbf{Type $E_7$ (see \cite[Planche VI]{Bourbaki}):}
Write $\operatorname{Dyn}(\Delta)$:
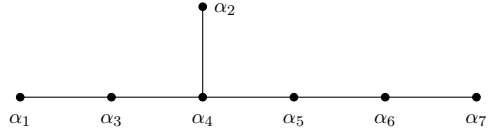
\[\dynkin[label,label macro/.code={\alpha_{#1}},edge length=1.2cm] E{7} \]
Denote by $n_1 n_2 n_3 n_4 n_5 n_6 n_7$ the root $\sum_{i=1}^7 n_i \alpha_i \in \Phi^+$.
Then the roots of $\Phi^+$ having a coefficient $n_i > 1$ for some $i$ are
\begin{multline*}
\lbrace 0112111,1112111,0112211,1122111,1112211,0112221,1122211,1112221,1122221,\\
1123211,1123221,1223211,1123321,1223221,1223321,1224321,1234321,2234321\rbrace.
\end{multline*}
One can take 
\[\Psi = \lbrace 1223211,1123321,1223221,1223321,1224321,1234321,2234321 \rbrace.\]
Condition~\ref{cond closure} is satisfied because there are all the roots having coefficient in $\alpha_4$ greater or equal to $3$ and coefficient in $\alpha_2$ greater or equal to $2$, together with the root $1123321$.
Condition~\ref{cond commutative} is satisfied considering the coefficient in $\alpha_4$.
The subset $\Psi$ is generating because of the differences
$\alpha_1=2234321-1234321$, 
$\alpha_2=1223321-1123321$,
$\alpha_3=1234321-1224321$,
$\alpha_4=1224321-1223321$,
$\alpha_5=1223321-1223221$,
$\alpha_6=1223221-1223211$.
Hence $\Psi$, being of cardinality $\ell=7$, answers to the lemma.

\textbf{Type $E_8$ (see (10) of \cite[VI.§4]{Bourbaki}):}
Write $\operatorname{Dyn}(\Delta)$:
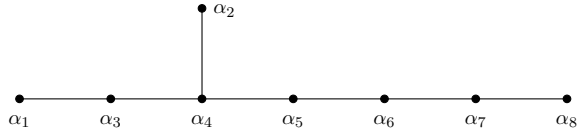
\[\dynkin[label,label macro/.code={\alpha_{#1}},edge length=1.2cm] E{8} \]
Denote by $n_1 n_2 n_3 n_4 n_5 n_6 n_7 n_8$ the root $\sum_{i=1}^8 n_i \alpha_i \in \Phi^+$.
Consider the subset $\Psi$ of $\Phi^+$ consisting in roots such that $n_4 \geq 5$.
One can take
\[\Psi = \lbrace 23354321,22454321,23454321,23464321,23465321,23465421,23465431,23465432\rbrace.\]
Condition~\ref{cond closure} is satisfied because there are all the roots such that $\sum_{i=1}^8 n_i \geqslant 23$ and having coefficient $n_4$ in $\alpha_4$ greater or equal to $5$.
Condition~\ref{cond commutative} is satisfied considering the coefficient in $\alpha_4$.
The generating condition comes from the differences
$\alpha_8=23465432-23465431$,
$\alpha_7=23465431-23465421$,
$\alpha_6=23465421-23465321$,
$\alpha_5=23465321-23464321$,
$\alpha_4=23464321-23454321$,
$\alpha_3=23454321-23354321$,
$\alpha_2=23454321-22454321$.
Hence $\Psi$, being of cardinality $\ell=8$, answers to the lemma.

\textbf{Type $F_4$, (see \cite[Planche VIII]{Bourbaki}):}
Write $\operatorname{Dyn}(\Delta)$:
\[\dynkin[label,label macro/.code={\alpha_{#1}},edge length=1.2cm] F{4} \]
Denote by $n_1 n_2 n_3 n_4$ the root $\sum_{i=1}^4 n_i \alpha_i \in \Phi^+$.
Consider the subset $\Psi$ of $\Phi^+$ consisting in roots such that $n_3 \geq 3$.
Consider
\[\Psi = \lbrace 1232,1242,1342,2342\rbrace.\]
Hence $\Psi$ answers to the lemma.

\textbf{Type $G_2$, (see \cite[Planche IX]{Bourbaki}):}
Write $\operatorname{Dyn}(\Delta)$:
\[\dynkin[labels={1,2},label macro/.code={\alpha_{#1}},edge length=1.2cm] G{2} \]
Then
\[\Phi^+ = \left\lbrace \alpha_1,\alpha_2,\alpha_1+\alpha_2,2\alpha_1+\alpha_2,3\alpha_1+\alpha_2,3\alpha_1+2\alpha_2\right\rbrace.\]
Hence $\Psi = \lbrace 3\alpha_1+\alpha_2,3\alpha_1+2\alpha_2\rbrace$ answers to the lemma.
\end{proof}

\begin{corollary}\label{cor good increasing subsets}
Let $\Phi$ be a root system with positive roots $\Phi^+$. Let $V= \operatorname{Vect}_\mathbb{R}(\Phi)$ and denote by $m$ its dimension.
There exists a sequence of subsets $\Psi_i^\emptyset$ of $\Phi^+$ for $0 \leq i \leq m$ such that for any $i$, the subset $\Psi_i^\emptyset$ satisfies conditions~\ref{cond closure},~\ref{cond commutative}, is linearly independent in $V$ and the family $V_i := \operatorname{Vect}_\mathbb{R}(\Psi_i^\emptyset)$, for $0 \leq i \leq m$, is a complete flag of $V$.
\end{corollary}

\begin{proof}
Let $\Psi_m^\emptyset$ be the subset of $\Phi^+$ given by Proposition~\ref{prop good roots}.
Pick a numbering on $\Psi_m^\emptyset = \{\alpha_1,\dots,\alpha_m\}$ satisfying~\eqref{cond decreasing}.
Then, for any $i \in \llbracket 1,m\rrbracket$, we know by Lemma~\ref{lem any subset} that $\Psi_i^\emptyset := \{\alpha_1,\dots,\alpha_i\}$ also satisfies conditions~\ref{cond closure} and~\ref{cond commutative}.
Moreover, it spans a complete flag by cardinality and linear independence.
\end{proof}

\section{Stabilizer of points in the Borel variety: integral points and some of their unipotent subgroups}\label{section Stabilizer of points in the Borel variety}

In this section, we assume that $A$ is an integral domain, that $k$ is its fraction field and that $\mathbf{G}$ is a split Chevalley group scheme over $\mathbb{Z}$.
In subsections~\ref{sec upper bound SLn} and~\ref{sec upper bound}, we will furthermore assume that the domain $A$ is Dedekind and, in subsection~\ref{sec upper bound SLn}, that $\mathbf{G} = \mathrm{SL}_n$  or $\mathrm{GL}_n$.
We denote by $\rho: \mathbf{G} \to \mathrm{SL}_{n,\mathbb{Z}}$ a closed embedding of smooth $\mathbb{Z}$-group schemes. We recall the following classical lemma which provides two point of views on the arithmetic group $\mathbf{G}(A)$: either the $A$-integral points of a group scheme, or the matrices with entries in $A$ via $\rho$.

\begin{lemma}
We have $\rho\big(\mathbf{G}(A)\big) = \rho\big(\mathbf{G}(k) \big) \cap \mathrm{SL}_n(A)$.
\end{lemma}

\begin{proof}
Obviously, $\rho\big(\mathbf{G}(A)\big) \subseteq \rho\big(\mathbf{G}(k) \big) \cap \mathrm{SL}_n(A)$.

Since $\rho : \mathbf{G} \to \mathrm{SL}_{n,\mathbb{Z}}$ is faithful, it corresponds to a surjective homomorphism $\rho^* \in \operatorname{Hom}_{\mathbb{Z}-alg}\big( \mathbb{Z}[\mathrm{SL}_n],\mathbb{Z}[\mathbf{G}]\big)$ \cite[1.4.5]{BT2}.
Let $g \in \mathbf{G}(k) = \operatorname{Hom}_{\mathbb{Z}-\text{alg}}(\mathbb{Z}[\mathbf{G}],k)$ and let $h = g \circ \rho^* \in \mathrm{SL}_n(k)$.
Suppose that $\rho(g) = g \circ \rho^* \in \mathrm{SL}_n(A)$.
Then $g\big(\mathbb{Z}[\mathbf{G}]\big) = g \circ \rho^*\big( \mathbb{Z}[\mathrm{SL}_n] \big) \subseteq A$.
Thus $g \in \mathbf{G}(A)$, whence the converse inclusion $\rho\big( \mathbf{G}(k) \big) \cap \mathrm{SL}_n(A) \subseteq \rho\big(\mathbf{G}(A)\big) $ follows.
\end{proof}

Let $\mathbf{B}$, $\mathbf{U}^+ \subset \mathbf{G}$ and $N^\mathrm{sph} \subset \mathbf{G}(A)$ as defined in \S~\ref{intro Chevalley groups}.
The group $\mathbf{G}(A)$ naturally acts on the Borel variety $(\mathbf{G}/\mathbf{B})(k)$ and we would like to understand stabilizer in $\mathbf{G}(A)$ of a point $D \in (\mathbf{G}/\mathbf{B})(k)$ (identified with a $k$-vector chamber).
We have seen in Lemma~\ref{lemma WUW} there are elements $u \in \mathbf{U}^+(k)$ and $n \in N^{\mathrm{sph}}$ such that $u \cdot D = n \cdot D_0$.
We introduce the subset
\begin{equation}\label{eq def H}
H := \{ n^{-1} u: n \in N^\mathrm{sph},\ u\in \mathbf{U}^+(k) \} \subset \mathbf{G}(k),
\end{equation}
so that for any $k$-vector chamber $D$, there is $h \in H$ such that $h \cdot D = D_0$.
Thus
\[ \stab_{\mathbf{G}(A)}(D) = h^{-1} \stab_{h \mathbf{G}(A) h^{-1}}(D_0) h\]
By Lemma~\ref{lemma stab standard face}, we deduce that
\[ \stab_{\mathbf{G}(A)}(D) = h^{-1} \Big( h \mathbf{G}(A) h^{-1} \cap \mathbf{B}(k) \Big) h\]
Thus, the group $\stab_{\mathbf{G}(A)}(D)$ is solvable and its unipotent part is contained in $ \mathbf{G}(A) \cap h^{-1} \mathbf{U}^+(k) h$.
We study some commutative subgroup of this group or, equivalently, of $h \mathbf{G}(A) h^{-1}  \cap \mathbf{U}^+(k)$.

Let $\Psi \subset \Phi^+$ be a closed subset satisfying~\ref{cond commutative} so that $\mathbf{U}_\Psi = \prod_{\alpha \in \Psi} \mathbf{U}_\alpha$ is a commutative subgroup of $\mathbf{U}^+$.
There is a natural isomorphism of group schemes:
\[ \psi = \prod_{\alpha \in \Psi} \theta_\alpha: \big(\mathbb{G}_a\big)^\Psi \to \mathbf{U}_\Psi\]
More generally, for such a subset $\Psi$ and an arbitrary element $h \in \mathbf{G}(k)$, we define a subgroup of the $k$-vector space $k^\Psi$ by:
\begin{equation}\label{eq def MPsi}
    M_\Psi(h) := \psi^{-1} \Big( \mathbf{U}_{\Psi}(k) \cap h \mathbf{G}(A) h^{-1} \Big).
\end{equation}

In particular, when $\Psi = \{\alpha\}$ is a root, it defines a subgroup\footnote{It follows from a straightforward calculation, using Proposition~\ref{prop suitable polynomials}, that when $n=1$, i.e.~$h = u \in \mathbf{U}^+(k)$, then $ M_\Psi(u)$ is an $A$-module and $M_\alpha(u)$ is a fractional $A$-ideal.} of the additive group $k$ by
\begin{equation}
    M_\alpha(h) := \theta_\alpha^{-1} \Big( \mathbf{U}_{\alpha}(k) \cap h \mathbf{G}(A) h^{-1} \Big).
\end{equation}

If we take $h = n_w^{-1} u \in H$ with $w \in W^\mathrm{sph}$ and $u \in \mathbf{U}^+(k)$, then 
\[ n_w \theta_\alpha\big( M_\alpha(h) \big) n_w^{-1} = n_w \mathbf{U}_\alpha(k) n_w^{-1} \cap u \mathbf{G}(A) u^{-1}= \mathbf{U}_\beta(k) \cap u \mathbf{G}(A) u^{-1}\] 
where $\beta = w(\alpha) \in \Phi$ since $n_w \mathbf{U}_\alpha n_w^{-1} = \mathbf{U}_\beta$.
Thus we extend the definition by setting:
\begin{equation}
    M_\beta(u) := \theta_\beta^{-1}\big( \mathbf{U}_{\beta}(k) \cap u \mathbf{G}(A) u^{-1}\big)
\end{equation}
for any $\beta \in \Phi$ and any $u \in \mathbf{U}^+(k)$.

The goal of this section is to provide some $A$-modules which bound the subgroups $M_\Psi(h)$. 
More specifically, in this section, on one hand, we prove that, for $\beta \in \Phi$ and $u \in \mathbf{U}^+(k)$, the group $M_\beta(u)$ always contains a non-zero $A$-ideal, denoted by $J_\beta(u)$ in Proposition~\ref{prop ideal contained}.
On the other hand, we prove that, for suitable subsets $\Psi \subset \Phi^+$, there are non-zero fractional $A$-ideals $\mathfrak{q}_{\Psi,\alpha}(h)$ for $h \in H$ and $\alpha \in \Psi$ such that
\[
    M_\Psi(h) \subseteq \bigoplus_{\alpha \in \Psi} \mathfrak{q}_{\Psi,\alpha}(h).
\]

\begin{example}
Here we exhibit the preceding ideals $J_{\alpha}(h)$ and $\mathfrak{q}_{\Psi,\alpha}(h)$ in the context where $\mathbf{G} = \mathrm{SL}_2$. Let $(\mathbf{B},\mathbf{T})$ be the pair of upper triangular and diagonals matrices. Then, the unipotent subgroup defined from to the unique root $\alpha \in \Phi^{+}$ is exactly the set of unipotent upper triangular matrices. Also, recall that $W^{\mathrm{sph}}= \lbrace \mathrm{id}, w \rbrace$ in this context. In particular, this implies that $\Psi=\Phi^{+}=\lbrace \alpha\rbrace$, whence the unique relevant ideals are $J_{\alpha}(h)$ and $\mathfrak{q}_{\alpha}(h):=\mathfrak{q}_{\Psi,\alpha}(h)$, for $h \in H$. Also, recall that $(\mathbf{G}/\mathbf{B})(k) \cong \mathbb{P}^1(k)$ in this context. Thus, to fix an element $x \in \mathbb{P}^1(k)$ is equivalent to fix a $k$-vector chamber in $\mathcal{X}(\mathbf{G}, k)$. Moreover, the canonical vector chamber $D_0$ corresponds with $\infty \in \mathbb{P}^1(k)$ via the preceding identification. To start, let $x=\infty$, for which we can consider $h=\mathrm{id}$. Thus
$$\mathrm{Stab}_{\mathbf{G}(A)}(x) = \mathbf{G}(A) \cap \mathbf{B}(k)= \mathbf{B}(A) = \left\lbrace \begin{pmatrix} a & c\\0 & b\end{pmatrix} : a,b \in \mathbb{F}^{\times}, \, \, ab=1, \, \,  c \in A \right\rbrace.$$
In particular, we get that $M_{\alpha}(h)=\mathfrak{q}_{\alpha}(h)=J_{\alpha}(h)=A$ in this case.

Now, assume that $x \in k$. Then, we can consider 
\[
    h = n_w \theta_\alpha(x) = \small \begin{pmatrix} 0 & -1\\1 & 0\end{pmatrix} \begin{pmatrix} 1 & -x\\ 0 & 1\end{pmatrix}= \begin{pmatrix} 0 & -1\\ 1 & -x\end{pmatrix} \normalsize.
\]
When $x \in \mathbb{F}$, it follows from \cite[Lemma 3.2]{M} and Proposition~\ref{prop igual stab} that $ \mathrm{Stab}_{\mathbf{G}(A)}(x) = h^{-1} \mathbf{B}(A)h$. So, also in this case, we conclude $M_{\alpha}(h)=\mathfrak{q}_{\alpha}(h)=J_{\alpha}(h)=A$.

Finally, assume that $x \in k \smallsetminus \mathbb{F}$. Then, it follows from \cite[Lemma 3.4]{M} and Proposition~\ref{prop igual stab} that
$$\mathrm{Stab}_{\mathbf{G}(A)}(x) = \small \left\lbrace \begin{pmatrix} a+xc & (b-a)x-x^2c\\c & b-xc\end{pmatrix} \normalsize
: \begin{matrix}
a,b \in \mathbb{F}^{\times}, \, \,  ab=1, \\
c \in A \cap A x^{-1} \cap ((b-a)x^{-1}+Ax^{-2})
\end{matrix}
\right\rbrace. $$
Thus, we obtain
\begin{multline*}
   \mathbf{G}(A) \cap h^{-1} \mathbf{U}^{+}(k) h  = \left\lbrace \begin{pmatrix} 1+xc & -x^2c\\c & 1-xc\end{pmatrix} : 
c \in A \cap A x^{-1} \cap Ax^{-2}
\right\rbrace, \\
= h^{-1} \left\lbrace \begin{pmatrix} 1 & c\\0 & 1\end{pmatrix} : 
c \in A \cap A x^{-1} \cap Ax^{-2}
\right\rbrace  h. \quad \quad\quad \quad \quad \quad
\end{multline*}
In particular, we conclude that $M_{\alpha}(h)=\mathfrak{q}_{\alpha}(h)=J_\alpha(h) =A \cap  A x^{-1} \cap Ax^{-2}$.
\end{example}

\begin{example}
In this example, we provide a set $M_{\Psi}(h)$ that is not a product of fractional $A$-ideals. Let $\mathbf{G}=\mathrm{SL}_3$ and let $A=\mathbb{F}[t]$. Let $(\mathbf{B},\mathbf{T})$ be the pair of upper triangular and diagonal matrices in $\mathbf{G}$. Then, the set of positive roots $\Phi^{+}=\Phi(\mathbf{B}, \mathbf{T})^{+}$ of $\mathrm{SL}_3$ equals $\lbrace \alpha, \beta, \alpha+ \beta \rbrace$, when each root defines exactly one of the following unipotent subgroups:
\begin{align*}
    \mathbf{U}_\alpha = & \left\{ 
   \small \left(
        \begin{matrix}1&z&0\\0&1&0 \\0&0&1\end{matrix} \right) \normalsize
    : z \in \mathbb{A}^1\right\}, &
    \mathbf{U}_\beta = & \left\{ 
    \small \left(
        \begin{matrix}1&0&0\\0&1&z \\0&0&1\end{matrix} \right) \normalsize
    : z \in \mathbb{A}^1\right\},\\
    \mathbf{U}_{\alpha+\beta} = & \left\{ 
        \small \left(
        \begin{matrix}1&0&z\\0&1&0 \\0&0&1\end{matrix} \right) \normalsize : z \in \mathbb{A}^1\right\}.
\end{align*}
Consider the subset $\Psi=\lbrace \alpha+ \beta, \alpha \rbrace$ and the subgroup $\mathbf{U}_{\Psi}=   \left\{ \left(
        \begin{smallmatrix}1&x&z\\0&1&0 \\0&0&1\end{smallmatrix} \right)
    : x,z \in \mathbb{A}^1\right\}.$
Now, let us consider the element $h=n_{w} u$, where
$$ n_w= \small \left(
        \begin{matrix}0&0&1\\0&-1&0 \\1&0&0\end{matrix} \right) \normalsize \in N^{\mathrm{sph}} , \text{ and } u=\small \left(
        \begin{matrix}1&0&1/t\\0&1&1/t \\0&0&1\end{matrix} \right) \in \mathbf{U}(k)^{+} \normalsize. $$
We claim that $M_{\Psi}(h)$ is not equal to the product of two fractional ideals. Indeed, it is easy to check that
\small
$$ h \mathbf{U}_{\Psi} h^{-1}=  \left\{ 
         \left(
        \begin{matrix}1+z/t& -x/t& 1/t \cdot (x/t-z/t)
        \\z/t&1-x/t& 1/t \cdot (x/t-z/t) \\
        z&-x&1+x/t-z/t\end{matrix} \right)  : x,z \in \mathbb{A}^1\right\}. $$
\normalsize
Thus, we deduce that $h \mathbf{U}_{\Psi} h^{-1} \in \mathbf{G}(\mathbb{F}[t])$ if and only if $x, z, x/t-z/t \in t \mathbb{F}[t]$. Set $\mathfrak{q}_{\alpha}=\mathfrak{q}_{\alpha+\beta}= t \mathbb{F}[t]$. Then, we have that
\begin{multline*}
\quad\quad\quad\quad\quad \quad M_{\Psi}(h)= \left\lbrace (x,z) \in \mathfrak{q}_{\alpha} \times \mathfrak{q}_{\alpha+\beta}  : x-z \in t^2\mathbb{F}[t]\right\rbrace, \\
 = \left\lbrace (tx_0,tz_0) \in \mathfrak{q}_{\alpha} \times \mathfrak{q}_{\alpha+\beta} :  x_0-z_0 \in t\mathbb{F}[t]\right\rbrace. \quad \quad\quad\quad\quad\quad    
\end{multline*}
Now, note that for each $tx_0 \in  \mathfrak{q}_{\alpha} $, we have that $z_0:=x_0 $ satisfies $(tx_0, tz_0)\in M_{\Psi}(h)$. Then, we obtain $\pi_{\alpha}(M_{\Psi}(h))=\mathfrak{q}_{\alpha}$. Analogously, we deduce $\pi_{\alpha+\beta}(M_{\Psi}(h))=\mathfrak{q}_{\alpha+\beta}$. But $(t(t+1), t^2) \in \mathfrak{q}_{\alpha} \times \mathfrak{q}_{\alpha+\beta} \smallsetminus M_{\Psi}(h)$. Thus, the claim follows.

\end{example}

\subsection{Lower bound}

\begin{lemma}\label{lemma polynomial ideal contained}
Let $\big( P_i(t) \big)_{1 \leq i \leq r} \subset k[t]$ be a finite family of polynomials satisfying $P_i(0)=0$, for all $i \in \llbracket 1,r \rrbracket$.
Let $\big( I_i \big)_{1 \leq i \leq r}$ be a finite family of non-zero (resp. fractional) $A$-ideal.
Set $ M= \lbrace y \in k: P_i(y) \in I_i,\ \forall i \in \llbracket 1,r \rrbracket  \rbrace.$
Then, there exists a non-zero (resp. fractional) $A$-ideal $J$ of $A$ contained in $M$.
\end{lemma}

\begin{proof}
Without loss of generality, we can assume that every $P_i$ is non-zero.
Thus, each polynomial $P_i(t)$ can be written as the quotient of a polynomial $Q_i(t)= a_{1,i}t + \cdots + a_{d(i),i} t^{d(i)} \in A[t]$ by a non-zero element $b_i \in A \smallsetminus \{0\}$. For each $i \in \llbracket 1,r \rrbracket$, we set $J_i= b_i I_i$. Then, we get $y \in M$ if and only if $Q_i(y) \in J_i$, for all $i \in \llbracket 1,r \rrbracket$. In particular, we have that $J= \bigcap_{i=1}^{r} J_i$ is contained in $M$.
Since $A$ is assumed to an integral domain, $J$ is non-zero as finite intersection of non-zero fractional $A$-ideals. Thus the result follows.
\end{proof}

\begin{proposition}\label{prop ideal contained}
For every $\alpha \in \Phi$ and every $h \in \mathbf{G}(k)$, there exists a non-zero ideal $J_\alpha(h)$ of $A$ such that $J_\alpha(h) \subseteq M_\alpha(h)$.
\end{proposition}

\begin{proof}
For $1 \leq i,j \leq n$, denote by $\pi_{i,j} : \mathrm{SL}_{n,\mathbb{Z}} \to \mathbb{A}^1_{\mathbb{Z}}$ the canonical $\mathbb{Z}$-morphism of projection onto the $(i,j)$-coordinate.
Consider the composite $\mathbb{Z}$-morphism:
\[ f_\alpha^{i,j} := \pi_{i,j} \circ \rho \circ \theta_\alpha : \mathbb{G}_{a,\mathbb{Z}} \to \mathbb{A}^1_{\mathbb{Z}}.
\]
It corresponds to a ring morphism $\big(f_\alpha^{i,j}\big)^*: \mathbb{Z}[t] \to \mathbb{Z}[t]$. Let us define
\[ Q_{i,j}(t) := \big(f_\alpha^{i,j}\big)^*(t) - \delta_{i,j} \in \mathbb{Z}[t].\]

Thus, for any $y \in k$, we have that $\rho \circ \theta_\alpha(y) = I_n + \big( Q_{i,j}(y) \big)_{1 \leq i,j \leq n} \in \mathrm{SL}_n(k)$
where $Q_{i,j}(0) = 0$ since $\rho \circ \theta_\alpha(0) = I_n$.
Hence there are matrices $\big( N_i \big)_{1 \leq i \leq d} \in \mathcal{M}_n(\mathbb{Z})$ for $d = \max \{ \operatorname{deg}( Q_{i,j}),\ 1 \leq i,j \leq n \}$ such that
\[ \rho \circ \theta_\alpha(y) = I_n + \sum_{i=1}^d y^i N_i\qquad \forall y \in \mathbb{G}_a.\]
Set $U = \rho(h) \in \mathrm{SL}_n(k)$ and $V = U^{-1} = \rho(h^{-1}) \in \mathrm{SL}_n(k)$.
Then
\[ \rho( h^{-1} \theta_\alpha(y) h) = I_n + \sum_{i=1}^d y^i V N_i U\]
Let $\big( P_{i,j} \big)_{1 \leq i,j \leq n}$ be the family of polynomials in $t k[t]$ such that 
\[\sum_{i=1}^d y^i V N_i U = \big( P_{i,j}(y) \big)_{1 \leq i,j \leq n} \]
Applying Lemma~\ref{lemma polynomial ideal contained} to the family $\big( P_{i,j} \big)_{1 \leq i,j \leq n}$ and ideals $I_{i,j}= A$, we get an $A$-ideal $J$ such that $y \in J \Rightarrow P_{i,j}(y) \in A$.
Thus, by setting $J_\alpha(h) = J$, we get that
\[ \rho \big( h^{-1} \theta_\alpha( J ) h \big) \subseteq \mathrm{SL}_n(k) \cap \mathcal{M}_n(A) \cap \rho\big(\mathbf{G}(k)\big) = \mathrm{SL}_n(A) \cap \rho\big(\mathbf{G}(k)\big) = \rho\big( \mathbf{G}(A) \big).\]
Since $\rho$ is faithful, we deduce that
\[\theta_\alpha\big( J_\alpha(h)\big) \subseteq h \mathbf{G}(A) h^{-1} \cap \mathbf{U}_\alpha(k) = \theta_\alpha\big( M_\alpha(h) \big)\]
whence the result follows since $\theta_\alpha$ is an isomorphism.
\end{proof}

\subsection{Upper bound for \texorpdfstring{$\mathrm{SL}_n$ and $\mathrm{GL}_n$}{SL(n) and GL(n)}}\label{sec upper bound SLn}

Let $n \in \mathbb{N}$. 
In this subsection only, we assume that $\mathbf{G} = \mathrm{SL}_{n}$ or $ \mathrm{GL}_{n}$ over $\mathbb{Z}$ and we consider its usual matrix realization.
Let $\mathbf{T} \subset \mathbf{G}$ be the maximal $k$-torus consisting in diagonal matrices.
Recall that roots of $\mathbf{G}$ with respect to $\mathbf{T}$ are the characters $\chi_{i,j} : \mathbf{T} \to \mathbb{G}_m \in \Phi(\mathbf{G},\mathbf{T})$ given by \[\chi_{i,j} (\begin{pmatrix} d_1 & & \\ & \ddots & \\&&d_n \end{pmatrix})= d_i d_j^{-1}\] for $i \neq j$.
Let $\mathbf{B}$ (resp. $\mathbf{U}$) be the closed solvable (resp. unipotent) subgroup scheme of $\mathbf{G}$ consisting in upper triangular matrices.
Denote by $E_{i,j}$ the matrix with coefficient $1$ at position $(i,j)$ and $0$ elsewhere.
Then the group scheme homomorphism $e_{i,j} : \mathbb{G}_a \to \mathbf{G}$ given by $e_{i,j}(x) = I_n + x E_{i,j}$ parametrizes the root group $U_{i,j}$ of $\mathbf{G}$ with respect to $\chi_{i,j}$.
Recall that, for any ordering of $\lbrace (i,j): 1\leq i<j \leq n\rbrace$, there is an isomorphism of $k$-varieties:
\[\begin{array}{cccc}\phi: &\mathbb{A}^{\frac{n(n-1)}{2}}_k &\to& \mathbf{U}_k\\
&(x_{i,j})_{1 \leq i < j \leq n} & \to & \prod_{1\leq i<j \leq n} e_{i,j}(x_{i,j})
\end{array}.\]

\begin{lemma}\label{lemma upper bound SL(n,A)}
Assume that $\mathbf{G} = \mathrm{GL}_n$ or $\mathrm{SL}_n$.
Let $A$ be a Dedekind domain and $k$ be its fraction field.
For any $h \in \mathbf{G}(k)$.
There are non-zero fractional $A$-ideals $\big( J_{i,j}(h) \big)_{1 \leq i < j \leq n}$ (that does not depend on the ordering of the product) such that 
\[ \prod_{1 \leq i < j \leq n} e_{i,j}(x_{i,j}) \in h \mathbf{G}(A) h^{-1} \Rightarrow \forall 1 \leq i < j \leq n,\ x_{i,j} \in J_{i,j}(h)\]
\end{lemma}

\begin{proof}
Recall that, in a Dedekind domain, a finite sum and product of non-zero fractional ideals is also a non-zero fractional ideal (c.f.~\cite[\S 6]{Lang}). We use this fact along the proof. 
In the following, given fractional $A$-ideals $I_{i,j}$ for $1 \leq i,j \leq n$, we denote by $\big( I_{i,j} \big)_{i,j=1}^n$ the subset of matrices in $\mathcal{M}_n(k)$ whose coefficient at entry $(i,j)$ is contained in $I_{i,j}$.
Let us write $h = \big( h_{i,j} \big)_{i,j=1}^n$ and $h^{-1} = \big( h'_{i,j} \big)_{i,j=1}^n$
Then, an elementary calculation shows that
\[h \mathbf{G}(A) \subseteq h \mathcal{M}_n(A) \subseteq (I'_{i,j})_{i,j=1}^{n},\]
where, for any $(i,j)$, the ideal $I'_{i,j}= \sum_{\ell=1}^{n} h_{i,\ell} A$ is a non-zero fractional $A$-ideal as finite sum of fractional $A$-ideals with at least one non-zero since there is $\ell$ such that $h_{i,\ell} \neq 0$ because $h \in \mathbf{G}(k)$.
Thus, an analogous elementary calculation shows that
\begin{equation}\label{eq conj SL(n,A) in ideals}h \mathbf{G}(A) h^{-1} \subseteq \big( I'_{i,j} \big)_{i,j=1}^n h^{-1} \subseteq \big( I_{i,j} \big)_{i,j=1}^n,\end{equation}
where, for any $(i,j)$, the ideal $I_{i,j}= \sum_{\ell=1}^{n} (h'_{\ell,j}) I'_{i,\ell}$ is a non-zero fractional $A$-ideal as finite sum of fractional $A$-ideals with at least one non-zero.
Since for $i<j$ and $i'<j'$ we have $E_{i,j}E_{i',j'}$ equals $E_{i,j'}$ with $i <j'$, if $j=i'$, and $0$ otherwise, a straightforward matrix calculation gives that
\begin{equation}\label{eq straightforward in U(n)}\prod_{1 \leq i < j \leq n} (I_n + x_{i,j} E_{i,j}) = I_n + \sum_{1 \leq i < j\leq n} \Big( x_{i,j} + \sum_{\substack{2 \leq s \leq j-i\\i = \ell_0 < \cdots < \ell_s=j}} \big( \varepsilon_{\ell_0,\dots,\ell_s} \prod_{t=1}^s x_{\ell_{t-1},\ell_t} \big) \Big) E_{i,j}\end{equation}
where each $\varepsilon_{\ell_0,\dots,\ell_s} \in \{0,1\}$ depends on the ordering of the factors $(I_n + x_{i,j} E_{i,j})$ in the product.
Hence Inclusion~\eqref{eq conj SL(n,A) in ideals} and Equation~\eqref{eq straightforward in U(n)} give that
\begin{multline}\label{implication straightforward}
    \prod_{1 \leq i < j \leq n} (I_n + x_{i,j} E_{i,j}) \in g \mathbf{G}(A) g^{-1}  \Rightarrow 
   \\\bigg( \forall 1 \leq i < j \leq n,\ \Big( x_{i,j} + \sum_{\substack{2 \leq s \leq j-i\\i = \ell_0 < \cdots < \ell_s=j}} \big( \varepsilon_{\ell_0,\dots,\ell_s} \prod_{t=1}^s x_{\ell_{t-1},\ell_t}\big) \Big) \in I_{i,j}\bigg).
\end{multline}
We define, by induction on $r:=j-i$ for $1 \leq i < j \leq n$, fractional $A$-ideals by  
\[J_{i,j} := I_{i,j} + \Big( \sum_{\substack{2 \leq s \leq j-i\\i = \ell_0 < \cdots < \ell_s=j}} \prod_{t=1}^s J_{\ell_{t-1},\ell_t}\Big)\]
By induction on $r=j-i$, we obtain that $J_{i,j}$ is a non-zero fractional $A$-ideal as finite sum and product of such fractional $A$-ideals.
With this definition of $J_{i,j}$, the result then follows from~\eqref{implication straightforward}.
\end{proof}

\subsection{Upper bound: general case}\label{sec upper bound}

\begin{lemma}\label{lemma monomials in Dedekind}
Let $A$ be a Dedekind domain and let $k$ be its fraction field.
Let $J$ be a non-zero fractional $A$-ideal and $P \in k[t]$ a non-constant monomial.
There is a non-zero fractional $A$-ideal $\mathfrak{q}$ such that for any $x \in k$
\[ P(x) \in J \Rightarrow x \in \mathfrak{q}.\]
\end{lemma}

\begin{proof}
Let us write $P(t)= z t^n$, where $n \in \mathbb{N}$ and $z \in k^*$. Since $A$ is a Dedekind domain, there exist $s \in \mathbb{N}$, non-trivial pairwise distinct prime ideals $P_i$ and non-zero integers $a_i \in \mathbb{Z} \setminus \{0\}$, for $i \in \llbracket 1,s\rrbracket$, such that we can write $J \cdot (z)^{-1}   = P_1^{a_1} \cdots P_s^{a_s}$. See~\cite[Page 20]{Lang} for details.
For each $i \in \llbracket 1,s \rrbracket$, set $b_i \in \mathbb{Z}$ such that $\lceil a_i/n \rceil \geq b_i$.

Let $x \in k$ be an element satisfying that $P(x) = zx^n \in J$. The latter condition is equivalent to:
$$(z) \cdot (x) ^n =(  zx^n ) \subseteq J.$$
In other words, we get that  $(x) ^n \subseteq J \cdot (z) ^{-1}   = P_1^{a_1} \cdots P_s^{a_s}.$
Then, let us write $(x) =  Q_1^{c_1} \cdots Q_r^{c_r},$ where $r \in \mathbb{N}$, $\big( Q_j \big )_{1 \leq j\leq r}$ is a family of pairwise distinct non-trivial prime ideals of $A$, and $\lbrace c_j \rbrace_{j=1}^r \subset \mathbb{Z} \smallsetminus \{0\}$ is a set of non-zero integers.
Therefore, we get that:
\begin{equation}\label{eq ideals}
  Q_1^{n c_1} \cdots Q_r^{n c_r} \subseteq  P_1^{a_1} \cdots P_s^{a_s}.  
\end{equation}
This implies that $\{P_1,\dots,P_s\} \subseteq \{Q_1,\dots,Q_r\} $ and, in particular, that $s\leq r$. See \cite[Page 21]{Lang} for details. Thus, up to exchanging the numbering of these prime ideals, we can assume that, for each $i \in \llbracket 1,s \rrbracket$, we have $P_i=Q_i$. Moreover, the inclusion~\eqref{eq ideals} implies that, for any $i \in \llbracket 1,s \rrbracket$, we have $c_i \geq a_i/n$, whence we deduce $c_i \geq b_i$.
In particular, we obtain that $$x \in (x) =  Q_1^{c_1} \cdots Q_r^{c_r} \subseteq P_1^{c_1} \cdots P_s^{c_s}  \subseteq  P_1^{b_1} \cdots P_s^{b_s}.$$
Thus, the result follows by taking $\mathfrak{q} =P_1^{b_1} \cdots P_s^{b_s}$ which does not depend on $x$.
\end{proof}

The following proposition generalizes Lemma~\ref{lemma upper bound SL(n,A)} under a suitable assumption on $\Psi$.

\begin{proposition}\label{prop fractional ideals}
Let $A$ be a Dedekind domain and $k$ be its fraction field.
Let $\Psi \subset \Phi^+$ be a closed subset of roots satisfying condition~\ref{cond commutative} which is a linearly independent family in $\operatorname{Vect}_\mathbb{R}(\Phi)$.
Let $h \in H$ and let $M_\Psi(h)$ as defined by equation~\eqref{eq def MPsi}.
There are non-zero fractional $A$-ideals $\mathfrak{q}_{\Psi,\alpha}(h)$ such that
\[
\big( x_\alpha \big)_{\alpha \in \Psi} \in M_\Psi(h) \Rightarrow \forall \alpha \in \Psi,\ x_\alpha \in \mathfrak{q}_{\Psi,\alpha}(h).\]
\end{proposition}

\begin{proof}
At first we provide a suitable faithful linear representation of $\mathbf{G}_k$ which is a closed embedding. Let $\chi_{ij}$, $\mathbf{T}_n$, $\mathbf{U}_n$, $E_{ij}$, $e_{ij}$ and $U_{ij}$ be defined as in the beginning of \S \ref{sec upper bound SLn}. Recall that, for any ordering of $\lbrace (i,j): 1\leq i<j \leq n\rbrace$, there is an isomorphism of $k$-varieties:
\[\begin{array}{cccc}\phi: &\mathbb{A}^{\frac{n(n-1)}{2}}_k &\to& \mathbf{U}_n\\
&(z_{i,j})_{1 \leq i < j \leq n} & \to & \prod_{1\leq i<j \leq n} e_{i,j}(z_{i,j})
\end{array}.\]

Let $(\mathbf{B}, \mathbf{T})$ be a Killing couple of $\mathbf{G}$. By conjugacy over a field of Borel subgroups, there is an element $\tau_1 \in \mathrm{SL}_n(k)$ such that $\tau_1 \rho\big( \mathbf{B} \big) \tau_1^{-1} \subset \mathbf{T}_n$ \cite[20.9]{BoA}.
Let $T_1$ be a maximal $k$-torus of $\mathbf{T}_n$ containing the $k$-torus $\tau_1 \rho\big( \mathbf{T} \big) \tau_1^{-1}$
By conjugacy over a field $k$ of Levi factors \cite[20.5]{BoA}, there is an element $\tau_2 \in \mathbf{U}_n(k) = \mathcal{R}_u(\mathbf{T}_n)(k)$ such that $\tau_2 T_1 \tau_2^{-1} \subset \mathbf{D}_n$.
Thus, we deduce a faithful linear representation over $k$ which is a closed embedding:
\[\begin{array}{cccc} \rho_1: & \mathbf{G}_k & \to & \mathrm{SL}_{n,k}\\& g & \mapsto & \tau \rho(g) \tau^{-1}\end{array}\]
where $\tau = \tau_2 \tau_1 \in \mathbf{G}(k)$ so that $\rho_1(\mathbf{B}) \subset \mathbf{T}_n$ and $\rho_1(\mathbf{T}) \subset \mathbf{D}_n$.

Write $\Psi = \{\alpha_1,\dots,\alpha_m\}$ for $m \in \mathbb{N}$.
Since $\Psi$ satisfies (C2) the morphism $\psi = \prod_{i=1}^m \theta_{\alpha_i}$ is an isomorphism of $k$-group schemes $\big( \mathbb{G}_a \big)^m \to \mathbf{U}_\Psi$. This isomorphism
For $1 \leq i < j \leq n$, we can define morphisms $f_{i,j}$ of $k$-varieties, from $\mathbb{A}_k^m$ to $\mathbb{A}^1_k$, by composing:
$$
\big( \mathbb{G}_{a,k}\big)^m \xrightarrow{\psi} \mathbf{U}_{\Psi} \xrightarrow{\rho_1} \mathbf{U}_n \xrightarrow{\phi^{-1}} \mathbb{A}^{\frac{n(n-1)}{2}}_k \xrightarrow{\pi_{i,j}} \mathbb{A}^1_k,
$$
where $\pi_{i,j}$ is the canonical projection onto the coordinate $(i,j)$.
Let $f_{i,j}^{*}:k[t] \rightarrow k[X_1,\dots,X_m]$ be the $k$-algebra homomorphism associated to $f_{i,j}$, and let 
$$r_{i,j}(X_1,\dots,X_m)=f_{i,j}^{*}(t) \in k[X_1,\dots,X_m].$$

Thus, there are polynomials $r_{i,j} \in k[X_1,\dots,X_m]$ such that for any $(x_1,\dots,x_m) \in \big( \mathbb{G}_{a,k} \big)^m$, we have that
\begin{equation}\label{eq param by polynomials}
    \rho_1 \circ \psi (x_1,\dots,x_m) = \prod_{1 \leq i<j \leq n} e_{i,j}\big(r_{i,j}(x_1,\dots,x_m)\big).
\end{equation}

\textbf{First claim:} The $r_{i,j}$ are, in fact, monomials.

For $i \in \llbracket 1,m\rrbracket$, since $\theta_{\alpha_i}$ parametrizes a root group, we have, for any $t \in \mathbf{T}$ and $x_i \in \mathbb{G}_a$, that
\[ t \cdot \theta_{\alpha_i}(x_i) \cdot t^{-1} = \theta_\alpha\big(\alpha_i(t) x_i\big).\]
Applying $\rho_1$ to these equalities, we get that
\[\forall t \in \mathbf{T},\quad \rho_1(t) \cdot \rho_1 \circ \psi(x_1,\dots,x_m) \cdot \rho_1(t)^{-1} = \rho_1 \circ \psi\big(\alpha_1(t)x_1,\dots,\alpha_m(t)x_m\big).\]
Hence, for the given ordering of positive roots of $\mathrm{SL}_n$, we have on the one hand that for any $t \in \mathbf{T}$
\[\prod_{1\leq i<j\leq n} \rho_1(t) \cdot e_{i,j}\big( r_{i,j}(x_1,\dots,x_m) \big) \cdot \rho_1(t)^{-1}\\
    = \prod_{1\leq i<j \leq n} e_{i,j}\Big( r_{i,j}\big( \alpha_1(t) x_1,\dots,\alpha_m(t) x_m \big)\Big),\]
and on the other hand, since $\rho_1(t) \in \mathbf{D}_n$, that for any $t \in \mathbf{T}$
\[\prod_{1\leq i<j\leq n} \rho_1(t) \cdot e_{i,j}\big( r_{i,j}(x_1,\dots,x_m) \big) \cdot \rho_1(t)^{-1}\\
    = \prod_{1\leq i<j \leq n} e_{i,j}\Big( \chi_{i,j}\big( \rho_1(t) \big) r_{i,j}\big( x_1,\dots,x_m \big)\Big).\]
Hence, the isomorphism $\phi$ provides the equalities
\[\chi_{i,j} \circ \rho_1(t) \cdot r_{i,j}(x_1,\dots,x_m) = r_{i,j}\big( \alpha_1(t) x_1,\dots,\alpha_m(t)x_m \big)\]
for any $1\leq i<j\leq n$, any $(x_1,\dots,x_m) \in \big( \mathbb{G}_{a,k}\big)^m$ and any $t \in \mathbf{T}$.

Fix $1\leq i<j\leq n$.
Write the polynomial 
\[r_{i,j} = \sum_{\underline{\ell}=(\ell_1,\dots,\ell_m) \in \mathbb{Z}_{\geq 0}^m} a_{\underline{\ell}} X_1^{\ell_1} \cdots X_m^{\ell_m} \in k[X_1,\dots,X_m]\]
Then the polynomial identity
\[\sum_{\underline{\ell} \in \mathbb{Z}_{\geq 0}^m} \chi_{i,j} \circ \rho_1(t) a_{\underline{\ell}} x_1^{\ell_1} \cdots x_m^{\ell_m} = \sum_{\underline{\ell} \in \mathbb{Z}_{\geq 0}^m} a_{\underline{\ell}} \alpha_1(t)^{\ell_1} \cdots \alpha_m(t)^{\ell_m} x_1^{\ell_1} \cdots x_m^{\ell_m}\]
for any $(x_1,\dots,x_m) \in \big( \mathbb{G}_{a,k} \big)^m$ and $t \in \mathbf{T}$ allows us to identify the coefficients
\begin{equation}\label{eq on each factor}
a_{\underline{\ell}} \cdot \chi_{i,j} \circ \rho_1(t) =a_{\underline{\ell}} \cdot \alpha_1(t)^{\ell_1} \cdots \alpha_m(t)^{\ell_m}
\end{equation} 
for any $t \in \mathbf{T}$ and any $\underline{\ell} = (\ell_1,\dots,\ell_m) \in \mathbb{Z}_{\geq 0}^m$.

Suppose that $a_{\underline{\ell}} \neq 0$.
Since equation~\eqref{eq on each factor} holds for any $t \in \mathbf{T}$ and $\chi_{i,j} \circ \rho : \mathbf{T} \to \mathbf{G}_{m,k}$ is a character on $\mathbf{T}$, we deduce the equation in the module of characters $X^*(\mathbf{T})$ (with additive notation):
\[
    \chi_{i,j} \circ \rho_1 = \ell_1 \alpha_1 + \dots + \ell_m \alpha_m
\]
Since $\Psi$ is a linearly independent family, the elements $\ell_1,\dots,\ell_m$ are uniquely determined.
Hence $r_{i,j}$ is a monomial whence the first claim follows.

\textbf{Second claim:}
For any $\ell \in \llbracket 1,m\rrbracket$, there is an index $(i_\ell,j_\ell)$ and a non-constant monomial $R_\ell \in k[t]$ such that $r_{i_\ell,j_\ell}(x_1,\dots,x_m) = R_\ell(x_\ell)$.

Let $\ell \in \llbracket 1,m\rrbracket$. For $1 \leq i < j \leq n$, define $R_{i,j,\ell}(t) = r_{i,j}(0,\dots,0,t,0,\dots,0) \in k[t]$ where the non-zero coefficient is at coordinate $\ell$.
By the first claim, $R_{i,j,\ell}$ is a monomial and there is a $k$-group schemes homomorphism
\[
    \begin{array}{cccc}
        \rho_1 \circ \theta_{\alpha_\ell} :& \mathbb{G}_{a,k}& \to& \mathbf{U}_n\\
        & x & \mapsto & \prod_{1 \leq i < j \leq n} e_{i,j}\big( R_{i,j,\ell}(x) \big)
    \end{array}
\]
deduced from Equation~\eqref{eq param by polynomials}.
Since it is non-constant by composition of closed embedding, there is an index $(i_\ell,j_\ell)$ such that $R_{i_\ell,j_\ell,\ell}$ is non-constant.
But since $r_{i_\ell,j_\ell} \in k[X_1,\dots,X_m]$ is a monomial, we deduce that $$ r_{i_\ell,j_\ell}(X_1,\dots,X_m) = r_{i_\ell,j_\ell}(0,\dots,X_\ell,\dots,0) = R_{i_\ell,j_\ell,\ell}(X_\ell),$$
whence the second claim follows.

As a consequence, we deduce from equations~\eqref{eq def MPsi} and~\eqref{eq param by polynomials} that
\begin{align*}
    (x_1,\dots,x_m) \in M_\Psi(h) \Rightarrow & \rho_1 \circ \psi(x_1,\dots,x_m) \in \rho_1(h) \rho_1\big( \mathbf{G}(A) \big) \rho_1(h^{-1})\\
    \Rightarrow & \prod_{1 \leq i < j \leq n} e_{i,j}\big( r_{i,j}(x_1,\dots,x_m) \big) \in \rho_1(h) \tau \mathrm{SL}_n(A) \tau^{-1} \rho_1(h)^{-1}
\end{align*}
Applying Lemma~\ref{lemma upper bound SL(n,A)} with $g=\rho_1(h) \tau \in \mathrm{SL}_n(k)$, we get that there are non-zero fractional $A$-ideals $\big( J_{i,j} \big)_{1 \leq i < j \leq n}$ such that
\[ (x_1,\dots,x_m) \in M_\Psi(h) \Rightarrow r_{i,j}(x_1,\dots,x_m) \in J_{i,j}.\]
Applying Lemma~\ref{lemma monomials in Dedekind}, for each $\ell \in \llbracket 1,m\rrbracket$, to the monomial $ r_{i_\ell,j_\ell}$ given by the second claim and to the non-zero fractional $A$-ideal $J_{i_\ell,j_\ell} $, we deduce that there are non-zero fractional $A$-ideals $\big( \mathfrak{q}_\ell \big)_{1 \leq \ell \leq m}$ such that
\[ (x_1,\dots,x_m) \in M_\Psi(h) \Rightarrow x_\ell \in \mathfrak{q}_\ell,\quad \forall \ell \in \llbracket 1,m\rrbracket.\]
We conclude by setting $\mathfrak{q}_{\Psi,\alpha_\ell}(h) := \mathfrak{q}_\ell$ for $1 \leq \ell \leq m$.
\end{proof}

\section{Image of sector faces in \texorpdfstring{$\mathbf{G}(A) \backslash \mathcal{X}_k$}{G(A)\textbackslash Xk}}\label{section structure of X}

 In order to describe the images of sector faces in the quotient space $\mathbf{G}(A) \backslash \mathcal{X}_k$, we will consider different foldings via some elements of root groups. 

\subsection{Specialization of vertices and scalar extension}~

 Along this section, in order to have enough foldings via unipotent elements, we have to introduce several well-chosen coverings $\mathcal{D} \to \mathcal{C}$, their associated ring extensions $B/A$ and field extensions $\ell/k$, and their associated embeddings of buildings $\mathcal{X}_k \to \mathcal{X}_\ell$.
In return, it will provide information of the action of $\mathbf{G}(A)$ on $\mathcal{X}_k$ by considering the action of $\mathbf{G}(B)$ over $\mathcal{X}_\ell$. Thus, we start by introducing the following arithmetical result that we extensively use along this section.

\begin{lemma}\label{lem curve extension}
Let $\mathcal{C}$ be a smooth projective curve over $\mathbb{F}$ and let $A$ be the ring of functions on $\mathcal{C}$ that are regular outside $\lbrace \p \rbrace$. Let $e$ be a positive integer. Then, there exist a finite extension $\mathbb E / \mathbb F$, a smooth projective curve $\mathcal{D}$ defined over $\mathbb E$, a cover $\phi: \mathcal{D} \to \mathcal{C}$, and a unique closed point $\p'$ of $\mathcal{D}$ over $\p$ such that:
\begin{itemize}
    \item the ring $B$ of functions on $\mathcal{D}$ that are regular outside $\lbrace \p' \rbrace$ is an extension of $A$,
    \item the field $\ell= \mathrm{Frac}(B)$ is an extension of $k$ of degree $e$, which is totally ramified at $\p$ and
    \item $\mathbb E$ is contained in the algebraic closure $\tilde{\mathbb{F}}$ of $\mathbb F$ in $k$.
\end{itemize}
\end{lemma}

\begin{proof}
Let $\pi_{\p} \in \mathcal{O}_{\p}$ be a local uniformizing parameter.
Since $k$ is dense in the completion $K=k_{\p}$ we can assume that $\pi_{\p}$ belongs to $k$. Set $p(x)=x^{e}-\pi_\p$. This is an Eisenstein polynomial at $\p$, whence it is irreducible over $K$, and then also over $k$. Thus, the rupture field $\ell$ of $p(x)$ is an extension of $k$ of degree $e$. Moreover, $\ell$ is a function field, since it is a finite extension of $k$ (c.f.~\cite[3.1.1]{Stichtenoth}). Hence, there exists a finite extension $\mathbb{E}$ of $\mathbb{F}$, and a smooth projective curve $\mathcal{D}$ defined over $\mathbb{E}$, such that $\ell= \mathbb{E}(\mathcal{D})$. Moreover, since $k \subseteq \ell$, there exists a branched covering $\phi: \mathcal{D} \to \mathcal{C}$. Recall that, at each point $Q$ of $\mathcal{C}$ we have that
$$ e = [\ell:k]= \sum_{Q'/Q} e(\ell_{Q'}/k_Q) f(\ell_{Q'}/k_Q),$$
where $Q'$ is a point over $Q$, $e(\ell_{Q'}/k_Q) $ is the local ramification index at $Q'$, and $f(\ell_{Q'}/k_Q)$ is the local inertia degree at $Q'$. This identity shows that there exists a unique point $\p'$ of $\mathcal{D}$ over $\p$, and that $\ell_{\p'}$ is a totally ramified extension of $k_{\p}$.

Set $E=\mathbb{E}(\mathcal{C})$. Then, $E $ is a function field satisfying that $k \subseteq E \subseteq \ell$. Note that $ [E:k] \leq [\mathbb{E}:\mathbb{F}]$, with the equality whenever $k$ is purely transcendental over $\mathbb{F}$. Since there exists exactly one point $\p'$ in $\mathcal{D}$ over $\p$, we get that there exists precisely one point $\p''$ of $E$ over $\p$. Moreover, it follows from \cite[Theo. 3.6.3, \S 3.6]{Stichtenoth}
that $E$ is an unramified extension of $k$ at each closed point. In particular, this is unramified at $\p''$ over $\p$. Thus, we have that $f(E_{\p''}/k_{\p})=[E:k]$. So, since $f(E_{\p''}/k_{\p})$ divides $f(\ell_{\p'}/k_{\p})=1$, we get that $[E:k]=1$. This implies that $E=k$, whence $\mathbb{E} \subset k$. So, since each element of $\mathbb{E}$ is algebraic over the ground field $\mathbb{F}$ of $k$, we finally conclude that $\mathbb{E} \subseteq \tilde{\mathbb{F}}$. 

Now, let $B$ be the integral closure of $A$ in $\ell$. Recall that $A$ can be characterized as the ring of functions of $\mathcal{C}$ with a positive valuation at each closed point different from $\p$. Then $B$ is the ring of functions of $\mathcal{D}$ with a positive valuation at each closed point different from $\p'$. This implies that $B$ contains only functions that are regular outside $\lbrace \p' \rbrace$. Hence, the result follows.
\end{proof}

\begin{definition}
A point of a building $\mathcal{X}(\mathbf{G},k)$ is called a $k$-center\index{$k$-center} if it is the arithmetic mean of the vertices of a $k$-face of $\mathcal{X}(\mathbf{G},k)$ as points in a real affine space.

For instance, $k$-vertices and middle of $k$-edges are $k$-centers.
\end{definition}

\begin{lemma}\label{lemma becomes special}
There exists a finite extension $\ell$ of $k$, which is totally ramified at $\p$, such that all the $k$-centers (e.g.~$k$-vertices) of $\mathcal{X}(\mathbf{G},k)$ are in the same $\mathbf{G}(\ell)$-orbit.
In particular, $k$-centers are special $\ell$-vertices (see Figure~\ref{figure walls extension}).
\end{lemma}

\begin{proof}
Because the valuation $\nu_{\p}$ is discrete, the set of $k$-vertices $V_{0,k}:=\operatorname{vert}(\mathbb{A}_0,k)$ forms a full lattice in the finite dimensional $\mathbb{R}$-vector space $\mathbb{A}_0$.
By construction, $\mathcal{N}_\mathbf{G}(\mathbf{T})(k)$ is the stabilizer of $\mathbb{A}_0$ in $\mathbf{G}(k)$ \cite[13.8]{L} and the subgroup $\mathbf{T}(k)$ acts by translations on $\mathbb{A}_0$, stabilizes $V_{0,k}$ and the set of these translations forms a sublattice $\Lambda_{0,k}$ of $V_{0,k}$ \cite[1.3,1.4]{L}.
Moreover, for any $k$-face $F$ with $d_F$ vertices, the $k$-center of $F$ is contained in $\frac{1}{d_F} V_{0,k}$ as arithmetic mean of $d_F$ points.
Let $V_{1,k} \subseteq \frac{1}{N} V_{0,k}$ be the lattice spanned by the $k$-centers of the $k$-faces of $\mathbb{A}_0$, where $N$ is the least common multiple of the $d_F$. We observe that $\Lambda_{0,k}$ is a cocompact sublattice of $V_{1,k}$.
Thus, there is a positive integer $e \in \mathbb{N}$ such that $V_{1,k}$ becomes a sublattice of $\frac{1}{e} \Lambda_{0,k}$.
Let $\mathcal{D}$ be a curve defined over a finite extension $\mathbb E / \mathbb F$, as in Lemma~\ref{lem curve extension}, and let $B/A$ be the corresponding ring extension.
Let $\ell = \operatorname{Frac}(B)$ be the fraction field of $B$.
Since there exists a unique closed point $\p'$ of $\mathcal{D}$ over $\p$, we get that there exists a unique discrete valuation $\nu_{\p'}$ on $\ell$ extending $\nu_{\p}$ on $k$.
Moreover, since the local ramification index is $e(\ell_{\p'}/k_{\p})=e$, we have that $\nu_{\p'}(\ell^\times) = \frac{1}{e} \nu_{\p}(k^\times)$.
Considering the canonical embedding $\mathcal{X}(\mathbf{G},k) \hookrightarrow \mathcal{X}(\mathbf{G},\ell)$ as defined in \S~\ref{intro rational building},
we claim that the $k$-centers of $\mathcal{X}(\mathbf{G},k)$ are in a single $\mathbf{G}(\ell)$-orbit in $\mathcal{X}(\mathbf{G},\ell)$.
Indeed, $\mathbf{G}(k)$ acts transitively on the set of $k$-apartments and any $k$-vertex is contained in some $k$-apartment by definition. Since $\mathbf{G}(k)$ is a subgroup of $\mathbf{G}(\ell)$, it suffices to prove it for $k$-centers of the standard $k$-apartment $\mathbb{A}_{0,k}$.
By construction of the Bruhat-Tits building and the action of $\mathbf{G}(\ell)$ on it, the group $\mathbf{T}(\ell)$ acts by translations on the standard $\ell$-apartment $\mathbb{A}_{0,\ell}$ and the set of translations is precisely $\Lambda_{0,\ell} = \frac{1}{e} \Lambda_{0,k}$ \cite[5.1.22]{BT2}.
Thus, all $k$-centers of the image of $\mathbb{A}_{0,k}$ in $\mathbb{A}_{0,\ell}$ are in the same $\mathbf{T}(\ell)$-orbit, whence all $k$-centers of the image of $\mathcal{X}(\mathbf{G},k)$ in $\mathcal{X}(\mathbf{G},\ell)$ are in the same $\mathbf{G}(\ell)$-orbit.
\end{proof}

\begin{figure}
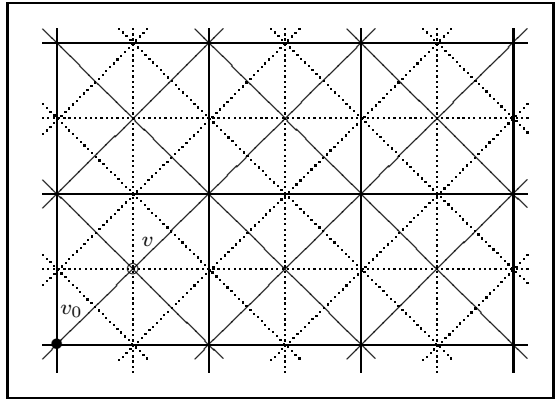

$$ 
\fbox{ 
\xygraph{
!{<0cm,0cm>;<1.0cm,0cm>:<0cm,1.0cm>::}
!{(0,-0.48) }*+{}="a0"
!{(2,-0.48) }*+{}="a1"
!{(4,-0.48) }*+{}="a2"
!{(6,-0.48) }*+{}="a3"
!{(0,4.3) }*+{}="b0"
!{(2,4.3) }*+{}="b1"
!{(4,4.3) }*+{}="b2"
!{(6,4.3) }*+{}="b3"
!{(-0.3,0) }*+{}="c0"
!{(-0.3,2) }*+{}="c1"
!{(-0.3,4) }*+{}="c2"
!{(6.3,0) }*+{}="d0"
!{(6.3,2) }*+{}="d1"
!{(6.3,4) }*+{}="d2"
!{(-0.3,1.7) }*+{}="p1"
!{(-0.3,-0.3) }*+{}="p2"
!{(1.7,-0.3) }*+{}="p3"
!{(3.7,-0.3) }*+{}="p4"
!{(2.3,4.3) }*+{}="q1"
!{(4.3,4.3) }*+{}="q2"
!{(6.3,4.3) }*+{}="q3"
!{(6.3,2.3) }*+{}="q4"
!{(2.3,-0.3) }*+{}="r0"
!{(4.3,-0.3) }*+{}="r1"
!{(6.3,-0.3) }*+{}="r2"
!{(6.3,1.7) }*+{}="r3"
!{(-0.3,2.3) }*+{}="s0"
!{(-0.3,4.3) }*+{}="s1"
!{(1.7,4.3) }*+{}="s2"
!{(3.7,4.3) }*+{}="s3"
!{(1,-0.48) }*+{}="x0"
!{(3,-0.48) }*+{}="x1"
!{(5,-0.48) }*+{}="x2"
!{(1,4.3) }*+{}="y0"
!{(3,4.3) }*+{}="y1"
!{(5,4.3) }*+{}="y2"
!{(-0.3,1) }*+{}="z0"
!{(-0.3,3) }*+{}="z1"
!{(6.3,1) }*+{}="w0"
!{(6.3,3) }*+{}="w1"
!{(-0.3,2.7) }*+{}="u1"
!{(-0.3,0.7) }*+{}="u2"
!{(0.7,-0.3) }*+{}="u3"
!{(2.7,-0.3) }*+{}="u4"
!{(4.7,-0.3) }*+{}="u5"
!{(1.3,4.3) }*+{}="v1"
!{(3.3,4.3) }*+{}="v2"
!{(5.3,4.3) }*+{}="v3"
!{(6.3,3.3) }*+{}="v4"
!{(6.3,1.3) }*+{}="v5"
!{(1.3,-0.3) }*+{}="e0"
!{(3.3,-0.3) }*+{}="e1"
!{(5.3,-0.3) }*+{}="e2"
!{(6.3,0.7) }*+{}="e3"
!{(6.3,2.7) }*+{}="e4"
!{(-0.3,1.3) }*+{}="f0"
!{(-0.3,3.3) }*+{}="f1"
!{(0.7,4.3) }*+{}="f2"
!{(2.7,4.3) }*+{}="f3"
!{(4.7,4.3) }*+{}="f4"
!{(1,1) }*+{\circ}="vert2"
!{(1.2,1.3) }*+{{}^{v}}="vert2f"
!{(0.2,0.4) }*+{{}^{v_0}}="vert1f"
!{(0.0,0.0) }*+{\bullet}="vert1"
"a0"-"b0" "a1"-"b1" "a2"-"b2" "a3"-"b3" 
"c0"-"d0" "c1"-"d1" "c2"-"d2" 
 "p1"-"q1" "p2"-"q2" "p3"-"q3" "p4"-"q4" 
 "r0"-"s0" "r1"-"s1" "r2"-"s2" "r3"-"s3" 
 "x0"-@{.}"y0" "x1"-@{.}"y1" "x2"-@{.}"y2" 
 "z0"-@{.}"w0" "z1"-@{.}"w1" 
 "u1"-@{.}"v1" "u2"-@{.}"v2" "u3"-@{.}"v3" "u4"-@{.}"v4" "u5"-@{.}"v5" 
 "e0"-@{.}"f0" "e1"-@{.}"f1" "e2"-@{.}"f2" "e3"-@{.}"f3" "e4"-@{.}"f4" 
 }}
$$
\caption{Continuous lines represent the walls in the affine building $\mathcal{X}_k$ of $\mathrm{Sp}_4$ over $k$. Here, dotted lines describe the new walls in the affine building $\mathcal{X}_{\ell}$ of $\mathrm{Sp}_4$ over $\ell$, for $e=2$. Note that $v \notin \mathrm{Sp}_4(k) \cdot v_0$, while $v \in \mathrm{Sp}_4(\ell) \cdot v_0$.} \label{figure walls extension}
\end{figure}

\subsection{Stabilizers of ``far enough'' points}\label{subsection stab far point}

\begin{lemma}\label{lem subsector face and roots}
Let $\Phi$ be a root system with basis $\Delta$ and let $\Theta \subset \Delta$.
Let $\Psi \subseteq \Phi^+_\Theta$ be a subset of positive roots.

\begin{enumerate}[label=(\arabic*)]
\item\label{item subsector real} For any family of fixed real numbers $(n_{\alpha})_{\alpha \in \Psi}$ and any $w_0 \in \mathbb{A}_0$, there exists a subsector face $Q(w_1,D_0^\Theta) \subset Q(w_0,D_0^\Theta)$, such that $\alpha(v):= \langle v, \alpha \rangle > n_{\alpha}$, for any point $v \in Q(w_1,D_0^\Theta)$ and any $\alpha \in \Psi$.

\item\label{item subsector enclosure} Moreover, if $w_0$ is a special vertex, for any such $w_1 \in Q(w_0,D_0^\Theta)$, there is a special vertex $w_2 \in Q(w_1,D_0^\Theta)$ such that $\cl\big( Q(w_2,D_0^\Theta) \big) \subset Q(w_1,D_0^\Theta)$.
\end{enumerate}
\end{lemma}

Recall that, for a subset $\Omega$ of an apartment $\mathbb{A}$, $\cl(\Omega)$ denotes the enclosure of $\Omega$ as defined in \S \ref{intro rational building}.

\begin{proof}
\ref{item subsector real} We know that for any $\alpha \in \Phi^{+}$ there are non-negative integers $\lbrace b_{\alpha,\beta} \rbrace_{\beta \in \Delta}$ such that $ \alpha= \sum_{\beta \in \Delta} b_{\alpha,\beta} \beta$.
For $\alpha \in \Psi \subseteq \Phi_\Theta^+$, we define a positive integer
\[c_\alpha := \sum_{\substack{\beta \in \Delta \smallsetminus \Theta\\ b_{\alpha,\beta} \neq 0}}  b_{\alpha,\beta} \geq 1.\]
For $\beta \in \Delta$, we define 
\[m_\beta = \max \left\{ \frac{n_\alpha-\langle w_0,\alpha \rangle}{c_\alpha}: \alpha \in \Psi \text{ and } b_{\alpha,\beta} > 0 \right\}.\]
Then, for $\beta \in \Delta \smallsetminus \Theta$, set $n^*_{\beta} > \max(m_{\beta},0)$ and, for $\beta \in \Theta$, set $n^*_{\beta} = 0$.
Since $\Delta$ is a basis of $V^*$, there exists a unique $v^* \in \Theta^\perp \subseteq V$ such that $\langle v^*, \beta \rangle =  n^*_{\beta }$, for any $\beta \in \Delta$.
Note that $v^* \in D_0^\Theta$ by definition.
Consider $w_1 = w_0 +v^* \in Q(w_0,D_0^\Theta)$.
For any $\alpha \in \Psi \subseteq\Phi_\Theta^+$, for any point $v \in Q(w_1,D_0^\Theta)$, since $v-w_1 \in D_0^\Theta$, we have that $\langle v-w_1, \alpha \rangle  > 0$. Thus
\[
    \langle v,\alpha \rangle
    > \langle w_1,\alpha \rangle
    = \langle w_0,\alpha \rangle + \sum_{\beta \in \Delta} b_{\alpha,\beta} \langle v^*, \beta \rangle
    =\langle w_0,\alpha \rangle + \sum_{\substack{\beta \in \Delta \smallsetminus \Theta\\ b_{\alpha,\beta} \neq 0}}  b_{\alpha,\beta} n^*_\beta.
\]
Since $n^*_\beta > \frac{n_\alpha-\langle w_0,\alpha \rangle}{c_\alpha}$ for any $\beta \in \Delta \smallsetminus \Theta$ such that $b_{\alpha,\beta} \neq 0$, we therefore have that
\[\langle v,\alpha \rangle  
    > \langle w_0,\alpha \rangle + \sum_{\substack{\beta \in \Delta \smallsetminus \Theta\\ b_{\alpha,\beta} \neq 0}}  b_{\alpha,\beta} \left( \frac{n_\alpha - \langle w_0,\alpha \rangle}{c_\alpha} \right) = n_\alpha.
\]

\ref{item subsector enclosure} Without loss of generalities, assume that the walls of $\mathbb{A}_0$ are the kernels of the affine roots $\alpha+k$ for $\alpha \in \Phi$ and $k \in \mathbb{Z}$.
Assume that $w_0$ is a special vertex and let $w_1$ be any point such that $\langle w_1,\alpha \rangle \geqslant n_\alpha,\ \forall \alpha\in \Psi$.
Applying the above construction with the family $\left( \langle w_1,\alpha \rangle \right)_{\alpha \in \Phi_\Theta^+}$ and considering that the numbers $n_\beta^*$ satisfying $n_\beta^* > \max(m_\beta,0)$ are non-negative integers for $\beta \in \Theta$, it provides a vector $v_2^*$ such that $\langle v_2^*,\beta \rangle = n_\beta^*$, for any $\beta \in \Delta$.
Set $w_2 = w_0+v_2^*$ so that $\langle \alpha, w_2 \rangle > \langle \alpha, w_1 \rangle$, for any $\alpha \in \Phi_\Theta^+$.
We have that $\langle w_0, \beta \rangle \in \mathbb{Z}$, for every $\beta \in \Phi$, by definition of a special vertex.
Hence, for any $\alpha = \sum_{\beta \in \Delta} b_{\alpha,\beta} \beta \in \Phi^+$, we have that $\langle w_2, \alpha \rangle = \sum_{\beta \in \Delta} b_{\alpha,\beta} \big( \langle w_0,\beta \rangle + n^*_\beta \big) \in \mathbb{Z}$.
Thus, $w_2 \in Q(w_1,D_0^\Theta)$ also is a special vertex.
Hence, the closure of $Q(w_2,D_0^\Theta)$ is enclosed, whence any point $v \in \cl\big( Q(w_2,D_0^\Theta\big)$ satisfies $\langle v,\beta \rangle \geqslant \langle w_2,\beta \rangle > \langle w_1,\beta \rangle$, for any $\beta \in \Phi_\Theta^+$.
Thus $v \in Q(w_1,D_0^\Theta)$.
\end{proof}

\begin{lemma}\label{lemma finite special vertices}
Let $\Phi$ be a root system with basis $\Delta$ and let $\Theta \subseteq \Delta$.
Let $Q(x,D_0^\Theta)$ be a sector face of $\mathbb{A}_0$ with an arbitrary tip $x \in \mathbb{A}_0$.
There exists a finite subset $\Omega$ of $\cl\big( Q(x,D_0^\Theta) \big)$ consisting of special vertices such that any special vertex of $\cl\big( Q(x,D_0^\Theta) \big)$ belongs to $\cl\big(Q(\omega,D_0^\Theta)\big) = \overline{Q}(\omega,D_0^\Theta)$, for some $\omega \in \Omega$.
\end{lemma}

\begin{proof}
Let $(\varpi_\alpha)_{\alpha \in \Delta}$ be the basis of fundamental coweight associated to the basis $\Delta$.
In this basis, the set of special vertices in $\mathbb{A}_0$ is the lattice $\Lambda = \bigoplus_{\alpha\in \Delta} \mathbb{Z}\varpi_\alpha$ and, from the definition given in \S \ref{intro vector faces}, the vector face $D_0^\Theta$ can be written as $D_0^\Theta=\left\{ \sum_{\alpha \in \Delta \smallsetminus \Theta} x_\alpha \varpi_\alpha,\ \forall \alpha \in \Delta \smallsetminus \Theta,\ x_\alpha > 0\right\}$.
\footnote{Note that we focus on special vertices. According to \cite[Chap.VI,\S2.2, Cor. of Prop.~5]{Bourbaki}, the lattice of (non necessarily special) vertices would be $\displaystyle \bigoplus_{\alpha\in \Delta} \mathbb{Z} \tfrac{\varpi_\alpha}{n_\alpha}$ where the highest root of $\Phi$ is $\displaystyle \sum_{\alpha \in \Delta} n_\alpha \alpha$ with $n_\alpha$ some positive integers.}

Denote by $\Lambda':=\Lambda \cap \cl\big(Q(x,D_0^\Theta)\big)$ the set of special vertices of the enclosure of $Q(x,D_0^\Theta)$.
Note that $\cl\big(Q(\omega,D_0^\Theta)\big) \subset \cl\big(Q(x,D_0^\Theta)\big)$ for any $\omega \in \Lambda'$.
Consider the subset $\Omega$ of special vertices of $\cl\big(Q(x,D_0^\Theta)\big)$ that are not contained in any enclosed sector face directed by $D_0^\Theta$ whose tip is a special vertex of $\cl\big( Q(x,D_0^\Theta)\big)$, in which they are not the tip, in other words:
\[\Omega := \left\{ \omega \in \Lambda',\ \forall \omega' \in \Lambda' \smallsetminus \{\omega\},\ \omega \not\in \overline{Q}(\omega',D_0^\Theta)\right\}.\]

\textbf{First claim:} For any $\omega' \in \Lambda'$, there exists $\omega \in \Omega$ such that $\omega' \in \overline{Q}(\omega,D_0^\Theta)$.

We can define a sequence  $\omega'_m$ such that $\omega'_0 := \omega'$ and $\omega'_m \in \overline{Q}(\omega'_{m+1},D_0^\Theta)$.
Indeed, for any $\omega'_m \in \Lambda'$, we have either $\omega'_{m+1}:=\omega'_m \in \Omega$ or there exists $\omega'_{m+1} \in \Lambda'$ such that $\omega'_{m+1} \in \Lambda' \smallsetminus \{\omega'_m\}$ and $\omega'_m \in \overline{ Q}(\omega'_{m+1},D_0^\Theta)$.
Therefore, this defines a sequence $(\omega'_m)_{m}$ by induction.
For any $\alpha \in \Delta$, we have that $\alpha(\omega'_{m+1}) \leqslant \alpha(\omega'_m)$ since $\omega'_m \in \overline{Q}(\omega'_{m+1},D_0^\Theta)$.
Each coordinate is lower bounded by $\alpha(\omega'_m) \geqslant \alpha(x) -1$ since $\omega'_m \in \cl\big(Q(x,D_0^\Theta)\big)$, whence it is eventually constant since $\alpha(\omega'_m) \in \mathbb{Z}$.
Since $\Delta$ is finite, the sequence $\omega'_m$ also is eventually constant, whence eventually belongs to $\Omega$ by construction.
Hence, for any $\omega'_0 \in \Lambda'$, there is $m\in\mathbb{N}$ such that $\omega'_0 \in \overline{Q}(\omega'_m,D_0^\Theta)$ and $\omega:=\omega'_m \in \Omega$.

\textbf{Second claim:} $\Omega$ is finite.

Suppose by contradiction that $\Omega$ is infinite and let $(\omega_n)_{n\in\mathbb{N}}$ be a sequence of pairwise distinct special vertices in $\Omega$.

For any $\alpha \in \Theta$, we have that $ \alpha(x)-1\leqslant\alpha(\omega_n) \leqslant \alpha(x)+1$, whence $\alpha(\omega_n)$ can take finitely many values.
Hence, there exists an extraction $\varphi_0$ such that  $\left(\omega_{\varphi_0(n)}\right)_n$ is constant in the coordinate $\alpha$, for every $\alpha \in \Theta$.

Let $\Delta\smallsetminus\Theta = \{\alpha_1,\dots,\alpha_m\}$. For any $1 \leq i \leq m$, we have that $\forall n \in \mathbb{N},\ \alpha_i(\omega_n) \geqslant \alpha(x) -1$ since $\omega_n \in \cl\big( Q(x,D_0^\Theta)\big)$.
One can extract recursively non-decreasing subsequences.
Hence, there are extractions $\varphi_1,\dots,\varphi_m$ such that for any $1 \leq i \leq m$ the sequence $\alpha_i\left(\omega_{\varphi_0 \circ \cdots \circ \varphi_m(n)}\right)$ is non-decreasing.
Hence, by definition, we have that $\forall n \in \mathbb{N},\ \omega_{\varphi_0 \circ \cdots \circ \varphi_m(n)} \in \overline{Q}\big( \omega_{\varphi_0 \circ \cdots \circ \varphi_m(0)}, D_0^\Theta\big)$, which contradicts $\omega_{\varphi_0 \circ \cdots \circ \varphi_m(1)}  \in \Omega$ since the elements of this sequence are assumed to be pairwise distinct. Whence the claim follows.
\end{proof}

 The first main result of this section claims that, ``far enough'', for the action of $\mathbf{G}(A)$ on $\mathcal{X}_k$, the stabilizer of a point coincides with the stabilizer of a sector face.

\begin{proposition}\label{prop igual stab}
Let $x$ be any point of $\mathcal{X}_k$, and let $Q(x,D)$ be a $k$-sector face of $\mathcal{X}_k$. There exists a subsector face $Q(y_1,D) \subseteq Q(x,D)$
\nomenclature{$y_1$}{Proposition~\ref{prop igual stab}}
such that for any point $v \in \overline{Q}(y_1,D)$:
\begin{equation}\label{eq igual stab}
\mathrm{Stab}_{{\mathbf{G}(A)}}(v)=\mathrm{Fix}_{{\mathbf{G}(A)}}\big(Q(v,D)\big).  
\end{equation}
\end{proposition}

\begin{proof}
If $D = 0$, there is nothing to prove. Thus, we assume in the following that $D\neq 0$.
We consider a finite ramified extension $\ell/k$ as given by Lemma~\ref{lemma becomes special} and we denote by $L$ the completion of $\ell$ with respect to $\omega_\p$.

We firstly assume that $x$ is a $k$-vertex.
Let $v_0$ be the standard vertex in $\mathcal{X}(\mathbf{G},L)$ so that $x \in \mathbf{G}(\ell) \cdot v_0$.
Then, it follows from Lemma~\ref{lema0} that there exists $\tau \in \mathbf{G}(\ell)$ such that $\tau \cdot Q(x,D)=Q(v_0,D_0^{\Theta}) \subset \mathbb{A}_0$ for some standard vector face $D_0^{\Theta}$ with $\Theta \subsetneq \Delta$.
For any point $z \in Q(x,D)$,
extending the notation in the standard apartment of~\cite[7.4.4]{BT},
we denote by $\widehat{P}_z:=\stab_{\mathbf{G}(L)}(z)$ the stabilizer of $z$ and by $\widehat{P}_{z+D}:=\fix_{\mathbf{G}(L)}\big( Q(z,D) \big)$ the pointwise stabilizers of the subsector face $Q(z,D)$ for the action of $\mathbf{G}(L)$ on $\mathcal{X}(\mathbf{G},L)$.

We want to show that
\begin{equation}\label{statement stab} \exists y_1 \in Q(x,D),\ \forall z \in \overline{Q}(y_1,D),\ \widehat{P}_z \cap \mathbf{G}(A) = \widehat{P}_{z+D} \cap \mathbf{G}(A).\end{equation}
Up to translating by $\tau$, statement~\eqref{statement stab} is equivalent to show that
\begin{equation}
\exists w_0 \in Q(v_0,D_0^{\Theta}), \, \forall w \in \overline{Q}(w_0,D_0^{\Theta}),\ \widehat{P}_w \cap \tau \mathbf{G}(A) \tau^{-1}= \widehat{P}_{w+D_0^{\Theta}} \cap \tau \mathbf{G}(A) \tau^{-1}.
\end{equation}

For any $w \in Q(v_0,D_0^\Theta) \subset \mathbb{A}_0$ and any $z \in \overline{Q}(w,D_0^\Theta)$, $z \neq w$, we denote by $[w,z)$ the half-line with origin $w$ and direction $z -w$ and, for $z=w$, we denote $[w,z)=\{w\}$ by convention.
We denote by $\widehat{P}_{[w,z)}$ the pointwise stabilizer in $\mathbf{G}(L)$ of $[w,z)$.
We have that $\overline{Q}(w,D_0^\Theta) = \bigcup_{z \in \overline{Q}(w,D_0^\Theta)} [w,z)$, whence \[\widehat{P}_{w+D_0^\Theta} = \bigcap_{z \in \overline{Q}(w,D_\Theta^0)} \widehat{P}_{[w,z)}\] according to \cite[7.1.11]{BT}. Because $\widehat{P}_{w+D} \subset \widehat{P}_w$, we are reduced to prove that
\begin{multline}\label{eq fixator of line}
\exists w_0 \in Q(v_0,D_0^{\Theta}), \, \forall w \in \overline{Q}(w_0,D_0^{\Theta}),\ \forall z \in \overline{Q}(w,D_0^\Theta),\\
\widehat{P}_w \cap \tau \mathbf{G}(A) \tau^{-1}\subseteq \widehat{P}_{[w,z)} \cap \tau \mathbf{G}(A) \tau^{-1}.
\end{multline}

Now, using a building embedding, we show that it is enough to verify Identity~\eqref{eq fixator of line} in the context where $\mathbf{G}= \mathrm{SL}_n$ and $\Theta \neq \Delta$.
Indeed, let $\tau_0=\rho(\tau) \in \rho(\mathbf{G}(\ell))$.
The $\mathbf{G}(L)$-equivariant immersion $j: \mathcal{X}_L\hookrightarrow \mathcal{X}(\mathrm{SL}_n, L)$ introduced in~§\ref{intro building embedding} sends the standard vertex $v_0 \in \mathrm{vert}(\mathbb{A}_0)$ onto the standard vertex $v_0' \in \mathcal{X}(\mathrm{SL}_n, L)$ and embeds the closure $\overline{Q}(v_0,D_0)$ of the standard sector chamber of $\mathcal{X}_L$ into the closure $\overline{Q}(v_0',D_0')$ of the standard sector chamber of $\mathcal{X}(\mathrm{SL}_n,L)$.
Denote by $\Phi'$ the canonical root system of $\mathrm{SL}_n$ and by $\Delta'$ the canonical basis of $\Phi'$.
Since $j(v_0) = v'_0$ and $j\big(\overline{Q}(v_0,D_0^\Theta)\big) \subseteq j\big(\overline{Q}(v_0,D_0)\big) \subseteq \overline{Q}(v'_0,D_0')$, the enclosure of $j\big(\overline{Q}(v_0,D_0)\big)$ is a closed sector face of $\overline{Q}(v'_0,D_0')$.
In other words, there is a subset of simple roots $\Theta' \subseteq \Delta'$ such that $\operatorname{cl}\Big(j\big(Q(v_0,D_0^\Theta)\big)\Big) = \overline{Q}(v'_0,D_0^{\Theta'})$.
More precisely, $\Theta' = \left\{ \alpha \in \Delta',\ \alpha\Big(j\big(Q(v_0,D_0^\Theta)\big)\Big) = \{0\} \right\}$.
Assume that Identity~\eqref{eq fixator of line} holds for $B$, $\mathrm{SL}_n$, $\tau_0$ and $\Theta'$, that is 
\begin{multline}\label{eq fixator of line SLn}
\exists w'_0 \in Q(v'_0,D_0^{\Theta'}), \, \forall w' \in \overline{Q}(w'_0,D_0^{\Theta'}),\ \forall z' \in \overline{Q}(w',D_0^{\Theta'}),\\
\widehat{P}_{w'} \cap \tau_0 \mathrm{SL}_n(B) \tau_0^{-1}\subseteq \widehat{P}_{[w',z')} \cap \tau_0 \mathrm{SL}_n(B) \tau_0^{-1}.
\end{multline}

We claim that the intersection $j\big(\overline{Q}(v_0,D_0^\Theta)\big) \cap \overline{Q}(w'_0,D_0^{\Theta'})$ is nonempty.
Indeed, recall that the standard vertex $v'_0$ satisfies $\alpha(v'_0)=0,\ \forall \alpha \in \Phi'$, so that $j\big( \overline{Q}(v_0,D_0^\Theta) \big)$ is seen as a non-empty convex cone with tip $v'_0=0$ in $\mathbb{A}'_0 \cong \mathbb{R}^{n-1}$.
By definition of $\Theta'$, for any $\alpha \in \Delta' \smallsetminus \Theta'$, there exists $\delta_\alpha \in j\big( \overline{Q}(v_0,D_0^\Theta) \big)$ such that $\alpha(\delta_\alpha) >0$.
Since $\delta_\alpha \in \overline{Q}(v'_0,D_0^{\Theta'})$, we also have that $\beta(\delta_\alpha) \geq 0$ for any $\beta \in \Delta' \smallsetminus \Theta'$, and $\beta(\delta_\alpha) = 0$ for any $\beta \in \Theta'$.
Let $\delta = \sum_{\alpha \in \Theta'} \delta_\alpha$.
Then, $\alpha(\delta) > 0$, for all $\alpha \in \Delta' \smallsetminus \Theta'$ and $\alpha(\delta)=0$, for all $\alpha \in \Theta'$.
Hence, there exists $t \in \mathbb{R}$ such that $\alpha(t\delta) \geq \alpha(w'_0)$, for all $\alpha \in \Delta' \smallsetminus \Theta'$ and $\alpha(t \delta)=0$, for all $\alpha \in \Theta'$.
Thus $t \delta \in \overline{Q}(w'_0,D_0^{\Theta'})$ by definition.
Moreover, $t \delta \in j\big( \overline{Q}(v_0,D_0^\Theta) \big)$ as positive linear combination of such elements in a convex cone.
Hence there exists $w_0 \in \overline{Q}(v_0,D_0^\Theta)$ such that $j(w_0) = t \delta \in \overline{Q}(w'_0,D_0^{\Theta'})$. Whence the claim follows.

Thus, for any $w \in \overline{Q}(w_0,D_0^\Theta)$ and any $z \in \overline{Q}(w,D_0^\Theta)$, we have that
\[j(w) \in j\big( \overline{Q}(w_0,D_0^\Theta) \big) \subseteq \overline{Q}\big( j(w_0), D_0^{\Theta'}\big) \subseteq \overline{Q}(w'_0,D_0^{\Theta'}),\]
and 
\[ j(z) \in j\big(\overline{Q}(w,D_0^\Theta)\big) \subseteq \overline{Q}\big( j(w), D_0^{\Theta'}\big).\]
Moreover, $j([w,z)) = [j(w),j(z))$ whence
\[ \widehat{P}_{j(w)} \cap \tau_0 \mathrm{SL}_n(B) \tau_0^{-1}\subseteq \widehat{P}_{j([w,z))} \cap \tau_0 \mathrm{SL}_n(B) \tau_0^{-1}.\]
Thus, if we intersect the previous inclusion with $\tau_0 \rho(\mathbf{G}(A)) \tau_0^{-1}$ we obtain
$$\widehat{P}_{j(w)} \cap \rho(\tau \mathbf{G}(A) \tau^{-1})=\widehat{P}_{j([w,z))} \cap \rho(\tau \mathbf{G}(A) \tau^{-1}).$$
Since the immersion $j: \mathcal{X}_L\hookrightarrow \mathcal{X}(\mathrm{SL}_n, L)$ is $\mathbf{G}(L)$-equivariant, we have that $\widehat{P}_{j(w)} \cap \rho(\mathbf{G}(L)) = \rho(\widehat{P}_w)$ and $\widehat{P}_{j([w,z))} \cap \rho(\mathbf{G}(L)) = \rho(\widehat{P}_{[w,z)})$.
We conclude that $\widehat{P}_w \cap \tau \mathbf{G}(A) \tau^{-1} \subseteq \widehat{P}_{[w,z)} \cap \tau \mathbf{G}(A) \tau^{-1}$.
Thus, Identity~\eqref{eq fixator of line SLn} implies Identity~\eqref{eq fixator of line}. In particular, we are reduced to prove~\eqref{eq fixator of line SLn}.

By abuse of notation, let us denote by $\nu$ the valuation map on $L$ induced by $\p'$.
We assume that we are in the situation with $\mathbf{G} = \mathrm{SL}_n$ and any $\Theta' \subseteq \Delta' \subset \Phi'$.
Let $w' \in \overline{Q}(v'_0,D_0^{\Theta'})$ be any point.
The stabilizer in $\mathbf{G}(L)$ of the point $w' \in \mathbb{A}_0$ is, with the notation\footnote{We are in the case with $E=\{0\}$ and $\delta(g)=0$ because $g \in \mathrm{SL}_n$ has determinant $1$.} of \cite[10.2.8]{BT}:
\begin{equation}\label{equality Pw}
\widehat{P}_{w'}=\left\lbrace g=(g_{ij})_{i,j=1}^n \in \mathrm{SL}_n(L) : \nu(g_{ij})+(a_j-a_i)(w') \geq 0,\ 1 \leq i,j \leq n \right\rbrace,
\end{equation}
where $(a_i)_{1 \leq i \leq n}$ is the canonical basis of $\mathbb{R}^n$ seen as linear forms on $\mathbb{A}_0 \simeq \mathbb{R}^{n-1}$ so that $\Phi' = \{ \alpha_{ij} := a_j-a_i,\ 1\leq i, j \leq n \text{ and } i \neq j\}$ and $\Delta' = \{\alpha_{i\ i+1},\ 1 \leq i \leq n-1\}$.
Let $z' \in \overline{Q}(w',D_0^{\Theta'})$ and $g = (g_{ij})_{i,j=1}^n \in \widehat{P}_{[w',z')}$.
Denote by $\delta := z' - w'$ the direction of the half-line $[w',z')$.
We have $\delta \in \overline{D_0^{\Theta'}}$ so that
\[\left\{\begin{array}{ll}\alpha(\delta) \geq 0 & \text{ for } \alpha \in {\Phi'}_{\Theta'}^{+},\\
\alpha(\delta) = 0 &  \text{ for } \alpha \in {\Phi'}_{\Theta'}^{0},\\
\alpha(\delta) \leq 0 &  \text{ for } \alpha \in {\Phi'}_{\Theta'}^{-}.\\
\end{array} \right.\]
For any $t \in \mathbb{R}_{\geqslant 0}$, consider $z'_t = w' + t \delta \in [w',z')$.
Because $g \in \widehat{P}_{[w',z')} \subset \widehat{P}_{z'_t}$, we have, by Equality~\eqref{equality Pw}, that:
\begin{equation*}
    \nu(g_{i j}) \geq (a_i-a_j)(w') + t (a_i - a_j)(\delta),\ \forall 1 \leq i,j \leq n,\ i \neq j, \forall t \in \mathbb{R}_{\geq 0}
\end{equation*}
For $1 \leq i,j \leq n$, $i \neq j$, we deduce that:
\begin{equation*}
    \left\{\begin{array}{cc}
      g_{ij} = 0   &  \text{ if } (a_j-a_i)(\delta) < 0,\\
      \nu(g_{ij}) \geq (a_i - a_j)(w')   &  \text{ if } (a_j-a_i)(\delta) \geq 0.\\
    \end{array}\right.
\end{equation*}
Conversely, Equality~\eqref{equality Pw} immediately gives that any such $g=(g_{ij})$ satisfies $g \in \widehat{P}_{z'_t}$ for all $t \geq 0$, whence $g \in \widehat{P}_{[w',z')}$.
Hence we get that:
\begin{equation}\label{equality P[w[}
\widehat{P}_{[w',z')}= \left\lbrace (g_{ij})_{i,j=1}^n \in \mathrm{SL}_n(L):
\begin{array}{rl}
g_{ij}=0,& \text{ if } (a_j-a_i)(\delta)<0 \\
\nu(g_{ij})+(a_j-a_i)(w') \geq 0,& \text{ if } (a_j-a_i)(\delta)\geq 0
\end{array}
\right\rbrace.
\end{equation}

Now, let $\mathcal{I} = (I_{i j})_{1 \leq i,j \leq n}$ be a family of proper fractional $B$-ideals.
We denote by $\mathcal{M}_n(\mathcal{I})$ the $A$-module of matrices whose $(i,j)$-coefficient is $I_{ij}$, for any $i,j \in \lbrace 1, \cdots, n \rbrace$.
Then, we have that the group $\tau_0 \mathrm{SL}_n(A) \tau_0^{-1} \subseteq \tau_0 \mathrm{SL}_n(B) \tau_0^{-1}$ is contained in $\mathcal{M}_n(\mathcal{I})$, for some family of fractional ideals $\mathcal{I}$. In particular, the set 
\begin{multline*}
    \Pi:= \left\lbrace \nu(g_{ij}): g_{ij}\neq 0 \text{ and } g=(g_{ij})_{i,j=1}^n \in \tau_0\mathrm{SL}_n(A)\tau_0^{-1} \right\rbrace\\
    \subseteq \bigcup_{1 \leq i,j \leq n} \nu\big( I_{ij} \smallsetminus \{0\} \big) \subset \frac{1}{N}\mathbb{Z}
\end{multline*} is upper bounded since so are each $\nu\big( I_{i,j} \smallsetminus \{0\}\big)$.

Since $\Pi$ is an upper bounded set, it follows from Lemma~\ref{lem subsector face and roots}\ref{item subsector real} that there exists $w'_0 \in Q(v'_0,D_0^{\Theta'})$ such that for any $\alpha = \alpha_{ij} \in {\Phi'}_{\Theta'}^{-}$ we have $\max(\Pi) + \alpha_{ij}(w'_0) < 0$.
Thus, for any $g \in \widehat{P}_{w'} \cap \tau_0 \mathrm{SL}_n(A) \tau_0^{-1}$ such that $g_{ij} \neq 0$, we have $\nu(g_{ij}) + \alpha_{ij}(w'_0) < 0$.
Hence
\[\forall w' \in \overline{Q}(w'_0,D_0^{\Theta'}),\ \nu(g_{ij}) + \alpha_{ij}(w') < \alpha_{ij}(w') - \alpha_{ij}(w'_0) \leq 0\]
because $\alpha_{ij} \in {\Phi'}_{\Theta'}^{-}$ and $w'-w'_0 \in \overline{D_0^{\Theta'}}$.
Thus, Equality~\eqref{equality Pw} implies that $g_{ij}=0$ for any point $w' \in \overline{Q}(w'_0,D_0^{\Theta'})$, any negative root $\alpha_{ij} = a_j - a_i \in {\Phi'}_{\Theta'}^{-}$, and any $g \in \widehat{P}_{w'} \cap \tau_0 \mathrm{SL}_n(A) \tau_0^{-1}$.
Hence, Equality~\eqref{equality P[w[} gives that any such $g$ satisfies $g \in \widehat{P}_{[w',z')}$ for any $z' \in \overline{Q}(w',D_0^{\Theta'})$.
Therefore, we conclude that Statement~\eqref{eq fixator of line SLn} holds.
Hence, Statement~\eqref{statement stab} holds when $x$ is a $k$-vertex.\\

Secondly, assume that $x$ is a special $\ell$-vertex. Applying the previous situation replacing $\ell/k$ by some extension $\ell'/\ell$ given by Lemma~\ref{lem curve extension}, we deduce that there exists a subsector $Q(y_1,D)$ such that for any point $v \in \overline{Q}(y_1,D)$, we have that
\[
\stab_{\mathbf{G}(B)}(v) = \fix_{\mathbf{G}(B)}\big( Q(v,D) \big).\]
Hence, the proposition remains true for special $\ell$-vertices by intersecting the previous equality with $\mathbf{G}(A)$, which is a subgroup of $\mathbf{G}(B)$.\\

Finally, assume that $x \in \mathcal{X}(\mathbf{G},k)$ is any point.
Let $\mathbb{A}$ be a $k$-apartment containing the sector face $Q(x,D)$ and note that the enclosure $\operatorname{cl}_k\big(Q(x,D)\big)$ is contained in $\mathbb{A}$ by definition.
By Lemma~\ref{lemma finite special vertices}, there is a finite subset $\Omega$ of special $\ell$-vertices such that any special $\ell$-vertex and, in particular, any $k$-center, of $\operatorname{cl}_k\big(Q(x,D)\big)$ is contained in some $Q(\omega,D)$ for $\omega \in \Omega$.

For any point $z \in \mathbb{A}$, we denote by $z_c$ the $k$-center of the $k$-face containing $z$.
Because the $k$-centers span a maximal rank lattice of $\mathbb{A}$, hence cocompact, there is a positive real number $\eta_1$ such that $\forall z \in \mathbb{A},\ d(z,z_c) \leq \eta_1$.
Because $\Omega$ is finite, there is a positive real number $\eta_2$ such that $\forall \omega \in \Omega,\ d(x,\omega) < \eta_2$.
Let $V = \operatorname{Vect}(D)$ be the $\mathbb{R}$-vector subspace of $\mathbb{A}$ spanned by $D$.
Let $C'$ be the closed ball centered at $0$ of radius $\eta_1 + \eta_2$ in $V$.
For $\omega \in \Omega$, denote by $y_\omega^1$ the $y_1$ given by Statement~\eqref{statement stab} applied to $Q(\omega,D)$.
By construction of $V$, there is a subsector face $Q(y_\omega,D) \subseteq Q(y_\omega^1,D)$ such that $Q(y_\omega,D) + C' \subseteq Q(y_\omega^1,D)$.
Denote by $x_\omega := x + y_\omega - \omega$.
Then $x_\omega - x = y_{\omega}-\omega \in D$ because $y_\omega \in Q(y_\omega^1,D) \subseteq Q(\omega,D)$.
Hence $x_\omega \in Q(x,D)$.
Thus, there is a point $y_1 \in Q(x,D)$ such that $Q(y_1,D) \subseteq \bigcap_{\omega \in \Omega} Q(x_\omega,D)$ as finite intersection of subsector faces of $Q(x,D)$.

Let $z \in Q(y_1,D)$. The group $\mathcal{P}_z$ stabilizes the $k$-face of $z$.
Hence, it fixes the $k$-center $z_c$ of the face containing $z$, whence $\widehat{P}_z \subseteq \widehat{P}_{z_c}$.
Let $\omega \in \Omega$ be such that $z_c \in Q(\omega,D)$.
Let $y := y_\omega + z - x_\omega$.
Then $y \in Q(y_\omega,D)$ since $z-y_1 \in D$ and $y_1 - x_\omega \in D$.
Moreover, we have $z_c-y = (z_c - \omega) + (\omega-y_\omega) + (y_\omega - y) \in V$ and $z_c - y = z_c - y_\omega + x_\omega - z = (z_c - z) + (x-\omega)$, so that $d(z_c,y) \leq d(z_c,z) + d(x,\omega) \leq \eta_1 + \eta_2$.
Thus $z_c - y \in C'$ and $y \in Q(y_\omega,D)$ provides $z_c \in y + C' \subseteq Q(y_\omega,D) + C' \subseteq Q(y_\omega^1,D)$.
Thus, by~\eqref{statement stab} with $z_c \in Q(y_\omega^1,D)$, we have that $\widehat{P}_{z_c} \cap \mathbf{G}(A) = \widehat{P}_{Q(z_c,D)} \cap \mathbf{G}(A)$.
Hence $\widehat{P}_{z} \cap \mathbf{G}(A)$ fixes $z \cup Q(z_c,D)$ so that it fixes the closed convex hull of $z \cup Q(z_c,D)$ which contains $Q(z,D)$.
Therefore, $\widehat{P}_z\cap \mathbf{G}(A) \subseteq \widehat{P}_{z+D} \cap \mathbf{G}(A)$.
Whence the result follows.
\end{proof}

\subsection{Foldings along a subsector face}\label{subsection foldings}

Let $J$ be a non-zero fractional $A$-ideal. See it as a line bundle, associated to a divisor $D_J$, on the affine curve $\mathrm{Spec}(A)= \mathcal{C} \smallsetminus \lbrace \p \rbrace$.
We write $\deg(J)=\deg(D_J)$.\nomenclature[]{$\deg(J)$}{degree of the Cartier divisor $D_J$}
Normalize $\nu_\p$ so that $\nu_\p(k^\times) = \mathbb{Z}$ and let $\pi \in \mathcal{O}_{\p}$ be a uniformizer.\nomenclature[]{$\pi$}{uniformizer in $A$ of $\nu_\p$}
Let $m \in \mathbb{Z}_{\geq 0}$.
Define $ J[m]$ as the vector bundle $\mathcal{L}(-D_J+m\p)$ associated to the divisor $-D_J+m\p$.\nomenclature[]{$J[m]$}{some truncated ideal}
Equivalently, $J[m]:=\left\lbrace x \in J:\nu_\p(x) \geq -m\right\rbrace = J \cap \pi^{-m}\mathcal{O}$.
This set is always a finite dimensional $\mathbb{F}$-vector space (c.f.~\cite[\S 1, Prop.~1.4.9]{Stichtenoth}).
We denote by $g$ the genus of $\mathcal{C}$ and by $d$ the degree of $\p$.\nomenclature[]{$g$}{genus of $\mathcal{C}$}\nomenclature[]{$d$}{degree of $\p$}
Then, by the Riemann-Roch Theorem (c.f.~\cite[\S 1, Thm.~1.5.17]{Stichtenoth}) the finite-dimensional vector space $J[m]$ satisfies
$$\operatorname{dim}_{\mathbb{F}}(J[m])= \deg(-D_J+m\p)+ 1-g,$$
whenever $\deg(-D_J+m\p)\geq 2g-1$.
Hence, since
$$\deg(-D_J+m\p)=-\deg(J)+m\deg(\p),$$
we finally get,
\begin{equation}\label{eq rr0}
\operatorname{dim}_{\mathbb{F}}(J \cap \pi^{-m}\mathcal{O})= -\deg(J)+md+ 1-g,\quad\text{ whenever }\quad md\geq \deg(J)+2g-1.
\end{equation}

\begin{notation}\label{notation directed star}
Let $x \in \mathcal{X}_k$ be a point.
Let $\Phi_x = \Phi_{x,k}$ be the sub-root system of $\Phi$ associated to $x$ as defined in §\ref{intro rational building}.
We define the star\index{star} of $x$ in $\mathcal{X}_k$ as the subcomplex $\mathcal{X}_k(x)$\nomenclature[]{$\mathcal{X}_k(x)$}{star of $x$}
of $\mathcal{X}_k$ whose faces $F$ are the $k$-faces of $\mathcal{X}_k$ containing $x$ in their closure:
\[ \mathcal{X}_k(x) = \left\{ F \text{ is a } k\text{-face of } \mathcal{X}_k,\ x \in \overline{F}\right\}.\]
By abuse of notation, we also denote by $\mathcal{X}_k(x)$ the set of points $z \in \mathcal{X}_k$ such that there is a $k$-face $F \in \mathcal{X}_k$ such that $z \in F$.
The complex $\mathcal{X}_k(x)$ is a spherical building of type $\Phi_x$ whose the set of its chambers is
\[ \mathcal{R}_k(x) := \left\{ C \text{ is a } k\text{-alcove of } \mathcal{X}_k,\ x \in \overline{C}\right\},\]
\nomenclature[]{$\mathcal{R}_k(x)$}{residue at $x$}
and the set of its apartments is
\[ \mathcal{A}_k(x) = \left\{ \mathbb{A} \cap \mathcal{X}_k(x),\ \mathbb{A} \text{ is a } k\text{-apartment of } \mathcal{X}_k \right\}.\]
\nomenclature[]{$\mathcal{A}_k(x)$}{apartments of the star of $x$} 
See \cite[4.6.33]{BT2}, we give details of these facts in the proof of Proposition~\ref{prop directed building}.
If $F$ is a $k$-face and $x \in F$ is a generic point, we define $\mathcal{X}_k(F) = \mathcal{X}_k(x)$ and $\mathcal{R}_k(F) = \mathcal{R}_k(x)$.
\end{notation}

\begin{definition}
Let $\mathbb{A}$ be a $k$-apartment containing $x$ and $V$ be a vector space so that $x+V$ is an affine subspace of $\mathbb{A}$.
We define the star directed\index{star!directed} by $V$ of $x$ by:
\[ \mathcal{X}_k(x,V) = \left\{ z \in \mathcal{X}_k(x),\ \exists \mathbb{A}'\ k\text{-apartment},\ \mathbb{A}' \supseteq z \cup (x+V)\right\}
.\]
\nomenclature[]{$\mathcal{X}_k(x,V)$}{star of $x$ directed by $V$}
Note that, if $V=0$, then $\mathcal{X}_k(x,V) = \mathcal{X}_k(x)$;
and that, if $x+V = \mathbb{A}$, then $\mathcal{X}_k(x,V) = \mathcal{X}_k(x) \cap \mathbb{A}$.
We denote by
\[ \mathcal{R}_k(x,V) = \mathcal{R}_k(x) \cap \mathcal{X}_k(x,V)\]
\nomenclature[]{$\mathcal{R}_k(x,V)$}{chambers of $\mathcal{X}_k$ in $\mathcal{X}_k(x,V)$}
and by
\[\mathcal{A}_k(x,V) = \left\{ \mathbb{A}' \cap \mathcal{X}_{k}(x,V),\ \mathbb{A}' \text{ is a } k\text{-apartment containing } x+V \right\}.\]
\nomenclature[]{$\mathcal{A}_k(x,V)$}{apartments of $\mathcal{X}_k(x,V)$}
\end{definition}
The star directed by $V$ of $x$ is not a classical definition coming from Bruhat-Tits theory. We introduce it since it appears to be useful in the following.\\

In the following, we will focus on the case where $V$ is the vector subspace of an apartment spanned by the direction of some sector face. Roughly speaking, $\mathcal{X}_k(x,V_0^\Theta)$ will be the subset of the spherical building $\mathcal{X}_k(x)$ obtained by folding the standard apartment with respect to its affine subspace $(x+\Theta^\perp)$ for some $\Theta \subset \Delta$.
Hence $\mathcal{X}_k(x,V_0^\Theta)$ becomes a spherical building of type $\Phi_\Theta^0 \cap \Phi_x$, with Weyl group contained in $W_\Theta$.

\begin{lemma}\label{lemma subroot system}
Let $\Phi$ be a root system and let $\Delta$ be a basis of $\Phi$ and denote by $\Phi^+$ the positive roots.
Let $\Phi'$ be a sub-root system of $\Phi$ and let $\Delta'$ be the basis of $\Phi'$ such that the positive roots associated to $\Delta'$ in $\Phi'$ are $\Phi^+ \cap \Phi'$.
For any $\Theta \subset \Delta$, there is a unique $\Theta' \subset \Delta'$ such that $\Phi' \cap \Phi_\Theta^0 = (\Phi')_{\Theta'}^0$. Moreover $\Theta'$ is a basis of this sub-root system.
\end{lemma}

\begin{proof}
Let $\Theta' = \Delta' \cap \Phi_\Theta^0$.
By construction $(\Phi')_{\Theta'}^0$ is a root system with basis $\Theta'$.
For $\beta \in \Phi' \cap \Phi^+$, write $\beta = \sum_{\alpha \in \Delta} n_\alpha(\beta) \alpha = \sum_{\alpha' \in \Delta'} n'_{\alpha'}(\beta) \alpha'$ where the $n_\alpha(\beta)$ and $n'_{\alpha'}(\beta)$ are non-negative integers uniquely determined by the bases $\Delta$ and $\Delta'$.
For $\alpha' \in \Delta'$, if $\forall \alpha \in \Delta\smallsetminus \Theta,\ n_\alpha(\alpha') =0$, then $\alpha' \in \Phi_\Theta^0 \cap \Delta' = \Theta'$.

If $\beta \in (\Phi')_{\Theta'}^0$, then $\beta \in \operatorname{Vect}(\Theta') \subset \operatorname{Vect}(\Theta)$ since $\Theta' \subset \Phi_\Theta^0$.
Thus $\beta \in \Phi' \cap \Phi_\Theta^0$.

If $\beta \in (\Phi^+ \cap \Phi') \smallsetminus (\Phi')_{\Theta'}^0$, then there exists $\alpha' \in \Delta' \smallsetminus \Theta'$ such that $n'_{\alpha'}(\gamma) > 0$. Since $\alpha' \in \Delta' \smallsetminus \Theta'$, there is $\alpha \in \Delta \smallsetminus \Theta$ such that $n_\alpha(\alpha') > 0$.
Thus $n_\alpha(\beta) = \sum_{\alpha' \in \Delta'} n'_{\alpha'}(\beta) n_\alpha(\alpha') > 0$.
Hence $\beta \not\in \Phi_\Theta^0$.
Replacing $\beta$ by $-\beta$, which changes all the signs, we get that the same holds for $\beta \in (\Phi^- \cap \Phi') \smallsetminus (\Phi')_{\Theta'}^0$. Hence $\Phi' \cap \Phi_\Theta^0 = (\Phi')_{\Theta'}^0$.
\end{proof}

By definition, the group $\mathbf{G}(k)$ acts transitively on the set of $k$-apartments of $\mathcal{X}_k$. Thus, in order to describe $\mathcal{X}_k(x,V)$, we can assume without loss of generalities that $x+V$ is an affine subspace of the standard apartment $\mathbb{A}_0$.
Because the star of a point is a spherical building, we get, using the corresponding BN-pair, the following description of the star of a point directed by an affine subspace of some apartment.

\begin{proposition}\label{prop directed building}
Let $\Delta \subset \Phi$ be a basis and $x \in \mathbb{A}_0$.
\begin{enumerate}[label=(\arabic*)]
    \item\label{item dir building full} $\mathcal{X}_k(x)$ together with the collection of apartments $\mathcal{A}_k(x)$ is a spherical building of type $\Phi_x$ on which the group $\mathbf{U}_x(k)$ acts strongly transitively;
    \item\label{item dir building partial} for any $\Theta\subseteq \Delta$, the set $\mathcal{X}_k(x,V_0^\Theta)$ together with the collection of apartment $\mathcal{A}_k(x,V_0^\Theta)$ is naturally equipped with a spherical building structure of type $\Phi_\Theta^0 \cap \Phi_x$ (which is a root system by Lemma~\ref{lemma subroot system}) on which the group $\mathbf{U}_{x+V_0^\Theta}(k)$ acts strongly transitively;
\end{enumerate}

Note that~\ref{item dir building full} is the particular case of~\ref{item dir building partial} where $\Theta = \Delta$.
\end{proposition}

\begin{proof}

Using~\cite[6.4.23]{BT}, one can define\footnote{Precisely, we take $f=f_x$, we take $X^*=X$ to be the pointwise stabilizer of $\mathcal{X}_k(x)$ in $\mathbf{T}(K)$, so that the group $P_x = X \cdot \mathbf{U}_{x}(K)$ and $P_x^*$ correspond respectively to $X \cdot U_f$ and $X^* \cdot U_f^*$ and $f^*=f_{\Omega_w}$.}
subgroups $P_x$ and $P^*_x$ of $\mathbf{G}(K)$ that satisfies the following properties:
\begin{itemize}
    \item $P^*_x $ fixes $\mathcal{X}_k(x)$;
    \item for any $\alpha \in \Phi$, we have $P_x \supset \mathbf{U}_{\alpha,x}(K)$;
    \item the quotient group $\overline{G}_x = P_x / P^*_x$ admits a root group datum $\left(\overline{T}_x,(\overline{U}_\alpha)_{\alpha \in \Phi_x}\right)$ of type $\Phi_x$;
    \item by construction, $\overline{G}_x$ is generated by the $\overline{U}_\alpha$.
\end{itemize}

Denote by $\mu : P_x \to \overline{G}_x$ the canonical projection.
By construction, $\overline{U}_\alpha = \mu\left( \mathbf{U}_{\alpha,x}(K) \right)$ for any $\alpha \in \Phi_x$.
By definition of $\mathbf{U}_{\alpha,x}(K)$, of $P^*_x$ and by the writing in~\cite[6.4.9]{BT}, we have that:
\begin{equation}\label{eq Ubar alpha}
\mu\big( \mathbf{U}_{\alpha,x}(K) \big) \cong \theta_\alpha\big( \{ x \in K,\ \nu_\p(x) \geq - \alpha(x) \} \big) / \theta_\alpha\big( \{ x \in K,\ \nu_\p(x) > - \alpha(x) \} \big)
\end{equation}
In particular, $\mu\left( \mathbf{U}_{\alpha,x}(K) \right)$ is trivial for any $\alpha \in \Phi \smallsetminus \Phi_x$.
The existence of the root group datum of type $\Phi_x$ of $\overline{G}_x$ induces a spherical building structure on the star $\mathcal{X}_k(x)$ of $x$ together with a strongly transitive action of $\overline{G}_x$ on this building.
Its apartments are the intersection of $K$-apartments of $\mathcal{X}_K$ with $\mathcal{X}_k(x)$ (see~\cite[4.6.35]{BT2}).
Because $K$ is the completion of $k$ and the both have same residue field, we deduce from~\eqref{eq Ubar alpha} and isomorphism $\theta_\alpha: \mathbb{G}_a \to \mathbf{U}_\alpha$ that $\mu\big(\mathbf{U}_{\alpha,x}(K)\big) = \mu\big( \mathbf{U}_{\alpha,x}(k) \big)$ for any $\alpha \in \Phi$.
Hence $\mathcal{A}_k(x)$ is, in fact, the set of apartments of $\mathcal{X}_k(x)$ since $\overline{G}_x$ acts transitively on it and is generated by the $\mu\big( \mathbf{U}_{\alpha,x}(k) \big)$.
Since $\overline{G}_x$ acts strongly transitively on $\mathcal{X}_k(x)$, we deduce that $P_x$ acts strongly transitively on the spherical building $\mathcal{X}_k(x)$ via $\mu$.
Since $\mathbf{U}_x(k)$ is, by definition, generated by the $\mathbf{U}_{\alpha,x}(k)$ for $\alpha \in \Phi$, the group $\mathbf{U}_x(k)$ acts strongly transitively on $\mathcal{X}_k(x)$ via $\mu$.

Let $A_x := \mathbb{A}_0 \cap \mathcal{X}_k(x) \in \mathcal{A}_k(x)$, and let $\Omega_{x,\Theta} := \cl\big( x+V_0^\Theta \big)  \cap \mathcal{X}_k(x) \subset A_x$.
Denote by $\overline{G}_{x,\Theta} = \mu\big( \mathbf{U}_{x+V_0^\Theta}(K) \big) \subset \overline{G}_x$ and by $\overline{N}_{x,\Theta}= \mu\big( \mathbf{N}(K) \cap \mathbf{U}_{x+V_0^\Theta}(K) \big)$, so that $\overline{N}_{x,\Theta} = \overline{N}_x \cap \overline{G}_{x,\Theta}$ is the setwise stabilizer in $\overline{N}_x$ of $\Omega_{x,\Theta}$.
Indeed, $\mathbf{U}_{x+V_0^\Theta}(K)$ pointwise stabilizes $x+V_0^\Theta$ in $\mathcal{X}_K$ \cite[9.3]{L}, whence it setwise stabilizes its enclosure and, $\overline{N}_x$ setwise stabilizes $\mathcal{X}_k(x)$.

By definition, if $\alpha \in \Phi_\Theta^0$, then $\mathbf{U}_{\alpha,x+V_0^\Theta}(K)=\mathbf{U}_{\alpha,x}(K)$; and $\mathbf{U}_{\alpha,x+V_0^\Theta}(K)$ is trivial otherwise. Whence by~\cite[6.4.9]{BT}, we have
\[\mathbf{U}_{x+V_0^\Theta}(K) = \left(\prod_{\alpha \in \Phi_\Theta^0 \cap \Phi^+} \mathbf{U}_{\alpha,x}(K) \right)\left( \prod_{\alpha \in \Phi_\Theta^0 \cap \Phi^-} \mathbf{U}_{\alpha,x}(K) \right) N_{x+V_0^\Theta}\]
where $N_{x+V_0^\Theta} = \mathbf{N}(K) \cap \mathbf{U}_{x+V_0^\Theta}(K)$ according to~\cite[8.9(iv)]{L}.
Hence the group $\overline{G}_{x,\Theta}$ admits a generating root group datum $\left( \overline{T}_x, \left( \overline{U}_\alpha \right)_{\alpha \in \Phi_x \cap \Phi_\Theta^0} \right)$ inducing the BN-pair $\big( \overline{B}_{x,\Theta}, \overline{N}_{x,\Theta} \big)$ of type $\Phi_\Theta^0 \cap \Phi_x$ where $\overline{B}_{x,\Theta} = \overline{T}_x \cdot \left(\prod_{\alpha \in \Phi_\Theta^0 \cap \Phi_x^+} \overline{U}_\alpha\right)$ for any ordering on the product \cite[Prop. 10.5]{RouEuclidean}.

According to~\cite[9.7(i)]{L}, the group $\mathbf{U}_{x+V_0^\Theta}(K)$ acts transitively on the set of $K$-apartments of $\mathcal{X}_K$ containing $x+V_0^\Theta$. Hence, since the group $\overline{G}_{x,\Theta} = \mu\big( \mathbf{U}_{x + V_0^\Theta}(K) \big)$ is generated by the $\overline{U}_\alpha = \mu\big( \mathbf{U}_{\alpha,x}(k) \big)$ for $\alpha \in \Phi_\Theta^0 \cap \Phi_x $, it acts transitively on the set $\mathcal{A}_k(x,V_0^\Theta)$ of intersections of a $k$-apartment containing $x+V_0^\Theta$ with $\mathcal{X}_k(x)$.
Hence, the action of $\overline{G}_x$ on $\mathcal{X}_k(x)$ induces an action of the group $\overline{G}_{x,\Theta}$ on the subset $\mathcal{X}_k(x,V_0^\Theta)$.

The set of points of $\mathcal{X}_k(x,V_0^\Theta)$ fixed by $\overline{T}_x = \overline{B}_{x,\Theta} \cap \overline{N}_{x,\Theta}$ is
\[\{ z \in\mathcal{X}_k(x,V_0^\Theta),\ \forall t \in \overline{T}_x,\ t \cdot z = z\}= A_x \cap \mathcal{X}_k(x,V_0^\Theta) = A_x.\]
Since $A_x$ is the apartment of the spherical building $\mathcal{X}_k(x)$ corresponding to $\overline{T}_x = \overline{B}_x \cap \overline{N}_x$, we know that the pointwise stabilizer of $A_x$ in $\overline{G}_x$ is $\overline{T}_x$.
Hence, the pointwise stabilizer of $A_x$ in $\overline{G}_{x,\Theta}$ is $\overline{T}_x$.

Let $\mathfrak{C}_{x,\Theta}$ be the set of points in $\mathcal{X}_k(x,V_0^\Theta)$ fixed by $\overline{B}_{x,\Theta}$.
Then $\mathfrak{C}_{x,\Theta} \subset A_x$ since it is fixed by $\overline{T}_x$.
For any $\alpha \in \Phi_\Theta^0 \cap \Phi_x^+$, the root group $\overline{U}_\alpha$ fixes exactly an half-apartment of $A_x$ that we denote by $D_{\alpha,x} = \{ z \in A_x,\ -\alpha(z) \geq -\alpha(x)\}$.
Hence 
\begin{equation}\label{eq directed chamber}
\mathfrak{C}_{x,\Theta} = \bigcap_{\alpha \in \Phi_\Theta^0 \cap \Phi_x^+} D_{\alpha,x} = \{ z \in A_x,\ \forall \alpha \in \Phi_\Theta^0 \cap \Phi_x^+,\ -\alpha(z) \geq -\alpha(x)\}.
\end{equation}
Let $g \in \overline{G}_{x,\Theta}$ be a element pointwise stabilizing $\mathfrak{C}_{x,\Theta}$.
Then, there exists $\tilde{g} \in \mathbf{U}_{\mathfrak{C}_{x,\Theta}}(K)$ such that $g = \mu(\tilde{g})$.
Write it $\tilde{g} = nu^-u^+$ with $n \in \mathbf{N}(K) \cap \mathbf{U}_{\mathfrak{C}_{x,\Theta}}(K)$, $u^- \in \prod_{\alpha \in \Phi_\Theta^0 \cap \Phi_x^-} \mathbf{U}_{\alpha,\mathfrak{C}_{x,\Theta}}(K)$ and $u^+ \in \prod_{\alpha \in \Phi_\Theta^0 \cap \Phi_x^+} \mathbf{U}_{\alpha,\mathfrak{C}_{x,\Theta}}(K)$.
Then, by definition, $\mu\big( \mathbf{U}_{\alpha,\mathfrak{C}_{x,\Theta}}(K) \big)$ is trivial, for any $\alpha \in \Phi_\Theta^0 \cap \Phi_x^-$. Thus $\mu(u^-)$ is trivial and $\mu(u^+) \in \overline{B}_{x,\Theta}$ fixes $\mathfrak{C}_{x,\Theta}$.
Hence $\mu(n) \in \overline{N}_{x,\Theta}$ fixes $\mathfrak{C}_{x,\Theta}$ which is the Weyl chamber of $A_x$ with respect to the Weyl group $\overline{N}_{x,\Theta} / \overline{T}_x$, whence $\mu(n) \in \overline{T}_x$.
Hence $\overline{B}_{x,\Theta}$ is the pointwise stabilizer in $\overline{G}_{x,\Theta}$ of $\mathfrak{C}_{x,\Theta}$.

As a consequence, the BN-pair $\big( \overline{B}_{x,\Theta}, \overline{N}_{x,\Theta} \big)$ induces a spherical building structure of type $\Phi_\Theta^0 \cap \Phi_x$ on $\mathcal{X}_k(x,V_0^\Theta)$ whose the set of apartments is $\mathcal{A}_k(x,V_0^\Theta) = \overline{G}_{x,\Theta} \cdot A_x$ and the set of chambers is $\overline{G}_{x,\Theta} \cdot \mathfrak{C}_{x,\Theta}$.
Since $\overline{G}_{x,\Theta}$ is generated by the $\overline{U}_\alpha = \mu\big( \mathbf{U}_{\alpha,x+V_0^\Theta}(k) \big)$ and $\mathbf{U}_{x+V_0^\Theta}(k)$ is generated by the $\mathbf{U}_{\alpha,x+V_0^\Theta}(k)$, we get the result via $\mu$.
\end{proof}

\begin{proposition}
Let $x \in \mathcal{X}_k$ be any point.
Let $\mathbb{A}$ be a $k$-apartment containing $x$ and $V$ be the vector space spanned by some vector face $D$ of $\mathbb{A}$ so that $x+V$ is the affine subspace of $\mathbb{A}$ spanned by the sector face $Q(x,D)$.
Then, the star $\mathcal{X}_k(x,V)$ directed by $V$ of $x$ is naturally equipped with a structure of spherical building and the intersection of its chambers is
\[\Omega(x,V) := \cl(x+V) \cap \mathcal{X}_k(x).\]

Moreover, for any $k$-apartment $\mathbb{A}'$ containing $\Omega(x,V)$ and any vector space $V'$ spanned by some vector face $D'$ such that $x+V'$ is the affine subspace of $\mathbb{A}'$ spanned by $Q(x,D')$ such that $(x+V') \cap \mathcal{X}_k(x) = (x+V) \cap \mathcal{X}_k(x)$, the spherical buildings $\mathcal{X}_k(x,V)$ and $\mathcal{X}_k(x,V')$ are the same.
\end{proposition}

\begin{proof}
Because $\mathbf{G}(k)$ acts transitively on the set of $k$-apartments of $\mathcal{X}_k$ by definition, we can assume without loss of generalities that $\mathbb{A} = \mathbb{A}_0$.
By definition, every apartment containing $x+V$ contains its enclosure. Because the Weyl group induces a subgroup of $\mathbf{N}(k) \subset \mathbf{G}(k)$, we can assume without loss of generalities that $V=V_0^\Theta$ for some $\Theta \subset \Delta$.

We keep the notation introduced in the proof of Proposition~\ref{prop directed building}.
Since, by the writing~\eqref{eq directed chamber}, $\mathfrak{C}_{x,\Theta}$ is the intersection with $A_x$ of some half-apartments $D_{\alpha,x}$ that contain $x+V_0^\Theta$, we get that $\mathfrak{C}_{x,\Theta} \supset \cl(x+V_0^\Theta) \cap A_x$.
For any $g \in \overline{G}_{x,\Theta}$, since $g$ fixes $(x+V_0^\Theta) \cap A_x$, we get that $g \cdot \mathfrak{C}_{x,\Theta}$ also contains $\cl(x+V_0^\Theta) \cap A_x$.
Thus the intersection of these chambers contains $\cl(x+V_0^\Theta) \cap A_x$.

Conversely, the intersection of the chambers $g \cdot\mathfrak{C}_{x,\Theta}$, for $g \in \overline{G}_{x,\Theta}$, is contained in the intersection of the $n \cdot \mathfrak{C}_{x,\Theta}$, for $n \in \overline{N}_{x,\Theta}$, which is $\{z \in A_x,\ \forall \alpha \in \Phi_\Theta^0 \cap \Phi_x,\ -\alpha(z) =-\alpha(x)\}$ as intersection of the Weyl chambers of $A_x$, which is precisely $\cl(x+V_0^\Theta) \cap A_x$ thanks to a straightforward calculation.

According to~\cite[6.4.9]{BT} again, one can write
\[\mathbf{U}_{\Omega(x,V_0^\Theta)}(K) = \left( \prod_{\alpha \in \Phi^+} \mathbf{U}_{\Omega(x,V_0^\Theta)}(K) \right)\left( \prod_{\alpha \in \Phi^-} \mathbf{U}_{\Omega(x,V_0^\Theta)}(K) \right) N_{\Omega(x,V_0^\Theta)}\]
where $N_{\Omega(x,V_0^\Theta)} = \mathbf{N}(K) \cap \mathbf{U}_{\Omega(x,V_0^\Theta)}(K)$.
Observe that, by definition, $\mu\big(N_{\Omega(x,V_0^\Theta)}\big) = \mu\big( N_{x+V_0^\Theta}\big) = \overline{N}_{x+V_0^\Theta}$ is the stabilizer in $\overline{N}_x$ of $\Omega(x,V_0^\Theta)$.
For any $\alpha \in \Phi$, if $\alpha \not\in \Phi_x \cap \Phi_0^\Theta$, we have that $\mu\big( \mathbf{U}_{\alpha,\Omega(x,V_0^\Theta)}(K) \big)$ is trivial and, if $\alpha \in \Phi_x \cap \Phi_0^\Theta$, we have $\mu\big( \mathbf{U}_{\alpha,\Omega(x,V_0^\Theta)}(K) \big) = \overline{U}_\alpha$.
Hence $\mu\big( \mathbf{U}_{\Omega(x,V_0^\Theta)}(K)\big) = \mu\big( \mathbf{U}_{x+V_0^\Theta}(K)\big) = \overline{G}_{x,\Theta}$.

As a consequence, for any $k$-apartment $\mathbb{A}'$ containing $\Omega(x,V_0^\Theta)$ and any vector space $V'$ so that $x+V'$ is an affine subspace of $\mathbb{A}'$ such that $(x+V') \cap \mathcal{X}_k(x) = (x+V_0^\Theta) \cap \mathcal{X}_k(x)$, we have that $\Omega(x,V_0^\Theta) = \Omega(x,V')$ and, therefore, $\mathcal{X}_k(x,V_0^\Theta)= \overline{G}_{x,\Theta} \cdot A_x= \mu\big( \mathbf{U}_{\Omega(x,V_0^\Theta)}(K)\big) \cdot A_x  = \mathcal{X}_k(x,V') $.
\end{proof}

\begin{notation}\label{not directed face}
Given a $k$-sector face $Q(v,D)$ with tip $v$ and direction $D$ of the building $\mathcal{X}_k$, we denote by $\mathfrak{C}(v,D)$ the maximal $k$-face of the enclosure of the germ at $v$ of the $k$-sector face $Q(v,D)$.\nomenclature[]{$\mathfrak{C}(v,D)$}{see Notation~\ref{not directed face}}
\end{notation}

\begin{proposition}\label{prop starAction}
Let $Q(x,D)$ be any $k$-sector face of $\mathcal{X}_k$.
There exists a subsector face $Q(y_2,D)\subseteq Q(x,D)$ and a $k$-apartment $\mathbb{A}$ containing $Q(y_2,D)$ such that for any point $v \in Q(y_2,D)$:
\nomenclature[]{$y_2$}{see Proposition~\ref{prop starAction}}
\begin{enumerate}[label=(\arabic*)]
    \item \label{item starAction fix} the stabilizer of $v$ in $\mathbf{G}(A)$ fixes $Q(v,D)$;
    \item \label{item starAction orbits} if $V$ is the vector space such that $v+V$ is the affine subspace spanned by $Q(v,D)$ in $\mathbb{A}$, then the $\stab_{\mathbf{G}(A)}(v)$-orbits of elements in $\mathcal{X}_k(v,V)$ cover $\mathcal{X}_k(v)$;
    \item \label{item starAction fixed faces} for any face $F \in \mathcal{X}_k(v)$, if for any $g \in \stab_{\mathbf{G}(A)}(v)$, we have $g \cdot F = F$, then $\mathcal{R}_k\big(\mathfrak{C}(v,D)\big) \cap \mathcal{R}_k(F) \neq \emptyset$
\end{enumerate}

Moreover, there is $h \in H$ such that $h \cdot Q(y_2,D) \subset h \cdot \mathbb{A} = \mathbb{A}_0$.
\end{proposition}

\begin{proof}
According to Proposition~\ref{prop igual stab}, it suffices to find a subsector face of $Q(y_1,D)$ satisfying statements~\ref{item starAction orbits} and~\ref{item starAction fixed faces}. In order to provide it, we extend the ideas of Soul\'e, exposed in \cite[1.2]{So}, to our context.
Using Lemma~\ref{lemma WUW}, we consider $n \in N^{\mathrm{sph}}$, $\Theta \subseteq \Delta$ and $u \in \mathbf{U}^+(k)$ such that $u \cdot D = n \cdot D_0^\Theta$.
We denote by $h = n^{-1} u \in H$ where $H$ is defined as in Formula~\eqref{eq def H}.
Let $\mathbb{A} = h^{-1} \cdot \mathbb{A}_0 = u^{-1} \cdot \mathbb{A}_0$.
Then, there is a subsector $Q(y'_1,D)$ of $Q(y_1,D)$ such that $Q(y'_1,D) \subset \mathbb{A}$ and $h \cdot Q(y'_1,D) \subset \mathbb{A}_0$.
Let $V_0^\Theta = \operatorname{Vect}(D_0^\Theta) \subseteq \mathbb{A}_0$ and let $V = h^{-1} \cdot V_0^\Theta$ so that $y'_1+V$ is the affine subspace of $\mathbb{A}$ spanned by $Q(y'_1,D)$.
Let $v \in Q(y'_1,D)$ and set $w = h \cdot v \in \mathbb{A}_0$.
Then $\mathcal{X}_k(w) = h \cdot \mathcal{X}_k(v)$ and $\mathcal{X}_k(w,V_0^\Theta) = h \cdot \mathcal{X}_k(v,V)$.
Hence $$ \stab_{\mathbf{G}(A)}(v) \cdot \mathcal{X}_k(v,V) = \mathcal{X}_k(v),$$ is equivalent to $$ \left( h \mathbf{G}(A) h^{-1} \cap \stab_{\mathbf{G}(K)}(w) \right) \cdot \mathcal{X}_k(w,V_0^\Theta) = \mathcal{X}_k(w).$$
Thus, we are reduced to work inside the standard apartment $\mathbb{A}_0$ in which it suffices to prove that for any $h \in H $, any $\Theta \subset \Delta$ and any $w_0 \in \mathbb{A}_0$, there exists a subsector face $Q(y_2',D_0^\Theta) \subseteq Q(w_0, D_0^\Theta)$ such that
\begin{multline}\label{equality action on star}
    \left( h \mathbf{G}(A) h^{-1} \cap \stab_{\mathbf{G}(K)}(w) \right) \cdot \mathcal{X}_k(w,V_0^\Theta) = \mathcal{X}_k(w),\\ \text{ for any point } w \in Q(y_2',D_0).
\end{multline}
For any root $\alpha \in \Phi$, we fix a non-zero $A$-ideal $J_{\alpha}(h) \subseteq M_{\alpha}(h)$ given by Proposition~\ref{prop ideal contained}.

%
It follows from Lemma~\ref{lem subsector face and roots}\ref{item subsector real} that there exists a subsector $Q(y_2',D_0^\Theta)$ of $Q(w_0,D_0^\Theta)$ such that for any $w \in Q(y_2',D_0^\Theta)$ and any $\alpha \in \Phi_\Theta^+$, we have
$$\big(\alpha(w)-1\big) d\geq \deg\big(J_{\alpha}(h)\big)+2g-1,$$
and then, we deduce from the Riemann-Roch identity~\eqref{eq rr0} that the equality:
\begin{equation}\label{rri}
\dim_{\mathbb{F}}\left( J_{\alpha}(h) \cap \pi^{-m}\mathcal{O} \right) =m\deg(\p)-\deg\big(J_{\alpha}(h)\big)+1-g,
\end{equation}
holds for any $m \in \mathbb{N}$ such that $m \geq \alpha(w)-1$.

Consider any point $w$ of $Q(y_2',D_0^\Theta)$ and denote by $A_w = \mathbb{A}_0 \cap \mathcal{X}_k(w)$.
By definition of buildings, for any point $z \in \mathcal{X}_k(w)$, there is a $k$-apartment $\mathbb{A}_z$ of $\mathcal{X}_k$ containing both $z$ and the $k$-face $\mathfrak{C}(w,D_0^{\Theta})$.
According to~\cite[9.7(i)]{L}, we know that $\mathbf{U}_{\mathfrak{C}(w,D_0^\Theta)}(K)$ acts transitively on the set of $K$-apartments of $\mathcal{X}_K$ containing $\mathfrak{C}(w,D_0^\Theta)$.
Hence, there is an element $g \in \mathbf{U}_{\mathfrak{C}(w,D_0^\Theta)}(K)$ such that $g \cdot \mathbb{A}_z = \mathbb{A}_0$.
Since $g \cdot w = w$ and $z \in \mathcal{X}_k(w)$, we also have that $g \cdot z \in g \cdot \mathcal{X}_k(w) = \mathcal{X}_k(w)$.
Hence $g \cdot z \in A_w$.

According to~\cite[6.4.9]{BT}, one can write
\[\mathbf{U}_{\mathfrak{C}(w,D_0^\Theta)}(K) = N_{\mathfrak{C}(w,D_0^\Theta)} \left(\prod_{\alpha \in \Phi^-} \mathbf{U}_{\alpha,\mathfrak{C}(w,D_0^\Theta)}(K) \right) \left(\prod_{\alpha \in \Phi^+} \mathbf{U}_{\alpha,\mathfrak{C}(w,D_0^\Theta)}(K) \right),\]
for any ordering on each of both products of root groups,
where $N_{\mathfrak{C}(w,D_0^\Theta)} = \mathbf{N}(K) \cap \mathbf{U}_{\mathfrak{C}(w,D_0^\Theta)}(K)$ is the subgroup of element of $\mathbf{N}(K)$ pointwise stabilizing $\mathfrak{C}(w,D_0^\Theta)$ according to~\cite[8.9(iv)]{L}.
In particular, decomposing $\Phi^+ = (\Phi_\Theta^0 \cap \Phi^+) \sqcup \Phi_\Theta^+$, one can write $g = n g_- g_0 g_+$ with
\begin{align*} n \in& N_{\mathfrak{C}(w,D_0^\Theta)},& g_- \in& \left(\prod_{\alpha \in \Phi^-} \mathbf{U}_{\alpha,\mathfrak{C}(w,D_0^\Theta)}(K) \right),\\
g_0 \in & \left(\prod_{\alpha \in \Phi_\Theta^0 \cap \Phi^+} \mathbf{U}_{\alpha,\mathfrak{C}(w,D_0^\Theta)}(K) \right), &
g_+ \in & \left(\prod_{\alpha \in \Phi_\Theta^+} \mathbf{U}_{\alpha,\mathfrak{C}(w,D_0^\Theta)}(K) \right).
\end{align*}

Let $\mu : \mathbf{U}_w(K) \to \overline{G}_w$ be the quotient morphism defined as in proof of Proposition~\ref{prop directed building}.
Because $\mathbf{N}(K)$ stabilizes $\mathbb{A}_0$, we have that $n \in N_{\mathfrak{C}(w,D_0^\Theta)}$ stabilizes $\mathbb{A}_0 \cap \mathcal{X}_k(w)$, whence $\mu(n) \cdot A_w = A_w$.
For $\alpha \in \Phi_\Theta^-$, the group $\mu\left(\mathbf{U}_{\alpha,\mathfrak{C}(w,D_0^\Theta)}(K)\right)$ is trivial, whence 
$\mu(g_-) \in \left(\prod_{\alpha \in \Phi_\Theta^0 \cap \Phi^-} \overline{U}_\alpha\right)$.
For $\alpha \in \Phi \smallsetminus \Phi_w$, the group $\mu\big( \mathbf{U}_{\alpha,\mathfrak{C}(w,D_0^\Theta)}(K)\big)$ is trivial, whence $\mu(n),\mu(g_0),\mu(g_+) \in \overline{G}_{w,\Theta}$ with the notation of the proof of Proposition~\ref{prop directed building}.
Thus
\[g_+ \cdot z = \mu(g_+) \cdot z \in \overline{G}_{w,\Theta} \cdot A_w =\mathcal{X}_k(w,V_0^\Theta).\]

Thus, in order to deduce Equality~\eqref{equality action on star}, it suffices to prove that any such $g_+$ satisfies $\mu(g_+) \in \mu\left( h \mathbf{G}(A) h^{-1} \cap \stab_{\mathbf{G}(K)}(w)\right)$.
In particular, because $\mu\big( \mathbf{U}_{\alpha,\mathfrak{C}(w,D_0^\Theta)}(K) \big)$ is trivial whenever $\alpha \in \Phi \smallsetminus \Phi_w$, it suffices to prove that:
\begin{equation}\label{equation root groups of stabilizer}
    \mu\left(\mathbf{U}_{\alpha,w}(K) \right) = \overline{U}_\alpha \subseteq \mu\left( h \mathbf{G}(A) h^{-1} \cap \mathbf{U}_w(K)\right) \qquad \forall \alpha \in \Phi_w \cap \Phi_\Theta^+.
\end{equation}

Let $\alpha \in \Phi_w \cap \Phi_\Theta^+$.
By definition of $\Phi_w$, we have that $n(\alpha,w) := \alpha(w) \in \Gamma'_\alpha$ where $\Gamma'_\alpha$ is the set of values associated to $\alpha$.
Because $\mathbf{G}$ is split, we have $\Gamma'_\alpha = \mathbb{Z}$ with the usual normalization $\nu_\p(k^\times) = \mathbb{Z}$.
Thus, it follows from the isomorphism $\theta_\alpha: \mathbb{G}_a \stackrel{\simeq}{\to} \mathbf{U}_\alpha$, the definitions of $\mathbf{U}_{\alpha,w}(K)$, of $\Phi_w$ and of $\mu$ that
$$\mu\big( \mathbf{U}_{\alpha,w}(K) \big) = \theta_\alpha\big(\pi^{-n(w,\alpha)} \mathcal{O} \big) / \theta_\alpha\big(\pi^{-n(w,\alpha)+1} \mathcal{O} \big)
\cong \mathbf{U}_\alpha(\mathbb{F}(\p)),$$
where $\mathbb{F}(\p)$ is the residue field of $K$ and $\pi$ is a uniformizer.
Note that since $h \in \mathbf{G}(k)$, we have that
$$h\mathbf{G}(A) h^{-1}  \cap \mathbf{U}_{\alpha,w}(K) = h \mathbf{G}(A) h^{-1} \cap \mathbf{U}_\alpha(k) \cap \mathbf{U}_{\alpha,w}(K) = \theta_\alpha\left( M_\alpha(h) \cap \pi^{-n(w,\alpha)} \mathcal{O} \right).$$
Hence, since $M_{\alpha}(h) \supseteq J_\alpha(h)$ and $\mathbf{U}_w(K) \supseteq \mathbf{U}_{\alpha,w}(K)$, we have
$$\mu\Big( h \mathbf{G}(A) h^{-1} \cap \mathbf{U}_w(K) \Big) \supseteq \mu\bigg( \theta_\alpha\Big( J_\alpha(h) \cap \pi^{-n(w,\alpha)} \mathcal{O} \Big) \bigg).$$
But 
$$\mu\bigg( \theta_\alpha\Big( J_\alpha(h) \cap \pi^{-n(w,\alpha)} \mathcal{O} \Big) \bigg)
\supseteq \theta_\alpha\Big( J_\alpha(h) \cap \pi^{-n(w,\alpha)} \mathcal{O} \Big) / \theta_\alpha\Big( J_\alpha(h) \cap \pi^{-n(w,\alpha)+1} \mathcal{O} \Big).$$
The latter is equal to $\theta_\alpha\Big(\pi^{-n(w,\alpha)} \mathcal{O} \Big) / \theta_\alpha\Big(\pi^{-n(w,\alpha)+1} \mathcal{O} \Big)$ if and only if we have
\begin{equation}\label{eq J alpha last}
\dim_{\mathbb{F}}\left( ( J_{\alpha}(h) \cap \pi^{-n(w,\alpha)}\mathcal{O} )/( J_{\alpha}(h) \cap \pi^{-n(w,\alpha)+1}\mathcal{O} )\right)= \deg(\p).
\end{equation}
Applying Equality~\eqref{rri} to $m=n(w,\alpha)$ and to $m=n(w,\alpha)-1$, we deduce Equality~\eqref{eq J alpha last}.
Hence, we deduce Equality~\eqref{equation root groups of stabilizer} for any $\alpha \in \Phi_w \cap \Phi_\Theta^+$ and this concludes the proof of statement~\eqref{item starAction orbits}.

Consider the BN-pair $(\overline{B}_w,\overline{N}_w)$ introduced in the proof of Proposition~\ref{prop directed building} and denote by $C_w^f \in \mathcal{R}_k(w)$ the standard chamber of the spherical building $\mathcal{X}_k(w)$ associated to this BN-pair.

Let $C \in \mathcal{R}_k(w,V_0^\Theta)$ and note that $C_w^f \in \mathcal{R}_k(w,V_0^\Theta)$.
In the spherical building $\mathcal{X}_k(w,V_0^\Theta)$, the $k$-chambers $C$ and $C_w^f$ of $\mathcal{X}_k$ are contained, as subsets, in some chambers of the spherical building $\mathcal{X}_k(w,V_0^\Theta)$.
Hence there is an apartment of this spherical building containing both $C$ and $C_w^f$.
Thus, by definition of the set of apartments $\mathcal{A}_k(w,V_0^\Theta)$ of the building $\mathcal{X}_k(w,V_0^\Theta)$, there is a $k$-apartment $\mathbb{A}$ containing $C$, $C_w^f$ and $w+V_0^\Theta$.

Because the apartment $\mathbb{A}_0$ contains $C_w^f$ and $w+V_0^\Theta$, by~\cite[13.7(i)]{L}, there exists an element $g \in \mathbf{U}_{C_w^f \cup (w+ V_0^\Theta)}(K)$ such that $g \cdot \mathbb{A} = \mathbb{A}_0$.
Hence $g \cdot C \in \mathbb{A}_0 \cap \mathcal{R}_k(w) = \overline{N}_w \cdot C_w^f$.
Moreover $g$ stabilizes $C_w^f$ whence, by uniqueness of the writings of $\overline{B}_w = \prod_{\alpha \in \Phi_w^+} \overline{U}_\alpha \cdot \overline{T}_w$ and~\cite[6.4.9(iii)]{BT}
\[\mu\left( \mathbf{U}_{w+V_0^\Theta}(K) \right) = \left(\prod_{\alpha \in \Phi_0^\Theta \cap \Phi_w^+} \overline{U}_\alpha\right)\left(\prod_{\alpha \in \Phi_0^\Theta \cap \Phi_w^-} \overline{U}_\alpha\right) \mu\left(\mathbf{U}_{w+V_0^\Theta}(K) \cap \mathbf{N}(K)\right) ,\] we get
\[\mu(g) \in \overline{B}_w \cap \mu\left( \mathbf{U}_{w+V_0^\Theta}(K) \right) \subseteq \left(\prod_{\alpha \in \Phi_\Theta^0 \cap \Phi_w^+} \overline{U}_\alpha\right) \cdot \overline{T}_w.\]
Thus, we have that
\begin{equation}\label{eq directed star}
    \mathcal{R}_k(w,V_\Theta^0) = \left( \prod_{\alpha \in \Phi_\Theta^0 \cap \Phi_w^+} \overline{U}_\alpha \right) \overline{N}_w \cdot C_w^f.
\end{equation}
Analogously, using the same arguments, we get that $\mu\Big( \mathbf{U}_{\mathfrak{C}(w,D_0^\Theta)}(K) \Big)$ acts transitively on $\mathcal{R}_k\big( \mathfrak{C}(w,V_\Theta^0) \big)$, whence we also have that
\begin{equation}\label{eq directed residue}
    \mathcal{R}_k\big( \mathfrak{C}(w,D_0^\Theta) \big) =   \left( \prod_{\alpha \in \Phi_\Theta^0 \cap \Phi_w^+} \overline{U}_\alpha \right) \overline{N}_{\mathfrak{C}(w,D_0^\Theta)} \cdot C_w^f,
\end{equation}
where $\overline{N}_{\mathfrak{C}(w,D_0^\Theta)}$ denotes the subgroup of $\overline{N}_w$ preserving $\mathfrak{C}(w,D_0^\Theta)$.

Suppose that $F \in \mathcal{X}_k(w,V_0^\Theta)$ is a $k$-face such that $\mathcal{R}_k(F) \cap \mathcal{R}_k\big( \mathfrak{C}(w,D_0^\Theta)\big) = \emptyset$.
According to Equation~\eqref{eq directed star}, there is $u \in \prod_{\alpha \in \Phi_\Theta^0 \cap \Phi_w^+} \overline{U}_\alpha$ such that $u \cdot F \subset A_w$.
By Equation~\eqref{eq directed residue}, we also have that $u \cdot \mathcal{R}_k\big(\mathfrak{C}(w,D_0^\Theta)\big) = \mathcal{R}_k\big(\mathfrak{C}(w,D_0^\Theta)\big)$.
Thus $\mathcal{R}_k(u \cdot F) \cap \mathcal{R}_k\big(\mathfrak{C}(w,D_0^\Theta)\big) = u \cdot \Big( \mathcal{R}_k(F) \cap \mathcal{R}_k\big(\mathfrak{C}(w,D_0^\Theta)\big) \Big) = \emptyset$.
In particular, the convex sets of chambers $ \mathbb{A}_0 \cap \mathcal{R}_k\big(\mathfrak{C}(w,D_0^\Theta)\big)$ and $ \mathbb{A}_0 \cap \mathcal{R}_k\big(u \cdot F\big)$ of the apartment $A_w = \mathbb{A}_0 \cap \mathcal{X}_k(w)$ does not intersect.
Hence, in the spherical building $\mathcal{X}_k(w)$ of type $\Phi_w$, there is a wall associated to some root $\beta \in \Phi_w^+$ separating those sets.

We claim that $\beta \not\in \Phi_\Theta^0$.
Indeed, if $\beta$ were in $\Phi_\Theta^0$, then the reflection with respect to $\beta$ of the Weyl group $\overline{N}_w / \overline{T}_w$ would fix $w+V_0^\Theta \supset \mathfrak{C}(w,D_0^\Theta)$.
Thus, there would be an element $n_\beta \in \overline{N}_{\mathfrak{C}(w,D_0^\Theta)}$ lifting $s_\beta$.
We would have two chambers $n_\beta \cdot C_w^f$ and $C_w^f$ both in $\mathcal{R}_k\big( \mathfrak{C}(w,D_0^\Theta)\big)$ and separated by the wall with respect to $\beta$, which is a contradiction with the definition of $\beta$.
Thus $\beta \in \Phi_w^+ \smallsetminus \Phi_\Theta^0 = \Phi_\Theta^+ \cap \Phi_w$.

In the affine apartment $\mathbb{A}_0$, the wall with respect to $\beta$ containing $w$ is the affine subspace $\{ z \in \mathbb{A}_0,\ \beta(z) = \beta(w)\}$.
Because $\beta(C_w^f) > \beta(w)$, we have that $\beta(C) < \beta(w)$ for any $C\in \mathcal{R}_k(u \cdot F) \cap \mathbb{A}_0$. Hence $\beta(u \cdot F) < \beta(w)$ by the same argument as before.
Thus, there exists $u_\beta \in \overline{U}_\beta = \mu\big( \mathbf{U}_{w,\beta}(K) \big)$ that does not fix $u \cdot F$.
Indeed, the set of points fixed by any element of $\mathbf{U}_{w,\beta}(K)$ with valuation $-\beta(w)$ is exactly $\left\{z \in \mathbb{A}_0,\ \beta(z) \geqslant \beta(w) \right\}$ for $\beta \in \Phi_w$ \cite[7.4.5]{BT}.

Thus, we have shown that there exists a root $\beta \in \Phi_\Theta^+ \cap \Phi_w$, an element $u_\beta \in \overline{U}_\beta $ and an element $u \in \prod_{\alpha \in \Phi_\Theta^0 \cap \Phi_w^+} \overline{U}_\alpha$ such that $u_\beta u \cdot F \neq u \cdot F$.
For any $\alpha \in \Phi_\Theta^0 \cap \Phi_w^+$, the root group $\overline{U}_\alpha$ normalizes the group $\prod_{\gamma \in \Phi_\Theta^+ \cap \Phi_w} \overline{U}_\gamma$ (by axioms of a root group datum \cite[6.1.1(DR2)]{BT}).
Hence $u^{-1} u_\beta u \in \prod_{\gamma \in \Phi_\Theta^+ \cap \Phi_w} \overline{U}_\gamma$.
But since $\overline{U}_\gamma \subset \mu\big( \mathbf{U}_\gamma(K) \cap \stab_{\mathbf{G}(A)}(w) \big)$ for any $\gamma \in \Phi_\Theta^+ \cap \Phi_w$ according to~\eqref{equation root groups of stabilizer}, we deduce that there is $g \in \stab_{\mathbf{G}(A)}(w)$ such that $\mu(g) = u^{-1} u_\beta u$.
Hence $g \cdot F = u^{-1} u_\beta u \cdot F \neq F$.
Thus $\stab_{\mathbf{G}(A)}(w)$ does not stabilizes $F$.

As a consequence, if $F$ is stabilized by $\stab_{\mathbf{G}(A)}(w)$, we get that $F \subset \mathcal{X}_k(w,V_0^\Theta)$ by~\ref{item starAction orbits} and, in that case, we have shown that $\mathcal{R}_k(F) \cap \mathcal{R}_k\big( \mathfrak{C}(w,D_0^\Theta)\big) \neq \emptyset$. Whence the result of~\ref{item starAction fixed faces} follows.
\end{proof}

\begin{corollary}\label{corollary unique fixed face}
Let $\Theta \subset \Delta$. Let $Q(x,D)$ be a $k$-sector face such that $D \in \mathbf{G}(k) \cdot D_0^\Theta$.
Let $y_2$ as in Proposition~\ref{prop starAction} and let $z \in Q(y_2,D)$ be an point.
Recall that $\mathcal{X}_k(z)$ is a spherical building of type $\Phi_z$ (see Proposition~\ref{prop directed building}).
Let $\Delta_z$ be the basis of $\Phi_z$ induced by $\Delta$ and let $\Theta_z \subset \Delta_z$ so that $\Phi_z \cap \Phi_\Theta^0 = (\Phi_z)_{\Theta_z}^0$ (see Lemma~\ref{lemma subroot system}).
Consider the action of $\stab_{\mathbf{G}(A)}(z)$ on the set of faces of $\mathcal{X}_k(z)$.
Then the set of faces setwise stabilized by this action (i.e.~$F$ such that $g \cdot F = F\ \forall g \in \stab_{\mathbf{G}(A)}(z)$) contains a unique face of type $\Theta_z$, which is $\mathfrak{C}(z,D)$.
\end{corollary}

\begin{proof}
According to Proposition~\ref{prop starAction}\ref{item starAction fix}, the subset $Q(z,D) \cap \mathcal{X}_k(z)$ is fixed by $\stab_{\mathbf{G}(A)}(z)$, whence $\cl \big( Q(z,D) \big) \cap \mathcal{X}_k(z)$ is setwise stabilized by  $\stab_{\mathbf{G}(A)}(z)$.
Thus the maximal $k$-face $\mathfrak{C}(z,D)$ of $\cl \big( Q(z,D) \big) \cap \mathcal{X}_k(z)$ is setwise stabilized, which is of type $\Theta_z$ by construction.

Let $F$ be a $k$-face contained in $\mathcal{X}_k(z)$ which is setwise stabilized by $\stab_{\mathbf{G}(A)}(z)$.
According to Proposition~\ref{prop starAction}\ref{item starAction fixed faces}, there is a $k$-chamber $C \in \mathcal{R}_k\big(\mathfrak{C}(z,D)\big)$ such that $F$ is a $k$-face of $C$.
Suppose that $F$ is of type $\Theta_z$.
Thus inside the spherical building $\mathcal{X}_k(z)$ of type $\Phi_z$, since $\mathfrak{C}(z,D)$ is a $k$-face of $C$ of type $\Theta_z$, we get that $F = \mathfrak{C}(z,D)$, by uniqueness of a face with given type of a chamber.
Hence $F = \mathfrak{C}(z,D)$ is the unique face of $\mathcal{X}_k(z)$ of type $\Theta_z$ which is stabilized by $\stab_{\mathbf{G}(A)}(z)$.
\end{proof}

\begin{corollary}\label{cor-transportAlcoves}
Let $Q=Q(x,D)$ and $Q'=Q(x',D')$ be two $k$-sector faces in $\mathcal{X}_k$ such that $D,D' \in \mathbf{G}(k) \cdot D_0^{\Theta}$, for some $\Theta \subset \Delta$.
Let $Q(y_2,D) \subseteq Q(x,D)$ and $Q(y'_2,D) \subseteq Q(x',D')$ be the subsectors given by Proposition~\ref{prop starAction}.
For any two points $z \in Q(y_2,D) \subseteq Q$, $z' \in Q(y_2',D') \subseteq Q'$, and any $g \in \mathbf{G}(A)$:
\[
   g \cdot z=z'  \Rightarrow  g \cdot \mathfrak{C}(z,D) = \mathfrak{C}(z',D').
\]
\end{corollary}

\begin{proof}
Let $z \in Q(y_2,D)$, $z' \in Q(y_2',D')$ and $g \in \mathbf{G}(A)$ be such that $g \cdot z = z'$.
By definition, we have that $g \cdot \mathcal{X}_k(z) = \mathcal{X}_k(z')$ whence the spherical buildings $\mathcal{X}_k(z)$ and $\mathcal{X}_k(z')$ have the same type $\Phi_z = \Phi_{z'}$.
Then, according to Corollary~\ref{corollary unique fixed face}, the $k$-face $\mathfrak{C}(z,D)$ (resp.~$\mathfrak{C}(z',D')$) is the unique face of type $\Theta_z$ stabilized by the group $\stab_{\mathbf{G}(A)}(z)$ (resp.~$\stab_{\mathbf{G}(A)}(z')$), in the spherical building $\mathcal{X}_k(z)$ (resp.~$\mathcal{X}_k(z')$).
Thus, $g \cdot \mathfrak{C}(z,D) = \mathfrak{C}(z',g \cdot D)$ is a $k$-face which is a face of $g \cdot \mathcal{X}_k(z) = \mathcal{X}_k(z')$.
It is of type $\Theta_{z}=\Theta_{z'}$, by Lemma~\ref{lemma subroot system}, since $g \cdot D \in \mathbf{G}(k) \cdot D_0^\Theta$ and $\Phi_z = \Phi_{z'}$.
This face  $g \cdot \mathfrak{C}(z,D)$ is stabilized by the group $g \stab_{\mathbf{G}(A)}(z) g^{-1} = \stab_{\mathbf{G}(A)}(z')$ whence, by uniqueness, we get that $g \cdot \mathfrak{C}(z,D) = \mathfrak{C}(z',D')$.
\end{proof}

\begin{notation}
For any $\Theta \subset \Delta$, denote by $\Sec^\Theta_k$ the set of $k$-sector faces whose the direction $D$ belongs to $\mathbf{G}(k) \cdot D_0^\Theta$.
We denote by $\sSec_k^\Theta$ the $k$-sector faces in $\Sec_k^\Theta$ whose the tip is a special vertex.

For instance, $\Sec_k^\Delta$ is the set of points and $\sSec_k^\Delta$ is the set of special $k$-vertices.
The set of rational sector chambers is $\Sec_k^\emptyset = \Sec(\mathcal{X}_k)$.
\end{notation}

We now recursively apply Corollary \ref{cor-transportAlcoves} in order to prove that, inside a small enough subsector face of a given $k$-sector face $Q(x,D)$, if any two vertices are in the same $\mathbf{G}(A)$-orbit, the action of a certain element of $\mathbf{G}(A)$ onto the subsector stabilizes its visual boundary. More precisely,

\begin{proposition}\label{proposition facets in the same orbit ArXiv}
For any $\Theta \subset \Delta$, there exists a map $t_\Theta: \sSec_k^\Theta \to \sSec_k^\Delta$ such that for any sector faces $Q(x,D), Q(x',D') \in \Sec_k^\Theta$ whose tips $x,x'$ are special vertices, the two subsector faces $Q\big(y_3,D) \subseteq Q(x,D)$ and $Q(y_3',D') \subseteq Q(x',D')$ with $y_3 = t_\Theta\big( Q(x,D) \big)$ and $y'_3 = t_\Theta\big( Q(x',D') \big)$ satisfy that,
for any $v \in Q(y_3,D)$, any $w \in Q(y_3',D')$ and any $g \in \mathbf{G}(A)$: 
\[
  w=g \cdot v \Rightarrow  D'= g \cdot D.
\]
\end{proposition}

In fact, one can omit the ``special vertices'' assumption making a use of ramified extensions. See Corollary~\ref{cor conclusion without special hyp} for a more general statement.

\begin{proof}
Given a $k$-sector face $Q(x,D)$, there exists a point $y_2(x,D) \in \mathcal{X}_k$ such that $Q\big(y_2(x,D),D\big) \subseteq Q(x,D)$ satisfies the conditions of Proposition~\ref{prop starAction}.
According to Lemma~\ref{lem subsector face and roots}\ref{item subsector enclosure}, whenever $x$ is special, there exists a special vertex $y_3(x,D) \in Q\big(y_2(x,D),D\big)$ such that $\cl_k\Big( Q\big( y_3(x,D), D\big) \Big) \subset Q\big(y_2(x,D),D\big) \subseteq Q(x,D)$.
Hence, we can define a special $k$-vertex by $t_\Theta\big(Q(x,D)\big) := y_3(x,D) \in \mathcal{X}_k(D)$.
This defines a map $t_\Theta : \sSec_k^\Theta \to \sSec_k^\Delta$.

Consider $Q(x,D)$ and $Q(x',D') \in \sSec_k^\Theta$ and denote by $y_3 := t_\Theta\big( Q(x,D) \big)$ and by $y'_3 := t_\Theta\big( Q(x',D') \big)$.
We introduce some notation for this proof.
Let $\mathbb{A}$ be an apartment of $\mathcal{X}_k$ that contains $Q(y_3,D)$ and let $T_1$ be the maximal $k$-torus of $\mathbf{G}_k := \mathbf{G} \otimes k$ associated to $\mathbb{A}$.
Denote by $\Phi_1=\Phi(\mathbf{G}_k,T_1)$ the root system of $\mathbf{G}_k$ associated to $T_1$.
There is a vector chamber $\widetilde{D}$ so that $D$ is a face of $\widetilde{D}$ and a basis $\Delta_1=\Delta(\widetilde{D})$ of $\Phi_1$ associated to $\widetilde{D}$.
Denote by $\Phi_1^+$ the subset of positive roots of $\Phi_1$ associated to $\Delta_1$.
Let $\Theta_D \subset \Delta_1$ the subset of simple roots associated to $D$, i.e.
\[D=\{x \in \mathbb{A}: \forall \alpha \in \Delta_1 \smallsetminus \Theta_D, \, \alpha(x) > 0\  \text{ and } \forall \alpha \in \Theta_D, \, \alpha(x)=0\ \}.\]
As in section~\ref{intro vector faces}, introduce $\Phi_{D}^0 = \{\alpha \in \Phi:\ \alpha \in \operatorname{Vect}_\mathbb{R}(\Theta_D)\}$, $\Phi_D^+ = \Phi_1^+ \smallsetminus \Phi_{D}^0$ and $\Phi_D^- = \Phi_1^- \smallsetminus \Phi_{D}^0$ so that $\Phi_1 = \Phi_D^+ \sqcup \Phi_D^0 \sqcup \Phi_D^-$.
We identify the roots of $\Phi_1$ with linear forms on $\mathbb{A}$ so that, for all $v \in Q(y_3,D)$, we have:
\begin{equation}\label{eq Q en A}
    Q(v,D) = \{ z \in \mathbb{A}:\ \forall \alpha \in \Delta_1 \smallsetminus \Theta_D, \, \alpha(z) \geqslant \alpha(v) \text{ and } \forall \alpha \in \Theta_D, \, \alpha(z) = \alpha(v)\}.
\end{equation}
If $D=D'=0$ (i.e.~$\Theta=\Delta$), there is nothing to prove. Thus, in the following, we assume that $D\neq 0$, i.e.~$\Theta_D \neq \Delta_1$.

For any point $v \in Q(y_3,D)$, we define recursively subsets $\Omega_n(v,D)$ by:
\begin{itemize}
    \item $\Omega_0(v,D) := \{v\}$ and $\Omega_1(v,D) := \cl\big(\mathfrak{C}(v,D)\big)$;
    \item $\displaystyle \Omega_{n+1}(v,D) := \bigcup_{w \in  \Omega_n(v,D)} \cl\big(\mathfrak{C}(w,D)\big)$.
\end{itemize}
The sequence $\big( \Omega_n(v,D) \big)_{n \in \mathbb{N}}$ is increasing and we set $\Omega_{\infty}(v,D) := \bigcup_{n \in \mathbb{N}} \Omega_n(v,D)$.
We divide the proof in five steps, where the first two steps study the subsets $\Omega_\infty(v,D)$.

\textbf{First claim:} For every point $v \in \operatorname{cl} \big( Q(y_3,D) \big)$, we have $\cl\big(\mathfrak{C}(v,D)\big) \subset \operatorname{cl} \big( Q(y_3,D) \big)$.

Let $v \in \operatorname{cl} \big( Q(y_3,D) \big)$ be an arbitrary point.
By definition~\ref{not directed face}, $\cl\big(\mathfrak{C}(v,D)\big)$ is the enclosure of the germ at $v$ of the $k$-sector face $Q(v,D)$. This enclosure is contained in the enclosure of the germ at $v$ of $Q(y_3,D)$, which is contained in $\operatorname{cl} \big( Q(y_3,D) \big)$. Thus, $\cl\big(\mathfrak{C}(v,D)\big)$ is contained in $\operatorname{cl} \big( Q(y_3,D) \big)$, whence the claim follows.

\textbf{Second claim:} For the enclosure of $\Omega_\infty(v,D)$, we have the inclusions
\[ Q(v,D) \subseteq \operatorname{cl} \big( Q(v,D) \big)\subseteq \operatorname{cl}\big( \Omega_\infty(v,D) \big) \subseteq \operatorname{cl} \big( Q(y_3,D) \big).\]

We start by proving the last inclusion. By induction, we prove that $\Omega_n(v,D) \subset \operatorname{cl} \big( Q(y_3,D) \big)$.
Indeed, $\Omega_0(v,D) = \{v\} \subset Q(y_3,D) \subseteq \operatorname{cl} \big( Q(y_3,D) \big)$ by definition. Moreover, if $\Omega_n(v,D) \subset \operatorname{cl} \big( Q(y_3,D) \big)$, then for every $w \in \Omega_n(v,D)$, we have by first claim that $\cl\big(\mathfrak{C}(w,D)\big) \subset \operatorname{cl} \big( Q(y_3,D) \big)$.
Therefore $\Omega_{n+1}(v,D) \subset \operatorname{cl} \big( Q(y_3,D) \big)$.
Thus $\Omega_\infty(v,D) = \bigcup_{n \in \mathbb{N}} \Omega_n(v,D) \subseteq \operatorname{cl} \big( Q(y_3,D) \big)$. Hence, we get $\operatorname{cl}\big( \Omega_\infty(v,D) \big) \subseteq \operatorname{cl} \big( Q(y_3,D) \big)$.

In order to prove that $\operatorname{cl} \big( Q(v,D) \big)\subseteq \operatorname{cl}\big( \Omega_\infty(v,D) \big)$, we consider the quantity
\[\varepsilon = \min_{\alpha \in \Phi^+_D} \bigg( \inf_{w_0 \in \operatorname{vert}\big( \operatorname{cl}\big( Q(v,D) \big)\big)}\Big( \sup \left\{\alpha(w) - \alpha(w_0),\ w \in \mathfrak{C}(w_0,D) \right\} \Big) \bigg).\]
Since any $k$-face $\mathfrak{C}(w_0,D)$ is a bounded open subset of $\mathbb{A}$ we have that $\varepsilon < \infty$. Moreover, since  $\mathbb{A}$ contains a finite number of
$k$-faces up to isomorphism
\[\inf_{w_0 \in \operatorname{vert}\big( \operatorname{cl}\big( Q(v,D) \big)\big)}\Big( \sup \left\{\alpha(w) - \alpha(w_0),\ w \in \mathfrak{C}(w_0,D) \right\} \Big)>0,\]
whence $\varepsilon>0$ by finiteness of $\Phi_D^+$. We prove by induction that
\begin{equation}\forall n \in \mathbb{N},\ \forall \alpha \in \Phi^+_D,\ \exists w \in \operatorname{vert}\big( \operatorname{cl}\big( \Omega_n(v,D) \big)\big),\ \alpha(w) - \alpha(v) \geqslant n \varepsilon
\label{eq far vertices}\end{equation}
For $n=0$, it is obvious.
Induction step: Let $\alpha \in \Phi^+_D$ and $n \in \mathbb{N}$.
Assume that $\exists w \in \operatorname{vert}\big( \Omega_{n-1}(v,D) \big),\ \alpha(w) -\alpha(v) \geqslant (n-1) \varepsilon$.
Since $\cl\big(\mathfrak{C}(w,D)\big) = \overline{\mathfrak{C}(w,D)}$ is a polytope in an Euclidean space, the supremum
\[\sup \left\{\alpha(w') - \alpha(w),\ w' \in \mathfrak{C}(w,D) \right\} = \max \left\{\alpha(w') - \alpha(w),\ w' \in  \cl\big(\mathfrak{C}(w,D)\big) \right\}\]is a maximum reached at some vertex $w_\alpha \in \cl\big(\mathfrak{C}(w,D)\big)$.
Since $w \in \operatorname{vert}\big( \Omega_{n-1}(v,D) \big)$, we have by definition that $w_\alpha \in \Omega_n(v,D)$.
By definition of $\varepsilon$, we get that $\alpha(w_\alpha) - \alpha(w) \geqslant \varepsilon$.
Thus $$\alpha(w_\alpha) - \alpha(v) = \alpha(w_\alpha)- \alpha(w) + \alpha(w) - \alpha(v) \geqslant \varepsilon + (n-1)\varepsilon = n \varepsilon. $$
It concludes the induction and proves property~\eqref{eq far vertices}.

For $\alpha \in \Phi_1$ and $\lambda \in \mathbb{R}$, denote by $D(\alpha,\lambda) = \{ w\in \mathbb{A},\ \alpha(w) \geqslant \lambda\}$ the half-space of $\mathbb{A}$ with boundary in the hyperplane $H(\alpha, \lambda)=\ker (\alpha - \lambda)$.
For $\alpha \in \Phi^-_D$, the property~\eqref{eq far vertices} applied to $-\alpha \in \Phi_D^+$ gives that 
$$\forall n \in \mathbb{N},\ \exists w \in \Omega_\infty(v,D), \alpha(w) \leqslant  \alpha(v) - n \varepsilon$$
Thus, there is no half-apartment of the form $D(\alpha,\lambda)$, with $\alpha \in \Phi_D^{-}$, which contains $\Omega_\infty(v,D)$. For $\alpha \in \Phi_D^{0} \cup \Phi_D^{+}$, we prove that, if $D(\alpha, \lambda)$ contains $\Omega_{\infty}(v,D)$, then $D(\alpha, \lambda)$ contains $Q(v,D)$. Indeed, since $v \in \Omega_\infty(v,D)$, if $D(\alpha, \lambda)$ contains $\Omega_{\infty}(v,D)$, then $\alpha(v) \geq \lambda$. Moreover, since $\alpha \in \Phi_D^{0} \cup \Phi_D^{+}$, it follows from Equation~\eqref{eq Q en A} that $\alpha(z) \geq \alpha(v)$, for all $z \in Q(v,D)$, with equality for all $z$ exactly when $\alpha \in \Phi_D^{0}$. Hence, we get $\alpha(z) \geq \alpha(v) \geq \lambda$, whence we deduce that $D(\alpha, \lambda)$ contains $Q(v,D)$. Thus, by definition of the enclosure, we conclude that $\operatorname{cl}\big( \Omega_\infty(v) \big) \supseteq \operatorname{cl}\big( Q(v,D) \big)$. This proves the second claim.

\textbf{Third claim:} The visual boundary of $ \operatorname{cl} \big( Q(y_3,D) \big)$ equals $\partial_\infty Q(y_3,D) = \partial_\infty D$.

Recall that visual equivalence on geodesical rays is introduced in~\ref{intro visual boundary}.
Obviously, $\partial_\infty \cl\big( Q(y_3,D)\big) \supseteq \partial_\infty Q(y_3,D)$.

Conversely, by definition $Q(y_3,D) = \bigcap_{\alpha \in \Phi_D^{0}} H(\alpha,\lambda_{\alpha}) \cap \bigcap_{\alpha \in \Phi_D^{+}} D(\alpha, \lambda_{\alpha})$, where $\lambda_{\alpha}=\alpha(y_3) \in \mathbb{R}$. Then, by definition of the enclosure, we obtain that
\[\operatorname{cl} \big( Q(y_3,D) \big) = \bigcap_{\alpha \in \Phi_D^0 \cup \Phi_D^+} D(\alpha,b_\alpha)\]
where $b_{\alpha}= \max\{\mu \in \Gamma_\alpha,\ \mu \leq \lambda_\alpha\}$ (with $\Gamma_\alpha = \mathbb{Z}$ since $\mathbf{G}$ is split).
Let $\mathfrak{r}$ be a geodesical ray in $\operatorname{cl} \big( Q(y_3,D) \big)$, which is contained in some $k$-apartment. For any point $w \in \mathfrak{r}$ and any $\alpha \in \Phi_D^{0}$, we have that $b_{\alpha} \leq \alpha(w) \leq - b_{-\alpha}$.
This implies that the linear form $\alpha$ is bounded whence constant on $\mathfrak{r}$.
Let us fix a point $w_0 \in \mathfrak{r}$. Then, $\mathfrak{r}$ is in the same parallelism class as $\mathfrak{r}-w_0+y_3$, which is contained in $Q(y_3,D)$, whence the result follows.

\textbf{Fourth claim:} For any vertices $v \in Q(y_3,D)$, $w \in Q(y_3',D')$, any $g \in \mathbf{G}(A)$ such that $w = g \cdot v$, and any $n \in \mathbb{N}$, we have $g \cdot \Omega_n(v,D) = \Omega_n(w,D')$.

We prove the fourth claim by induction on $n$. For $n=1$, since $v$ is a point of $Q(y_3,D)$ and $w$ is a point of $Q(y_3',D')$, Corollary~\ref{cor-transportAlcoves} shows that $g \cdot \mathfrak{C}(v,D) = \mathfrak{C}(w,D')$, whence $g \cdot \Omega_1(v,D) = \Omega_1(w,D')$ by taking the enclosure.

Now, suppose $n \geqslant 2$ and $g \cdot \Omega_{n-1}(v,D) = \Omega_{n-1}(w,D')$. Let $\omega \in \Omega_{n}(v,D)$ be any point. Then, by definition, there is a point $\nu \in  \Omega_{n-1}(v,D) $ such that $\omega \in \cl\big(\mathfrak{C}(\nu,D)\big)$.
Since $\nu \in \Omega_\infty(v,D) \subseteq \cl\big(Q(y_3,D)\big) \subset Q(y_2,D)$, we deduce that $\nu$ is an point of $Q(y_2,D)$.
Moreover, since $g \cdot \nu \in \Omega_{n-1}(w,D') \subseteq \cl\big(Q(y_3',D')\big) \subset Q(y_2',D')$, the point $g \cdot \nu$ is a point in $Q(y_2',D')$.
Thus, by Corollary~\ref{cor-transportAlcoves}, we have that $g\cdot \omega \in g \cdot \mathfrak{C}(\nu,D) = \mathfrak{C}(g \cdot \nu,D')$.
Since $\nu \in \Omega_{n-1}(v,D)$, we have that $g \cdot \nu \in \Omega_{n-1}(w,D')$ by induction assumption.
Thus $g \cdot \omega \in \Omega_n(w,D')$ by definition.
Hence, we have shown that $g \cdot \Omega_n(v,D) \subseteq \Omega_n(w,D')$.
Since $v = g^{-1} \cdot w$ with $g^{-1} \in \mathbf{G}(A)$, we deduce the converse inclusion, which concludes the induction step.

\textbf{End of the proof:}
Let $v \in Q(y_3,D)$, $w \in Q(y_3',D')$ be two points and $g \in \mathbf{G}(A)$ be such that $g \cdot v= w$.
Then, the fourth claim gives that $g \cdot \Omega_\infty(v,D) = \Omega_\infty(w,D')$.
Since the action of $\mathbf{G}(A)$ preserves the simplicial structure, we have that
$g \cdot \operatorname{cl}\big( \Omega_\infty(v,D) \big)= \operatorname{cl}\big(\Omega_\infty(w,D')\big)$.
Since, by definition and the third claim, we have that $\partial^\infty Q(v,D) = \partial^\infty \operatorname{cl}\big( Q(y_3,D) \big)= \partial_\infty D$ and $\partial^\infty Q(w,D') = \partial^\infty \operatorname{cl}\big( Q(y_3',D') \big)= \partial_\infty D'$, the second claim gives that $\partial^\infty \operatorname{cl}\big( \Omega_\infty(v,D) \big) = \partial_\infty D$ and $\partial^\infty \operatorname{cl}\big( \Omega_\infty(w,D') \big) = \partial_\infty D'$, whence $g \cdot \partial_\infty D = \partial_\infty D'$.
Thus, by the correspondence between vector chambers and their visual boundaries (c.f.~\S \ref{intro visual boundary}), we deduce that $g \cdot D = D'$.
\end{proof}

\subsection{\texorpdfstring{$\mathbf{G}(A)$}{G(A)}-orbits of special vertices of subsector faces}\label{subsection G-orbits}

As stated in \S~\ref{sec upper bound}, there are various cases in which one can upper bound the abstract group $M_\Psi(h)$ by a suitable $A$-module but, because of non-commutativity of $\mathbf{U}_\Psi(k)$, we cannot upper bound this group for an arbitrary Chevalley group and an arbitrary ground field $\mathbb{F}$. Thus, we do the following assumption in order to provide some finite dimensional properties on some $\mathbb{F}$-vector spaces:

\begin{lemma}\label{lemma finite dimensional}
Let $\Psi \subset \Phi^+$ be a closed subset of positive roots and let $h \in \mathbf{G}(k)$.
Assume one of the following cases is satisfied:
\begin{enumerate}[label=(case \Roman*)]
    \item\label{case SLn} $\mathbf{G} = \mathrm{SL}_n$ or $\mathrm{GL}_n$ for some $n \in \mathbb{N}$,
    \item\label{case Borel} $\Psi$ is linearly independent in $\operatorname{Vect}_{\mathbb{R}}(\Phi)$ and $h \in H = \{n^{-1}u,\ n \in N^{\mathrm{sph}},\ u \in \mathbf{U}^+(k)\}$ as defined in \S~\ref{section Stabilizer of points in the Borel variety}, Equation~\eqref{eq def H},
    \item\label{case finite} $\mathbb{F}$ is a finite field.
\end{enumerate}
Then, for any family $\underline{z} = \big( z_\alpha \big)_{\alpha \in \Psi}$ of values $z_\alpha \in \mathbb{R}$, the $\mathbb{F}$-vector space spanned by
\[
    \Big\{ \big(x_\alpha\big)_{\alpha \in \Psi} \in M_\Psi(h): \nu(x_\alpha) \geq z_\alpha,\ \forall \alpha \in \Psi\Big\},
\]
in the $k$-vector space $k^\Psi \cong \big(\mathbb{G}_a(k)\big)^\Psi= \psi^{-1}\big(\mathbf{U}_\Psi(k)\big)$ is finite dimensional over $\mathbb{F}$.
\end{lemma}

\begin{notation}
For any family $\underline{z} = \big( z_\alpha \big)_{\alpha \in \Psi}$ of values $z_\alpha \in \mathbb{R}$, we denote by $V_\Psi(h)[\underline{z}]$ the $\mathbb{F}$-vector space spanned by $\Big\{ \big(x_\alpha\big)_{\alpha \in \Psi} \in M_\Psi(h): \nu(x_\alpha) \geq z_\alpha,\ \forall \alpha \in \Psi\Big\}$, as given by Lemma~\ref{lemma finite dimensional}.
\nomenclature[]{$V_\Psi(h)[\underline{z}]$}{some finite dimensional $\mathbb{F}$-vector space}
\end{notation}

\begin{proof}
Let $\pi \in k$ be a uniformizer of the complete discretely valued ring $\mathcal{O}$.
For any $\alpha \in \Psi$, denote by $m_\alpha \in \mathbb{Z}$ the integer such that $\{ x \in \mathcal{O},\ \nu(x) \geqslant z_\alpha\} = \pi^{m_\alpha} \mathcal{O}$.

In the \ref{case SLn}, consider the non-zero fractional $A$-ideals $J_{i,j}(h)$ defined by Lemma~\ref{lemma upper bound SL(n,A)}.
For $\alpha = \alpha_{i,j} \in \Psi$, denote by $\mathfrak{q}_\alpha(h) = J_{i,j}(h)$.
In the \ref{case Borel}, for $\alpha \in \Psi$, denote by $\mathfrak{q}_\alpha(h) = \mathfrak{q}_{\Psi,\alpha}(h)$ the non-zero fractional $A$-ideals given by Proposition~\ref{prop fractional ideals}.
Then, in both cases, 
\[
    M_\Psi(h) \subseteq \bigoplus_{\alpha \in \Psi} \mathfrak{q}_\alpha(h).
\]
Since $\mathbb{F} \subseteq A \cap \mathcal{O}$, we have that any intersection of a fractional $A$-ideal with a fractional $\mathcal{O}$-ideal is an $\mathbb{F}$-vector space. Thus, we deduce that
\[
    V_\Psi(h)[\underline{z}] \subseteq \bigoplus_{\alpha \in \Psi} \big( \mathfrak{q}_{\alpha}(h) \cap \pi^{m_\alpha} \mathcal{O} \big).
\]
Moreover, since each $\mathbb{F}$-vector space $\mathfrak{q} \cap \pi^m \mathcal{O}$ is finite dimensional (c.f.\cite[\S 1, Prop.~1.4.9]{Stichtenoth}), we deduce that $\bigoplus_{\alpha \in \Psi} \big( \mathfrak{q}_{\alpha}(h) \cap \pi^{m_\alpha} \mathcal{O} \big)$ is also a finite dimensional $\mathbb{F}$-vector space. Hence, so is $V_\Psi(h)[\underline{z}]$ as subspace.

Now, assume that $\mathbb{F}$ is a finite field \ref{case finite}. As in \cite[Lemma 1.1]{M}, we know that $\nu_{\p}(x) \leq 0$, for all $x \in A$. Then, for each $x \in A$, the open ball centered at $x$ of radius $r<1$ does not intersect $A$. This proves that the ring $A$ is a discrete subset of the local field $K$ for the $\nu_\p$-topology.
Hence, the topological group $\mathbf{G}(A)$ is a discrete subgroup of $\mathbf{G}(K)$.
Thus, we get that $\mathbf{U}_\Psi(k) \cap h \mathbf{G}(A) h^{-1}$ is a discrete subgroup of $\mathbf{U}_\Psi(K)$.
Since the isomorphism $\psi : K^\Psi \to \mathbf{U}_\Psi(K)$ is naturally a homeomorphism, the group $M_\Psi(h)$ is a discrete subgroup of $K^\Psi$.
As intersection of the discrete subset $M_\Psi(h)$ with the compact subset $\bigoplus_{\alpha \in \Psi} \pi^{m_\alpha} \mathcal{O}$ of $K^\Psi$, the set $\Big\{ \big(x_\alpha\big)_{\alpha \in \Psi} \in M_\Psi(h): \nu(x_\alpha) \geq z_\alpha,\ \forall \alpha \in \Psi\Big\}$ is finite.
Hence it spans a finite dimensional $\mathbb{F}$-vector space of $k^\Psi.$
\end{proof}

\begin{proposition}\label{prop racional ArXiv}
Let $Q=Q(x,D)$ be a $k$-sector face of $\mathcal{X}_k$.
Assume that $x$ is a special $k$-vertex and that one of the following cases is satisfied:
\begin{enumerate}[label=(case \Roman*)]
    \item\label{case SLn rat} $\mathbf{G} = \mathrm{SL}_n$ or $\mathrm{GL}_n$ for some $n \in \mathbb{N}$;
    \item\label{case Borel rat} $Q$ is a sector chamber;
    \item\label{case finite rat} $\mathbb{F}$ is a finite field.
\end{enumerate}
Then there exists a subsector face $Q(y_3,D)$ such that any two different special $k$-vertices $v,w \in Q(y_3,D)$ are not $\mathbf{G}(A)$-equivalent.
\end{proposition}

\begin{proof}
We get, as in the proof of Proposition~\ref{prop starAction} that there are a subset $\Theta \subset \Delta$, elements $u \in \mathbf{U}^+(k)$ and $n \in N^{\mathrm{sph}}$ such that the element $h = n^{-1} u \in \mathbf{G}(k)$ satisfies $h \cdot D= D_0^\Theta$.
Let $y_3 = t_\Theta\big( Q(x,D) \big)$ be the special $k$-vertex given by Proposition~\ref{proposition facets in the same orbit ArXiv}.
Define $y_3' = h \cdot y_3\in \mathbb{A}_0$.

Let $v,w \in Q(y_3,D)$ be two arbitrary $k$-special vertices, and set $v'=h \cdot v$, $w' =h \cdot w \in Q(y_3',D_0^\Theta)$.
Assume that there exists $g \in \mathbf{G}(A)$ such that $g \cdot v=w$. We want to prove that $v=w$.
By Proposition~\ref{proposition facets in the same orbit ArXiv}, we have $g \cdot D= D$. Equivalently, the element $b=hgh^{-1}$ satisfies $b \cdot D_0^\Theta=D_0^\Theta$.
Hence, it follows from Lemma~\ref{lemma stab standard face} that $b \in \mathbf{P}_\Theta(k)$.
Note that by definition
\begin{equation}\label{eq stab 1}
g \, \stab_{\mathbf{G}(k)}(v) \, g^{-1}=\stab_{\mathbf{G}(k)}(w).
\end{equation}
Moreover, $b \cdot v' = hg \cdot v = h \cdot w = w'$. Hence, we get that
\begin{equation}\label{eq stab 2}
   b\, \stab_{\mathbf{G}(k)}(v')\, b^{-1} =\stab_{\mathbf{G}(k)}(w').
\end{equation}
Consider the subgroup $G^h= h \mathbf{G}(A) h^{-1} \subseteq \mathbf{G}(k)$.
Since $b \in G^{h}$, we deduce by intersecting both sides of Equality~\eqref{eq stab 2} with $G^{h}$, that
\begin{equation}\label{eq stab 3}
b \, \stab_{G^h}(v')\, b^{-1}= \stab_{G^h}(w').
\end{equation}

Let $\big( \Psi_i^\Theta \big)_{1 \leq i \leq m}$ be the family of subsets of $\Phi^+$ given by Proposition~\ref{prop good increasing subsets} when $\Theta \neq \emptyset$, or by Corollary~\ref{cor good increasing subsets} when $\Theta=\emptyset$, and pick a family $(\beta_i)_{1 \leq i \leq m}$ of positive roots such that $\beta_i \in \Psi_i^\Theta \smallsetminus \Psi_{i-1}^\Theta$ (with $\Psi_0^\Theta = \emptyset$).
We prove by induction on $i \in \llbracket 1,m \rrbracket$ that $\beta_i(v') = \beta_i(w')$.

Intersecting both sides of the Equality~\eqref{eq stab 3} with the subgroup $\mathbf{U}_{\Psi_i^\Theta}(k)$, we get
$$ b \, \stab_{G^h}(v')\, b^{-1} \cap \mathbf{U}_{\Psi_i^\Theta}(k)  = \stab_{G^h}(w') \cap \mathbf{U}_{\Psi_i^\Theta}(k).$$
Since $\Psi_i^\Theta$ is $W_\Theta$-stable, and satisfies~\ref{cond closure} and~\ref{cond commutative}, applying  Corollary~\ref{cor k-linear action parabolic}, we get that $\mathbf{U}_{\Psi_i^\Theta}$ is normalized by $\mathbf{P}_\Theta$. Hence
\begin{equation}\label{eq stab 4}
b \, \big(\stab_{G^h}(v') \cap \mathbf{U}_{\Psi_i^\Theta}(k) \big) \, b^{-1}  = \stab_{G^h}(w') \cap \mathbf{U}_{\Psi_i^\Theta}(k).
\end{equation} 

Denote by $\psi_i= \prod_{\alpha \in \Psi_i^\Theta} \theta_{\alpha} : k^{\Psi_i^\Theta} \to \mathbf{U}_{\Psi_i^\Theta}(k)$ the natural group isomorphism deduced from the \'epinglage.
For any special vertex $z \in Q(y_3',D_0^\Theta)$, we consider the group
\begin{align}
    M_i[z]:=& \psi_i^{-1}\Big( \stab_{G^h}(z) \cap \mathbf{U}_{\Psi_i^\Theta}(k) \Big)\notag\\
    =& \psi_i^{-1}\big( G^h \cap \mathbf{U}_{\Psi_i^\Theta}(k) \big) \cap
\psi_i^{-1}\big(\stab_{\mathbf{G}(K)}(z) \cap \mathbf{U}_{\Psi_i^\Theta}(k) \big) \subseteq k^{\Psi_i^\Theta},\label{eq def Miz}
\end{align}
and we analyse separately the both sides of the latter intersection.
On the one hand, by definition given by~\eqref{eq def MPsi}, we have that
\[ G^h \cap \mathbf{U}_{\Psi_i^\Theta}(k) = \psi_i\big( M_{\Psi_i^\Theta}(h) \big).\]
On the other hand, since $z$ is a special vertex, for each $\alpha \in \Psi_i^\Theta$ we have that $\alpha(z) \in \mathbb{Z}$.
Thus, it follows from \cite[6.4.9]{BT}, applied to $\mathbf{G}$ which is split, that 
\[ \stab_{\mathbf{G}(K)}(z) \cap \mathbf{U}_{\Psi_i^\Theta}(k) = \prod_{\alpha \in \Psi_i^\Theta} \mathbf{U}_{\alpha,z}(k) = \prod_{\alpha \in \Psi_i^\Theta} \theta_{\alpha}\big(\pi^{-\alpha(z)} \mathcal{O} \big) = \psi_i\left( \bigoplus_{\alpha \in \Psi_i^\Theta} \pi^{-\alpha(z)} \mathcal{O} \right),\]
where $\pi \in k$ is a uniformizer of $\mathcal{O}$. Thus Equation~\eqref{eq def Miz} becomes
\begin{equation}\label{eq Mi intersection}
    M_i[z] = M_{\Psi_i^\Theta}(h) \cap \bigoplus_{\alpha \in \Psi_i^\Theta} \pi^{-\alpha(z)} \mathcal{O}.
\end{equation}

If $Q$ is a sector chamber \ref{case Borel rat}, then $\Theta = \emptyset$, whence $\Psi_i^\emptyset$ is linearly in dependant according to Corollary~\ref{cor good increasing subsets}.
Hence, in any of the cases~\ref{case SLn rat}, \ref{case Borel rat} or~\ref{case finite rat}, we deduce by Lemma~\ref{lemma finite dimensional} that
\[
    V_i[z] := V_{\Psi_i^\Theta}\Big[\big( -\alpha(z) \big)_{\alpha \in\Psi_i^\Theta}\Big] = \operatorname{Vect}_\mathbb{F}\big( M_i[z] \big),
\]
is a finite dimensional $\mathbb{F}$-vector space.
Define
\[ V_{\beta_i}:= \psi_i \circ \theta_{\beta_i}^{-1}\big( J_{\beta_i}(h) \big),\]
where $J_{\beta_i}(h)$ is one of the non-zero $A$-ideals considered in Riemann-Roch identity~\eqref{rri} (this is possible because $\Phi_z = \Phi$ since $z$ is special, whence $\beta_i \in \Phi_z^+$).
In particular, $V_{\beta_i}$ is an $\mathbb{F}$-vector space contained in the subset $M_{\Psi_i^\Theta}(h)$.
Thus, since $V_{\beta_i} \subseteq M_{\Psi_i^\Theta}(h)$, we get that
\begin{multline*}
    V_{\beta_i} \cap \big(\bigoplus_{\alpha \in \Psi_i^\Theta} \pi^{-\alpha(z)} \mathcal{O} \big)
     \subseteq V_{\beta_i} \cap M_{\Psi_i^\Theta}(h) \cap \big(\bigoplus_{\alpha\in \Psi_i^\Theta} \pi^{-\alpha(z)} \mathcal{O} \big) \\
     \subseteq V_{\beta_i} \cap V_i[z]
     \subseteq V_{\beta_i} \cap \big(\bigoplus_{\alpha \in \Psi_i^\Theta} \pi^{-\alpha(z)} \mathcal{O} \big).
\end{multline*}
Hence we get that
\begin{equation}\label{eq intersection with subspace}
    V_i[z] \cap V_{\beta_i} = V_{\beta_i} \cap \big(\bigoplus_{\alpha \in \Psi_i^\Theta} \pi^{-\alpha(z)} \mathcal{O} \big) = \psi_i \circ \theta_{\beta_i}^{-1}\big( J_{\beta_i}(h) \cap  \pi^{-{\beta_i}(z)} \mathcal{O} \big).
\end{equation}

According to Corollary~\ref{cor k-linear action parabolic} applied to $\Psi_i^\Theta$, we obtain a  $k$-linear automorphism $f_b : k^{\Psi_i^\Theta} \to k^{\Psi_i^\Theta}$ such that
\[ \psi_i \circ f_b \circ \psi_i^{-1}: x \mapsto bxb^{-1}.\]
Thus, using Equality~\eqref{eq stab 4}, it induces an $\mathbb{F}$-linear isomorphism
\[ V_i[w'] = \operatorname{Vect}_\mathbb{F}\Big( M_i[w'] \Big) = \operatorname{Vect}_\mathbb{F}\Big( f_b\big( M_i[v'] \big) \Big)
=f_b\big( V_i[v'] \big) \cong V_i[v'].
\]

Up to exchanging $v'$ and $w'$, we can assume without loss of generality that $\beta_i(v') \leq \beta_i(w')$ so that $\pi^{-\beta_i(v')} \mathcal{O} \subseteq \pi^{-\beta_i(w')} \mathcal{O} $.

By construction, the family $\left((\beta_i)_{|\Theta^\perp}\right)_{1 \leq i \leq m}$ is a basis which is adapted to the complete flag $\operatorname{Vect}_\mathbb{R}\left(\left(\Psi_i^\Theta\right)_{|\Theta^\perp}\right)$.
Moreover, by induction assumption $\beta_j(v') = \beta_j(w')$ for each $j <i$.
In particular, since $D_0^\Theta \subset \Theta^\perp$, we have that $\alpha(v') = \alpha(w')$ for any $\alpha \in \Psi_{i-1}^\Theta$.
Since $z' = v'-w' \in \left( \Theta \cup \Psi_{i-1}^\Theta\right)^\perp$, we deduce by Proposition~\ref{prop good increasing subsets} that for any $\alpha \in \Psi_{i}^\Theta \smallsetminus \Psi_{i-1}^\Theta$, the sign of $\alpha(z')$ is the same as that of $\beta_i(z')$.
As a consequence, we have $\alpha(v') \leq \alpha(w')$ for any $\alpha \in \Psi_i^\Theta$.

Hence, Equation~\eqref{eq Mi intersection} gives that $M_i[v'] \subseteq M_i[w'] $ and therefore $V_i[v']$ is a subspace of $V_i[w']$. Since they have the same dimension over $\mathbb{F}$, we deduce that $V_i[v']=V_i[w']$.
Hence, Equality~\eqref{eq intersection with subspace} applied to $v'$ and to $w'$ gives that
\[
    \operatorname{dim}_{\mathbb{F}} J_{\beta_i}(h)\big[\beta_i(v')\big] 
    = \operatorname{dim}_{\mathbb{F}} V_i[v'] \cap V_{\beta_i}
    = \operatorname{dim}_{\mathbb{F}} V_i[w'] \cap V_{\beta_i}
    = \operatorname{dim}_{\mathbb{F}} J_{\beta_i}(h)\big[\beta_i(w')\big],
\]
since $\psi_i^{-1}\circ \theta_{\beta_i}$ is an injective $k$-linear endomorphism.
As a consequence, because $v',w' \in Q(y_3',D_0^\Theta)$, the Riemann-Roch Equality~\eqref{rri} gives that $\beta_i(v') = \beta_i(w')$, whence we conclude the induction.

Therefore, we deduce that $\beta_i(v')=\beta_i(w')$ for every $i \in \llbracket 1,m\rrbracket$.
Because $\left(\Psi_m^\Theta\right)_{\Theta^\perp}$ is a generating set of $(\Theta^\perp)^*$, we deduce that $v'=w'$ and therefore $v=w$, which concludes the proof.
\end{proof}

\subsection{\texorpdfstring{$\mathbf{G}(A)$}{G(A)}-orbits of arbitrary points}\label{subsection G-orbits ar points}

\begin{lemma}\label{lemma integral polytope}
Let $A \in \mathcal{M}_{m,n}(\mathbb{Z})$.
There exists an integer $d_A \in \mathbb{N}$, depending on $A$, such that for any $b \in \left( \mathbb{Z} \right)^n$, the vertices of the polytope
$P_b := \left\{ z \in \mathbb{R}^n,\ A z \leqslant b \right\}$
belongs to $\left( \frac{1}{d_A} \mathbb{Z} \right)^n$.
\end{lemma}

\begin{proof}
Define $d \in \mathbb{N}$ as the lowest common multiple of the determinants of submatrices of rank $n$ of $A$.
Note that $d \neq 0$ is well-defined by assumption on the rank equal to $n$ of the considered submatrices, and that $d=1$ whenever $A$ is not of rank $n$.
Let $b \in \left( \mathbb{Z}\right)^n$.
Let $y \in P_b$ be any vertex (if there exists).
Then, there is a submatrix $A'$ of $A$ of rank $n$ and a corresponding submatrix $b'$ of $b$ so that $A'y = b'$.
Thus $y' = (A')^{-1} b' = \frac{1}{\operatorname{det}(A')} \operatorname{Com}(A')^\top b' \in \left( \frac{1}{d} \mathbb{Z} \right)^n$
since $\operatorname{det}(A') | d$.
Whence the result follows.
\end{proof}

Let $\Phi$ be a root system with Weyl group $W=W(\Phi)$ and let $\mathbb{A}_0$ be the Coxeter complex associated to $\Phi$ whose walls are the hyperplanes associated to the affine forms $\alpha + \ell$ for $\alpha \in \Phi$ and $\ell \in \mathbb{Z}$.
Let $\Delta$ be a basis of $\Phi$ and $W=$ be the finite Weyl group associated to $\Phi$.
Denote by $\varpi_\alpha$ the fundamental coweight associated to $\alpha \in \Delta$.
Let $\operatorname{Aff}(\mathbb{A}_0) = \mathrm{GL}(\mathbb{A}_0) \ltimes \mathbb{A}_0$ be the affine group the real affine space $\mathbb{A}_0$.
Then, the subgroup of $\operatorname{Aff}(\mathbb{A}_0)$ preserving the tiling of $\mathbb{A}_0$ is $\widetilde{W} = W \ltimes \bigoplus_{\alpha \in \Delta} \mathbb{Z} \varpi_\alpha$\footnote{In general, the affine Weyl group $W^\mathrm{aff} = W \ltimes \bigoplus_{\alpha \in \Delta} \mathbb{Z} \alpha^\vee$ is a proper subgroup of $\widetilde{W}$.}.

\begin{proposition}\label{prop fixed point Weyl affine}
Let $V = \mathbb{R}^n$ be a finite dimensional vector space and let $\Phi \subset V^*$ be a root system of rank $n$, with basis $\Delta$ and Weyl group $W \subset \mathrm{GL}_n(\mathbb{Z})$.
Denote by $\mathbb{A}_0$ the affine Coxeter complex structure on $V$ associated to $\Phi$ and let $\widetilde{W} = W \ltimes \mathbb{Z}^n$ be the affine subgroup of $\mathrm{GL}(V) \ltimes V$ preserving the Coxeter complex structure.

There exists a positive integer $e\in \mathbb{N}$ (only depending on $\Phi$) such that for any $\widetilde{g} \in \widetilde{W}$ and any point $x \in V$ fixed by $\widetilde{g}$, there exists $z \in V$ such that the translation $\tau_z$ by $z$ commutes with $\widetilde{g}$ and $\tau_z \cdot x = x+z \in \cl(x) \cap \left( \frac{1}{e} \mathbb{Z}\right)^n \subset \mathbb{R}^n$.
\end{proposition}

\begin{proof}
Let $V_\mathbb{Q} = \mathbb{Q}^n$ and let $(\varpi_1,\dots,\varpi_n)$ be its canonical basis of fundamental coweights.
Let $w \in W$ and consider the eigenspace $V_{w}(1) = \ker (w-\operatorname{id})$ in $V_\mathbb{Q}$.
Since $W$ is a finite group, $w \in \mathrm{GL}_n(\mathbb{Z})$ is semisimple over $\mathbb{Q}$, whence there exists a $w$-stable complementary subspace $V'_w(1)$ of $V_w(1)$ in $V_\mathbb{Q}$.
Since $w-\operatorname{id}$ is invertible over $V'_w(1) \cong \mathbb{Q}^m$, there is a positive integer $t_w$ so that $(w-\operatorname{id})_{V'_w}^{-1} = \big( u_{i,j} \big)_{1 \leq i,j \leq m} \in \mathrm{GL}_m(\mathbb{Q})$ has coefficients in $\frac{1}{t_w} \mathbb{Z}$.

Let $(e_1,\dots,e_n)$ be an adapted basis to the decomposition $V_\mathbb{Q} = V'_w(1) \oplus V_w(1)$ so that $(e_1,\dots,e_m)$ is a basis of $V'_w(1)$ and $(e_{m+1},\dots,e_n)$ is a basis of $V_w(1)$ (which are eigenvectors).
There exist positive integers $r_w,s_w \in \mathbb{N}$ and rational coefficients $\lambda_{i,j} \in \frac{1}{r_w} \mathbb{Z}$ and $\mu_{i,j} \in \frac{1}{s_w} \mathbb{Z}$ so that $\varpi_i = \sum_{j=1}^n \lambda_{i,j} e_j$ and $e_j = \sum_{i=1}^n \mu_{i,j} \varpi_i$.
Define $c_w = r_w s_w t_w \in \mathbb{N}$.

We claim that for any $y \in \mathbb{Z}^n \cap V'_w(1)$, we have that $(w-\operatorname{id})^{-1}(y) \in \left( \frac{1}{c_w} \mathbb{Z}\right)^n$.
Indeed, write $y = \sum_{i=1}^n y_i \varpi_i = \sum_{\substack{1 \leq i \leq n\\ 1 \leq j \leq m}} y_i \lambda_{i,j} e_j$.
Then \[(w-\operatorname{id})^{-1}(y) = \sum_{\substack{1 \leq i \leq n\\ 1 \leq j \leq m\\ 1 \leq k \leq m}} y_i \lambda_{i,j} u_{k,j} e_k = \sum_{\substack{1 \leq i \leq n\\ 1 \leq j \leq m\\ 1 \leq k \leq m\\ 1 \leq \ell \leq n}}  y_i \lambda_{i,j} u_{k,j} \mu_{\ell,k} \varpi_\ell\] where $y_i \lambda_{i,j} u_{k,j} \mu_{\ell,k} \in \frac{1}{r_ws_wt_w} \mathbb{Z} = \frac{1}{c_w} \mathbb{Z}.$


Let $\widetilde{g}\in \widetilde{W}$ and let $x \in V = V_\mathbb{Q} \otimes \mathbb{R}$ be a fixed point of $\widetilde{g} = (w,v)$ with $w \in W$ and $v \in \mathbb{Z}^n$.
Write $x = x_1 + x'$ and $v=v_1+v'$ with $x_1,v_1 \in V_w(1) \otimes \mathbb{R}$ and $x',v' \in V'_w(1) \otimes \mathbb{R}$.
Then $\widetilde{g} \cdot x = w(x) + v = x$, whence $w(x_1) -x_1 = -v_1$ and $w(x')-x' = -v'$.
Thus $v_1 = 0$ and $x' = (w-\operatorname{id})_{V'_w}^{-1}(-v')$.

By construction $x' \in \left( \frac{1}{c_w} \mathbb{Z} \right)^n$.
For $\alpha \in \Phi$, define $b_\alpha = \lfloor \alpha(x) \rfloor - \alpha(x') \in \frac{1}{c_w} \mathbb{Z}$.
Since the roots $\alpha \in \Phi$ are linear forms with coefficients in $\mathbb{Z}$, the $c_w b_\alpha$ belongs to $\mathbb{Z}$ and $w \in \mathrm{GL}_n(\mathbb{Z})$,
the polytope
\begin{align*}
    P\ :=&\big(x' + V_1(w)\otimes \mathbb{R}\big) \cap \cl(x)\\
    =&\left\{x'+z,\ z \in V_1(w)\otimes \mathbb{R},\ \text{ and } \forall \alpha \in \Phi, \alpha(x'+z) \geqslant \lfloor \alpha(x) \rfloor \right\}\\
    =&x' + \left\{ z \in \mathbb{R}^n,\ (w-\operatorname{id})(z) \geqslant 0 \text{ and } (\operatorname{id}-w)(z) \geqslant 0 \text{ and } \forall \alpha \in \Phi,\ c_w \alpha(z) \geqslant c_w b_\alpha \right\}
\end{align*}
is defined by inequalities with coefficients in $\mathbb{Z}$.
Hence, according to Lemma~\ref{lemma integral polytope}, there is an integer $d_w$ which only depends on $w \in W$ and $\Phi$, but not on $v'$, such that the vertices of $P$ belongs to $\left( \frac{1}{d_w} \mathbb{Z} \right)^n$.

Define $e := \operatorname{lcm}(d_w,\ w \in W)$ and let $y \in \cl(x)$ be any vertex of $P$.
Then $y \in \left( \frac{1}{e} \mathbb{Z} \right)^n$.
Define $z = y-x \in V_w(1)\otimes \mathbb{R}$ and $\tau_z$ the translation by $z$, so that $\tau_z(x) = y \in \cl(x) \cap \left( \frac{1}{e} \mathbb{Z} \right)^n$.
Then $z$ is a fixed point of $w$ so that $\tau_z$ commutes with $\widetilde{g}$.
Whence the result follows for the integer $e \in \mathbb{N}$, only depending on $\Phi$.
\end{proof}

\begin{corollary}\label{cor image of sector faces}
Let $Q=Q(x,D)$ be a $k$-sector face of $\mathcal{X}_k$ whose tip $x$ is an arbitrary point (not necessarily a vertex).
Assume that $\mathbf{G} = \mathrm{SL}_n$ or $\mathrm{GL}_n$, or that $Q$ is a sector chamber, or that $\mathbb{F}$ is finite.
Then, there exists a subsector face $Q(y_4,D)$ such that any two different points $v,w \in Q(y_4,D)$ are not $\mathbf{G}(A)$-equivalent.
\end{corollary}

\begin{proof}
Consider the positive integer $e \in \mathbb{N}$ given by Proposition~\ref{prop fixed point Weyl affine} applied to the root system $\Phi$.
Let $\mathcal{D}\to\mathcal{C}$ (resp. $B/A$, $\ell/k$, etc.) be a finite extension of degree $e$ ramified at $\p$ as given by Lemma~\ref{lem curve extension}.

We prove this corollary in $3$ steps:
firstly, $x,v,w$ are assumed to be $\ell$-vertices and $x$ is furthermore assumed to be special; secondly $x$ is assumed to be a special $\ell$-vertex and $v,w$ are any points in the sector face; finally, $x$, $v$, $w$ are in the general consideration of the statement.

\textbf{First step:} Assume that $x \in \mathcal{X}_k$ is a special $\ell$-vertex.

Let $\ell' / \ell$ be a field extension given by Lemma~\ref{lemma becomes special}, and $B'/B$ the corresponding ring extension.
Let $i : \mathcal{X}_\ell \to \mathcal{X}_{\ell'}$ be the canonical $\mathbf{G}(\ell)$-equivariant embedding as introduced in section~\ref{intro rational building}.
Note that since $\mathbf{G}$ is split, one can identify the standard apartment of $\mathcal{X}_{\ell'}$ with that of $\mathcal{X}_\ell$ \cite[9.1.19(b)]{BT} and the $\ell$-walls are sent onto some $\ell'$-walls by $i$.
Hence, the vertices of $\mathcal{X}_\ell$ identifies via $i$ with the $\ell$-vertices of $\mathcal{X}_{\ell'}$.
In particular, the standard vertex $x_0=i(x_0)$ is an $\ell$-vertex of $\mathcal{X}_{\ell'}$.
Thus, by assumption on $\ell'/\ell$, the $\ell$-vertex $x$, identified with $i(x)$, is in the $\mathbf{G}(\ell')$-orbit of the standard special vertex $i(x_0)$ of $\mathcal{X}_{\ell'}$ because both are $\ell$-vertices.
Since $\mathbf{G}$ is split, the root systems of $\mathbf{G}_\ell$ and $\mathbf{G}_{\ell'}$ are the same, whence $Q(x,D)$ identifies with an $\ell'$-sector face in $\mathcal{X}_{\ell'}$.
According to Proposition~\ref{prop racional ArXiv}, there is a point $y_3$ of $Q(x,D)$ such that any special $\ell'$-vertices $v,w$ of the subsector face $Q(y_3,D)$ are not $\mathbf{G}(B')$-equivalent, therefore not $\mathbf{G}(B)$-equivalent, whence not $\mathbf{G}(A)$-equivalent.
In particular, this proves the corollary for $\ell$-vertices of $\ell$-sector faces whose tip is a special $\ell$-vertex.

According to  Lemma~\ref{lem subsector face and roots}\ref{item subsector enclosure}, since $x$ is a special $\ell$-vertex, there exists a special $k$-vertex $y_4$ such that $\cl_\ell\big( Q(y_4,D)\big) \subset Q(y_3,D)$.

\textbf{Second step:} Assume that $v,w$ are any two points of $Q(y_4,D)$ that are $\mathbf{G}(A)$-equivalent and write $w = g \cdot v$ with $g \in \mathbf{G}(A)$.

Let $\overline{F_v}=\cl_\ell(v)$ and $\overline{F_w}=\cl_\ell(w)$ be the $\ell$-enclosure of the points $v$ and $w$ respectively (that are closure of $\ell$-faces).
These enclosure of faces are contained in $Q(y_3,D)$.
Because the action of $\mathbf{G}(A) $ preserves the simplicial structure of $\mathcal{X}_\ell$, we have that $g$ sends the $\ell$-vertices of $\overline{F_w} \subset Q(y_3,D)$ onto the $\ell$-vertices of $\overline{F_v} \subset Q(y_3,D)$. Hence $g$ fixes these vertices according to the previous step of the proof, whence $g$ fixes $\overline{F_w} = \overline{F_v}$. In particular, $w=v$.

\textbf{Third step:} Assume that $x$ is any point of $\mathcal{X}_k$.

Let $h = n^{-1}u \in H$ (with $u \in \mathbf{U}^+(k)$ and $n \in N^{\mathrm{sph}}$) so that $D = h^{-1} \cdot D_0^\Theta$ and $\mathbb{A} := h^{-1} \cdot \mathbb{A}_0$ is a $k$-apartment containing $Q(y_2,D)$ as given by Proposition~\ref{prop starAction}.

According to Lemma~\ref{lemma finite special vertices}, there is a finite subset $\Omega$ consisting of special $\ell$-vertices such that any special $\ell$-vertex of $\cl_k\big(Q(x,D)\big)$ belongs to $\cl_k\big(Q(\omega,D)\big)$ for some $\omega \in \Omega$.

For any $\omega \in \Omega$, there is a point $y_4(\omega)$, given by second step, such that any two points $v,w \in Q\big(y_4(\omega),D\big)$ are not $\mathbf{G}(A)$-equivalent.
Let $T_1$ be the $K$-torus associated to $\mathbb{A}$ as in the proof of Proposition~\ref{proposition facets in the same orbit ArXiv}, and let $\Theta_D$ be the simple roots in a basis $\Delta_D$ defining the vector face $D$.
Let $y_4(x)$ be the point of $Q(x,D) \subset \mathbb{A}$ defined so that 
\[
\left\{
\begin{array}{ll}
\forall \alpha \in \Theta_D,&\alpha\big( y_4(x) \big) = \max \Big\{ \alpha\big( y_4(\omega)\big)+1,\ \omega \in \Omega \Big\};\\
\forall \alpha \in \Delta_D \smallsetminus \Theta_D,&\alpha\big( y_4(x)\big) = \alpha(x).
\end{array}
\right.
\]
Let $z$ be any special $\ell$-vertex of $\cl_k\Big( Q\big(y_4(x),D\big)\Big)$.
Then $z \in \cl_k\Big( Q\big(x,D\big)\Big)$ and, by definition of $\Omega$, there is $\omega \in \Omega$ such that $z \in \cl_k\big(Q(\omega,D)\big)$.
Moreover, for any $\alpha \in \Theta_D$, we have that
\[\alpha(z) > \alpha\big( y_4(x)\big)-1 \geqslant y_4(\omega).\]
Hence $z \in Q\big(y_4(\omega),D\big)$.

Let $v,w \in Q\big(y_4(x),D\big)$ and $g \in \mathbf{G}(A)$ so that $g \cdot v = w$.
Let $v_0 = h \cdot v$ and $w_0 = h \cdot w \in \mathbb{A}_0$. Let $g_0 = hgh^{-1}$ so that $w_0 = g_0 \cdot v_0$.
Let $F_v,F_w$ be the $k$-faces containing $v$ and $w$ respectively.
Let $z_v$ be any $k$-vertex of $\overline{F_v}$.
Since the action of $\mathbf{G}(A)$ preserves the $k$-structure, $z_w := g \cdot z_v$ is a $k$-vertex of $\overline{F_w}$.
These two vertices $z_v,z_w$ are special $\ell$-vertices, hence they belongs to some $Q\big(y_4(\omega_v),D\big)$ and $Q\big(y_4(\omega_w),D\big)$ respectively, for some $\omega_v,\omega_w \in \Omega$.
Therefore, by Proposition~\ref{proposition facets in the same orbit ArXiv} applied with these two special $\ell$-vertices $z_v,z_w$ in the suitable subsector faces $Q\big(y_4(\omega_v),D\big)$ and $Q\big(y_4(\omega_w),D\big)$, we deduce that $g \cdot D = D$ so that $g_0 \cdot D_0^\Theta = D_0^\Theta$,
whence $g_0 \in \mathbf{P}_\Theta(k)$ by Lemma~\ref{lemma stab standard face}.

Applying \cite[7.4.8]{BT} to $g_0$, there exists $n \in \mathbf{N}(K)$ such that $\forall z \in \mathbb{A}_0 \cap g_0^{-1} \mathbb{A}_0, g_0 \cdot z = n \cdot z$.
Hence $n^{-1}g_0$ fixes $ \mathbb{A}_0 \cap g_0^{-1} \mathbb{A}_0$ which contains a subsector face of the form $Q(y_0,D_0^\Theta)$.
Hence $n^{-1}g_0 \cdot D_0^\Theta = D_0^\Theta$ and $g_0 \cdot D_0^\Theta = D_0^\Theta$, whence $n \in \mathbf{N}(K) \cap \mathbf{P}_\Theta(K)$.
Thus, $n \in N_\Theta \mathbf{T}(K)$ has image $\widetilde{g} \in W_\Theta \ltimes \bigoplus_{\alpha \in \Delta} \mathbb{Z} \varpi_\alpha$ in the affine group $\mathrm{Aff}(\mathbb{A}_0)$, and $\forall z \in \mathbb{A}_0 \cap g_0^{-1} \mathbb{A}_0,\ g_0 \cdot z = \widetilde{g} \cdot z = n \cdot z$.
In particular, for any $z \in \cl_k\big( \mathfrak{C}(v_0,D_0^\Theta) \big)$, we have that $\widetilde{g} \cdot z= g_0 \cdot z$.

Decompose $\mathbb{A}_0 = V_0^\Theta \oplus V_0^{\Delta \smallsetminus \Theta}$.
In the basis of fundamental coweights, $V_0^\Theta = \bigoplus_{\alpha \in \Delta \smallsetminus \Theta} \mathbb{R} \varpi_\alpha$.
Write $v_0 = v_\Theta + \hat{v}$, $w_0 = w_\Theta + \hat{w}$ where $v_\Theta, w_\Theta \in V_0^\Theta$ and $\hat{v},\hat{w} \in V_0^{\Delta\smallsetminus \Theta}$.
Write $\widetilde{g} = \widetilde{w} \ltimes (\tau_\Theta+\hat{\tau})$
with $\widetilde{w}\in W_\Theta$, $\tau_\Theta \in V_0^\Theta$ and $\hat{\tau} \in V_0^{\Delta \smallsetminus \Theta}$.
Write $v_\Theta = \sum_{\alpha \in \Delta\smallsetminus \Theta} z_\alpha \varpi_\alpha$ and define $v'_\Theta = \sum_{\alpha \in \Delta\smallsetminus \Theta} \lceil z_\alpha \rceil \varpi_\alpha \in \bigoplus_{\alpha \in \Delta\smallsetminus \Theta} \mathbb{Z} \varpi_\alpha$.
Then $z = v'_\Theta - v_\Theta \in D_0^\Theta$ is fixed by $\widetilde{w}$ so that the translation by $z$ commutes with $\widetilde{g}$.
Define $w'_0 := w_0 + z \in Q(w_0,D_0^\Theta)$ and $v' := v_0+z = \hat{v}+v_0'$, so that $w'_0 = \widetilde{g} \cdot v'_0$ since $\tau_z$ commutes with $\widetilde{g}$.

Note that $\alpha(v_0) = \alpha(w_0)$ for all $\alpha \in \Theta$ since $v_0,w_0 \in Q\big(h \cdot y_4(x),D_0^\Theta\big)$.
Then, we have that 
$\alpha(w'_0)=\alpha(w_0) + \alpha(z) = \alpha(v_0)+ \alpha(z)=\alpha(v'_0)$
for all $\alpha \in \Theta$.
Thus $\hat{v} = \hat{w}$.
Let $\hat{w}_\Theta= \widetilde{w} \ltimes \hat{\tau}$.
Then $\widetilde{g} \cdot v'_0 = w'_0$ gives $\hat{w}_\Theta \cdot v'_0 = w'_0$, whence $\hat{w}_\Theta \cdot \hat{v} = \hat{w}$.
Hence $\hat{v}$ is a fixed point of $\hat{w}_\Theta$ and therefore $v'_0$ is a fixed points of $\hat{w}_\Theta$.
By Proposition~\ref{prop fixed point Weyl affine}, there is $\hat{z}\in V_0$ such that the translation $\tau_{\hat{z}}$ by $\hat{z}$ commutes with $\hat{w}_\Theta$, whence it commutes with $\widetilde{g}$, and such that $v'_0+\hat{z} \in \cl_k(v'_0) \cap \bigoplus_{\alpha \in \Delta} \left( \frac{1}{e} \mathbb{Z} \varpi_\alpha \right)$. Note that, by definition of $\ell$, the set $\bigoplus_{\alpha \in \Delta} \left( \frac{1}{e} \mathbb{Z} \varpi_\alpha \right)$ is that of special $\ell$-vertices of $\mathbb{A}_0$.

Define $v''_0 = v'_0+\hat{z} \in \cl_k(v'_0)\subset \cl_k\big( Q(v_0,D_0^\Theta) \big)$ and $w''_0 = \widetilde{g} \cdot v''_0$.
Note that $v''_0$ is a special $\ell$-vertex, whence so is $w''_0$.
Since $g_0 \cdot Q(v_0,D_0^\Theta) = Q(w_0,D_0^\Theta)$ and it preserves the simplicial $k$-structure, we have that $g_0 \cdot \cl_k(v'_0) = \cl_k(w'_0) \subset \mathbb{A}_0$.
Hence $v''_0 \in \mathbb{A}_0 \cap g_0^{-1} \mathbb{A}_0$ so that $g_0 \cdot v''_0=\widetilde{g} \cdot v''_0= w''_0$.

Since $h^{-1} v'_0 \in Q\big(y_4(x),D\big)$, there is $\omega \in \Omega$ such that $h^{-1} \cdot v''_0, \in Q\big(y_4(\omega),D\big)$ as being a special $\ell$-vertex contained in $\cl_k\Big(Q\big(y_4(x),D\big)\Big)$.
Since $\widetilde{g} \in W_\Theta \ltimes V_0$, we have, for any $\alpha \in \Delta_D \smallsetminus \Theta_D$ that $\alpha(\omega) = \alpha(h^{-1} \cdot v''_0) = \alpha(h^{-1} \cdot w''_0)$.
Furthermore, because $h^{-1} \cdot w'_0 \in Q\big(y_4(x),D\big)$ and $h^{-1} \cdot w''_0 \in \cl_k(w'_0)$, for any $\alpha \in \Theta_D$ we have that $\alpha(h^{-1} \cdot w''_0) \geqslant \max \left\{ \alpha\big(y_4(\omega)\big), \omega \in \Omega \right\}$ by definition of $y_4(x)$.
Hence $h^{-1} \cdot w''_0 \in Q\big( y_4(\omega),D\big)$.

Hence, according to first step applied with the special vertices $h^{-1} \cdot v''_0, h^{-1} \cdot w''_0 \in Q\big(y_4(\omega),D\big)$, we have that $h^{-1} \cdot v''_0 = h^{-1} \cdot w''_0$.
The translation $\widetilde{\tau} := \tau_z \tau_{\hat{z}}$ commutes with $\widetilde{g}$, hence
\[\widetilde{\tau} \cdot v_0 =v''_0 =w''_0= \widetilde{g} \cdot v''_0 = \widetilde{g} \widetilde{\tau} \cdot v_0 = \widetilde{\tau} \widetilde{g} \cdot v_0 = \widetilde{\tau} \cdot w_0.\]
Therefore $v_0 = w_0$, whence $v=w$.
\end{proof}

Note that, in third step of the proof, the set $\Omega$ depends on $\ell, x, D$ but not on $g,v,w$.
Thus, we can summarize this section in the following result:

\begin{corollary}\label{cor conclusion without special hyp}
For any $\Theta \subset \Delta$, one can define a map $t_\Theta: \Sec_k^\Theta \to \Sec_k^\Delta$ by $Q(x,D) \mapsto y_4(x)$ that satisfies:
\begin{itemize}
    \item for any sector faces $Q(x,D), Q(x',D') \in \Sec_k^\Theta$, the two subsector faces $Q\big(y_4,D) \subseteq Q(x,D)$ and $Q(y_4',D') \subseteq Q(x',D')$ with $y_4 = t_\Theta\big( Q(x,D) \big)$ and $y'_4 = t_\Theta\big( Q(x',D') \big)$ satisfy that,
for any $v \in Q(y_4,D)$, any $w \in Q(y_4',D')$ and any $g \in \mathbf{G}(A)$: 
\begin{equation}\label{eq equal visual boundary 2}
  w=g \cdot v \Rightarrow  D'= g \cdot D;
\end{equation}
\item for any sector face $Q(x,D) \in \Sec_k^\Theta$, the subsector face $Q\big(y_4,D) \subseteq Q(x,D)$ with $y_4 = t_\Theta\big( Q(x,D) \big)$ satisfy that,
for any $v,w \in Q(y_4,D)$ and any $g \in \mathbf{G}(A)$: 
\begin{equation}
  w=g \cdot v \Rightarrow  w=v.
\end{equation}
\end{itemize} 
\end{corollary}

\section{Structure of the quotient of the Bruhat-Tits building by \texorpdfstring{$\mathbf{G}(A)$}{X}}
\label{section structure orbit space}

Let $\mathrm{pr}: \mathcal{X} \to \mathbf{G}(A) \backslash \mathcal{X}$ be the canonical projection. In this section, we describe the topological structure of $\mathrm{pr}(\mathcal{X}_k)$.
We start this section by describing the image of certain $k$-sector chambers in the quotient space $\mathrm{pr}(\mathcal{X}_k)$. In order to do this, let us introduce the following concept.

\begin{definition}
We define a cuspidal rational sector chamber\index{sector chamber!cuspidal} of $\mathcal{X}_k$ as a $k$-sector chamber $Q \subset \mathcal{X}$, such that:
\begin{enumerate}[label=(Cusp\arabic*)]
\item\label{cusp folding} for any vertex $y \in Q$, any $k$-face of $\mathcal{X}_k(y)$ is $\mathbf{G}(A)$-equivalent to a $k$-face contained in $\mathcal{X}_k(y) \cap Q$,
\item\label{cusp spreading} any two different points in $Q$ belong to different $\mathbf{G}(A)$-orbits.
\end{enumerate}
\end{definition}

The following result describes the image in the quotient $\mathrm{pr}(\mathcal{X}_k)$ of any cuspidal rational sector chamber.

\begin{lemma}\label{inc of cuspidal sector chambers}
Let $Q$ be a cuspidal rational sector chamber of $\mathcal{X}_k$.
Then, its image $\mathrm{pr}(Q) \subseteq \mathrm{pr}(\mathcal{X}_k)$ is topologically and combinatorially isomorphic to $Q$.
\end{lemma}

\begin{proof} Let us write $Q=Q(x,D)$, for certain point $x \in \mathcal{X}_k$ and certain vector chamber $D$ of $\mathcal{X}_k$.

We claim that, given a $k$-face $F$ in the complex-theoretical union
\begin{align*}
    \mathcal{X}_k(Q) := \bigcup \lbrace \mathcal{X}_k(v): v \in \mathrm{vert}(Q) \rbrace,
\end{align*}
there exists a unique $k$-face $F'$ in $Q$ such that $F$ and $F'$ belong to the same $\mathbf{G}(A)$-orbit.
\nomenclature[]{$\mathcal{X}_k(Q)$}{$1$-combinatorial neighbourhood of $Q$}
Indeed, by definition of $\mathcal{X}_k(Q)$, there exists a vertex $v$ of $Q$ such that $F \subset\mathcal{X}_k(v) $.
It follows from Property~\ref{cusp folding} that
there exists $g \in \mathbf{G}(A)$ such that $g \cdot F \subset \mathcal{X}_k(v) \cap Q$.
Therefore, given a $k$-face $F$ in $\mathcal{X}_k(Q)$, there exists at least one $k$-face $F'$ in $Q$ that belongs to the $\mathbf{G}(A)$-orbit of $F$.
Moreover, applying Property~\ref{cusp spreading}, we deduce that the preceding $k$-face $F'\subset Q$ is unique. Thus, the claim follows.

Let $v$ be a vertex of $Q$. It follows from the preceding claim that the image of the star $\mathcal{X}_k(v)$ in the quotient $\mathrm{pr}(\mathcal{X}_k)$ is isomorphic to $\mathcal{X}_k(v) \cap Q$. In particular, the neighboring vertices of $\mathrm{pr}(v)$ in $\mathrm{pr}(\mathcal{X}_k)$ are exactly the vertices $\mathrm{pr}(w)$, with $w \in \mathrm{vert}(\mathcal{X}_k(v) \cap Q)$. This implies that $\mathrm{pr}(Q)$ is combinatorially isomorphic to $Q$.
Moreover, since for any vertex $v \in Q$, the image of the star $\mathcal{X}_k(v)$ in the quotient $\mathrm{pr}(\mathcal{X}_k)$ is isomorphic to $\mathcal{X}_k(v) \cap Q$, the projection $\mathrm{p}: Q \to \mathrm{pr}(Q)$ is a local homeomorphism. Since Property~\ref{cusp spreading} implies that the restriction $\mathrm{p}$ of $\mathrm{pr}$ to $Q$ is bijective, we deduce that $\mathrm{p}$ is an homeomorphism, which concludes the proof.
\end{proof}


On the set of cuspidal rational sector chambers, we define an equivalence relation by setting:
\begin{equation}\label{eq rel on cup sec cham}
Q \sim_{\mathbf{G}(A)} Q' \text{ if and only if } \exists g \in \mathbf{G}(A): \, \, g \cdot \partial_{\infty}Q=\partial_{\infty} Q'.
\end{equation}
Let us fix a system of representatives $C_{\emptyset}$ for the preceding equivalence relation.

Note that, from the definition of $\sim_{\mathbf{G}(A)}$, we know that any subsector chamber $Q'$ of a given chamber $Q$ satisfies $Q' \sim_{\mathbf{G}(A)} Q$.
Then, up to replacing the elements of $C_{\emptyset}$ by their subsector chambers given by Corollary~\ref{cor conclusion without special hyp}, we assume that, for any $Q,Q' \in C_{\emptyset}$, we have that

\begin{equation}\label{cusp visual boundary}
\forall v \in Q,\ \forall w \in Q', \ g \cdot v = w \Rightarrow g \cdot \partial_{\infty}(Q)= \partial_{\infty}(Q').
\end{equation}

Thus, it follows from Property~\ref{cusp spreading}, Property~\eqref{cusp visual boundary} and the definition of $\sim_{\mathbf{G}(A)}$ that, for any $Q,Q' \in C_{\emptyset}$, with $Q\neq Q'$, we have $\mathrm{pr}(Q) \cap \mathrm{pr}(Q')=\emptyset$. It follows from Proposition~\ref{prop starAction} and Corollary~\ref{cor image of sector faces} that given an arbitrary $k$-sector chamber, it contains a subsector chamber that is a cuspidal rational sector chamber. Since $\partial_{\infty}Q=\partial_{\infty} Q'$ implies $Q \sim_{\mathbf{G}(A)} Q'$, we have that the $\mathbf{G}(A)$-orbits of the visual boundaries of chambers in $C_{\emptyset}$ cover all the chambers in $\partial_{\infty}(\mathcal{X}_k)$.

In order to describes the topological structure of $\mathrm{pr}(\mathcal{X}_k)$, in the next result, we introduce some sets $C_\Theta$ consisting of $k$-sector faces of type $\Theta$, for $\Theta \subset \Delta$.

\begin{theorem}\label{main teo 2 new}
There exists a family $\left( C_\Theta \right)_{\Theta \subset \Delta}$ of sets $C_{\Theta}=\lbrace Q_{i,\Theta}: i\in \mathrm{I}_\Theta \rbrace$ consisting of $k$-sector faces of type $\Theta$ of $\mathcal{X}_k$ and there exists a subspace $\mathcal{Y}$ of $\mathcal{X}_k$ such that
\begin{enumerate}[label=(\arabic*)]
    \item\label{item mainTh1} for each $i\in \mathrm{I}_\Theta $, there exists $x \in \mathcal{X}_k$ and $D \in \mathbf{G}(k) \cdot D_0^{\Theta}$ such that $Q_{i,\Theta}$ equals $Q(x,D)$, 
    \item\label{item mainTh2} the set $C_{\emptyset}$ is a system of representatives for the equivalence relation defined on the set of cuspidal rational sector chambers by Equation~\eqref{eq rel on cup sec cham}. In particular, any two different sector chambers in $C_{\emptyset}$ do not intersect.
    \item\label{item mainTh3} given a non-empty set $\Theta \subset \Delta$, for each pair of indices $i,j \in \mathrm{I}_\Theta$, if $\mathbf{G}(A) \cdot Q_{i,\Theta} \cap \mathbf{G}(A) \cdot Q_{j,\Theta} \neq \emptyset$, then $\mathbf{G}(A) \cdot \partial^{\infty}(Q_{i,\Theta}) \cap \mathbf{G}(A) \cdot \partial^{\infty}(Q_{j,\Theta}) \neq \emptyset$,
    \item\label{item mainTh4} each sector chamber in $C_{\emptyset}$ embeds in the quotient space $\mathrm{pr}(\mathcal{X}_k)=\mathbf{G}(A)\backslash \mathcal{X}_k$,
    \item\label{item mainTh5} the quotient space $\mathrm{pr}(\mathcal{X}_k)$ is a CW-complex obtained as the attaching space of the images $\mathrm{pr}( \overline{Q_{i,\Theta}}) \subseteq \mathrm{pr}(\mathcal{X}_k)$ of the closures $\overline{Q_{i,\Theta}}$ of all the $k$-sector faces $Q_{i,\Theta}$ and the image $\mathrm{pr}(\mathcal{Y})$ of $\mathcal{Y}$ along certain subsets.
\end{enumerate}
Moreover, when $\mathbf{G}=\mathrm{SL}_n$ or $\mathrm{GL}_n$ or $\mathbb{F}$ is finite, we have that any sector face in any $C_{\Theta}$ embeds in $\mathrm{pr}(\mathcal{X}_k)$.

\end{theorem}

\begin{proof} Note that Statement~\ref{item mainTh2} defines the set $C_{\emptyset}$.
Moreover, note that Statement~\ref{item mainTh4} directly follows Lemma~\ref{inc of cuspidal sector chambers}.
When $\Theta$ is not empty, we define the sets $C_{\Theta}$ by induction on the cardinality of $\Theta$. 
Firstly, assume that $\Theta$ has just one element. 
We denote by $C^0_{\Theta}$ the set of $k$-sector faces $Q(x,D)$ in $\mathcal{X}_k \smallsetminus \bigcup \lbrace g \cdot Q_{i, \emptyset} : i \in \mathrm{I}_{\emptyset}, g \in \mathbf{G}(A) \rbrace$ whose direction $D$ belong to $\mathbf{G}(k) \cdot D_0^{\Theta}$.
Given a $k$-sector face $Q=Q(x,D) \in C^0_{\Theta}$ there exists a subsector face of $Q(y_4,D) \subseteq Q$ that satisfies the Statements~\ref{item starAction fix}, \ref{item starAction orbits} and~\ref{item starAction fixed faces} of Proposition~\ref{prop starAction}, Equation~\eqref{eq equal visual boundary 2} in Corollary~\ref{cor conclusion without special hyp}, and without two different points in the same $\mathbf{G}(A)$-orbit when $\mathbf{G}=\mathrm{SL}_n$ or $\mathrm{GL}_n$ or $\mathbb{F}$ is finite, as Corollary~\ref{cor image of sector faces} shows.
We denote $C^1_{\Theta}$ the set of all the aforementioned $k$-subsector faces. Then, we define $C_{\Theta}$ as a representative set of $C^1_{\Theta}$ by the equivalence relation
\begin{equation}\label{eq rel on sector faces}
Q \sim^{*} Q' \text{ if and only if } \exists g \in \mathbf{G}(A): \, \, g \cdot Q=Q',
\end{equation} 
Assume that we have defined $C_{\Theta'}$, for each set $\Theta'$ such that $\mathrm{Card}(\Theta') \leq n-1$. Then, we are able to define $C_{\Theta}$, for each set $\Theta$ with $n$ elements. Indeed, let us denote by $C_{\Theta}^{0}$ the set of $k$-sector faces $Q(x,D)$ contained in 
$$\mathcal{X}_k \smallsetminus \bigcup \left\lbrace g \cdot Q_{i, \Theta'} : i \in \mathrm{I}_{\Theta'}, \mathrm{Card}(\Theta')\leq n-1, g \in \mathbf{G}(A) \right\rbrace,$$
whose direction $D$ belongs to $\mathbf{G}(k) \cdot D_0^{\Theta}$.
As in the case of cardinality one, recall that, given a $k$-sector face $Q=Q(x,D) \in C^0_{\Theta}$ there exists a subsector face of $Q(y_4,D) \subseteq Q$ that satisfies the Statements~\ref{item starAction fix}, \ref{item starAction orbits} and~\ref{item starAction fixed faces} of Proposition~\ref{prop starAction}, Equation~\eqref{eq equal visual boundary 2} in Corollary~\ref{cor conclusion without special hyp}, and without two different points in the same $\mathbf{G}(A)$-orbit when $\mathbf{G}=\mathrm{SL}_n$  or $\mathrm{GL}_n$ or $\mathbb{F}$ is finite, as Corollary~\ref{cor image of sector faces} shows.
We denote $C^1_{\Theta}$ the set of such $k$-subsector faces.
Then, we define $C_{\Theta}$ as a representative set for the equivalence relation defined in $C_{\Theta}^{1}$ by Equation~\eqref{eq rel on sector faces}.
Thus, by induction, we have defined all the sets $C_{\Theta}$, with $\Theta \subset \Delta$. Thus, Statement~\ref{item mainTh1} follows by definition. Moreover, Statement~\ref{item mainTh3} directly follows from Corollary~\ref{cor conclusion without special hyp}.

Now, we define $\mathcal{Y}$ as the space $\mathcal{X}_k \smallsetminus \bigcup \left\lbrace g \cdot Q_{i, \Theta} : i \in \mathrm{I}_{\Theta}, \Theta \subset \Delta, g \in \mathbf{G}(A) \right\rbrace$.
Let us denote by $\mathrm{pr}( \overline{Q_{i,\Theta}} )$ and $\mathrm{pr}(\mathcal{Y})$ the image in $\mathrm{pr}(\mathcal{X}_k)$ of $\overline{Q_{i,\Theta}}$, with $Q_{i,\Theta} \in C_{\Theta}$, and $\mathcal{Y}$, respectively.
Since $\mathcal{X}_k=\mathcal{Y} \cup \bigcup \left\lbrace g \cdot \overline{Q_{i, \Theta}} : i \in \mathrm{I}_{\Theta}, \Theta \subset \Delta, g \in \mathbf{G}(A) \right\rbrace,$ we have
\[ \mathrm{pr}(\mathcal{X}_k)= \mathrm{pr}(\mathcal{Y}) \cup \bigcup \left\lbrace \mathrm{pr}(\overline{Q_{i, \Theta}}): i \in \mathrm{I}_{\Theta}, \Theta \subset \Delta \right\rbrace. \]
whence Statement~\ref{item mainTh5} follows.
\end{proof}


The following result describes the intersection between certain cuspidal rational sector chambers.

\begin{lemma}\label{int with the closure}
Let $Q$ and $Q'$ be two rational sector chambers contained in a system of representatives for the equivalence relation defined in Equation~\eqref{eq rel on cup sec cham}. If $\overline{Q} \cap g \cdot Q'\neq \emptyset$, for some $g \in \mathbf{G}(A)$, then $Q=Q'$ and $g \in \mathrm{Stab}_{\mathbf{G}(A)}( \partial_{\infty}(Q))$.
\end{lemma}

\begin{proof}
Let $w$ be a point in $\overline{Q} \cap g \cdot Q'$.

We claim that $Q \cap g \cdot Q' \neq \emptyset$. If $w \in Q$, then the claim follows. In any other case, $w$ belongs to $\overline{Q} \smallsetminus Q$. Since $g \cdot Q'$ is open, there exists an open ball $O \subset g \cdot Q'$ containing $w$. Since $w \in \overline{Q} \smallsetminus Q$ and $O$ is open, the intersection $O \cap Q$ is not empty, whence the claim follows.

Let $w'$ be a point in $Q \cap g \cdot Q'$. Since $w' \in Q$ and $g^{-1} \cdot w' \in Q'$ belong to the same $\mathbf{G}(A)$-orbit, it follows from Equation~\eqref{cusp visual boundary} that $\partial_{\infty}(Q)=g^{-1} \cdot \partial_{\infty}(Q')$.
Moreover, since $Q,Q'\in C_{\emptyset}$, we deduce that $Q=Q'$, whence $g \in \mathrm{Stab}_{\mathbf{G}(A)}( \partial_{\infty}(Q))$.
\end{proof}

\begin{theorem}\label{theorem quotient and cuspidal rational sector chambers}
There exists a system of representatives $C=\lbrace Q_i: i \in \mathrm{I}_{\emptyset} \rbrace$ for the equivalence relation on the cuspidal rational sector chambers defined in Equation~\eqref{eq rel on cup sec cham}, and a closed connected subspace $\mathcal{Z}$ of $\mathcal{X}$ such that:
\begin{enumerate}[label=(\alph*)]
    \item\label{item 86a} the quotient space $\mathrm{pr}(\mathcal{X}_k)$ is the union of $\mathcal{Z}$ and of the $\mathrm{pr}( \overline{Q_{i}}) \subseteq \mathrm{pr}(\mathcal{X}_k)$, for $i \in \mathrm{I}_{\emptyset}$, 
    \item\label{item 86b'} any $\overline{Q_i}$ is contained in a cuspidal rational sector chamber, in particular,
    \item\label{item 86b} $\mathrm{pr}(\overline{Q_i})$ is topologically and combinatorially isomorphic to $\overline{Q}_i$,
    \item\label{item 86rem} for each point $x \in \overline{Q}_i$, we have $\mathrm{Stab}_{{\mathbf{G}(A)}}(x)=\mathrm{Fix}_{{\mathbf{G}(A)}}\big(x+\partial_{\infty}(Q_i)\big)$,
    \item\label{item 86c} for each $i \neq j$, $\mathrm{pr}( \overline{Q_i}) \cap \mathrm{pr}( \overline{Q_j})=\emptyset$,
    \item\label{item 86d} the intersection of $\mathcal{Z}$ with any $\mathrm{pr}( \overline{Q_i})$ is connected.
\end{enumerate}
\end{theorem}
\begin{proof}
Let $C_{\emptyset}=\lbrace Q_{i,\emptyset}: i \in \mathrm{I}_{\emptyset} \rbrace$ be given by Theorem~\ref{main teo 2 new}\ref{item mainTh2}. It follows from Proposition~\ref{prop igual stab} that, up to replacing some $Q_{i,\emptyset}$ by subsection if needed, we can assume that for each point $x \in \overline{Q}_{i,\emptyset}$ we have $\mathrm{Stab}_{{\mathbf{G}(A)}}(x)=\mathrm{Fix}_{{\mathbf{G}(A)}}\big(x+\partial_{\infty}(Q_i)\big).$

Let us define $C:=\lbrace Q_{i} :i \in \mathrm{I}_{\emptyset} \rbrace$ as a set of sector chambers satisfying $\overline{Q_{i}} \subseteq Q_{i,\emptyset}$, for all $i \in \mathrm{I}_{\emptyset}$ (c.f.~Lemma~\ref{lem subsector face and roots}). Then, Statements~\ref{item 86b'} and \ref{item 86rem} are satisfied.
Since $Q_{i,\emptyset}$ does not have two different points in the same $\mathbf{G}(A)$-orbit, the same holds for $\overline{Q_i}$.
Moreover, since $\mathrm{pr}(Q_{i,\emptyset})$ is topologically and combinatorially isomorphic to $Q_{i,\emptyset}$, Statement~\ref{item 86b} follows.
Now, assume, by contradiction, that $\mathrm{pr}(\overline{Q_i})\cap \mathrm{pr}(\overline{Q_j}) \neq \emptyset$.
Then, there exists $x_1 \in \overline{Q_{i}}$, $x_2 \in \overline{Q_{j}}$ and $g_1, g_2 \in \mathbf{G}(A)$ such that $g_1 \cdot x_1= g_2 \cdot x_2$.
Since $x_1 \in Q_{i,\emptyset}$ and $x_2 \in Q_{j,\emptyset}$ we deduce from Corollary~\ref{cor conclusion without special hyp} that $ (g_2^{-1} g_1) \cdot \partial_{\infty}(Q_{i,\emptyset}) =\partial_{\infty}(Q_{j,\emptyset})$.
Then, by definition of $C_{\emptyset}$, we deduce that $i=j$.
Thus, Statement~\ref{item 86c} follows.

In that follows we define $\mathcal{Z}$, we prove that it is closed connected and we check Statements~\ref{item 86a} and \ref{item 86d}.

\textbf{First step:} Let us denote by $\mathcal{Y}'$ the space $\mathcal{X}_k \smallsetminus \bigcup \lbrace g \cdot Q_{i} : i \in \mathrm{I}_{\emptyset}, g \in \mathbf{G}(A) \rbrace$.
We claim that $\mathcal{Y}'$ is closed, $\mathbf{G}(A)$-stable and that $\overline{Q_{i}} \cap \mathcal{Y}'$ is a stared space.

Since the sector chambers in $\mathcal{X}_k$ are open, the space $\mathcal{Y}'$ is closed.
Moreover, note that, since the union of the sector chambers $g \cdot Q_{i}$, for $i \in \mathrm{I}_{\emptyset}$ and $g \in \mathbf{G}(A)$, is the union of a certain number of $\mathbf{G}(A)$-orbits of sector chambers, we have that $\mathcal{Y}'$ is $\mathbf{G}(A)$-stable.
For each $i \in \mathrm{I}_{\emptyset}$, let us write $Q_{i}$ as $Q(x_i,D_i)$.
Note that $\overline{Q_{i}} \cap \mathcal{Y}'$ is contained in $\overline{Q_{i}} \cap (\mathcal{X} \smallsetminus Q_{i, \emptyset}) = \overline{Q_{i}} \smallsetminus Q_{i}$.

Now, we prove that $\overline{Q_{i}} \cap \mathcal{Y}'$ is a stared space.
Let $y$ be a point in $ \overline{Q_{i}} \cap \mathcal{Y}'$ and let $[x_i,y]= \lbrace ty+(1-t)x_i: t \in [0,1] \rbrace$ be the geodesic path joining $y$ with $x_i$. 
We have to prove that $[x_i,y]$ is contained in $\overline{Q_{i}} \cap \mathcal{Y}'$. 
Assume, by contradiction, that $[x_i,y]$ is not contained in $\overline{Q_{i}} \cap \mathcal{Y}'$.
Then, since $[x_i,y] \subset \overline{Q(x_i,D_i)}$, there exists a point $w \in [x_i,y]$ that is not contained in $\mathcal{Y}'$.
In other words, there exists a point $w$ that is contained in $ g \cdot Q_{j}$, for some $g \in \mathbf{G}(A)$ and some $j \in \mathrm{I}_{\emptyset}$.
Thus, it follows from Lemma~\ref{int with the closure} that $Q_{i}=Q_{j}$ and $g \in \mathrm{Stab}_{\mathbf{G}(A)}( D_{i})$.
In particular $g \cdot Q_{j}=Q(z,D_i)$, for some point $z \in \mathcal{X}_k$, and then $w \in Q(z,D_i)$.
If $w = y$, then $y \in Q(z,D_i)$, which contradicts the choice of $y$.
Otherwise, since $w \in [x_i,y]$, there is $\lambda > 0$ such that $y - w = \lambda (y - x_i) \in \mathbb{A}_i$, where $\mathbb{A}_i$ is an apartment containing $Q(x_i,D_i)$.
Since $y - x_i \in D_i$, we have that $y-w \in D_i$, whence $y \in Q(w,D_i)$.
Therefore, in an apartment containing $Q(z,D_i) \supset Q(w,D_i)$, we have that $y \in Q(z,D_i)$ since $w \in Q(z,D_i)$.
We deduce that $y \notin \mathcal{Y}'$, which contradicts the choice of $y$.
Thus, we conclude that $\overline{Q_{i}} \cap \mathcal{Y}'$ is stared.

\textbf{Second step:} Let us define $\mathcal{Z}'$ as the image of $\mathcal{Y}'$ in $\mathrm{pr}(\mathcal{X}_k)$.
We claim that $\mathcal{Z}'$ is closed and that $\mathcal{Z}'\cap \mathrm{pr}(\overline{Q_i})$ is connected.

Since $\mathrm{pr}: \mathcal{X}_k \to \mathbf{G}(A) \backslash \mathcal{X}_k$ is an open continuous map, it follows from the first step that $\mathrm{pr}(\mathcal{Y}'^c)$ is an open set of $\mathrm{pr}(\mathcal{X})$.
Since $\mathrm{pr}$ is surjective, we have that $\mathrm{pr}(\mathcal{Y}')^c \subseteq \mathrm{pr}(\mathcal{Y}'^c)$.
Now, let $y \in \mathrm{pr}(\mathcal{Y}'^c)$. Then $y=\mathrm{pr}(x)$, where $x \notin \mathcal{Y}'$. If $y=\mathrm{pr}(z)$, with $z \in \mathcal{Y}'$, then $x$ and $z$ belong to the same $\mathbf{G}(A)$-orbit.
Then, since $\mathcal{Y}'$ is $\mathbf{G}(A)$-stable, we deduce that $x,z \in \mathcal{Y}'$, which is impossible.
Thus $y \in \mathrm{pr}(\mathcal{Y}')^c$, whence we conclude that $\mathrm{pr}(\mathcal{Y}')^c = \mathrm{pr}(\mathcal{Y}'^c)$.
Hence, we get that $\mathcal{Z}'=\mathrm{pr}(\mathcal{Y}')$ is a closed set of $\mathrm{pr}(\mathcal{X})$.

Now, we prove that, for any $i \in \mathrm{I}_{\emptyset}$, we have that the intersection $\mathcal{Z}' \cap \mathrm{pr}(\overline{Q_{i}})=\mathrm{pr}(\mathcal{Y}') \cap \mathrm{pr}(\overline{Q_{i}})$ is connected.
Note that $y \in \mathrm{pr}(\mathcal{Y}') \cap \mathrm{pr}(\overline{Q_{i}})$ exactly when $y=\mathrm{pr}(x)$, where $x= g_1 \cdot x_1$ and $x= g_2 \cdot x_2$, for some $x_1 \in \mathcal{Y}'$, $x_2 \in \overline{Q_{i}}$ and $g_1, g_2 \in \mathbf{G}(A)$. Since $\mathrm{pr}(g_1^{-1} \cdot x)=\mathrm{pr}(x)$, up to replacing $x$ by another representative, we assume that $g_1$ is trivial.
Thus, we get that
\[ \mathrm{pr}(\mathcal{Y}') \cap \mathrm{pr}(\overline{Q_{i}})=\mathrm{pr}\big(\bigcup_{g \in \mathbf{G}(A)} \mathcal{Y}' \cap g \cdot \overline{Q_{i}} \big)=\bigcup_{g \in \mathbf{G}(A)} \mathrm{pr}\big(\mathcal{Y}' \cap g \cdot\overline{Q_{i}} \big).\]
Since $\mathcal{Y}'$ is $\mathbf{G}(A)$-invariant, we have that
\[\mathrm{pr}\big(\mathcal{Y}' \cap g \cdot \overline{Q_{i}} \big)= \mathrm{pr}\big(g^{-1} \cdot \mathcal{Y}' \cap \overline{Q_{i}} \big) = \mathrm{pr}\big(\mathcal{Y}' \cap \overline{Q_{i}} \big),\]
which implies that $\mathrm{pr}(\mathcal{Y}') \cap \mathrm{pr}(\overline{Q_{i}})=\mathrm{pr}\big(\mathcal{Y}' \cap \overline{Q_{i}} \big)$. Then, since $\mathrm{pr}$ is continuous and the first step shows that $\mathcal{Y'} \cap \overline{ Q_{i}}$ is connected, we have that $\mathrm{pr}\big(\mathcal{Y}' \cap  \overline{Q_{i}} \big)$ is also connected. This proves that $\mathcal{Z}'\cap \mathrm{pr}(\overline{Q_i})$ is connected.

\textbf{Third step:} For any $i \in \mathrm{I}_{\emptyset}$, we claim that $\mathcal{Z}' \cap \left( \mathrm{pr}(\overline{Q_{i}})) \smallsetminus \mathrm{pr}(Q_{i})) \right) \neq \emptyset$.

First, assume that $\mathcal{Y}' \cap (\overline{Q_{i}} \smallsetminus Q_i )$ is empty. Then $\overline{Q_{i}} \smallsetminus Q_i \subseteq \mathcal{Y}'^c$.
Let $x_i$ be the tip of $Q_i$, which is contained in $\overline{Q_{i}} \smallsetminus Q_i $.
Then $x_i \in g \cdot Q_j$, for some $g \in \mathbf{G}(A)$ and some $j \in \mathrm{I}_{\emptyset}$.
In particular $\overline{Q_i} \cap g \cdot Q_j \neq \emptyset$.
Hence, it follows from Lemma~\ref{int with the closure} that $Q_{i}=Q_{j}$.
Thus, we get that $x_i \in g \cdot Q_i$. In other words $x_i= g\cdot y_i$, for some $y_i \in Q_i$.
Since $\overline{Q}_i$ does not have two different points in the same $\mathbf{G}(A)$-orbit and $x_i \neq y_i$, we obtain a contradiction.
Thus, we conclude that $\mathcal{Y}' \cap (\overline{Q_{i}} \smallsetminus Q_i )$ is not empty.

Now, let $x \in \mathcal{Y}' \cap (\overline{Q_{i}} \smallsetminus Q_i )$.
Then, we have that $\mathrm{pr}(x)$ belongs to $\mathcal{Z}' \cap \left( \mathrm{pr}(\overline{Q_{i}}) \smallsetminus \mathrm{pr}(Q_{i}) \right)$.
Indeed, $\mathrm{pr}(x) \in \mathcal{Z}' \cap \mathrm{pr}(\overline{Q}_i)$.
Moreover, if $\mathrm{pr}(x)\in \mathrm{pr}(Q_i)$, then $x$ is in the same $\mathbf{G}(A)$-orbit as that of a point $y \in Q_i$, whence $x=y$, since $\overline{Q}_i$ does not have two different points in the same $\mathbf{G}(A)$-orbit.
This contradicts the choice of $x$.
Thus, the element $\mathrm{pr}(x)$ belongs to $\mathcal{Z}' \cap \left( \mathrm{pr}(\overline{Q_{i}}) \smallsetminus \mathrm{pr}(Q_{i}) \right)$. In particular, it is not empty.


\textbf{Fourth step:} Set $\mathcal{Z}:=\mathcal{Z'} \cup \bigcup \lbrace \mathrm{pr}(\overline{Q_i}) \smallsetminus \mathrm{pr}(Q_i) : i \in \mathrm{I}_{\emptyset} \rbrace$. We claim that $\mathcal{Z}$ satisfies Statements~\ref{item 86a} and~\ref{item 86d}.

Firstly, we check Statement~\ref{item 86a}. Since $\mathcal{X}= \mathcal{Y}' \cup \bigcup \lbrace g \cdot \overline{Q_{i}}: i \in \mathrm{I}_{\emptyset}, g \in \mathbf{G}(A) \rbrace$, we have that
\begin{equation}\label{eq pr(x), Z and cusps}
\mathrm{pr}(\mathcal{X})=\mathcal{Z}' \cup \bigcup \lbrace  \mathrm{pr}(\overline{Q_{i}}): i \in \mathrm{I}_{\emptyset}\rbrace.
\end{equation}
In particular, we deduce that $\mathrm{pr}(\mathcal{X})=\mathcal{Z} \cup \bigcup \lbrace  \mathrm{pr}(\overline{Q_{i}}): i \in \mathrm{I}_{\emptyset}\rbrace$, since $\mathcal{Z}' \subseteq \mathcal{Z}$. 

Secondly, we prove Statement~\ref{item 86d}.
For each $i \in \mathrm{I}_{\emptyset}$, we have that
\[ \mathrm{pr}(\overline{Q_i}) \cap  \mathcal{Z} = \mathrm{pr}(\overline{Q_i}) \cap \left( \mathcal{Z'} \cup \bigcup \lbrace \mathrm{pr}(\overline{Q_j}) \smallsetminus \mathrm{pr}(Q_j) : j \in \mathrm{I}_{\emptyset} \rbrace \right).\]
Then
$$\mathrm{pr}(\overline{Q_i}) \cap  \mathcal{Z}= \left(\mathrm{pr}(\overline{Q_i}) \cap  \mathcal{Z}' \right)  \cup  \left( \bigcup \lbrace \mathrm{pr}(\overline{Q_i}) \cap \left( \mathrm{pr}(\overline{Q_j}) \smallsetminus \mathrm{pr}(Q_j) \right) : j \in \mathrm{I}_{\emptyset} \rbrace \right),$$
which equals $\left( \mathrm{pr}(\overline{Q_i}) \cap  \mathcal{Z}' \right) \cup \left( \mathrm{pr}(\overline{Q_i}) \smallsetminus \mathrm{pr}(Q_i) \right)$, since $\mathrm{pr}(\overline{Q_j}) \cap \mathrm{pr}(\overline{Q_i})=\emptyset$ if $i \neq j$.
Note that $\left( \mathrm{pr}(\overline{Q_i}) \cap  \mathcal{Z}' \right) \cap \left( \mathrm{pr}(\overline{Q_i}) \smallsetminus \mathrm{pr}(Q_i) \right)$ equals $\mathcal{Z}' \cap \left( \mathrm{pr}(\overline{Q_{i}}) \smallsetminus \mathrm{pr}(Q_{i}) \right)$. In particular, the third step shows that $\left( \mathrm{pr}(\overline{Q_i}) \cap  \mathcal{Z}' \right) \cap \left( \mathrm{pr}(\overline{Q_i}) \smallsetminus \mathrm{pr}(Q_i) \right)$ is not empty.
Therefore, since $\left( \mathrm{pr}(\overline{Q_i}) \cap  \mathcal{Z}' \right)$ and $\left( \mathrm{pr}(\overline{Q_i}) \smallsetminus \mathrm{pr}(Q_i) \right)$ are connected, Statement~\ref{item 86d} follows.

\textbf{Fifth step:} The space $\mathcal{Z}$ is closed and connected.

First, we check that $\mathcal{Z}$ is connected.
Since $\mathcal{X}$ is connected and $\mathrm{pr}$ is continuous, we have that $\mathrm{pr}(\mathcal{X})$ is connected.
Note that the closure $\overline{Q}$ of any sector chamber $Q$ can be retracted onto its boundary, i.e.~on $\overline{Q} \smallsetminus Q$.
Hence, the space $\mathcal{W}$ obtaining by retracting each $\mathrm{pr}(\overline{Q_{i}}) \cong \overline{Q_{i}}$ onto its boundary is connected.
Applying this retraction in each term of the union in Equation~\eqref{eq pr(x), Z and cusps}, we obtain that \[ \mathcal{W}= \left( \mathcal{Z}' \smallsetminus \bigcup \lbrace  \mathcal{Z}' \cap \mathrm{pr}(Q_{j}): j \in \mathrm{I}_{\emptyset}\rbrace \right) \cup \bigcup \lbrace \mathrm{pr}(\overline{Q_j}) \smallsetminus \mathrm{pr}(Q_j): j \in \mathrm{I}_{\emptyset} \rbrace.\]
Let $\mathcal{W}_1=\left( \mathcal{Z}' \smallsetminus \bigcup \lbrace  \mathcal{Z}' \cap \mathrm{pr}(Q_{j}): j \in \mathrm{I}_{\emptyset}\rbrace \right)$ and $\mathcal{W}_2=\bigcup \lbrace \mathrm{pr}(\overline{Q_j}) \smallsetminus \mathrm{pr}(Q_j): j \in \mathrm{I}_{\emptyset} \rbrace$.
Recall that, by second step, any space $\mathcal{Z}' \cap \mathrm{pr}(\overline{Q_i})$ is connected. Moreover, note that $\mathcal{W} \cap \left( \mathcal{Z}' \cap \mathrm{pr}(\overline{Q_i})\right)$ equals the union of
$$ \mathcal{W}_1 \cap \left( \mathcal{Z}' \cap \mathrm{pr}(\overline{Q_i})\right)= \left(\mathcal{Z}' \cap \mathrm{pr}(\overline{Q_i}) \right) \smallsetminus \bigcup \lbrace  \mathcal{Z}' \cap \mathrm{pr}(Q_{j}): j \in \mathrm{I}_{\emptyset}\rbrace$$ 
with 
$$ \mathcal{W}_2 \cap \left( \mathcal{Z}' \cap \mathrm{pr}(\overline{Q_i})\right)=\bigcup \lbrace \left(\mathcal{Z}'\cap \mathrm{pr}(\overline{Q_i}) \right) \cap \left( \mathrm{pr}(\overline{Q_j})  \smallsetminus \mathrm{pr}(Q_j) \right): j \in \mathrm{I}_{\emptyset} \rbrace.$$
Then, since $\mathrm{pr}(\overline{Q_i}) \cap \mathrm{pr}(\overline{Q_j})=\emptyset$, if $i \neq j$, we have that
\[ \mathcal{W}_1 \cap \left( \mathcal{Z}' \cap \mathrm{pr}(\overline{Q_i})\right)= \mathcal{Z}' \cap \left( \mathrm{pr}(\overline{Q_{i}}) \smallsetminus \mathrm{pr}(Q_{i}) \right)=\mathcal{W}_2 \cap \left( \mathcal{Z}' \cap \mathrm{pr}(\overline{Q_i})\right).\]
Thus, we conclude that $\mathcal{W} \cap \left( \mathcal{Z}' \cap \mathrm{pr}(\overline{Q_i})\right)=\mathcal{Z}' \cap \left( \mathrm{pr}(\overline{Q_{i}}) \smallsetminus \mathrm{pr}(Q_{i}) \right)$. 
In particular, the third step implies that $\mathcal{W} \cap \left( \mathcal{Z}' \cap \mathrm{pr}(\overline{Q_i})\right)$ is not empty. 
Hence, since $\mathcal{W}$ and each $\mathcal{Z}' \cap \mathrm{pr}(\overline{Q_i})$ are connected, we conclude that $\mathcal{W} \cup \bigcup \lbrace \mathcal{Z}' \cap \mathrm{pr}(\overline{Q_i}) : i \in \mathrm{I}_{\emptyset} \rbrace$ is connected. 
But, by definition $\mathcal{W} \cup \bigcup \lbrace \mathcal{Z}' \cap \mathrm{pr}(\overline{Q_i}) : i \in \mathrm{I}_{\emptyset} \rbrace$ equals the union of $\mathcal{W}_1=\left( \mathcal{Z}' \smallsetminus \bigcup \lbrace  \mathcal{Z}' \cap \mathrm{pr}(Q_{j}): j \in \mathrm{I}_{\emptyset}\rbrace \right)$ with $\bigcup \lbrace \mathcal{Z}' \cap \mathrm{pr}(\overline{Q_i}) : i \in \mathrm{I}_{\emptyset} \rbrace$ and with $\mathcal{W}_2=\bigcup \lbrace \mathrm{pr}(\overline{Q_j}) \smallsetminus \mathrm{pr}(Q_j): j \in \mathrm{I}_{\emptyset} \rbrace.$ 
Therefore, since
$$\mathcal{Z}'= \left( \mathcal{Z}' \smallsetminus \bigcup \lbrace  \mathcal{Z}' \cap \mathrm{pr}(Q_{j}): j \in \mathrm{I}_{\emptyset}\rbrace \right) \cup \bigcup \lbrace \mathcal{Z}' \cap \mathrm{pr}(\overline{Q_i}) : i \in \mathrm{I}_{\emptyset} \rbrace,$$
we conclude that 
$$\mathcal{Z}=\mathcal{Z'} \cup \bigcup \lbrace \mathrm{pr}(\overline{Q_i}) \smallsetminus \mathrm{pr}(Q_i) : i \in \mathrm{I}_{\emptyset} \rbrace=\mathcal{W} \cup \bigcup \lbrace \mathcal{Z}' \cap \mathrm{pr}(\overline{Q_i}) : i \in \mathrm{I}_{\emptyset} \rbrace,$$
is connected.

We finally prove that $\mathcal{Z}$ is closed. Since the second step shows that $\mathcal{Z}'$ is closed, we just need to prove that $\mathcal{Q}:=\bigcup \lbrace \mathrm{pr}(\overline{Q_i}) \smallsetminus \mathrm{pr}(Q_i) : i \in \mathrm{I}_{\emptyset} \rbrace$ is closed. Let $(x_n)_{n=1}^{\infty}$ be a sequence in $\mathcal{Q}$ that converges to $x \in \mathrm{pr}(\mathcal{X})$.
In particular $(x_n)_{n=1}^{\infty}$ is a Cauchy sequence. 
Then, since $\mathrm{pr}(\overline{Q}_i) \subset \mathrm{pr}(Q_{i,\emptyset})$, where $\mathrm{pr}(Q_{i,\emptyset})$ is open, and since $\mathrm{pr}(Q_{i,\emptyset}) \cap \mathrm{pr}( Q_{j,\emptyset})=\emptyset$, if $i \neq j$, we have that there exists $i \in \mathrm{I}_{\emptyset}$ and $N \in \mathbb{Z}_{\geq 0}$ such that $x_n \in \mathrm{pr}(\overline{Q}_i) \smallsetminus \mathrm{pr}(Q_i)$, for all $n \geq N$. 
Since $ \mathrm{pr}(\overline{Q}_i) \smallsetminus \mathrm{pr}(Q_i) \cong \overline{Q}_i \smallsetminus Q_i $, which is a complete space, we deduce that $x \in \mathrm{pr}(\overline{Q}_i) \smallsetminus \mathrm{pr}(Q_i) \subseteq \mathcal{Q}$.
Thus, we conclude that $\mathcal{Z}$ is closed. 
\end{proof}

\begin{remark}

Note that, in Theorem~\ref{theorem quotient and cuspidal rational sector chambers}, the subset $\mathcal{Y}' \subset \mathcal{X}_k$ that we considered to build $\mathcal{Z}'$ and $\mathcal{Z}$ is not necessarily connected, even if so is $\mathcal{Z}$.

Furthermore, the connected set $\mathcal{Z}$ contains subsets that can be lifted in non-zero sector faces: it is far to deserve finite or compact properties.
In Theorem~\ref{main teo 2 new}, despite it is not stated, one can expect that $\mathcal{Y}$ is, in fact, finite whenever $\mathbb{F}$ is finite and, up to replacing the $Q_{i,\Theta}$ for $i \in \mathrm{I}_\Theta$ by $Q_{j,\Theta} \times C_j$ for some well-chosen finite CW-complexes $C_j$, one can expect to reduce the set of parameters $i\in \mathrm{I}_\Theta$ by some finite set of parameters $j \in \mathrm{J}_\Theta$.

\end{remark}

\section{Number of orbits of cuspidal rational sector faces}\label{sec number cusp faces}

In the previous section, we introduce an equivalence relation on the set of cuspidal rational sector chambers (c.f.~Equation~\eqref{eq rel on cup sec cham}).
We also fix a set $C_\emptyset$ of representatives for this equivalence relation.
In this section, we describe $C_{\emptyset}$ thanks to the \'etale cohomology.
Indeed, we show that $C_{\emptyset}$ is parametrized by the double coset $\mathbf{G}(A) \backslash \mathbf{G}(k) /\mathbf{B}(k)$, which is in bijection with $\mathrm{Pic}(A)^\mathbf{t},$ when $\mathbf{G}_k$ is semisimple and simply connected.
In order to establish this correspondence, we start by a general approach, studying the family of double cosets of the form $\mathbf{G}(A) \backslash \mathbf{G}(k) /\mathbf{P}_{\Theta}(k)$, where $\mathbf{P}_{\Theta}$ is the standard parabolic subgroup defined from $\Theta \subset \Delta$.

\subsection{Double cosets of the form \texorpdfstring{$\mathbf{G}(A) \backslash \mathbf{G}(k) /\mathbf{P}_{\Theta}(k)$}{X}.}

Recall that, according to \cite[Exp.~XXVI, Prop. 1.6]{SGA3-3}, any standard parabolic subgroup $\mathbf{P}_\Theta$ admits a Levi decomposition, i.e.~it can be written as the semi-direct product of a reductive group $\mathbf{L}_\Theta$ and its unipotent radical $\mathbf{U}_\Theta$.

\begin{proposition}\label{prop number of cusps equals some kernell}
Let $h : H^1_{\text{\'et}}\big(\mathrm{Spec}(A),\mathbf{L}_{\Theta}\big) \to H^1_{\text{\'et}}\big(\mathrm{Spec}(A),\mathbf{G}\big)$ be the map defined from the natural exact sequence $1 \to \mathbf{L}_{\Theta} \to \mathbf{G} \to \mathbf{G}/\mathbf{L}_{\Theta} \to 1$. There exists a bijective map from $\mathbf{G}(A) \backslash \mathbf{G}(k) /\mathbf{P}_{\Theta}(k)$ to $\mathrm{ker}(h)$.
\end{proposition}

\begin{proof}
Firstly, we show that $(\mathbf{G}/\mathbf{P}_{\Theta})(A)= (\mathbf{G}/\mathbf{P}_{\Theta})(k)$ by using a classical argument via the valuative criterion for properness and patching (c.f.~\cite[Prop. 1.6, \S 4]{Liu})
Set $x \in (\mathbf{G}/\mathbf{P}_{\Theta})(k)$.
By the valuative criterion of properness, for any prime ideal $\p'$ of $A$ there exists a unique element $x_{\p'} \in (\mathbf{G}/\mathbf{P}_{\Theta})(A_{\p'})$ such that $x=\operatorname{Res}(x_{\p'})$.
Since the functor of points $\mathfrak{h}_V$ of $V=\mathbf{G}/\mathbf{P}_{\Theta}$ is faithfully flat, we can suppose that $x_{\p'} \in V(A_{f_{\p'}})$ where $f_{\p'} \notin \p'$. Moreover, $\operatorname{Spec}(A)$ can be covered by a finite set $\lbrace\operatorname{Spec}(A_{f_{\p'_i}})\rbrace_{i=1}^n$.
Hence, by a patching argument, we find $\underline{x} \in V(A)$ such that $x_{\p'_i}=\operatorname{Res}(\overline{x})$, and then $x=\operatorname{Res}(\overline{x})$.
This element is unique by local considerations. Thus, we conclude the claim.

Now, let us consider the exact sequence of algebraic varieties
$$
1 \rightarrow \mathbf{P}_{\Theta} \xrightarrow{\iota} \mathbf{G} \xrightarrow{p} \mathbf{G}/\mathbf{P}_{\Theta} \rightarrow 1.
$$
It follows from \cite[\S 4, 4.6]{DG} that there exists a long exact sequence
\[
1 \to \mathbf{P}_{\Theta}(A) \rightarrow \mathbf{G}(A) \rightarrow (\mathbf{G}/\mathbf{P}_{\Theta})(A) \rightarrow H^1_{\text{\'et}}\big(\operatorname{Spec}(A), \mathbf{P}_{\Theta}\big) \xrightarrow{} H^1_{\text{\'et}}\big(\operatorname{Spec}(A),\mathbf{G}\big).
\]
Moreover, it follows from \cite[Exp.~XXVI, Cor. 2.3]{SGA3-3} that \[H^1_{\text{fppf}}\big(\operatorname{Spec}(A), \mathbf{P}_{\Theta}\big)= H^1_{\text{fppf}}\big(\operatorname{Spec}(A), \mathbf{L}_{\Theta}\big).\]
But, since $\mathbf{P}_{\Theta}$ and $\mathbf{L}_{\Theta}$ are both smooth over $\mathrm{Spec}(A)$, we have that \[H^1_{\text{fppf}}\big(\operatorname{Spec}(A),\mathbf{P}_{\Theta}\big) = H^1_{\text{\'et}}\big(\operatorname{Spec}(A), \mathbf{P}_{\Theta}\big), H^1_{\text{fppf}}\big(\operatorname{Spec}(A), \mathbf{L}_{\Theta}\big)=H^1_{\text{\'et}}\big(\operatorname{Spec}(A), \mathbf{L}_{\Theta}\big).\]
This implies that $H^1_{\text{\'et}}\big(\operatorname{Spec}(A), \mathbf{P}_{\Theta}\big)= H^1_{\text{\'et}}\big(\operatorname{Spec}(A), \mathbf{L}_{\Theta}\big)$. Thus, there exists a long exact sequence
$$
1 \rightarrow \mathbf{P}_{\Theta}(A) \rightarrow \mathbf{G}(A) \rightarrow (\mathbf{G}/\mathbf{P}_{\Theta})(k) \rightarrow H^1_{\text{\'et}}\big(\operatorname{Spec}(A), \mathbf{L}_{\Theta}\big) \rightarrow H^1_{\text{\'et}}\big(\operatorname{Spec}(A),\mathbf{G}\big).
$$
This implies that
$$\ker \left( H^1_{\text{\'et}}(\operatorname{Spec}(A), \mathbf{L}_{\Theta}) \to H^1_{\text{\'et}}(\operatorname{Spec}(A),\mathbf{G})\right) \cong \mathbf{G}(A) \backslash (\mathbf{G}/\mathbf{P}_{\Theta})(k).$$
According to~\cite[4.13(a)]{BoTi}, $(\mathbf{G}/\mathbf{P}_{\Theta})(k) = \mathbf{G}(k)/\mathbf{P}_{\Theta}(k)$, whence the result follows.
\end{proof}

\begin{corollary}\label{corollary cusps number in terms of pic}
Assume that the split reductive $k$-group $\mathbf{G}_k$ is semisimple and simply connected.
Denote by $\mathfrak{t}$ its semisimple rank which is the dimension of its maximal split torus $\mathbf{T}$.
Then, there is a one-to-one correspondence between $\mathbf{G}(A) \backslash \mathbf{G}(k) /\mathbf{B}(k)$ and $\operatorname{Pic}(A)^\mathfrak{t}$.
\end{corollary}

\begin{proof}

Note that, when $\Theta=\emptyset$, we have $\mathbf{P}_{\Theta}=\mathbf{B}$, whence $\mathbf{L}_{\Theta}=\mathbf{T}$.
Consider the map $h_0 : H^1_{\text{\'et}}\big(\mathrm{Spec}(A),\mathbf{T}\big) \to H^1_{\text{\'et}}\big(\mathrm{Spec}(A),\mathbf{G}\big)$ defined from the exact sequence of $k$-varieties $1 \to \mathbf{T} \to \mathbf{G} \to \mathbf{G}/\mathbf{T} \to 1$.
Then, it follows from Proposition~\ref{prop number of cusps equals some kernell} that there is a one-to-one correspondence between $\mathbf{G}(A) \backslash \mathbf{G}(k) /\mathbf{B}(k)$ and $\mathrm{ker}(h_0)$.

Since $\mathbf{T}$ is split over $\mathbb{Z}$, we have that $\mathbf{T} \cong \mathbb{G}_{m,\mathbb{Z}}^{\mathbf{t}}$. It follows from Hilbert's Theorem~90 (c.f.~\cite[Ch. III, Prop. 4.9]{Mi}) that 
\[H^1_{\text{Zar}}\big(\operatorname{Spec}(A), \mathbb{G}_m\big) = H^1_{\text{\'et}}\big(\operatorname{Spec}(A), \mathbb{G}_m\big) \cong \operatorname{Pic}(A).\]
Thus, we get that
$$H^1_{\text{Zar}}\big(\operatorname{Spec}(A), \mathbf{T}\big)= H^1_{\text{\'et}}\big(\operatorname{Spec}(A), \mathbf{T}\big) \cong \operatorname{Pic}(A)^\mathbf{t}.$$
Since $h_0\big(H^1_{\text{Zar}}(\operatorname{Spec}(A), \mathbf{T})\big) \subseteq H^1_{\text{Zar}}\big(\operatorname{Spec}(A), \mathbf{G}\big)$ and $H^1_{\text{Zar}}\big(\operatorname{Spec}(A), \mathbf{T}\big)$ equals $ H^1_{\text{\'et}}\big(\operatorname{Spec}(A), \mathbf{T}\big)$, we get that
\[
\mathrm{ker}(h_0) 
=\ker \left( H^1_{\text{Zar}}\big(\operatorname{Spec}(A), \mathbf{T}\big)  \to H^1_{\text{Zar}}\big(\operatorname{Spec}(A),\mathbf{G}\big)\right).
\]
Moreover, since $\mathbf{G}_k$ is a simply connected semisimple $k$-group scheme, according to~\cite[Th. 2.2.1 and Cor. 2.3.2]{H1}, $H^1_{\text{Zar}}(\operatorname{Spec}(A),\mathbf{G})$ is trivial, for any integral model $\mathbf{G}$ of $\mathbf{G}_k$.
We conclude that
$$\ker(h_0) = H^1_{\text{Zar}}\big(\operatorname{Spec}(A),\mathbf{T}\big) \cong  \operatorname{Pic}(A)^\mathbf{t},$$
whence the result follows.
\end{proof}

\begin{corollary}\label{coro finite kernell in cohomology}
Assume that $\mathbb{F}$ is finite. Then, for each $\Theta \subset \Delta$, the kernel of the map $h_{\Theta}: H^1_{\text{\'et}}(\mathrm{Spec}(A), \mathbf{L}_{\Theta}) \to H^1_{\text{\'et}}(\mathrm{Spec}(A),\mathbf{G})$ is finite.
\end{corollary}

\begin{proof}
Assume that $\mathbb{F}$ is finite. Then, it follows from Weyl Theorem (c.f.~\cite[Ch. II, \S 2.2]{S}) that $\mathrm{Pic}(A)$ is finite. Hence, Corollary~\ref{corollary cusps number in terms of pic} shows that $\mathbf{G}(A) \backslash \mathbf{G}(k) /\mathbf{B}(k)$ is finite. Moreover, since, for each $\Theta \subset \Delta$, we have $\mathbf{B}(k) \subseteq \mathbf{P}_{\Theta}(k)$, we deduce that $\mathbf{G}(A) \backslash \mathbf{G}(k) /\mathbf{P}_{\Theta}(k)$ is finite. Thus, the result follows from Proposition~\ref{prop number of cusps equals some kernell}.
\end{proof}

\subsection{Number of non-equivalent cuspidal rational sector chambers}

In this section we count the number of orbits of cuspidal rational sector chambers defined by Equation~\eqref{eq rel on cup sec cham}. In other words, we count the cardinality of $C_{\emptyset}$, or equivalently the cardinality of $C$ given by Theorem~\ref{theorem quotient and cuspidal rational sector chambers}, or $\mathrm{I}_{\emptyset}$ given by Theorem~\ref{main teo 2 new}.
 
\begin{proposition}\label{lemma number of cusps}
 The set $\partial_\infty C_\emptyset := \{ \partial_\infty Q,\ Q \in C_\emptyset \}$ is a set of representatives of the $\mathbf{G}(A)$-orbits of the set of chambers of the spherical building at infinity $\partial_{\infty}(\mathcal{X}_k)$.
In particular, there exists a one-to-one correspondence between $C_{\emptyset}$ and the double coset $ \mathbf{G}(A) \backslash  \mathbf{G}(k)/\mathbf{B}(k)$.
\end{proposition}

\begin{proof}
Let $C_{\emptyset}^1$ be the set of cuspidal rational sector chambers in $\mathcal{X}_k$. By definition, there are one-to-one correspondence between the sets $C_{\emptyset}$, $C_{\emptyset}^1/\sim_{\mathbf{G}(A)}$ and $\partial_\infty C_\emptyset$
(c.f.~Equation~\eqref{eq rel on cup sec cham}).
It follows from Proposition~\ref{prop starAction} and Corollary~\ref{cor image of sector faces} that given an arbitrary $k$-sector chamber, it contains a subsector chamber that is a cuspidal rational sector chamber. This implies that $\partial_{\infty}(C_{\emptyset}^1)$ covers all the chambers in $\partial_{\infty}(\mathcal{X}_k)$, whence the first statement follows.

Now, since Lemma~\ref{visual limit lemma} shows that the set of chambers $\partial_\infty \Sec(\mathcal{X}_k)$ in $\partial_{\infty}(\mathcal{X}_k)$ is in one-to-one correspondence with $\mathbf{G}(k) / \mathbf{B}(k)$, we conclude that there exists a bijection between $C_{\emptyset}$ and the double coset $ \mathbf{G}(A) \backslash  \mathbf{G}(k)/\mathbf{B}(k)$.
\end{proof}

\begin{corollary}\label{cor cusps and H1}
There exists a bijective map between $C_{\emptyset}$ and the kernel of the map $h: H^1_{\text{\'et}}(\mathrm{Spec}(A), \mathbf{T}) \to H^1_{\text{\'et}}(\mathrm{Spec}(A),\mathbf{G})$.
\end{corollary}

\begin{proof}
It is an immediate consequence of Proposition~\ref{lemma number of cusps} and Proposition~\ref{prop number of cusps equals some kernell}.
\end{proof}

\begin{corollary}\label{coro cusps and pic}
Assume that $\mathbf{G}_k$ is semisimple and simply connected.
\begin{enumerate}
    \item There exists a one-to-one correspondence between the set $C_{\emptyset}$ and $\mathrm{Pic}(A)^{\mathbf{t}}$.
    \item In particular, $C_{\emptyset}$ is finite whenever $\mathbb{F}$ is a finite field.
\end{enumerate}
\end{corollary}

\begin{proof}
The first statement directly follows from Corollary~\ref{corollary cusps number in terms of pic} and Corollary~\ref{cor cusps and H1}.
If $\mathbb{F}$ is finite, then $\operatorname{Pic}(A)$ is finite according to Weyl Theorem (c.f.~\cite[Ch. II, \S 2.2]{S}).
Therefore, the number $\operatorname{Card}(C_\emptyset) = \operatorname{Card}\big( \operatorname{Pic}(A) \big)^\mathrm{t}$ of cuspidal rational sector chambers is finite.
\end{proof}

\begin{remark}
Note that, when $\Theta \neq \emptyset$, the double coset $\mathbf{G}(A) \backslash \mathbf{G}(k)/\mathbf{P}_{\Theta}(k)$ does not count the number of sector faces $Q_{i,\Theta}$ in Theorem~\ref{main teo 2 new}.
Indeed, in rank $\mathbf{t} \geqslant 2$, there always exists infinitely many sector faces (that are not sector chambers) whose the points are pairwise non-equivalent but with the same visual boundary (for instance a family of sector faces contained in a given cuspidal sector chamber).
This shows that, using the cohomological approach via double cosets, the unique context where we can count the non-equivalent cuspidal rational sector faces of a given type $\Theta$ is for sector chambers, which involves the use of the minimal parabolic subgroups, i.e.~Borel subgroups.
\end{remark}

\section*{Acknowledgements}
The authors thank P.~Gille for valuable comments on \'Etale cohomology. 
The authors thank G.~Lucchini~Arteche and L.~Arenas-Carmona for useful discussions.

The first author was partially supported by Anid-Conicyt by the Doctoral fellowship No $21180544$.
The second author was partially supported by the GeoLie project (ANR-15-CE40-0012, The French National Research Agency).


\bibliographystyle{amsalpha}
\bibliography{refs.bib}

\providecommand{\bysame}{\leavevmode\hbox to3em{\hrulefill}\thinspace}
\providecommand{\MR}{\relax\ifhmode\unskip\space\fi MR }
\providecommand{\MRhref}[2]{%
  \href{http://www.ams.org/mathscinet-getitem?mr=#1}{#2}
}
\providecommand{\href}[2]{#2}
\begin{thebibliography}{{Sou}79}

\bibitem[AB08]{Brown}
Peter Abramenko and Kenneth~S. Brown, \emph{Buildings: Theory and
  applications}, 1 ed., Graduate Texts in Mathematics, vol. 248, Springer, New
  York, 2008, Theory and applications.

\bibitem[BKW13]{B}
Kai-Uwe Bux, Ralf K{\"o}hl, and Stefan Witzel, \emph{Higher finiteness
  properties of reductive arithmetic groups in positive characteristic: the
  rank theorem.}, Ann. Math. (2) \textbf{177} (2013), no.~1, 311--366
  (English).

\bibitem[BL22]{BravoLoisel}
Claudio Bravo and Benoit Loisel, \emph{On chevalley group schemes over function
  fields: quotients of the bruhat-tits building by $\{\wp\}$-arithmetic
  subgroups}, available at \url{https://arxiv.org/abs/2207.06546}, 2022.

\bibitem[Bor91]{BoA}
Armand Borel, \emph{Linear algebraic groups}, second ed., Graduate Texts in
  Mathematics, vol. 126, Springer-Verlag, New York, 1991.

\bibitem[Bou81]{Bourbaki}
Nicolas Bourbaki, \emph{\'{E}l\'ements de math\'ematique}, Masson, Paris, 1981,
  Groupes et alg{\`e}bres de Lie. Chapitres 4, 5 et 6.

\bibitem[BT65]{BoTi}
Armand Borel and Jacques Tits, \emph{Groupes r\'eductifs}, Institut des Hautes
  \'Etudes Scientifiques. Publications Math\'ematiques \textbf{27} (1965),
  55--150.

\bibitem[BT72]{BT}
François Bruhat and Jacques Tits, \emph{Groupes r\'eductifs sur un corps
  local}, Institut des Hautes \'Etudes Scientifiques. Publications
  Math\'ematiques \textbf{41} (1972), 5--251.

\bibitem[BT84]{BT2}
\bysame, \emph{Groupes r\'eductifs sur un corps local. {II}. {S}ch\'emas en
  groupes. {E}xistence d'une donn\'ee radicielle valu\'ee}, Institut des Hautes
  \'Etudes Scientifiques. Publications Math\'ematiques \textbf{60} (1984),
  197--376.

\bibitem[DG70a]{DG}
Michel Demazure and Pierre Gabriel, \emph{Groupes alg\'ebriques. {T}ome {I} :
  {G}\'eom\'etrie alg\'ebrique, g\'en\'eralit\'es, groupes commutatifs}, vol.
  Tome 1, North-Holland Publishing Co., Amsterdam, 1970, Avec un appendice
  {\it{Corps de classes local}} par Michiel Hazewinkel.

\bibitem[DG70b]{SGA3-3}
Michel Demazure and Alexander Grothendieck (eds.), \emph{Sch{\'e}mas en
  groupes. s{\'e}minaire de g{\'e}om{\'e}trie alg{\'e}brique du bois marie
  1962-64 (sga 3)}, Lecture notes in mathematics (153), vol.~3,
  Springer-Verlag, Berlin, New York, 1970, augmented and corrected 2008–2011
  re-edition of the original by Philippe Gille and Patrick Polo. Available at
  \url{http://www.math.jussieu.fr/~polo/SGA3}.

\bibitem[DG70c]{SGA3-1}
Michel Demazure and Alexander Grothendieck (eds.), \emph{Sch{\'e}mas en
  groupes. s{\'e}minaire de g{\'e}om{\'e}trie alg{\'e}brique du bois marie
  1962-64 (sga 3)}, Lecture notes in mathematics (151), vol.~1,
  Springer-Verlag, Berlin, New York, 1970, augmented and corrected 2008–2011
  re-edition of the original by Philippe Gille and Patrick Polo. Available at
  \url{http://www.math.jussieu.fr/~polo/SGA3}.

\bibitem[Har67]{H1}
G{\"u}nter Harder, \emph{Halbeinfache {Gruppenschemata} {\"u}ber
  {Dedekindringen}}, Invent. Math. \textbf{4} (1967), 165--191 (German).

\bibitem[Har77]{H2}
\bysame, \emph{Die kohomologie s-arithmetischer gruppen über
  funktionenkörpern.}, Inventiones mathematicae \textbf{42} (1977), 135--176
  (ger).

\bibitem[Lan94]{Lang}
Serge Lang, \emph{Algebraic number theory.}, 2nd ed. ed., Grad. Texts Math.,
  vol. 110, New York: Springer-Verlag, 1994 (English).

\bibitem[Lan96]{L}
Erasmus Landvogt, \emph{A compactification of the {B}ruhat-{T}its building},
  Lecture Notes in Mathematics, vol. 1619, Springer-Verlag, Berlin, 1996.

\bibitem[Lan00]{Landvogt-functoriality}
\bysame, \emph{Some functorial properties of the {B}ruhat-{T}its building},
  Journal f\"ur die Reine und Angewandte Mathematik. [Crelle's Journal]
  \textbf{518} (2000), 213--241.

\bibitem[Liu02]{Liu}
Qing Liu, \emph{Algebraic geometry and arithmetic curves}, Oxf. Grad. Texts
  Math., vol.~6, Oxford: Oxford University Press, 2002 (English).

\bibitem[Mar09]{Margaux}
Benedictus Margaux, \emph{The structure of the group {{\(G(k[t])\)}}:
  variations on a theme of {Soul{\'e}}.}, Algebra Number Theory \textbf{3}
  (2009), no.~4, 393--409 (English).

\bibitem[Mas01]{M}
Alexander~W. Mason, \emph{{Serre's generalization of Nagao's theorem: an
  elementary approach}}, {Trans. Am. Math. Soc.} \textbf{353} (2001), no.~2,
  749--767 (English).

\bibitem[Mil80]{Mi}
James~S. Milne, \emph{{\'E}tale cohomology}, Princeton Math. Ser., vol.~33,
  Princeton University Press, Princeton, NJ, 1980 (English).

\bibitem[Par00]{Parreau}
Anne Parreau, \emph{Immeubles affines: construction par les normes},
  Crystallographic Groups and Their Generalizations: Workshop, Katholieke
  Universiteit Leuven Campus Kortrijk, Belgium, May 26-28, 1999, vol. 262,
  American Mathematical Soc., 2000, p.~263.

\bibitem[Rou77]{Rousseau77}
Guy Rousseau, \emph{Immeubles des groupes r\'eductifs sur les corps locaux},
  U.E.R. Math\'ematique, Universit\'e Paris XI, Orsay, 1977, Th{\`e}se de
  doctorat, Publications Math{\'e}matiques d'Orsay, No. 221-77.68.

\bibitem[Rou09]{RouEuclidean}
\bysame, \emph{Euclidean buildings}, G\'eom\'etries \`a courbure n\'egative ou
  nulle, groupes discrets et rigidit\'es, S\'eminaires et Congr\`es, vol.~18,
  Soci\'et\'e Math\'ematique de France, Paris, 2009, pp.~77--116.

\bibitem[Ser03]{S}
Jean-Pierre Serre, \emph{Trees. {Transl}. from the {French} by {John}
  {Stillwell}.}, corrected 2nd printing of the 1980 original ed., Springer
  Monogr. Math., Berlin: Springer, 2003 (English).

\bibitem[{Sou}79]{So}
Christophe {Soul\'e}, \emph{{Chevalley groups over polynomial rings}},
  {Homological group theory, Proc. Symp., Durham 1977, Lond. Math. Soc. Lect.
  Note Ser. 36, 359-367.}, 1979.

\bibitem[Spr98]{Springer}
Tonny~Albert Springer, \emph{Linear algebraic groups}, second ed., Progress in
  Mathematics, vol.~9, Birkh\"auser Boston, Inc., Boston, MA, 1998.

\bibitem[Sti09]{Stichtenoth}
Henning Stichtenoth, \emph{Algebraic function fields and codes}, 2nd ed. ed.,
  Grad. Texts Math., vol. 254, Berlin: Springer, 2009 (English).

\bibitem[{Stu}80]{Stuhler}
Ulrich {Stuhler}, \emph{{Homological properties of certain arithmetic groups in
  the function field case}}, {Invent. Math.} \textbf{57} (1980), 263--281
  (English).

\end{thebibliography}

\end{document}